\documentclass[preprint,3p]{amsart}

\usepackage{amsmath}
\usepackage{amsfonts}
\usepackage{amsthm}
\usepackage{amssymb, textcomp}
\usepackage{url}
\usepackage[usenames]{color}
\usepackage{xcolor}
\usepackage{stmaryrd}
\usepackage{hyperref}
\usepackage{verbatim}

\usepackage{pgf,tikz}
\usetikzlibrary{shapes.geometric, arrows}

\tikzstyle{rect} = [rectangle, minimum width=3cm, minimum height=1cm, text centered, draw=black]
\tikzstyle{arrow} = [thick,->,>=stealth]

\usepackage{float}  

\newtheorem{lemma}{Lemma}[section]

\newtheorem{remark}[lemma]{Remark}
\newtheorem{question}[lemma]{Question}

\newtheorem{notation}[lemma]{Notation}
\newtheorem{proposition}[lemma]{Proposition}
\newtheorem{theorem}[lemma]{Theorem}

\def\be{\begin{eqnarray}}
\def\ee{\end{eqnarray}}
\def\beal{\begin{aligned}}
\def\enal{\end{aligned}}
\newcommand{\eps}{\varepsilon}
\newcommand{\ga}{\gamma}

\newcommand{\Ga}{\Gamma}
\newcommand{\al}{\alpha}
\newcommand{\dps}{\displaystyle}
\newcommand{\RR}{\mathbb{R}}
\newcommand{\NN}{\mathbb{N}}
\newcommand{\CC}{\mathbb{C}}
\newcommand{\TT}{\mathbb{T}}

\newcommand{\ZZ}{\mathbb{Z}}
\newcommand{\MM}{\mathcal{M}}

\newcommand{\PP}{\mathcal{P}}
\newcommand{\GG}{\mathcal{G}}

\newcommand{\BB}{\mathcal{B}}
\newcommand{\LL}{\mathcal{L}}
\newcommand{\QQ}{\mathcal{Q}}
\newcommand{\YY}{\mathcal{Y}}
\newcommand{\KK}{\mathcal{K}}

\newcommand{\SSS}{\mathcal{S}}
\newcommand{\DD}{\mathbb{D}}
\newcommand{\XX}{\mathcal{X}}
\newcommand{\OO}{\mathcal{O}}
\newcommand{\FF}{\mathcal{F}}
\newcommand{\HH}{\mathcal{H}}

\newcommand{\RRR}{\mathcal{R}}
\newcommand{\TTT}{\mathcal{T}}
\newcommand{\UU}{\mathcal{U}}
\newcommand{\NNN}{\mathcal{N}}
\newcommand{\EE}{\mathcal{E}}
\newcommand{\JJ}{\mathcal{J}}
\newcommand{\AAA}{\mathcal{A}}

\newcommand{\tZ}{\mathtt{Z}}
\newcommand{\tL}{\mathtt{L}}
\newcommand{\CCC}{\mathcal{C}}
\newcommand{\QQQ}{\mathcal{Q}}
\newcommand{\Id}{\mathrm{Id}}

\newcommand{\lln}{\llfloor}
\newcommand{\rrn}{\rrfloor}
\newcommand{\ii}{^{-1}}
\newcommand{\de}{\delta}
\newcommand{\pa}{\partial}
\newcommand{\la}{\lambda}

\newcommand{\kk}{\kappa}
\newcommand{\rr}{\rho}

\newcommand{\tet}{\theta}
\newcommand{\ol}{\overline}
\newcommand{\Tet}{\Theta}
\newcommand{\tte}{\tilde{\Theta}}
\newcommand{\tth}{\tilde{\Theta}}
\newcommand{\La}{\Lambda}
\newcommand{\bet}{\beta}

\newcommand{\lo}{\gamma}

\newcommand{\e}{e_c}

\newcommand{\vect}{\mathrm{vec}}

\newcommand{\s}{s}
\renewcommand{\Re}{\mathrm{Re\, }}
\renewcommand{\Im}{\mathrm{Im\,}}

\newcommand{\wt}{\widetilde}
\newcommand{\wh}{\widehat}
\newcommand{\Lip}{\mathrm{Lip}\,}
\newcommand{\El}{\mathrm{El}}
\newcommand{\Par}{\mathrm{Par}}

\newcommand{\loc}{\mathrm{loc}}

\newcommand{\h}{\mathrm{h}}
\newcommand{\het}{\mathrm{het}}
\newcommand{\ex}{\mathrm{e}}
\newcommand{\matrx}{\mathrm{mat}}

\newcommand{\alo}{\eta_0}
\newcommand{\beto}{\xi_0}
\newcommand{\vm}{{\sigma}}
\newcommand{\fl}{\mathrm{flow}}
\newcommand{\ovr}{\mathrm{ovr}}
\newcommand{\tY}{\mathtt{Y}}
\newcommand{\tU}{\mathtt{U}}
\newcommand{\tLa}{\mathtt{\Lambda}}
\newcommand{\tGa}{\mathtt{\Gamma}}
\newcommand{\tA}{\mathtt{A}}
\newcommand{\tB}{\mathtt{B}}
\newcommand{\tkk}{\tilde{\kappa}}
\newcommand{\tS}{\mathtt{S}}

\newcommand{\Psil}{\Psi_{\textrm{loc}}}
\newcommand{\Psiloc}[2]{\Psi_{\textrm{loc},#1,#2}}
\newcommand{\wtPsiloc}[2]{\wt \Psi_{\textrm{loc},#1,#2}}
\newcommand{\Psig}[1]{\Psi_{\textrm{glob},#1}}
\newcommand{\wtPsig}[1]{\wt \Psi_{\textrm{glob},#1}}
\newcommand{\M}{\mathcal{M}}

\newcommand{\fH}{{\rm \bf H1}}
\newcommand{\sH}{{\rm \bf H2}}
\newcommand{\etabad}{\eta_{\mathrm{bad}}}

\title{Hyperbolic dynamics and oscillatory motions in the 3 Body Problem}

\author[M. Guardia]{Marcel Guardia}
\address[MG]{ Departament de Matem\`atiques, Universitat Polit\`ecnica de Catalunya, Diagonal 647, 08028 Barcelona, Spain \& Centre de Recerca Matemàtica, Edifici C, Campus Bellaterra, 08193 Bellaterra, Spain}
\email{marcel.guardia@upc.edu}

\author[P. Mart\'in]{Pau Mart\'in}
\address[PM]{ Departament de Matem\`atiques, Universitat Polit\`ecnica de Catalunya, Diagonal 647, 08028 Barcelona, Spain \& Centre de Recerca Matemàtica, Edifici C, Campus Bellaterra, 08193 Bellaterra, Spain}
\email{p.martin@upc.edu}

\author[J. Paradela]{Jaime Paradela}
\address[JP]{ Departament de Matem\`atiques, Universitat Polit\`ecnica de Catalunya, Diagonal 647, 08028 Barcelona, Spain}
\email{jaime.paradela@upc.edu}

\author[T. M. Seara]{Tere M. Seara}
\address[TS]{Departament de Matem\`atiques, Universitat Polit\`ecnica de Catalunya, Diagonal 647, 08028 Barcelona, Spain \& Centre de Recerca Matemàtica, Edifici C, Campus Bellaterra, 08193 Bellaterra, Spain}
\email{tere.m-seara@upc.edu }

\begin{document}
\maketitle

\begin{abstract}
Consider the planar 3 Body Problem with masses $m_0,m_1,m_2>0$. In this paper we address two fundamental questions: the existence of oscillatory motions and of chaotic hyperbolic sets.

In 1922, Chazy classified the possible final motions of the three bodies, that is the behaviors the bodies may have when time tends to infinity. One of the possible behaviors are oscillatory motions, that is, solutions of the 3 Body Problem such that the positions of the bodies $q_0, q_1, q_2$ satisfy
\[
\liminf_{t\to\pm\infty}\sup_{i,j=0,1,2, i\neq j}\|q_i-q_j\|<+\infty \quad \text{ and }\quad \limsup_{t\to\pm\infty}\sup_{i,j=0,1,2, i\neq j}\|q_i-q_j\|=+\infty.
\]
Assume that all three masses $m_0,m_1,m_2>0$  are not equal. Then, we prove that such motions exists. We also prove  that one can construct solutions of the three body problem whose forward and backward final motions are of different type.

This result relies on constructing  invariant sets whose dynamics is conjugated to the (infinite symbols) Bernouilli shift. These sets are hyperbolic for the symplectically reduced planar 3 Body Problem. As  a consequence, we obtain the existence of chaotic motions, an infinite number of periodic orbits and  positive topological entropy for the 3 Body Problem.
\end{abstract}

\tableofcontents

\section{Introduction}
The 3 Body Problem models the motion of three punctual bodies $q_0, q_1, q_2$ of masses $m_0,m_1,m_2>0$ under the Newtonian gravitational force. In suitable units, it is given by the equations
\begin{equation}\label{eq:Newton}
\begin{split}
\ddot q_0&=m_1\frac{q_1-q_0}{\|q_1-q_0\|^3} +m_2\frac{q_2-q_0}{\|q_2-q_0\|^3}\\
\ddot q_1&=m_0\frac{q_0-q_1}{\|q_0-q_1\|^3} +m_2\frac{q_2-q_1}{\|q_2-q_1\|^3}\\
\ddot q_2&=m_0\frac{q_0-q_2}{\|q_0-q_2\|^3} +m_1\frac{q_1-q_2}{\|q_1-q_2\|^3}.
 \end{split}
\end{equation}
In this paper we want to address two fundamental questions for this classical model: The analysis of the possible \emph{Final Motions} and the existence of \emph{chaotic motions (symbolic dynamics)}. These questions go back to the first half of the XX century.
\vspace*{0.5cm}
\paragraph{\textbf{Final motions:}} We call final motions to the possible qualitative behaviors that the complete (i.e. defined for all time) trajectories of the 3 Body Problem may possess as time tends to infinity (forward or backward). The analysis of final motions was proposed by Chazy \cite{Chazy22}, who proved that the final motions of the 3 Body Problem should fall into one of the following categories. To describe them, we denote by $r_k$ the vector from the point mass $m_i$ to the point mass $m_j$ for $i\neq k$, $j\neq k$, $i<j$.

\begin{theorem}[Chazy, 1922, see also \cite{ArnoldKN88}]\label{thm:chazy}
 Every solution of the 3 Body Problem defined for all (future) time belongs to one of the following seven classes.
\begin{itemize}
\item Hyperbolic (\textit{H}): $|r_i|\to\infty$, $|\dot r_i|\to c_i>0$ as $t\to\infty$.
\item Hyperbolic--Parabolic (\textit{HP$_k$}): $|r_i|\to\infty$, $i=1,2,3$, $|\dot r_k|\to 0$, $|\dot r_i|\to c_i>0$, $i\neq k$, as $t\to\infty$.
\item Hyperbolic--Elliptic,  (\textit{HE$_k$}): $|r_i|\to\infty$, $|\dot r_i|\to c_i>0$, $i\neq k$, as $t\to\infty$, $\sup_{t\geq t_0}|r_k|<\infty$.
\item Parabolic-Elliptic (\textit{PE$_k$}): $|r_i|\to\infty$, $|\dot r_i|\to 0$, $i\neq k$, as $t\to\infty$, $\sup_{t\geq t_0}|r_k|<\infty$.
\item Parabolic (\textit{P}): $|r_i|\to\infty$, $|\dot r_i|\to 0$ as $t\to\infty$.
\item Bounded (\textit{B}):  $\sup_{t\geq t_0}|r_i|<\infty$.
\item Oscillatory (\textit{OS}):  $\limsup_{t\to\infty}\sup_{i}|r_i|=\infty$ and $\liminf_{t\to\infty}\sup_{i}|r_i|<\infty$.
\end{itemize}
\end{theorem}
Note that this classification applies both when $t\to+\infty$ or $t\to-\infty$. To distinguish both cases we add a superindex $+$ or $-$ to each of the cases, e.g $H^+$ and $H^-$.

At the time of Chazy all types of motions were known to exist except the oscillatory motions
\footnote{Indeed, note that in the  limit $m_1,m_2\to 0$, where the model becomes two uncoupled Kepler problems, all final motions are possible except \textit{OS}$^\pm$.}
. Their existence was proven later by Sitnikov \cite{Sitnikov60} for the Restricted 3 Body Problem and by Alekseev \cite{Alekseev68} for the (full) 3 Body Problem for some choices of the masses. After these seminal works, the study of oscillatory motions have drawn considerable attention  (see Section \ref{sec:literature} below) but all results apply under non-generic assumptions on the masses. 

Another question posed by Chazy was whether the future and past final motion of any trajectory must be of the same type. This was disproved by Sitnikov and Alekseev, who showed that there exist trajectories with all possible combinations of future and past final motions (among those permitted at an energy level). 

The first result of this paper is to construct oscillatory motions for the 3 Body Problem provided $m_0\neq m_1$ and to show that all possible past and future final motions at negative energy can be combined.

Besides the question of existence of such motions, there is the question about their abundance. As is pointed out in \cite{GorodetskiK12}, V. Arnol'd, in the conference in
honor of the 70th anniversary of Alexeev, posed the following question.

\begin{question}\label{ques:measure}
Is the Lebesgue measure of the set of oscillatory motions positive?
\end{question}

Arnol'd considered it the fundamental question in Celestial Mechanics. Alexeev conjectured in \cite{Alekseyev71} that the Lebesgue measure is zero (in the English version \cite{Alekseev81} he attributes this conjecture to Kolmogorov). This conjecture remains wide open.

\vspace*{0.5cm}
\paragraph{\textbf{Symbolic dynamics:}} The question on existence of chaotic motions in the 3 Body Problem can be traced back to Poincar\'e and Birkhoff. It has been a subject of deep research during the second half of the XX century. The second goal of this paper is to construct hyperbolic invariant sets for (a suitable Poincar\'e map and after symplectically reducing for the classical first integrals) of the 3 Body Problem whose dynamics is conjugated to that of the usual shift
\begin{equation}\label{shift}
\sigma:\NN^\ZZ\to \NN^\ZZ,\qquad  (\sigma\omega)_k = \omega_{k+1},
\end{equation}
in the space of sequences, one of the paradigmatic models of chaotic dynamics. Note that these dynamics implies positive topological entropy and an infinite number of periodic orbits.

The known results on symbolic dynamics in Celestial Mechanics require rather restrictive hypohteses on the parameters of the model (in particular, the values of the masses). Moreover, all the proofs of  existence of hyperbolic sets with symbolic dynamics in Celestial Mechanics deal with very symmetric configurations which allow to reduce the 3 Body Problem to a two dimensional Poincar\'e map (see the references in Section \ref{sec:literature} below).

\subsection{Main results}
The two main results of this paper are the following. The first theorem deals with the existence of different final motions and, in particular, of oscillatory motions.

\begin{theorem}\label{thm:Main1}
 Consider the 3 Body Problem with masses $m_0,m_1,m_2>0$ such that $m_0\neq m_1$. Then,
 \[
  X^-\cap Y^+\neq \emptyset\qquad \text{ with }\qquad X,Y=\textit{OS}, \textit{B}, \textit{PE$_3$}, \textit{HE$_3$}.
 \]
\end{theorem}

Note that this theorem gives the existence of orbits which are oscillatory in the past and in the future. It also gives different combinations of past and future final motions.  Indeed,
\begin{itemize}
\item The bodies of masses $m_0$ and $m_1$ perform (approximately) circular motions. That is, $|q_0-q_1|$  is aproximately constant.
\item The third body may have radically different behaviors: oscillatory, bounded, hyperbolic or parabolic.
\end{itemize}
The motions given by Theorem \ref{thm:Main1} have negative energy. In such energy levels, only $\textit{OS}$, $\textit{B}$, $\textit{PE$_k$}$, $\textit{HE$_k$}$ are possible and therefore we can combine all types of past and future final motions.

The second main result of this paper deals with the existence of chaotic dynamics for the 3 Body Problem.
\begin{theorem}\label{thm:Main2}
 Consider the 3 Body Problem with masses $m_0,m_1,m_2>0$ such that $m_0\neq m_1$. Fix the center of mass at the origin and denote by $\Phi_t$ the corresponding flow. Then, there exists a section $\Pi$ transverse to  $\Phi_t$ such that the induced Poincar\'e map
 \[
  \PP: \UU=\mathring\UU\subset\Pi\to\Pi
 \]
satisfies the following. There exists $M\in\NN$ such that the map $\PP^M$ has an invariant set $\XX$ which is homeomorphic to $\NN^\ZZ\times\TT$. Moreover, the dynamics $\PP^M:\XX\to\XX$ is topologically conjugated to
\[
\begin{split}
 \NN^\ZZ\times\TT&\rightarrow\NN^\ZZ\times\TT\\
 (\omega,\theta)&\mapsto (\sigma\omega,\theta+f(\omega))
\end{split}
\]
where $\sigma$ is the  shift introduced in  \eqref{shift} and $f: \NN^\ZZ\to\RR$ is a continuous function.
\end{theorem}

 The set $\XX$ is a hyperbolic set once the 3 Body Problem is reduced by its classical first integrals. This invariant set implies  positive topological entropy and an infinite number of periodic orbits for the 3 Body Problem for any values of the masses (except all equal). The oscillatory motions given by Theorem \ref{thm:Main1} also belong  to this invariant set $\XX$.  In fact, Theorem \ref{thm:Main1} will be a consequence of Theorem \ref{thm:Main2}.

\subsection{Literature}\label{sec:literature}
\paragraph{\textbf{Oscillatory motions:}}
The first proof of oscillatory motions was achieved by Sitnikov in~\cite{Sitnikov60} for what is called today the \emph{Sitnikov problem}. That is the Restricted 3 Body Problem when the two primaries have equal mass, equivalently the mass ratio is  $\mu=1/2$, and perform elliptic motion whereas the third body (of zero mass) is confined to the line perpendicular to the plane defined by the two primaries and passing through their center of mass. This configuration implies that this model can be reduced to a one and a half degrees of freedom Hamiltonian system, i.e. three dimensional phase space.

Later, Moser~\cite{Moser01} gave a new proof. His approach  to prove
the existence of such motions was to consider the invariant
manifolds of infinity and to prove that they intersect transversally.
Then, he established the existence of symbolic dynamics close to these invariant
manifolds which lead to the existence of oscillatory motions. His ideas have been very influential and are the base of the present work. In Section \ref{sec:Moser} we explain this approach and as well as the challenges to apply it to the 3 Body Problem.

Since the results by Moser, there has been quite a number of works dealing with the Restricted 3 Body Problem. In the planar setting, the first one was by Sim\'o and
Llibre~\cite{SimoL80}. Following the same approach as
in~\cite{Moser01}, they proved the existence of oscillatory motions
for the RPC3BP for small enough values of the mass ratio~$\mu$ between the two primaries.
One of the main ingredients of their proof, as in~\cite{Moser01}, was the study of
 the transversality of the intersection of the invariant manifolds of
 infinity. They were able to prove this transversality provided $\mu$ is exponentially small with respect to the Jacobi constant, which was taken large enough.
Their result was extended by Xia~\cite{Xia92} using the
real-analyticity of the invariant manifolds of infinity.
The problem of existence of oscillatory motions for the RPC3BP was closed by the authors of the present paper in \cite{GuardiaMS16}, which proved the existence of oscillatory motions for any value of the mass ratio $\mu\in (0,1/2]$. These oscillatory motions possess large Jacobi constant. The authors with M. Capinski and P. Zgliczy{\'n}ski showed the existence of oscillatory motions with ``low'' Jacobi constant relying on computer assisted proofs in \cite{CapinskiGSZ21}. 
A completely
different approach using Aubry-Mather theory and semi-infinite
regions of instability has been recently developed
in~\cite{GalanteK11, GalanteK10b, GalanteK10c}.
In \cite{GuardiaPSV20}, the Moser approach is applied to the Restricted Isosceles 3 Body Problem. 

The existence of oscillatory
motions has also been proven for the (full)
3 Body Problem by \cite{Alekseev68} and~\cite{LlibreS80}. The first paper deals with the Sitnikov problem with a third positive small mass and the second one with the collinear 3 Body Problem.

A fundamental feature of the mentioned models is that they can be reduced to two dimensional area preserving maps. In particular, one can implement the Moser approach~\cite{Moser01},
that is, they relate the oscillatory motions to transversal
homoclinic points to infinity and symbolic dynamics. The works by Galante and Kaloshin do not rely on the Moser approach but still rely on two dimensional area preserving maps tools. Moreover, most of these works require rather strong assumptions on the masses of the bodies.

Results on Celestial Mechanics models of larger dimension such as the  3 Body Problem or the Restricted Planar Elliptic 3 Body Problem are much more scarce.

The authors with L. Sabbagh (see \cite{GuardiaSMS17}) proved the existence of oscillatory motions for the Restricted Planar Elliptic 3 Body Problem for any mass ratio and small eccentricity of the primaries. This work relies on ``soft techniques'' which allow to prove that $\textit{OS}^\pm\neq\emptyset$ but unfortunately do not imply that $\textit{OS}^-\cap \textit{OS}^+\neq\emptyset$ (this stronger results could be proven with the tools developed in the present paper). The same result, that is $\textit{OS}^\pm\neq\emptyset$, is obtained for a Restricted Four Body Problem in \cite{Seara20}.

The present paper is the first one which ``implements'' the ideas developed by Moser to the planar 3 body problem (see Sections \ref{sec:Moser} and \ref{sec:outline} below). Conditional results had been previously obtained in \cite{Robinson84, Robinson15}, where C. Robinson,  proved the existence of oscillatory motions under the assumption that the so-called scattering map has a hyperbolic fixed point. As far as the authors know, such assumption has not been proven yet. We follow a different approach (see Section \ref{sec:outline}).

In \cite{Moeckel07}, R. Moeckel proves the existence of oscillatory motions for the 3 Body Problem relying on passage close to triple collision, and therefore for arbitrarily small total angular momentum. This result applies to a ``big'' set of mass choices but, unfortunately, the required hypotheses are not generic.

The mentioned works deal with the problem of  existence of
oscillatory motions in different models of Celestial Mechanics. As far as the authors know, there is only one result dealing with their abundance \cite{GorodetskiK12} (recall the fundamental  Question \ref{ques:measure}). Gorodetksi and Kaloshin analyze
the Hausdorff dimension of the set of oscillatory motions for the
Sitnikov example and the RPC3BP. They prove that for both problems and a Baire
generic subset of an open set of parameters (the eccentricity of the
primaries in the Sitnikov example and the mass ratio and the Jacobi
constant in the RPC3BP) the Hausdorff dimension is maximal.

A dynamics strongly related to oscillatory motions is the Arnold diffusion behavior attached to the parabolic invariant manifolds of infinity. Such unstable behavior leads to growth in angular momentum.
This is proven in \cite{DelshamsKRS19} for the Restricted Planar Elliptic 3 Body Problem for mass ratio and  eccentricity small enough (some formal computations on the associated Melnikov function had been done previously in ~\cite{MartinezP94}).


\vspace*{0.5cm}

\paragraph{\textbf{Symbolic dynamics and Smale horseshoes for the 3 Body Problem}}
Starting from the 90's, there is a wide literature proving the conjugacy or semi-conjugacy of the dynamics of  $N$-Body Problem models with the shift \eqref{shift}. These results give the existence of symbolic dynamics for such models. Results proving the existence of hyperbolic sets with symbolic dynamics  are much more scarce and, as far as the authors know, all of them are in models which can be reduced to 2 dimensional maps. Namely, until the present paper no  hyperbolic sets with symbolic dynamics had been proven to exist for the (symplectically reduced) planar 3 Body Problem.

Note also that all the previous results dealing with symbolic dynamics in Celestial Mechanics must impose non-generic conditions on the masses. These results usually rely  on dynamics either related to infinity manifolds (as the ones mentioned above), dynamics close to collision or close to central configurations.

Concerning the Restricted 3 Body Problem, there are several papers proving the existence of hyperbolic invariant sets with symbolic dynamics. On the one hand there are the results mentioned above which construct oscillatory motions relying  on the invariant manifolds of infinity. There is a wide literature constructing symbolic dynamics (providing semiconjugacy with the shift) by means of orbits passing very close to binary collision \cite{BolotinMcK00,BolotinMcK06, Bolotin06}.

For models which can be reduced to a two dimensional Poincar\'e map (such as the Restricted Circular Planar 3 Body Problem), there are also results which rely on Computer Assisted Proofs to show the existence of transverse homoclinic points and therefore symbolic dynamics (see for instance \cite{Arioli02,WilczakZ03,Capinski12, GierzkiewiczZ19}).


On the full 3 Body Problem, as far as the authors know, most of the results proving symbolic dynamics rely on dynamics close to triple collision \cite{Moeckel89, Moeckel07}. These great results give semiconjugacy between the 3 Body Problem and the shift \eqref{shift} and apply to a ``large'' open set (but unfortunately not generic) of masses (see also \cite{RobinsonS83}). However, they do not lead either to the existence of hyperbolic  sets with symbolic dynamics.

The results on symbolic dynamics for the $N$-Body Problem with $N\geq 4$ are very scarce (see \cite{ShaneMireles19} for chaotic motions in  a Restricted 4 Body Problem). They are more abundant for the  $N$ center problem \cite{BolotinN03,Terracini12}).

\subsection{The Moser approach}\label{sec:Moser}
The proof of Theorems \ref{thm:Main1} and \ref{thm:Main2} rely on the ideas developed by J. Moser \cite{Moser01} to prove the existence of symbolic dynamics and oscillatory motions for the Sitnikov problem. Let us explain here these ideas. Later, in Section \ref{sec:outline}, we explain the challenges we have to face to apply these ideas to the 3 Body Problem.

The Sitnikov problem models two particles of equal mass (take $m_0=m_1=1/2$) performing elliptic orbits with eccentricity $\eps$ and a third body of mass 0 which is confined along the line perpendicular to the ellipses plane and passing through the center of mass of the two first bodies. This  is a Hamiltonian system of one and a half degrees of freedom defined by
\begin{equation}\label{Sitnikov}
H(p,q,t)=\frac{p^2}{2}-\frac{1}{\sqrt{q^2+R(t)}}
\end{equation}
where $R(t)$ is the distance between each of the primaries to the center of mass and satisfies
\[
 R(t)=\frac{1}{2}+\frac{\eps}{2}\cos t+\OO(\eps^2).
\]
%
For this model, J. Moser proposed the following steps to construct oscillatory motions.

\begin{itemize}
 \item[1] One can consider $P=(q,p,t)=(+\infty, 0,t)$, $t\in\TT$,  as a periodic orbit at infinity. This periodic orbit is degenerate (the linearitzation of the vector field at it is the zero matrix). Nevertheless, one can prove that  it has stable and unstable invariant manifolds \cite{McGehee73}. Note that these manifolds correspond to the parabolic-elliptic motions (see Theorem \ref{thm:chazy}).

 \item[2] One can prove that these invariant manifolds intersect transversally, leading to transverse homoclinic orbits ``to infinity''. Indeed, when $\eps=0$ the Hamiltonian \eqref{Sitnikov} has one degree of freedom and is therefore integrable. Then, the invariant manifolds coincide. For $0<\eps\ll 1$ one can apply Melnikov Theory \cite{Melnikov63} to prove their splitting.
\end{itemize}
If $P$ would be a hyperbolic periodic orbit, one could apply the classical Smale Theorem \cite{Smale65}
 to construct symbolic dynamics and inside it oscillatory motions. However, since $P$ is  degenerate  one needs a more delicate analysis than rather just applying the Smale Theorem.
 In particular, one needs the further steps.

\begin{itemize}
 \item[3] Analyze the local behavior of \eqref{Sitnikov} close to the infinity periodic orbits $P$. In hyperbolic points/periodic orbits this is encoded in the classical Lambda lemma (see for instance \cite{PalisMelo82}). In this step one needs to prove a suitable version of the Lambda lemma for degenerate (parabolic) periodic orbits.

 \item[4] From Steps 2 and 3 one can construct a 2--dimensional  return map close to the invariant manifolds of the periodic orbit $P$. The final step is to construct a sequence of ``well aligned strips''  for this return map plus cone conditions. This leads to the existence to a hyperbolic set whose dynamics is conjugated to that of the shift \eqref{shift} (a Smale horseshoe with an ``infinite number of legs'').
\end{itemize}

\section*{Acknowledgements}
This project has received funding from the European Research Council (ERC) under the European Union’s
Horizon 2020 research and innovation programme (grant agreement No 757802). T. M. S. has also been partly supported by the Spanish MINECO-
FEDER Grant PGC2018-098676-B-100 (AEI/FEDER/UE) and the Catalan grant 2017SGR1049. M. G.
and T. M. S. are supported by the Catalan Institution for Research and Advanced Studies via an ICREA
Academia Prize 2019. P.M. has been partially supported by the Spanish MINECO-
FEDER Grant PGC2018-100928-B-I00 and the Catalan grant 2017SGR1049.
This work is also supported by the Spanish State Research Agency, through the Severo Ochoa and Mar\'ia de Maeztu Program for Centers and Units of Excellence in R\&D (CEX2020-001084-M).

\section{Outline of the proof}\label{sec:outline}
To apply the Moser ideas to the 3 Body Problem is quite challenging, even more if one wants to give results for a wide choice of masses. Note here the main difficulties:
\begin{itemize}
\item After reducing by the first integrals, Sitnikov model, Alekseev model, Restricted Planar Circular 3BP are 3 dimensional flows whereas the planar 3 Body Problem is a 5 dimensional flow. This is by no means a minor change. In particular infinity goes from a periodic orbit to a two dimensional family of periodic orbits. This adds ``degenerate'' dimensions which makes considerably more difficult to build hyperbolic sets.
\item We do not assume any smallness condition on the masses. This means that one cannot apply classical Melnikov Theory to prove the transversality between the invariant manifolds. We consider a radically different perturbative regime:  we take the third body far away from the other two (usually such regime is referred to as \emph{hierarchical}). This adds multiple time scales to the problem which leads to a highly anisotropic transversality between the invariant manifolds: in some directions the transversality is exponentially small whereas in the others is power-like.
\end{itemize}
These issues make each of the steps detailed in Section \ref{sec:Moser} considerably difficult to be implemented in the 3 Body Problem. In the forthcoming sections we detail the main challenges and the novelties of our approach.

We believe that the ideas developed for each of these steps have interest beyond the results of the present paper and could be used in other physical models (certainly in Celestial Mechanics) to construct all sorts of unstable motions such as chaotic dynamics or Arnold diffusion.
\subsection{Outline of Step 0: A good choice of coordinates}\label{sec:Step0}
Before implementing the Moser approach in the Steps 1, 2, 3, and 4 below, one has to consider first a preliminary step: to choose a good system of coordinates. This is quite usual in Celestial Mechanics where typically cartesian coordinates do not capture ``well'' the dynamics of the model and one has to look for suitable coordinates.

In this case, keeping in mind that we want to construct hyperbolic sets, it is crucial that
\begin{itemize}
\item We symplectically reduce the planar 3 Body Problem by the classical first integrals.
\item We consider coordinates which capture the near integrability of the model in such a way that the first two bodies perform close to circular motion whereas the third one performs close to parabolic motion (see Figure \ref{fig:parabolacircle}).
\end{itemize}

\begin{figure}[h]
\begin{center}
\includegraphics[height=5cm]{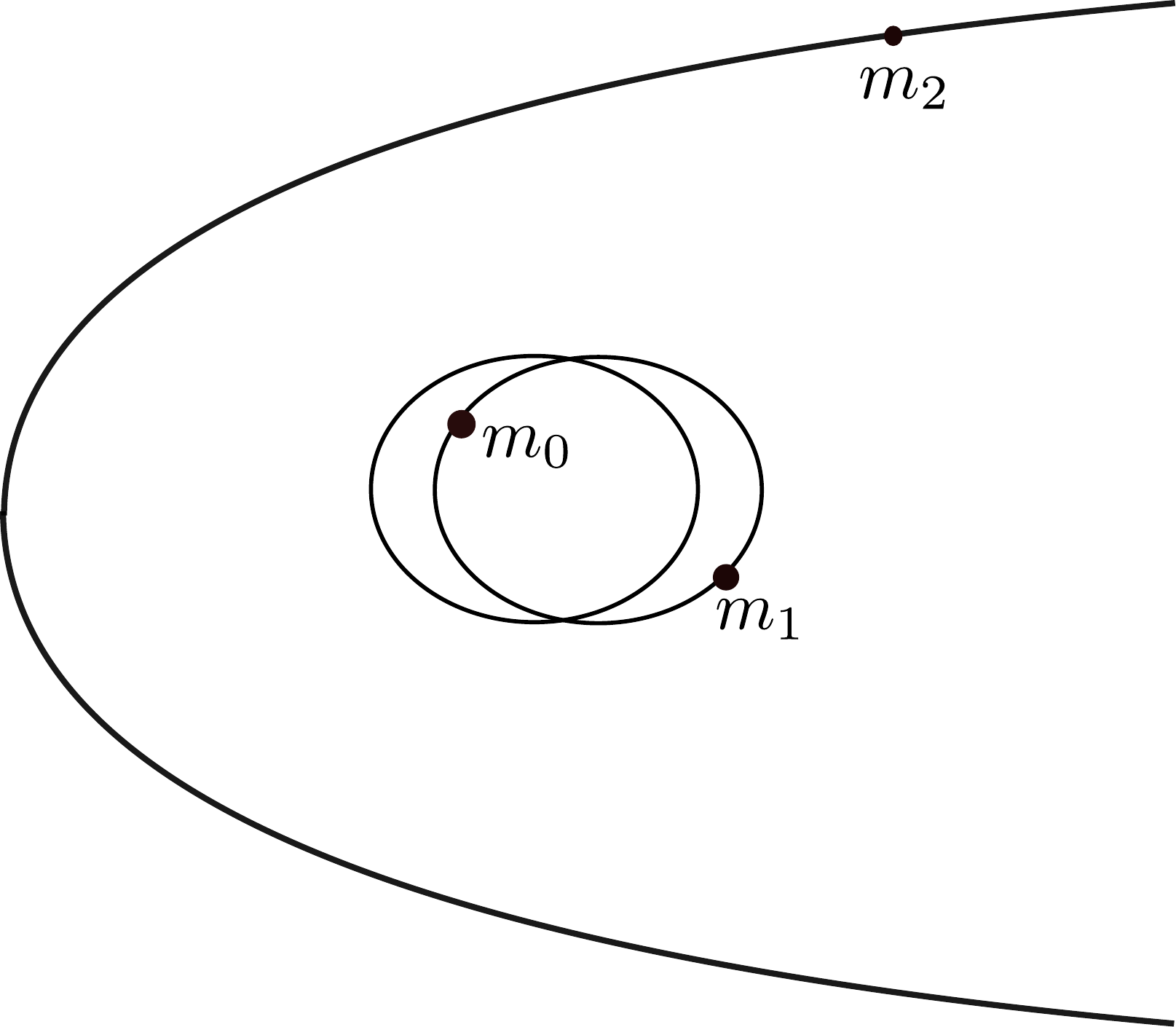}
\end{center}
\caption{We consider two bodies perform approximately circular motions while the third body is close to a  parabola, which is  arbitrarily large far from the other two bodies.}\label{fig:parabolacircle}
\end{figure}

To this end we first consider Jacobi coordinates $(Q_1,Q_2)$ as seen in Figure \ref{fig:Jacobi} and conjugate momenta $(P_1,P_2)$. This reduces the model to a 4 degrees of freedom Hamiltonian system.

Then, for the first pair $(Q_1,P_1)$, we consider the classical Poincar\'e variables $(\la,L,\eta,\xi)$ and for the second one $(Q_2,P_2)$ we consider polar coordinates $(r,y,\al,\Gamma)$ where $y$ is the radial momentum and $\Gamma$ is the angular momentum.  Finally, we ``eliminate'' the pair $(\al,\Gamma)$ by reducing the system by  rotations.

Therefore, we have finally a three degrees of freedom Hamiltonian system defined in the coordinates $(\la,L,\eta,\xi, r,y)$. In Section \ref{sec:Coordinates} we perform these changes of coordinates in full detail and give expressions for the resulting Hamiltonian.

We fix the total energy to a negative value. Following the Moser approach explained in Section \ref{sec:Moser}, we consider the ``parabolic infinity'' manifold, which is now defined by
\[
\EE_\infty = \{ r=\infty,\;  y=0\},
\]
and therefore can be parameterized by the coordinates $(\la,L,\eta,\xi)$ (we actually eliminate the variable $L$ by means of the energy conservation). More properly speaking, we consider McGehee coordinates
\[
 r=\frac{2}{x^2},
\]
so that ``infinity'' becomes $(x,y)=(0,0)$.

The dynamics at infinity is foliated by the periodic orbits $(\eta,\xi)=\text{constant}$ of the same period. The first step in our proof is to analyze the invariant manifolds of these periodic orbits and their intersections.

\subsection{Outline of Step 1: Transverse homoclinic orbits to infinity}
In suitably modified McGehee coordinates, the infinity manifold  becomes $\EE_\infty=\{(x,y)=(0,0),\, z\in U\subset\RR^2,\, t \in \TT\}$. The dynamics in a neighborhood of infinity  is given by
\[
\begin{aligned}
\dot x & = -  x^3 y(1+ \OO_2(x,y)), &  \dot { z} & =  \OO_6(x,y), \\
\dot y & = -  x^4(1+ \OO_2(x,y)), & \dot t&=1, \\
\end{aligned}
\]
for $x+y>0$, $z\in U\subset\RR^2$,  and $t \in \TT$. Note that infinity is foliated by the periodic orbits $z=\mathrm{constant}$. Thanks to \cite{BaldomaFM20a,BaldomaFM20b}, these periodic orbits have local stable and unstable invariant manifolds, which are analytic (away from infinity) and smooth with respect to parameters and to the base periodic orbit. The union of these invariant manifolds form the stable and unstable invariant manifolds of infinity, $W^s(\EE_\infty)$ and $W^u( \EE_\infty)$, which are four dimensional in a five dimensional energy level.

In order to control the globalization of these invariant manifolds, we consider a hierarchical regime in our system. We consider a configuration such that the first two bodies perform approximately circular motions whereas the third body performs approximately parabolic motion along a parabola which is taken arbitrarily large compared with the circle of the two first bodies (see Figure \ref{fig:parabolacircle}).

In other words, we choose the fixed value of the energy to be negative and of order 1, and take the total angular momentum $\Theta$ large. This choice has two consequences.  On the one hand, the motion of the third body takes place far from the first two. This  implies that the system becomes close to integrable, since, being far from the first bodies, the third one sees them almost as a single one and hence its motion is governed at a first order by a Kepler problem with zero energy --- since its motion is close to parabolic --- while the motion of the first bodies is given at first order by another Kepler problem with negative energy.
On the other hand, in this regime the system has two time scales, since the motion of the third body is $\OO(\Theta^{-3})$ slower than that of the first ones. This  implies that the coupling term between the two Kepler problems is a fast perturbation.

In the framework of averaging theory, the fact that the perturbation is fast implies that the difference between the stable and unstable invariant manifolds of infinity is typically exponentially small in $\Theta^{-3}$, which precludes the application of the standard Melnikov theory to compute the difference of these invariant manifolds. Indeed,  the perturbation can be averaged out up to any power of $\Theta^{-3}$, making the distance between the manifolds a beyond all orders quantity. We need to resort to more delicate techniques to obtain a formula of this distance which is exponentially small in $\Theta^{-3}$ (see Theorem~\ref{thm:MainSplitting}). From this formula we are able to deduce that the invariant manifolds of infinity do intersect transversally along two distinct intersections. These intersections are usually called homoclinic channels, which we denote by $\Gamma^1$ and $\Gamma^2$ (see Figure \ref{fig:canalhomoclinic}).


The fact that the perturbed invariant manifolds are exponentially close is usually referred to as  \emph{exponentially small splitting of separatrices}. This phenomenon was discovered by Poincar\'e \cite{Poincare90, Poincare99}. It was not until the 80's, with the  pioneering work  by Lazutkin for the \emph{standard map} (see~\cite{Lazutkin84russian,Lazutkin84}) that analytic tools to analyze this phenomenon were developed. Nowadays, there is quite a number of works proving the existence of transverse homoclinic orbits following the seminal ideas by Lazutkin, see for instance \cite{DelshamsS92, Gelfreich94,DelshamsS97,DelshamsGJS97, Gelfreich97,DelshamsR98,Gelfreich99,Gelfreich00,Lombardi00,GelfreichS01,BaldomaF04, GuardiaOS10,GaivaoG11, BFGS12, MartinSS11,BCS13,BCS18a,BCS18b}. Note however that most of these results deal with rather low dimensional models (typically area preserving two dimensional maps or three dimensional flows), whereas the model considered in the present paper has higher dimension (see also \cite{breathers}, which deals with an exponentially small splitting problem in infinite dimensions). The high dimension makes the analysis in the present paper considerably more intrincate.
%
%
Of special importance for the present paper are the works by Lochak, Marco and Sauzin (see \cite{LochakMS03,Sauzin01}) who analyze such phenomenon considering certain graph parameterizations of the invariant manifolds.  Other methods to deal with exponentially small splitting of separatrices are Treschev's \emph{continuous averaging}  (see \cite{Treshev97}) or ``direct'' series methods (see \cite{GGM99}).

As far as the authors know, the first paper to  prove an exponentially small splitting of separatrices in a Celestial Mechanics problem is \cite{GuardiaMS16} (see also \cite{GuardiaPSV20,BaldomaGG21a,BaldomaGG21a}).

The information in Theorem~\ref{thm:MainSplitting} will allow us to define and control two different return maps from a suitable section transverse to the unstable manifold of infinity. The combination of these return maps will give rise to chaotic dynamics. The section, four dimensional, will be close to $\EE_\infty$ and will include the time $t$ as a variable.
Each of these return maps will be, in turn, the composition of a \emph{local map}, that describes the passage close to infinity, and a \emph{global map}, following the dynamics along the global invariant manifolds. These are the subject of study of Steps~2 and 3 below.

\subsection{Outline of Step 2: The parabolic Lambda lemma and the Local map}\label{sec:Step2}
To analyze the local behavior close to infinity, we develop a \emph{parabolic Lambda lemma}. The classical Lambda lemma applies to (partially) hyperbolic invariant objects and is no longer true in the parabolic setting. The statement has to be adapted and the proof we provide has to face considerable subtleties.

The first step in proving a Lambda lemma is to perform a normal form procedure which straightens the invariant manifolds and the associated stable and unstable foliations. In the present paper, thus, we need to set up a \emph{parabolic normal form}. 
Indeed, for any fixed $N\geq 3$, we  construct local coordinates in a neighborhood of infinity in which the (symplectically reduced) 3BP  is written as
\begin{equation}\label{def:normalformintro}
\begin{aligned}
\dot q & = q ((q+p)^3+\OO_4(q,p)), & \dot{z} & = q^N p^N \OO_{4}(q,p), \\
\dot p & = -p ((q+p)^3+\OO_4(q,p)), & \dot t & = 1,
\end{aligned}
\end{equation}
where $p=q =0$ corresponds to the parabolic infinity, $\EE_\infty$. Note that in these coordinates the (local) unstable manifold of infinity is given by $p=0$ and the (local) stable manifold is $q=0$. The key point, however, is that the dynamics  on the ``center'' variables $z$ is extremely slow in a neighborhood of infinity.
This normal form is obtained in Theorem~\ref{prop:coordinatesatinfinity}.

The \emph{parabolic} Lambda Lemma is proven in these normal form variables. However, since the statement  fails at the infinity manifold, first we consider two 4-dimensional sections at a fixed but small distance of $\EE_\infty$: $\Sigma_1$ transverse to $W^{s}(\EE_\infty)$ and $\Sigma_2$ transverse to $W^{u}(\EE_\infty)$ (see Figure \ref{fig:LambdaLemma}).   We call \emph{local map} to the induced map by the flow between the sections $\Sigma_1$ and $\Sigma_2$.

\begin{figure}[h]
\begin{center}
\includegraphics[height=5cm]{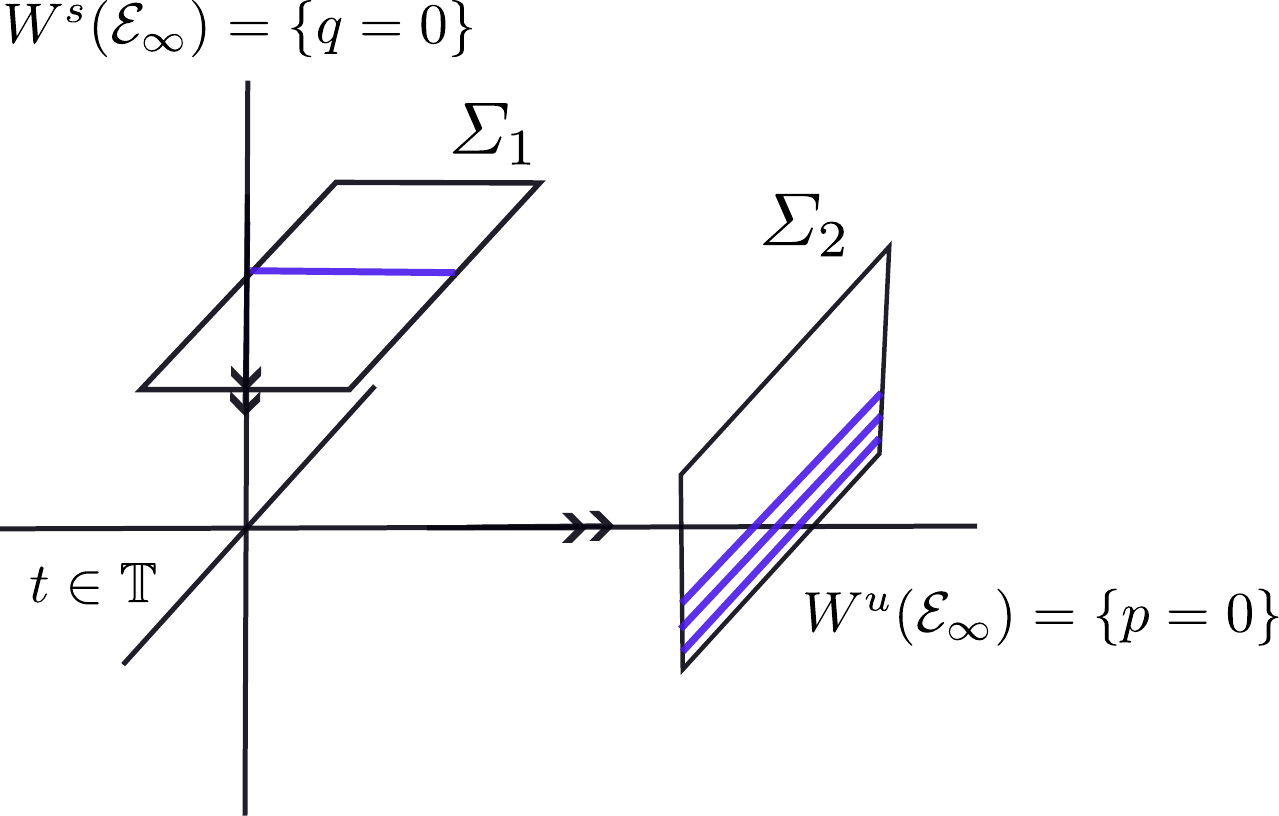}
\end{center}
\caption{Behavior of the local map from the section $\Sigma_1$ to the section $\Sigma_2$. We are omitting the dynamics of the $z$-components, which are very close to the identity.}\label{fig:LambdaLemma}
\end{figure}

The claim implies that the intersection of manifolds transverse to $W^s (\EE_\infty)$ within $\Sigma_1$ gets mapped by the local map to an immersed manifold which accumulates in a $\CCC^1$ way to  $W^{u}(\EE_\infty)\cap\Sigma_2$ (see Theorem~\ref{prop:lambdalemma}).  Furthermore, in the $z$-variables the local map is  close to the identity at the $\CCC^1$ level. As a consequence, the local map and its inverse have one and only one expanding direction.

When combining the local map with a global map
along a homoclinic chanel, this construction provides a map with a single expanding direction and a single contracting direction. This was enough for Moser since in the Sitnikov one deals with  2-dimensional sections. However,  in order to obtain a true hyperbolic object, we need   hyperbolicity in \emph{all four directions}. That is, we need to ``gain'' hyperbolicity in the $z$-directions,  whose dynamics  is given by the behavior of the global map. We will achieve hyperbolicity by combining two different global maps, related to the two different homoclinic channels obtained Step~1 (see Theorem~\ref{thm:MainSplitting}). The Lambda Lemma ensures that the dynamics in the $z$-variables induced by the travel along the homoclinic manifold is essentially preserved by the local passage.

%

%
%

\subsection{Outline of Step 3: The Scattering map and the Global maps}

A crucial tool to understand the dynamics close to the invariant manifolds to infinity is the so-called \emph{Scattering map}. The Scattering map was  introduced by Delshams, de la Llave and Seara \cite{DelshamsLS00,DelshamsLS06a,DelshamsLS08} to analyze the heteroclinic connections to a normally hyperbolic invariant manifold.
%
However, as shown in~\cite{DelshamsKRS19} (see Section~\ref{sec:scatteringstatements}), the theory in~\cite{DelshamsLS08} can be  adapted to the parabolic setting of the present paper.

%

From Theorem~\ref{thm:MainSplitting} we will obtain that the transversal intersection of the invariant manifolds $W^s(\EE_\infty)$ and $W^u (\EE_\infty)$
contains at least two homoclinic channels, $\Gamma^j \subset W^s (\EE_\infty )\cap W^u( \EE_\infty)$ $j=1,2$ (see Figure \ref{fig:canalhomoclinic}).
Then, associated to each homoclinic channel, one can define the scattering map $S^j$ as follows. We say that $x_+=S^j (x_-)$ if there exists a heteroclinic point in $\Gamma^j$ whose trajectory  is asymptotic to the trajectory of $x_+$ in the future and asymptotic to the trajectory of $x_-$ in the past. Such points $x_\pm$ are well defined even if $\EE_\infty$ is not a normally hyperbolic manifold. Once $\Gamma^j$ is fixed, thanks to the transversality between the invariant manifolds, the associated scattering map is locally unique and inherits the regularity of the invariant manifolds. 
%
%

\begin{figure}[h]
\begin{center}
\includegraphics[height=4cm]{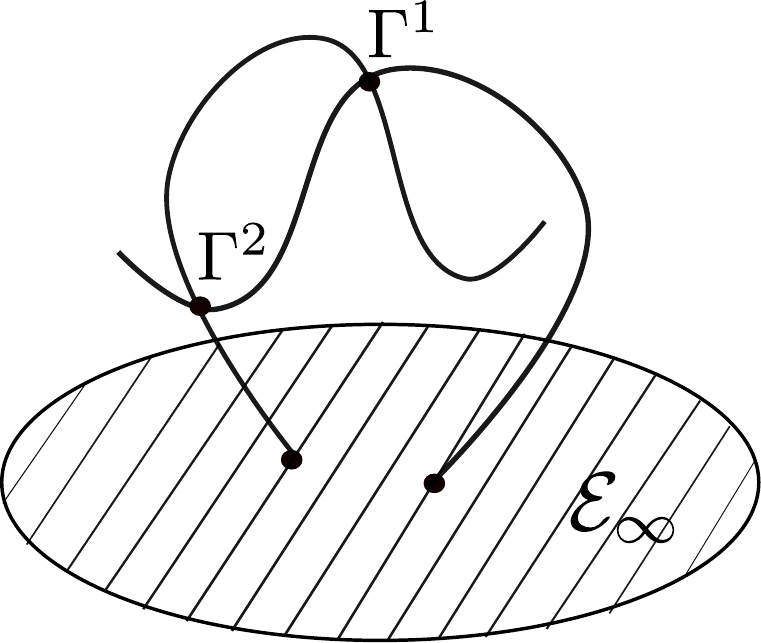}
\end{center}
\caption{Transverse intersection of the invariant manifolds of $\EE_\infty$ along two homoclinic channels $\Gamma_1$ and $\Gamma_2$.}\label{fig:canalhomoclinic}
\end{figure}

The construction of scattering maps in the parabolic setting was already done in \cite{DelshamsKRS19}. Note, however, that in the present paper the transversality between the invariant manifolds is  highly anisotropic (exponentially small in some directions and polynomially small in the others). This complicates considerably the construction of the scattering maps, which is done in Section \ref{sec:scatteringstatements} (see Section \ref{sec:scatteringglobal} for the proofs). Moreover, we show that the domains of the two scattering maps $S^1$ and $S^2$, associated to the two different channels, overlap. 

The scattering maps are crucial to analyze the global maps which have been introduced in Step 2 and are defined from the section $\Sigma_2$ to the section $\Sigma_1$. Indeed, we show that the dynamics of the $z$-variables in the two global maps are given by the corresponding variables of the associated scattering maps. The additional hyperbolicity in the $z$-directions we need will come from  a suitable high iterate of a combination of the two scattering maps $\wh S=(S^1)^M\circ S^2$ (for a suitable large $M$). To prove the existence of this hyperbolicity, we construct an isolating block for these combination.

\begin{figure}[h]
\begin{center}
\includegraphics[height=3cm]{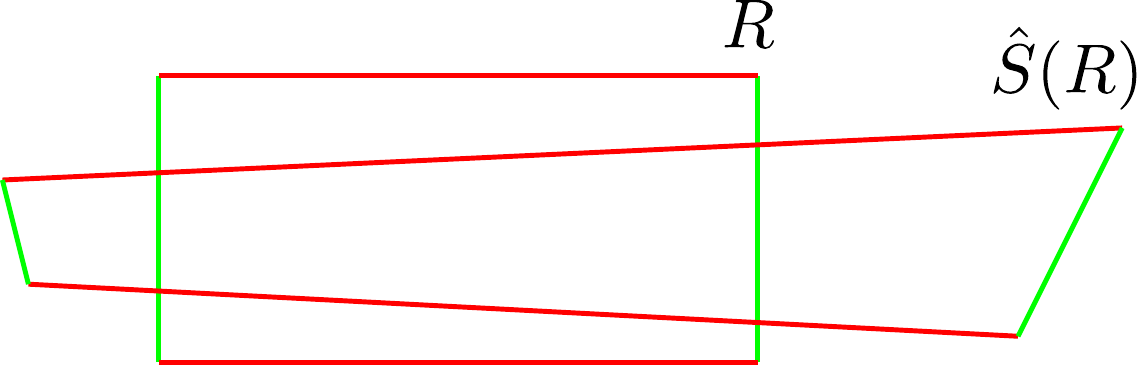}
\end{center}
\caption{The isolating block $R$ of the iterate of the scattering map $\wh S$}\label{fig:isolatingblockfigure}
\end{figure}

By isolating block we mean the following. There exists a small rectangle in the $z$-variables, in the common domain of the scattering maps, whose image under $\wh S$ is another rectangle ``correctly aligned'' with the original one, as seen in Figure~\ref{fig:isolatingblockfigure}, that is, the horizontal and vertical boundaries are mapped into horizontal and vertical boundaries,
respectively, it is stretched along the horizontal direction, shrunk in the vertical direction, and the left and right vertical boundaries are mapped to the left and right of the vertical boundaries, respectively, while the top and bottom horizontal boundaries are mapped below and above, respectively,
of the top and bottom horizontal boundaries.

To construct such isolating block we proceed as follows. Each of the scattering maps are nearly integrable twist maps around an elliptic fixed point (see Figure \ref{fig:2scattering}). These two fixed points are different but exponentially close to each other with respect $\Theta^{-3}$. Combining the two rotations around the distinct elliptic points, we use a transversality-torsion argument (in the spirit of \cite{Cresson03}) to build the isolating block.

\begin{figure}[h]
\begin{center}
\includegraphics[height=4cm]{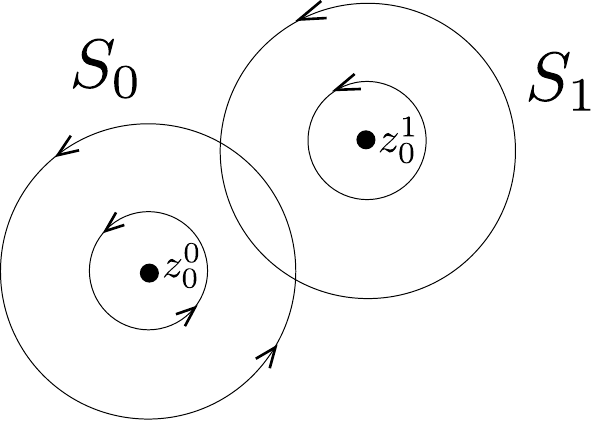}
\end{center}
\caption{The dynamics of the two scattering maps $S^1$ and $S^2$}\label{fig:2scattering}
\end{figure}

\subsection{Outline of Step 4: The isolating block for the return map} 
The last step of the proof combines Steps 2 and 3. We consider the return map $\Psi$ given by $M$ iterates of the return map along the first homoclinic chanel and $1$ iterate along the second homoclinic chanel. Each of the maps has two hyperbolic directions given by the passage close to the infinity manifolds as we have seen in Section \ref{sec:Step2}. The projection onto the $z$-variables  of each of the maps is close to the corresponding projection of the scattering maps. The same happens to the projection onto the $z$-variables of the whole composition $\Psi$. Hence, the map $\Psi$, possesses two ``stable'' and two ``unstable'' directions in some small domain. Even if the two stretching rates in the two expanding directions are drastically different, we are able to  check that the restriction of $\Psi$ to this small domain satisfies the standard hypotheses that ensure that $\Psi$ is conjugated to the Bernouilli shift with infinite symbols. In particular, we prove cone conditions for the return map $\Psi$.

In conclusion, we obtain a product-like structure as seen in Figure \ref{fig:isolatingtotal}. In left part of the figure, one obtains the usual structure of infinite horizontal and vertical strips as obtained by Moser in \cite{Moser01} whereas the right part of the figure corresponds to the isolating block construction in the $z$ directions. This structure leads to the existence of a hyperbolic set whose dynamics is conjugated to that of the usual shift \eqref{shift}. Since the strips accumulate to the invariant manifolds of infinity, one can check that there exists oscillatory orbits inside the hyperbolic invariant set.

\begin{figure}[h]
\begin{center}
\includegraphics[height=4cm]{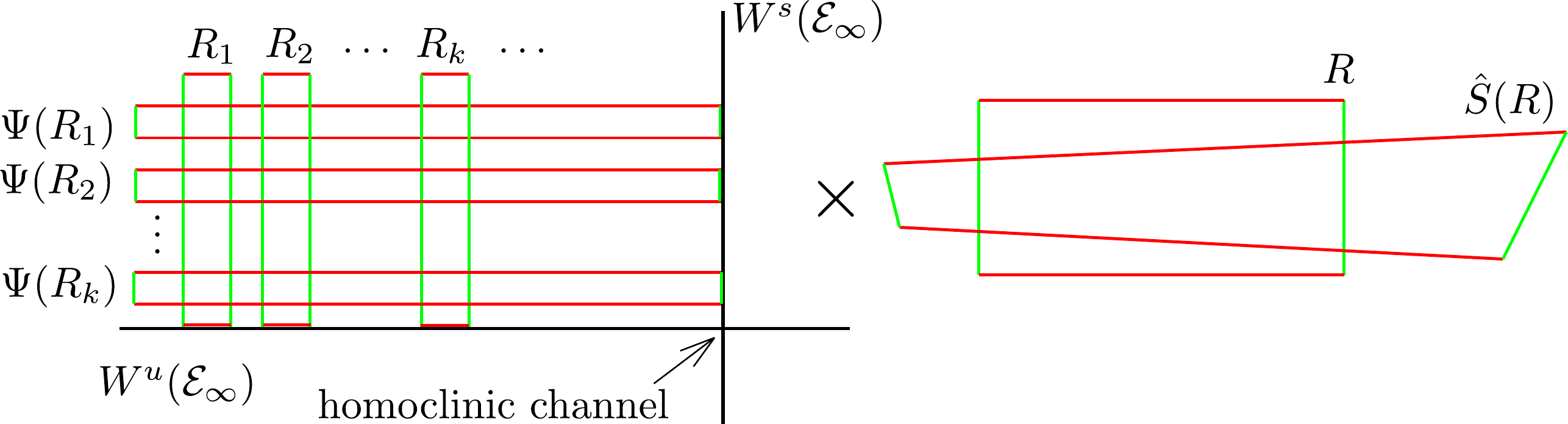}
\end{center}
\caption{The horizontal and vertical four dimensional strips which lead to the conjugation with the Bernouilli shift of infinite symbols.}\label{fig:isolatingtotal}
\end{figure}

\subsection{Summary of the outline and structure of the paper}

To summarize, we present here in a diagram the main steps in the proof of Theorems \ref{thm:Main1} and \ref{thm:Main2}.

\begin{center}
\begin{tikzpicture}[node distance=2.5cm]
\node(start)[rect]{\begin{tabular}{c}Transversality of the\\invariant manifolds\\ of infinity (Theorem \ref{thm:MainSplitting})\end{tabular}};

\node(normalform)[rect, right of=start, xshift=4cm]{\begin{tabular}{c} Parabolic\\normal form\\(Theorem \ref{prop:coordinatesatinfinity})\end{tabular}};

\node(scat)[rect, below of=start]{\begin{tabular}{c} Two Scattering maps\\(Theorem \ref{prop:scatteringmap})\end{tabular}};

\node(lambda)[rect, right of=scat, xshift=4cm]{\begin{tabular}{c} Parabolic \\Lambda lemma\\ (Theorem \ref{prop:lambdalemma})\end{tabular}};

\node(global)[rect, below of=scat]{\begin{tabular}{c} Two global maps
\end{tabular}};

\node(local)[rect, below of=lambda]{\begin{tabular}{c}Local map
\end{tabular}};

\node(bloc)[rect, below of=global, xshift=3cm]{\begin{tabular}{c}
                              Isolating block for \\
                              a suitable iterate of the return map\\ (Theorem \ref{thm:lanostraferradura})\end{tabular}};

\node(thm1)[rect, below of=bloc, xshift=-3cm]{\begin{tabular}{c} Oscillatory motions \\ (Theorem \ref{thm:Main1})\end{tabular}};
\node(thm2)[rect, below of=bloc, xshift=3cm]{\begin{tabular}{c} Symbolic dynamics\\ (Theorem \ref{thm:Main2})\end{tabular}};

%
%
\draw [arrow] (start) -- (scat);
\draw [arrow] (scat) -- (global);
\draw [arrow] (normalform) -- (lambda);
\draw [arrow] (lambda) -- (local);
\draw [arrow] (global) -- (bloc);
\draw [arrow] (local) -- (bloc);
\draw [arrow] (bloc) -- (thm1);
\draw [arrow] (bloc) -- (thm2);

\end{tikzpicture}
\end{center}

\section{A good system of coordinates for the 3 Body Problem}\label{sec:Coordinates}
To analyze the planar 3 Body Problem \eqref{eq:Newton}, the first step  is to choose a good system of coordinates which, on the one hand,  reduces symplectically the classical first integrals of the model and, on the other hand, makes apparent the nearly integrable setting explained in Section \ref{sec:outline}. That is, we consider a good system of coordinates so that we obtain at first order that the two first bodies, $q_0,q_1\in\RR^2$, move on ellipses, whereas the  third body, $q_2\in \RR^2$, moves on a coplanar parabola which is far away from the ellipses (that is, arbitrarily large).


\subsection{Symplectic reduction of the 3 Body Problem}\label{sec:symplecticreduction}
Introducing the momenta  $p_i=m_i \dot q_i$, $i=1,2,3$, equation \eqref{eq:Newton} defines a six degrees of freedom Hamiltonian system. We start  by reducing it by  translations with  the classical Jacobi coordinates to obtain  a four degrees of freedom Hamiltonian system.
That is, we define the symplectic transformation
\[
 \begin{aligned}
 Q_0 &= q_0 &P_0 &= p_0+p_1+p_2\\
 Q_1&=q_1-q_0& P_1 &= p_1+\frac{m_1}{m_0+m_1}p_2\\
 Q_2&=q_2-\frac{m_0q_0+m_1q_1}{m_0+m_1} &P_2 &= p_2.
 \end{aligned}
\]

\begin{figure}[h]
\begin{center}
\includegraphics[height=5cm]{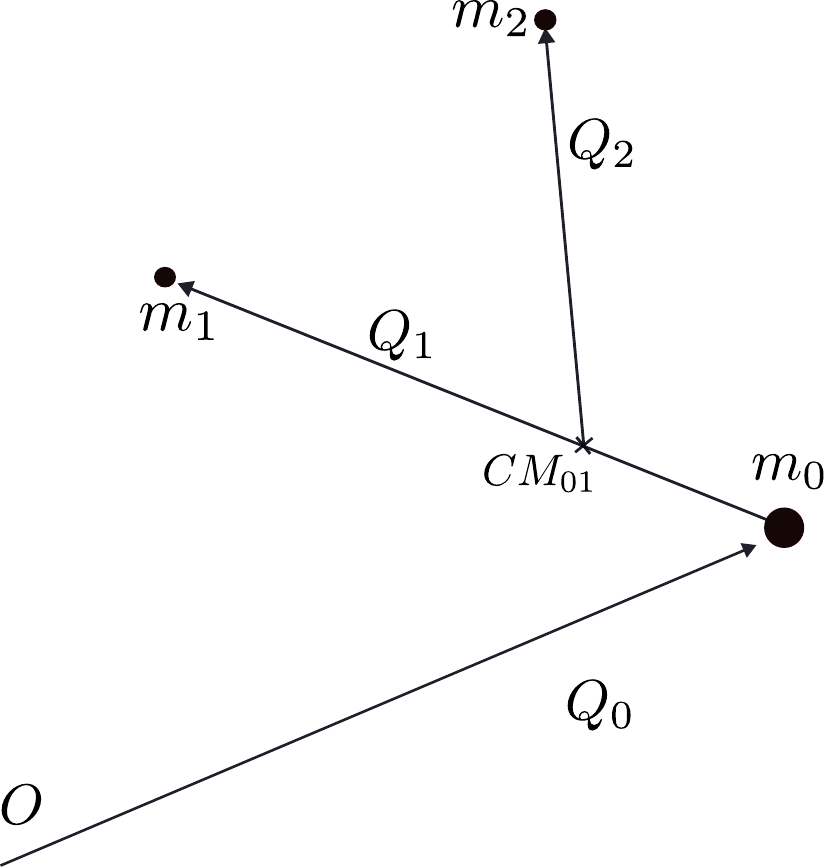}
\end{center}
\caption{The Jacobi coordinates. $CM_{01}$ stands for the center of mass of the bodies $q_0$ and $q_1$.}\label{fig:Jacobi}
\end{figure}
These coordinates allow to reduce by the total linear momentum since now $P_0$ is a first integral. Assuming $P_0=0$, the Hamiltonian of the 3 Body Problem becomes
\[
 \wt H(Q_1,P_1, Q_2,P_2)=\sum_{j=1}^2\frac{|P_j|^2}{2\mu_j}-\wt U(Q_1,Q_2)
\]
where
\[
\frac{1}{\mu_1}=\frac{1}{m_0}+\frac{1}{m_1},\qquad \frac{1}{\mu_2}=\frac{1}{m_0+m_1}+\frac{1}{m_2}
\]
and
\[
 \wt U(Q_1,Q_2)=\frac{m_0m_1}{\| Q_1\|}+\frac{m_0m_2}{\|Q_2+\sigma_0 Q_1\|}+\frac{m_1m_2}{\|Q_2-\sigma_1 Q_1\|}
\]
with
\begin{equation}\label{eq:sigma01}
 \sigma_0= \frac{m_1}{m_0+m_1},\qquad  \sigma_1= \frac{m_0}{m_0+m_1}=\frac{1}{1+\sigma_0}.
\end{equation}
Next step is to express the Hamiltonian $\wt H$ in polar coordinates.
Identifying $\RR^2$ with $\CC$, we consider the symplectic transformation
\[
Q_1=\rr e^{i\theta},\qquad Q_2=re^{i\al},\qquad  P_1=z e^{i\theta}+i\frac{\Gamma}{\rr}e^{i\theta},\qquad   P_2=y e^{i\al}+i\frac{G}{r}e^{i\al}
\]
which leads to  the Hamiltonian
\[
H^*(\rr,z,\tet,\Ga,r,y,\al,G)=
\frac{1}{\mu_1}\left(\frac{z^2}{2}+\frac{\Ga^2}{2\rr^2}\right)+\frac{1}{\mu_2}\left(\frac{y^2}{2}+\frac{G^2}{2r^2}\right)-
\wt U\left(\rr e^{i\theta}, re^{i\al}\right).
\]
where
\[
\begin{split}
\wt U\left(\rr e^{i\theta}, re^{i\al}\right)&= \frac{m_0m_1}{ \rr}+\frac{m_0m_2}{|r e^{i\al}+\sigma_0 \rr e^{i\theta}|}+\frac{m_1m_2}{|r e^{i\al}-\sigma_1 \rr e^{i\theta}|}\\
&= \frac{m_0m_1}{\rr}+\frac{1}{r}\left(\frac{m_0m_2}{|1+\sigma_0 \frac{\rr}{r} e^{i(\theta-\al)}|}+\frac{m_1m_2}{|1-\sigma_1 \frac{\rr}{r} e^{i(\theta-\alpha)}|}\right).
\end{split}
\]
We study the regime where  the third body is  far away  from the other two and its angular momentum is very large. That is,
\[
 r\gg \rr\,\,\,\text{ and }\,\,\,G\gg\Ga.
\]
Then, we have
\[
H^*(\rr,z,\tet,\Ga,r,y,\al,G)=
\frac{1}{\mu_1}\left(\frac{z^2}{2}+\frac{\Ga^2}{2\rr^2}\right)+\frac{1}{\mu_2}\left(\frac{y^2}{2}+\frac{G^2}{2r^2}\right)-
\frac{m_0m_1}{\rr}-\frac{m_2(m_0+m_1)}{r}+\OO\left(\frac{\rr^2}{r^3}\right).
\]
Thus, at first order, we have two uncoupled Hamiltonians, one for $(\rr,z,\tet,\Ga)$ and the other for $(r, y,\al,G)$,
\begin{equation}\label{def:UncoupledHam}
\begin{split}
H_\El(\rr,z,\tet,\Ga)&=\frac{1}{\mu_1}\left(\frac{z^2}{2}+\frac{\Ga^2}{2\rr^2}\right)-m_0m_1\frac{1}{\rr}\\
H_\Par(r,y,\al,G)&=\frac{1}{\mu_2}\left(\frac{y^2}{2}+\frac{G^2}{2r^2}\right)-m_2(m_0+m_1)\frac{1}{r}.
\end{split}
\end{equation}
To have the first order Hamiltonians $H_\El$ and $H_\Par$ independent of the masses, we make the following scaling to the variables, which is symplectic,
\[
 \rr=\frac{1}{\mu_1 m_0m_1}\wt\rr,\,\,\,\,z=\mu_1 m_0m_1 \wt z,\,\,\,\,r=\frac{1}{\mu_2 m_2(m_0+m_1)}\wt r\,\,\,\,\text{ and }\,\,\,\,y=\mu_2 m_2(m_0+m_1) \wt y.
\]
We also rescale time as
\[
 t=\frac{\tau}{\mu_2 m_2^2(m_0+m_1)^2}.
\]
Then, we obtain the Hamiltonian
\[
 \tilde{H}^*(\wt\rr,\wt z,\tet,\Ga,\wt r,\wt y,\al,G)=\nu\left(\frac{\wt
z^2}{2}+\frac{\Ga^2}{2\wt \rr^2}-\frac{1}{\wt
\rr}\right)+\left(\frac{\wt y^2}{2}+\frac{G^2}{2\wt r^2}-\frac{1}{\wt
r}\right)-W (\wt\rr,\wt r,\tet- \al).
\]
with
\begin{equation}\label{def:PotentialUtilde}
W (\wt\rr,\wt r,\tet- \al)=\frac{\wt\nu}{\wt r}\left(\frac{m_0}{|
1+\wt\sigma_0\frac{\wt \rr}{\wt r} e^{i(\tet-\al)}|}+\frac{m_1}{|1-\wt\sigma_1\frac{\wt
\rr}{\wt r} e^{i(\tet-\al)}|}-(m_0+m_1)\right),
\end{equation}
and
\begin{equation}\label{eq:sigmatilde01}
 \nu=\frac{\mu_1 m_0m_1}{\mu_2 m_2(m_0+m_1)},
\qquad \wt\nu=(m_0+m_1)m_2^2\qquad\text{and}\qquad\wt\sigma_i=\frac{\mu_2m_2(m_0+m_1)}{\mu_1 m_0m_1}\sigma_i.
\end{equation}
%
Note that the potential $W$ only depends on the angles through $\tet-\al$ due to the rotational symmetry of the system.

Now, we change the polar variables $(\wt\rr,\wt z,\tet,\Ga)$ to the classical Delaunay coordinates (see, for instance, \cite{Szebehely})
\begin{equation}\label{def:Delaunay}
(\wt\rr,\wt z,\tet,\Ga)\mapsto (\ell,L,g,\Ga).
\end{equation}
This change is symplectic. As usual, by the  Delaunay actions, which are the square of the semimajor axis $L$ and the angular momentum $\Ga$,
one can compute the eccentricity
\begin{equation}
\label{def:DeleLGamma}
 \e(L,\Ga)=\sqrt{1-\frac{\Ga^2}{L^2}}.
\end{equation}
The position variables $(\wt\rho,\theta)$ can be expressed in terms of Delaunay as
\begin{equation}
\label{def:Delvg} \wt\rr=\wt\rr(\ell,L,\Ga)=L^2(1-\e\cos E)
\qquad \text{and}\qquad \tet=\tet(\ell,L,g,\Ga)=v(\ell,L,\Ga)+g
\end{equation}
where the angles true anomaly $v$ and eccentric anomaly $E$ are defined in terms of the mean anomaly $\ell$ and eccentricity $e_c$  as
\begin{equation}
\label{def:Dellu}
 \ell=E-\e\sin E \qquad \text{and}\qquad  \tan \frac{v}{2}=\sqrt{\frac{1+\e}{1-\e}}\tan \frac{E}{2}.
\end{equation}
One could also write an expression for $\tilde z$, but it is not necessary to obtain the new Hamiltonian
\[
\HH(\ell, L,g,\Ga,\wt r,\wt
y,\al,G)=-\frac{\nu}{2L^2}+\left(\frac{\wt y^2}{2}+\frac{G^2}{2\wt
r^2}-\frac{1}{\wt r}\right)+W (\wt\rr(\ell,L,\Ga),\wt
r,v(\ell,L,\Ga)+g- \al)
\]
where $W$ is the potential introduced in \eqref{def:PotentialUtilde}.
Now, by~\eqref{def:Delvg}, the distance condition corresponds to
$\wt r\gg L^2$
and the \emph{first order} uncoupled Hamiltonians are
\[
\HH_\El(\ell, L,g,\Ga)=-\frac{\nu}{2L^2}\qquad\text{ and }\qquad
\HH_\Par(\wt r,\wt y,\al,G)=\frac{\wt
y^2}{2}+\frac{G^2}{2\wt r^2}-\frac{1}{\wt r}
\]
whereas $W=\mathcal{O}(\frac{\tilde \rho ^2}{\tilde r ^3})= \mathcal{O}(\frac{L ^4}{\tilde r ^3})$.

Now, we make the last reduction which uses the rotational symmetry.
We define the new angle $\phi = g-\alpha$.
To have a symplectic change of coordinates, we consider the transformation
\begin{equation}\label{def:rotationreduction}
(\ell, L,\phi,\Ga,\wt r,\wt y,\al,\Tet)=(\ell, L,g-\al,\Ga,\wt r,\wt y,\al,G+\Ga).
\end{equation}
Then, we obtain the following Hamiltonian, which is independent of $\alpha$,
\begin{equation}\label{def:Ham3dof}
\begin{split}
\wt \HH(\ell, L,\phi,\Ga,\wt r,\wt y;\Tet)=&\,\HH(\ell, L,\phi+\alpha,\Ga,\wt r,\wt y,\alpha,\Tet-\Ga)\\
=&\,-\frac{\nu}{2L^2}+\left(\frac{\wt y^2}{2}+\frac{(\Theta-\Ga)^2}{2\wt r^2}-\frac{1}{\wt r}\right)+W (\wt\rr(\ell,L,\Ga),\wt r,v(\ell,L,\Ga)+\phi).
\end{split}
\end{equation}
Since this Hamiltonian is independent of $\al$,  the total angular momentum $\Theta$ is a conserved quantity which can be taken as a parameter of the system. We assume $\Theta\gg 1$.

\subsection{The Poincar\'{e} variables}

We  consider nearly circular motions for the first two bodies.
Since Delaunay variables are singular at the circular motions $\Ga\simeq L$ (equivalently by \eqref{def:DeleLGamma}, $\e\simeq 0$), we introduce Poincar\'e variables
\[
(\ell, L,\phi,\Ga, \wt r,\wt y) \mapsto (\la, L,\eta,\xi, \wt r,\wt y),
\]
defined by
\begin{equation}\label{def:PoincareVariablesnr}
\begin{aligned}
\la=&\,\ell+\phi,& L=&\,L,\\
\eta=&\,\sqrt{L-\Ga}e^{i\phi}, & \xi=&\,\sqrt{L-\Ga}e^{-i\phi},
\end{aligned}
\end{equation}
which are symplectic in the sense that the form $d\ell\wedge dL+d\phi\wedge d\Gamma$ is mapped into $d\la\wedge dL+id\eta\wedge d\xi$. These coordinates  make  the Hamiltonian $\wt \HH$ well defined at  circular motions (i.e.  at $\eta = \xi = 0$). The transformed Hamiltonian can be written as
\begin{equation}\label{def:Ham:Poincarenr}
\wt\KK(\la,L,\eta,\xi,\wt r,\wt y;\Tet)=-\frac{\nu}{2 L^2}+\frac{\wt
y^2}{2}+\frac{(\Theta-L+\eta\xi)^2}{2\wt r^2}-\frac{1}{\wt
r}+\wt  W (\la, L,\eta,\xi, \wt r)
\end{equation}
where, using that
\begin{equation}\label{epoincare}
e^{i\phi}=\frac{\eta}{\sqrt{\eta\xi}},
\end{equation}
the potential becomes
\begin{equation}\label{def:potentialwtilde}
\begin{split}
\wt W( \la, L,\eta,\xi,\wt r)&= W (\wt\rr(\ell,L,\Ga),\wt r,v(\ell,L,\Ga)+\phi)\\
&=\frac{\wt\nu}{\wt r}\left(\frac{m_0}{|
1+\wt\sigma_0\frac{\wt \rr}{\wt r} \frac{\eta}{\sqrt{\eta\xi}} e^{iv}|}
+\frac{m_1}{|1-\wt\sigma_1\frac{\wt \rr}{\wt r}\frac{\eta}{\sqrt{\eta\xi}} e^{iv}|}-(m_0+m_1)\right),
\end{split}
\end{equation}
where the functions $v$ and $\tilde \rr$ are evaluated at
\begin{equation}\label{def:rvInPoinc}
(\ell,L,\Gamma)=\left(\la+\frac{i}{2}\log\frac{\eta}{\xi},L,L-\eta\xi\right).
\end{equation}
In particular, by~\eqref{def:DeleLGamma} and~\eqref{def:PoincareVariablesnr}, the eccentricity is given by
\begin{equation*}
\label{def:ecc:xieta}
 \e=\frac{1}{L}\sqrt{\eta\xi}\sqrt{2L-\eta\xi}.
\end{equation*}
The associated equations are
\begin{equation}\label{def:Poincarenr}
  \begin{aligned}
  \la'&= \frac{\nu}{L^3}-\frac{\Theta-L+\eta\xi}{\wt r^2}  +\pa_L \wt W ,
  &  L' &= -\pa_\la \wt W\\
\eta' &=-i\frac{\Theta- L+\eta\xi}{ \wt r^2}\eta-i\pa_\xi \wt W,
&\xi'&=i\frac{\Theta- L+\eta\xi}{ \wt  r^2}\xi +i\pa_\eta \wt W\\
\wt  r'&=\wt  y, & \wt  y' &=  \frac{(\Theta- L+\eta\xi)^2}{\wt  r^3}-
 \frac{1}{\wt  r^2}-\pa_r \wt W.
\end{aligned}
\end{equation}

\begin{remark}\label{rmk:analytic}
Notice that the Hamiltonian \eqref{def:Ham3dof} was not analytic at a neighborhood of circular motions for the two first bodies, that is $L=\Ga$. Nevertheless, it is well known that once this  Hamiltonian is expressed in Poincar\'e variables, that is Hamiltonian \eqref{def:Ham:Poincarenr}, the system becomes \emph{analytic} for $(\eta,\xi)$ in a neighborhood of $(0,0)$. See, for instance, \cite{Fejoz13}.
\end{remark}

\section{The manifold at infinity and the associated invariant manifolds}
\label{sec:varietatainfinit}

The Hamiltonian $\wt\KK$ in  \eqref{def:Ham:Poincarenr} has an invariant manifold at infinity. Indeed, the potential $\wt W$ in \eqref{def:potentialwtilde} satisfies  $\wt W = \OO(L^4/\widetilde r^3)$. Therefore,  the manifold
\begin{equation*}
\PP_{\infty}=\{(\la, L,\eta,\xi,\wt r,\wt y): \ \wt r=+\infty,\ \wt y=0\}
\end{equation*}
is invariant\footnote{To analyze this manifold properly, one should consider McGehee coordinates $r=2/x^2$. This is done in Section \ref{sec:LambdaLemma}.}.

Note that, at $\PP_\infty$, the Hamiltonian $\wt\KK$ satisfies
\[
\wt\KK_{\mid\PP_\infty}=-\frac{\nu}{2L^2}
\]
and  $\dot L_{\mid\PP_\infty}=0$. Therefore, we can fix $L=L_0$ and restrict to an energy level $\wt\KK=-\frac{\nu}{2L_0^2}$. We consider the restricted infinity manifold
%
\begin{equation}\label{def:inftyenergy}
\EE_{\infty}=\PP_{\infty}\cap \wt \KK^{-1}\left(-\frac{ \nu}{2L_0^2}\right)=\left\{
(\lambda, L,  \eta, \xi,\wt  r,\wt  y): L=L_0, \ \wt  r=+\infty , \ \wt  y=0, \ (\eta, \xi) \in \UU , \
\lambda \in \TT\right\},
\end{equation}
where $\UU\subset  \RR^2$ is an open set containing the origin\footnote{Observe that $(\eta,\xi)\in\CC^2$ but they satisfy $\xi=\ol\eta$.} which is specified below.
By the particular form of the  Hamiltonian  $\wt\KK$ in  \eqref{def:Ham:Poincarenr}, it is clear that
the manifold $\EE_\infty$
is foliated by  periodic orbits as
\[
\EE_{\infty}=\bigcup _{(\eta_0,\xi_0)\in\UU} P_{\eta_0,\xi_0}
\]
with
\[
P_{\eta_0,\xi_0} = \{(\lambda, L,  \eta, \xi,\wt  r,\wt  y): \  \ \eta = \eta_0,\ \xi= \xi_0, \ L=L_0, \ \wt  r=+\infty , \ \wt  y=0,\
\lambda \in \TT\},
\]
whose dynamics is given by
\[
\la(t)= \la_0+\frac{\nu}{L_0^3}t.
\]
These periodic orbits are parabolic, in the sense that its linearization (in McGehee coordinates) is degenerate. Nonetheless, they have stable and unstable invariant manifolds whose union  form the invariant manifolds of the infinity manifold $\EE_\infty$ (see Theorem \ref{thm:MainSplitting}).

The goal of this section is to analyze the stable and unstable invariant manifolds of $\EE_\infty$ and show that,  restricting to suitable open domains in $\EE_\infty$, they  intersect transversally  along two homoclinic channels $\Gamma^1$ and $\Gamma^2$ (see Figure \ref{fig:canalhomoclinic}). This will allow us to define two different scattering maps in suitable domains in $\EE_{\infty}$.

\subsection{The unperturbed Hamiltonian system}
Since we are considering the regime $\wt r\gg L^2$
 and $\wt W$ satisfies $\wt W = \OO(L^4/\widetilde r^3)$, we  first analyze the Hamiltonian $\wt\KK$ in \eqref{def:Ham:Poincarenr} with  $\wt W=0$. We consider this as the unperturbed Hamiltonian. In fact, when $\wt W=0$,  $\wt\KK$ becomes  integrable and therefore the invariant manifolds of the periodic orbits $P_{\eta_0,\xi_0}$ coincide.

Indeed, it is easy to check that
$L$ and $\eta\xi$ (and the Hamiltonian) are functionally independent first integrals.
Therefore, if we  restrict to the energy level $\wt\KK=-\frac{\nu}{2L_0}$
and we define
\begin{equation*}
G_0=\Tet- L_0+\eta_0 \xi_0,
\end{equation*}
the invariant manifolds of any periodic orbit $P_{\eta_0,\xi_0}$ should satisfy
\[\Theta- L+\eta\xi=G_0\]
and
therefore they must be a solution of the equations
\begin{equation}\label{eq:integrableJacobinrmanifold}
  \begin{aligned}
  \la'&=\frac{\nu }{L_0^3}-\frac{G_0}{\wt  r^2}, &  &\\
\eta' &=-i\frac{G_0}{\wt r^2}\eta, &\xi'&=i\frac{G_0}{\wt  r^2}\xi\\
\wt  r'&=\wt  y, & \wt  y' &=  \frac{G_0^2}{\wt  r^3}-\frac{1}{\wt  r^2}.
\end{aligned}
\end{equation}
The invariant manifolds of the periodic orbit $P_{\eta_0,\xi_0}$ associated to equation \eqref{eq:integrableJacobinrmanifold} are analyzed in the next lemma.

\begin{lemma}\label{lemma:homounperturbed}
 The invariant manifolds of the periodic orbit $P_{\eta_0,\xi_0}$ associated to equation \eqref{eq:integrableJacobinrmanifold} coincide along a homoclinic manifold which can be parameterized as
\begin{equation}\label{eq:homoclinicPoincare}
  \begin{aligned}
\la&=\ga+\phi_\h(u)&  L&=L_0\\
\eta&=
\eta_0 e^{i\phi_\h(u)}&
\xi&=
\xi_0e^{-i\phi_\h(u)}\\
\wt r&=G_0^2 \wh r_\h(u)&  \wt y&=G_0^{-1}\wh y_\h(u),
\end{aligned}
 \end{equation}
where $(\wh r_\h(u), \wh y_\h(u), \phi_\h(u))$ are defined as
 \begin{equation}\label{def:homoclinic}
\begin{aligned}
\wh r_\h(u) & =  r_0(\tau(u)), & \quad  r_0(\tau) & = \frac{1}{2}(\tau^2+1), \\
\wh y_\h(u) & =  y_0(\tau(u)), & \quad  y_0(\tau) & = \frac{2\tau}{(\tau^2+1)}, \\
\phi_\h(u) & = \phi_0(\tau(u)),& \quad  \phi_0(\tau) &  = i \log\left( \frac{\tau-i}{\tau+i}\right),
\end{aligned}
\end{equation}
where $\tau(u)$ is obtained through
\begin{equation*}
u = \frac{1}{2}\left(\frac{1}{3} \tau^3+\tau\right).
\end{equation*}
In particular
\begin{equation*}
e^{i\phi_0(\tau)}=\frac{\tau+i}{\tau-i}
\end{equation*}
and $\phi_h$ satisfies
\begin{equation}\label{def:philimit}
\lim_{u\to \pm\infty}\phi_\h(u)=0 \,\,(\text{mod } 2\pi)\qquad \text{and}\qquad \ \phi_\h(0)=\pi.
\end{equation}
Moreover, the dynamics in the homoclinic manifold \eqref{eq:homoclinicPoincare} is given by
\[
 u'=G_0^{-3}, \quad \ga'= \frac{\nu}{L_0^3}.
\]
\end{lemma}

Note that the dynamics in this homoclinic manifold makes apparent the slow-fast dynamics.
Indeed the motion on the $(\wt r, \wt y)$ is much slower than the rotation dynamics in the $\lambda$ variable.

\begin{proof}[Proof of Lemma \ref{lemma:homounperturbed}]
To prove this lemma   it is convenient to  scale the variables and time as
\begin{equation}\label{def:rescaling}
\wt r = G_0^2 \wh r,\,\,\,\wt y=G_0^{-1}\wh y, \quad \text{and}  \quad t=G_0^3 s
\end{equation}
in equation \eqref{eq:integrableJacobinrmanifold} to obtain
\begin{equation}\label{eq:integrableJacobi}
  \begin{aligned}
 \frac{d \la}{ds}&=\frac{\nu G_0^3 }{L^3}-\frac{1}{\wh  r^2}, &  &\\
\frac{d\eta}{ds} &=-\frac{i}{\wh r^2}\eta, &\frac{d\xi}{ds}&=\frac{i}{\wh  r^2}\xi\\
\frac{d \wh  r}{ds}&=\wh  y, & \frac{d\wh  y}{ds} &=  \frac{1}{\wh  r^3}-\frac{1}{\wh  r^2}.
\end{aligned}
\end{equation}
The last two equations are Hamiltonian with respect to
\begin{equation}\label{def:HamPendulum}
 h(\wh r, \wh y)=\frac{\wh y^2}{2}+\frac{1}{2\wh r^2}-\frac{1}{\wh r}
\end{equation}
and, following \cite{SimoL80}, they have a solution $(\wh r_\h(s), \wh y_\h(s))$ as given in \eqref{def:homoclinic} which satisfies
\begin{equation}\label{def:CondicioInicialHomo}
\lim_{s \to \pm\infty} (\wh r_\h(s),\wh y_\h(s)) = (\infty,0)\quad \text{ and }\quad \wh y_\h(0)=0.\,\,\,
\end{equation}
Moreover, for the $(\eta,\xi)$ components, it is enough to define the function $\phi_\h(s)$ which satisfies
\[
\frac{d \phi}{ds}=-\frac{1}{\wh  r_h^2} \qquad \text{and}\qquad \phi_\h(0)=\pi.
\]
Following again \cite{SimoL80}, it is given in \eqref{def:homoclinic} and  satisfies the asymptotic conditions in \eqref{def:philimit}.

To complete the proof of the lemma it is enough to integrate the rest of equations in \eqref{eq:integrableJacobi}
and undo the scaling \eqref{def:rescaling}.
\end{proof}

Observe that the union of the homoclinic manifolds of the periodic orbits $P_{\eta_0,\xi_0}$ form the homoclinic manifold of the infinity manifold (restricted to the energy level) $\EE_{\infty}$, which is  four dimensional.


\subsection{The invariant manifolds for the perturbed Hamiltonian}

In this section we analyze the invariant manifolds of the infinity manifold $\EE_\infty$ (see \eqref{def:inftyenergy}) and their transverse intersections  for the full Hamiltonian $\wt\KK$ in \eqref{def:Ham:Poincarenr} (that is, incorporating the potential $\wt W$ in \eqref{def:potentialwtilde}). Given a  periodic orbit $P_{\eta_0,\xi_0}\in \EE_\infty$, we want to study its unstable manifold and its possible intersections with the stable manifold of nearby periodic orbits
\[
P_{\eta_0+\delta\eta,\xi_0+\delta\xi}\in \EE_\infty\qquad \text{for some}\qquad |\delta\eta|,|\delta\xi|\ll 1.
\]
This will lead to heteroclinic connections and, therefore, to the definition of scattering maps. To this end, we consider parameterizations of a rather particular form. The reason, as it is explained in Section \ref{sec:adaptedcoordinates} below, is to keep track of the symplectic properties of these parameterizations. Using  the unperturbed parameterization introduced in Lemma \ref{lemma:homounperturbed} and the constant
\[
G_0=\Theta-L_0+\eta_0\xi_0,
\]
we define parameterizations of the following form, where $*$ stands for $ *=u,s$,
\begin{equation}\label{def:sigmaparamnousnr}
\begin{split}
\la =& \,  \gamma+\phi_\h(u)\\
L^*(u,\lo) =&\, L_0+\Lambda^*(u,\lo)\\
\eta^* (u,\lo)=&\,e^{i\phi_\h(u)}(\eta_0+\alpha^*(u,\lo))\\
\xi ^*(u,\lo)=&\,e^{-i\phi_\h(u)}(\xi_0+\beta^*(u,\lo))\\
\wt r =&\, G_0^2 \wh r_\h(u)\\
\wt y^*(u,\lo) =&\, \frac{ \wh y_\h(u)}{G_0}+\frac{ Y^*(u,\lo)}{G_0^2\wh y_\h(u)}+\frac{\Lambda^*(u,\lo)-(\eta_0+\alpha^*(u,\lo))(\xi_0+\beta^*(u,\lo))+\eta_0\xi_0}{G_0^2\wh y_\h(u)(\wh r_\h(u))^2}
\end{split}
\end{equation}
where
the functions $\La^*,\al^*,\beta^*, Y^*$ satisfy
\[
\begin{split}
(\La ^u(u,\lo),\al^u(u,\lo),\beta^u(u,\lo),(\wh y_\h(u))\ii Y^u (u,\lo)) &\to (0,0,0,0),\quad \mbox{as}\quad  u\to -\infty\\
(\La ^s(u,\lo),\al^s(u,\lo),\beta^s(u,\lo),(\wh y_\h(u))\ii Y^s (u,\lo)) &\to (0,\delta \eta,\delta\xi,0),\quad \mbox{as}\quad  u\to +\infty
\end{split}
\]

The rather peculiar form of these parameterizations relies on the fact that one can interpret them through the change of coordinates given by
\[
(\la,L,\eta,\xi,\wt r,\wt y) \to (\ga,\Lambda, \al, \bet, u, Y).
\]
Then, one can keep track of the symplectic properties of the invariant manifolds since this change is symplectic in the sense that it  sends the canonical form into  $ d\gamma \wedge d\Lambda+id\alpha\wedge d\beta+du\wedge d Y$. This is explained in full detail in Section \ref{sec:adaptedcoordinates} (see the symplectic transformation \eqref{def:ChangeThroughHomo}).

If the functions $(Y^*,\La ^*,\al^*,\beta^*)$ are small, as stated in Theorem \ref{thm:MainSplitting} below, these parameterizations are close to those of the unperturbed problem, given in Lemma \ref{lemma:homounperturbed}.
Furthermore, note that, to analyze the difference between the invariant manifolds, it is enough to measure the differences
\begin{equation}\label{def:diffdiff}
 (Y^s -Y^u,\La ^s-\La ^u,\al^s-\al^u,\beta^s-\beta^u)
\end{equation}
for $u$ in a suitable interval and $\lo\in\TT$.
The zeros of this difference  will lead to homoclinic connections to $P_{\eta_0,\xi_0}$, if one chooses $\de\eta=\de\xi=0$, and to heteroclinic connections between $P_{\eta_0,\xi_0}$ and $P_{\eta_0+\delta \eta,\xi_0+\delta \xi}$, otherwise.

The analysis of the difference \eqref{def:diffdiff} is done  in
Proposition \ref{prop:MelnikovPotential} and Theorem \ref{thm:MainSplitting} below.
First, in Proposition \ref{prop:MelnikovPotential}, we define  a  Melnikov potential, which  provides the first order of the difference between the invariant manifolds through the difference \eqref{def:diffdiff}.
Then, Theorem \ref{thm:MainSplitting} gives the existence of parameterizations of the form \eqref{def:sigmaparamnousnr} for the unstable manifold of $ {P}_{\eta_0, \xi_0}$ and the stable invariant manifold of $ {P}_{\eta_0+\de\eta, \xi_0+\de\xi}$  and shows that, indeed, the derivatives of the Melnikov potential given in Proposition \ref{prop:MelnikovPotential} plus an additional explicit term depending on $(\de\eta,\de\xi)$ gives the first order of their difference when the parameter $\Tet$ is large enough.

We then introduce a Melnikov potential
%
\begin{equation}\label{def:MelnikovPotential}
\LL(\vm,\eta_0,\xi_0)=G_0^3\int_{-\infty}^{+\infty}
\wt W\left( \vm+\omega s+\phi_\h(s), L_0,e^{i\phi_\h(s)}\eta_0,
e^{-i\phi_\h(s)}\xi_0,G_0^2\wh  r_\h(s)\right)ds,
\end{equation}
%
where  $\wt W$ is given in \eqref{def:potentialwtilde}, $(r_\h(u),\phi_\h(u))$ are introduced in Lemma \ref{lemma:homounperturbed} and
\begin{equation}\label{eq:omega}
\omega=\frac{\nu G_0^3}{L_0^{3}}, \quad \mbox{with}\quad G_0=\Theta-L_0+\eta_0\xi_0.
\end{equation}
Note that, as usual, it is just the integral of  the perturbing potential $\wt W$ evaluated at the unperturbed homoclinic manifold \eqref{eq:homoclinicPoincare} (at a given energy level).

To provide asymptotic formulas for the Melnikov potential $\LL$ we use the parameter
\begin{equation}\label{def:ThetaTilde}
\tte=\Tet-L_0.
\end{equation}

\begin{proposition}\label{prop:MelnikovPotential}
Fix $L_0\in [1/2,2]$.
Then, there exists $\Theta ^*\gg 1$ and $0<\varrho^*\ll 1$ such that for
$\Theta\ge \Theta ^*$ and $(\eta_0, \xi_0)$
satisfying $\xi_0=\ol{\eta_0}$ and
$|\alo|\Theta^{3/2} \leq \varrho^*$,
the Melnikov potential introduced in \eqref{def:MelnikovPotential} is $2\pi$-periodic in $\sigma$ and can be written as
\[
 \LL(\vm,\eta_0,\xi_0)=\LL^{[0]}(\eta_0,\xi_0) +\LL^{[1]}(\eta_0,\xi_0)e^{i\vm}+\LL^{[-1]}(\eta_0,\xi_0)e^{-i\vm}+\LL^\geq(\vm,\eta_0,\xi_0),
\]
%
%
and the Fourier coefficients satisfy $\LL^{[q]}(\alo,\beto)=\ol{\LL^{[-q]}}(\beto,\alo)$ and
\[
\begin{split}
\LL^{[0]}(\alo,\beto) =&\,
\tilde \nu \pi L_0^4 (\tte+\alo\beto)^{-3} \Bigg[\frac{ N_2}{8}\left(1+3\frac{\alo\beto}{L_0}
-\frac32 \frac{\alo^2\beto^2}{L_0^2}
\right)
- N_3\frac{15}{64}\  \frac{L_0^2}{\sqrt{2L_0}} \tte^{-2}(\alo+\beto)+\RRR_0(\alo,\beto)\Bigg]\\
\LL^{[1]}(\alo,\beto) =&\,
  \tilde\nu\ex^{-\frac{\tilde \nu (\tte+\alo\beto)^3}{3 L_0^3}} \Bigg[
  \frac{N_3}{32} \sqrt{\frac{\pi}{2}}L_0^6 \tte^{-\frac{1}{2}}-3\frac{N_2}{4}\sqrt{\pi}L_0^{\frac{7}{2}}\tte^{\frac32}\alo
 +\RRR_1(\alo,\beto)\Bigg]
\end{split}
\]
where  $\tilde\nu$ is the constant introduced in  \eqref{eq:sigmatilde01} and
\begin{equation}\label{def:massterm}
N_2=\frac{m_2^4(m_0+m_1)^5}{m_0^3m_1^3}, \ \quad
N_3=\frac{m_2^6(m_0+m_1)^7}{m_0^5m_1^5}(m_1-m_0),
\end{equation}
and
\[
\RRR_0(\alo,\beto)=\OO\left(\Tet^{-4}\right)+
\OO\left(\Tet^{-2}|\alo|^3\right),\qquad
\RRR_1(\alo,\beto)=  \OO\left(\Tet^{-1},|\alo|, |\alo|^2 \Tet^{5/2}\right)
\]
and, for $i,j\geq 1$,
\[
\left|\pa_{\alo}^i\pa_{\beto}^j\RRR_0(\alo,\beto)\right|\leq C(i,j)\Theta^{-2},\qquad
\left|\pa_{\alo}^i\pa_{\beto}^j\RRR_1(\alo,\beto)\right|\leq C(i,j)\Tet^{(-1+3i+3j)/2},
\]
for some constants $C(i,j)$ independent of $\Theta$.

Moreover, for  $i,j\geq 0$, $k\geq 1$,
\[
|\pa^i_{\alo}\pa_{\beto}^j\pa_\sigma^k\LL^\geq|
 \leq C(i,j,k) \Tet ^{7/2+3(i+j)/2} \ex^{-\frac{2\tilde \nu \Tet^3}{3 L_0^3}}.
\]
where $C(i,j,k)$ is a constant independent of $\Tet$.
\end{proposition}
This proposition is proven in Appendix \ref{sec:Melnikov}.

The next theorem gives an asymptotic formula for the diference between the  unstable manifold of
the periodic orbit ${P}_{\eta_0, \xi_0}$ and the stable manifold of the periodic orbit ${P}_{\eta_0+\delta\eta, \xi_0+\de\xi}$.

\begin{theorem}\label{thm:MainSplitting}
Fix $L_0\in [1/2,2]$ and  $u_1, u_2$ such that $u_1> u_2 >0$.
Then, there exists $\Theta ^*\gg 1$ and $0<\varrho^*\ll 1$ such that for
$\Theta\ge \Theta ^*$, $(\eta_0, \xi_0)$
satisfying $\xi_0=\ol{\eta_0}$ and
$|\alo|\Theta^{3/2} \leq \varrho^*$ and $(\delta\eta,\delta\xi)$ satisfying $\de\xi=\ol{\de\eta}$ and $|\de\eta|\Theta^{3}\leq \varrho^*$,
the unstable manifold of
 ${P}_{\eta_0, \xi_0}$ and the stable manifold of
 ${P}_{\eta_0+\de\eta, \xi_0+\de\xi}$ can be parameterized as  graphs with respect to $(u,\lo)\in (u_1,u_2)\times\TT$ as in
\eqref{def:sigmaparamnousnr}
 for some functions $(Y^*, \Lambda^*, \al^*, \bet^*)$, $*=u,s$ which satisfy
%
\begin{equation}\label{def:DerivadesManifolds}
 |Y^*|\leq \Theta^{-3},\quad  |\La^*|\leq C\Theta^{-6}, \quad
 |\al^*|\leq \Theta^{-3}, \quad  |\bet^*|\leq C\Theta^{-3}\ln\Theta.
\end{equation}
Moreover, its difference satisfies
\begin{equation}\label{def:SplittingFormula}
\begin{pmatrix}
Y^u(u,\ga, z_0) -Y^s(u,\ga,  z_0,\de z) \\
\La^u(u,\ga, z_0)-\La^s(u,\ga,  z_0,\de z)\\
\al^u(u,\ga, z_0) -\al^s(u,\ga,  z_0,\de z)\\
\bet^u(u,\ga,  z_0) -\bet^s(u,\ga,  z_0,\de z)
\end{pmatrix}=\NNN(u,\ga, z_0,\de z)\begin{pmatrix}
\MM_Y(u,\ga, z_0,\de z)\\
 \MM_\La(u,\ga,  z_0,\de z)\\
\de\eta+ \MM_\al(u,\ga,  z_0,\de z)\\
\de\xi+ \MM_\bet(u,\ga,  z_0,\de z)
\end{pmatrix}
\end{equation}
where $z_0=(\eta_0,\xi_0)$, $\de z=(\de\eta,\de\xi)$,  $\NNN$ is a matrix which satisfies
\begin{equation}\label{def:matriuN}
\NNN=\Id+\OO(\Tet^{-3}\log \Tet)=\Id+\OO(\Tet^{-3})\quad \text{ and }\quad \NNN_{21}=\OO\left(\Tet^{-9/2}\log^2\Tet\right)
\end{equation}
%
and the vector $\MM$ is of the form
\begin{equation}\label{def:MelnikvoWithErrors}
\begin{pmatrix}
\MM_Y(u,\ga,  z_0,\de z)\\
\MM_\La(u,\ga,  z_0,\de z)\\
\MM_\al(u,\ga,  z_0,\de z)\\
\MM_\bet(u,\ga,  z_0,\de z)
\end{pmatrix}
=
\begin{pmatrix}
\omega \pa_\vm \LL(\ga-\omega u,  z_0)
+\OO\left(\Tet^{-6}\ln^2\Tet\right)\\
-\pa_\vm \LL(\ga-\omega u,  z_0)
+\OO\left( \ex^{-\frac{\tilde \nu \tte^3}{3 L_0^3}}\Tet^{-5/2}\ln^2\Tet\right)\\
-i\pa_{\beto}\LL(\ga-\omega u,  z_0)
+\OO\left(\Tet^{-6}\right)\\
i\pa_{\alo}\LL(\ga-\omega u,  z_0)
+\OO\left(\Tet^{-6}\right),
\end{pmatrix}
\end{equation}
%
%
where $\LL$ is the Melnikov potential introduced in \eqref{def:MelnikovPotential} and $\omega$ is given in \eqref{eq:omega}.
%

Moreover, the function $\MM$ satisfies the following estimates
\begin{equation}\label{def:DerivMelnikovAltres}
\begin{split}
  |\pa_{\alo}^i\pa_{\beto}^j\pa^k_\ga (\MM_Y+\omega\pa_{\sigma}\LL)|&\leq  C(i,j,k)\Tet^{-6}\\
|\pa_{\alo}^i\pa_{\beto}^j\pa^k_\ga (\MM_\al+i\pa_{\beto}\LL)|,  |\pa_{\alo}^i\pa_{\beto}^j\pa^k_\ga (\MM_\bet-i\pa_{\alo}\LL)|&\leq  C(i,j,k)\Tet^{-6}.
\end{split}
 \end{equation}
Furthermore, both for the derivatives of  the component $\MM_\Lambda$ and the $\gamma$-derivatives of the other components one has the following exponentially small estimates.
For any $i,j,k\geq 0$
\begin{equation}\label{def:DerivMelnikovLambda}
\begin{split}
|\pa_{\alo}^i\pa_{\beto}^j\pa^k_\ga (\MM_\La+\pa_\sigma\LL)|&\leq  C(i,j,k)\Tet^{-5/2+3(i+j)/2}\ex^{-\frac{\tilde \nu \tte^3}{3 L_0^3}}\log^2\Theta\\
|\pa_{\alo}^i\pa_{\beto}^j\pa^{k+1}_\ga (\MM_Y-\omega\pa_{\vm} \LL)|&\leq  C(i,j,k)\Tet^{2-3(i+j)/2}\ex^{-\frac{\tilde \nu \tte^3}{3 L_0^3}}\\
|\pa_{\alo}^i\pa_{\beto}^j\pa^{k+1}_\ga (\MM_\al+i\pa_{\alo} \LL)|,|\pa_{\alo}^i\pa_{\beto}^j\pa^{k+1}_\ga (\MM_\bet-i\pa_{\beto} \LL)|&\leq  C(i,j,k) \Tet^{1/2-3(i+j)/2}\ex^{-\frac{\tilde \nu \tte^3}{3 L_0^3}}.
\end{split}
\end{equation}
%
%
%
\end{theorem}

Note that for $\MM_Y$ the error is \eqref{def:MelnikvoWithErrors} and \eqref{def:DerivMelnikovAltres} bigger than the first order. Modifying slightly the matrix $\NNN$ this error could be made smaller. However, this is not needed. The reason is that, by conservation of energy one does not need to take care of the distance between the invariant manifolds on the $Y$ component.

\begin{remark}\label{rmk:fitesMtilde}
The estimates in Proposition \ref{prop:MelnikovPotential} and the bounds \eqref{def:DerivMelnikovAltres} imply the following estimates that are needed for analyzing the scattering maps in Section \ref{sec:scatteringstatements},
\begin{equation*}
\begin{split}
|D^N_{\alo,\beto}\MM_Y|&\leq C(N) \Tet^{-6}\\
|D^2_{\alo,\beto}\MM_x|&\lesssim \Tet^{-5}+\Tet^{-3}(|\eta_0|+|\xi_0|) \quad\text{for}\quad x=\al,\bet \\
|D^N_{\alo,\beto}\MM_x|&\leq C(N) \Tet^{-3} \quad\text{for }\quad x=\al,\bet\, \text{ and }\,N\geq  1, \, N\neq 2,
\end{split}
\end{equation*}
where $C(N)$ is a constant only depending on $N$.

Analogously Proposition \ref{prop:MelnikovPotential} and the bounds in \eqref{def:DerivMelnikovLambda} imply
\begin{equation*}
 |\pa^k_\ga\MM_\La|\leq  C(k)(\Tet^{-1/2}+\Tet^{-3/2}|\eta_0|)\ex^{-\frac{\tilde \nu \tte^3}{3 L_0^3}},\qquad
 |\pa_{\alo}\MM_\La|\lesssim \Tet^{3/2}\ex^{-\frac{\tilde \nu \tte^3}{3 L_0^3}},\qquad
 |\pa_{\beto}\MM_\La|\lesssim \Tet^{3/2}\ex^{-\frac{\tilde \nu \tte^3}{3 L_0^3}}
\end{equation*}
and, when $i+j\geq 1$,
\[
 |\pa_{\alo}^i\pa_{\beto}^j\pa^k_\ga\MM_\La|\leq C(i,j,k) \Tet^{3(i+j)/2}\ex^{-\frac{\tilde \nu \tte^3}{3 L_0^3}},\qquad
 \]
where $C(k)$ and $C(i,j,k)$ are independent of $\Theta$.

Finally, $i,j,k\geq 0$,
%
\begin{equation*}
\begin{split}
|\pa_{\alo}^i\pa_{\beto}^j\pa^{k+1}_\ga\MM_Y|&\leq C(i,j,k) \Tet^{3-3(i+j)/2}\ex^{-\frac{\tilde \nu \tte^3}{3 L_0^3}}\\
|\pa_{\alo}^i\pa_{\beto}^j\pa^{k+1}_\ga\MM_\al|,|\pa_{\alo}^i\pa_{\beto}^j\pa^{k+1}_\ga \MM_\bet|&\leq C(i,j,k) \Tet^{3/2-3(i+j)/2}\ex^{-\frac{\tilde \nu \tte^3}{3 L_0^3}}.
\end{split}
\end{equation*}
%
\end{remark}

Theorem \ref{thm:MainSplitting} is proved in several steps.
First, in Section \ref{sec:ParamInvManifolds}, we prove the existence of parameterizations of the form  \eqref{def:sigmaparamnousnr} for the invariant manifolds.
These parameterizations are analyzed in complex domains.
They fail  to exist at $u=0$ since at this point, when written in the original coordinates $(\lambda, L, \eta,\xi,\wt r,\wt y)$, the unperturbed invariant manifold is not a graph over the variables $\wt r, \la$.
Thus, in Section \ref{sec:InvManifold:ExtensionUnstable}, we extend the unstable
invariant manifolds using a different parameterization. This allows us to have
at the end a common domain (intersecting the real line) where both manifolds have graph parameterizations
as  \eqref{def:sigmaparamnousnr}.
Finally in Section \ref{sec:Difference} we analyze the
difference between the invariant manifolds and complete the proof of Theorem
\ref{thm:MainSplitting}.

\subsection{The scattering maps associated to the invariant manifolds of infinity}\label{sec:scatteringstatements}

Once we have asymptotic formulas \eqref{def:SplittingFormula} for the difference of the stable and unstable manifolds for nearby periodic orbits in $\EE_\infty$,  next step is to look for their intersections to find heteroclinic connections between different periodic orbits. This will allow us to define scattering maps in suitable domains of $\EE_\infty$.
Now we   provide two homoclinic channels and the two associated scattering maps whose domains inside $\EE_\infty$ overlap.
%
The construction of the homoclinic channels relies on certain non-degeneracies of the Melnikov potential analyzed in Proposition \ref{prop:MelnikovPotential}.
In particular, we need non-trivial dependence on the angle $\ga$.
If one analyzes the first $\gamma$-Fourier coefficient of the potential $\LL(\gamma-\omega u, \alo,\beto)$ given by \eqref{def:MelnikovPotential}, one can easily see that it vanishes at a  point of the form
\begin{equation*}
\overline{\xi_{\mathrm{bad}}}=\etabad= \frac{N_3}{24\sqrt{2}N_2}L_0^{5/2}\tte^{-2}+\OO\left(\tte^{-5/2}\right).
\end{equation*}
Therefore, we will be able to define scattering maps for $|\eta_0|\ll \Theta^{-3/2}$ and $\xi_0=\ol{\eta_0}$ (that is the domain considered in Theorem \ref{thm:MainSplitting} minus a small ball around the point $(\etabad, \ol{\etabad})$.

The main idea behind Theorem \ref{prop:scatteringmap} is the following:
We  fix a section $u=u^*$   and, for $(\alo,\beto)$ in the good domain $\wt\DD$ introduced below (see \eqref{def:domainscattering}), we  analyze the zeros of equations \eqref{def:SplittingFormula}, which lead to  two solutions
\[
\ga^j=\ga^j(u^*,\alo,\beto), \
\delta \xi^j=\delta \xi^j(u^*,\alo,\beto), \
\delta \eta^j=\delta \eta^j(u^*, \alo,\beto), \ j=1,2.
\]
These solutions provide two heteroclinic points through the parameterization \eqref{def:sigmaparamnousnr} as
\[
\begin{split}
z_{\mathrm{het}}^j &= z_{\mathrm{het}}^j (u^*,\eta_0,\xi_0) = \left(\la_{\mathrm{het}}, L_{\mathrm{het}}, \eta_{\mathrm{het}},\xi_{\mathrm{het}}, \wt r_{\mathrm{het}},\tilde y_{\mathrm{het}}\right)\\
&=
\left(\la, L^u, \eta^u,\xi^u, \wt r,\tilde y^u\right)(u^*,\ga ^j(u^*,\alo,\beto)) \in W^u(P_{\eta_0,\xi_0})\cap
W^s(P_{\eta_0+\delta\eta^j,\xi_0+\delta\xi^j}).
\end{split}
\]
%
Varying $(\eta_0,\xi_0)$ and $u$, one has two 3--dimensional homoclinic channels which define homoclinic manifolds to infinity. These channels are defined by
\begin{equation}\label{def:homochannels}
\Gamma^j=\left\{ z_{\mathrm{het}}^j (u,\eta_0,\xi_0):\,u\in (u_1,u_2), (\eta_0,\xi_0)\in\wt\DD \right\}
\end{equation}
and associated to these homoclinic channels one can define scattering maps  which are analyzed in the next theorem.

To define the domains of the scattering maps we introduce the notation
\begin{equation}\label{def:disk}
\DD_\rr(\eta_0,\xi_0)=\left\{w\in\RR^2: |(\eta,\xi)-(\eta_0,\xi_0)|<\rr \right\}.
\end{equation}


\begin{theorem}\label{prop:scatteringmap}
Assume that $m_0\neq m_1$. Fix $L_0\in [1/2,2]$ and $0<\varrho\ll \varrho^*$ where $\varrho^*$ is the constant introduced in Theorem \ref{thm:MainSplitting}. Then, there exists $\Theta ^*\gg 1$ such that, if $\Theta \geq\Theta ^*$, one can define scattering maps
\[
\wt  \SSS^j:\,   \TT\times\left[\frac12,1\right]\times \wt \DD \to  \TT\times\left[\frac12,1\right]\times\CC, \quad j=1,2
\]
where (see Figure \ref{fig:Discoscattmap})
\begin{equation}\label{def:domainscattering}
  \wt \DD=\DD_{\varrho\Theta^{-3/2}}(0,0)\setminus \DD_{\varrho\Theta^{-2}}(\etabad,\ol{\etabad}),
\end{equation}
associated to the homoclinic channels $\Gamma^j$ introduced in \eqref{def:homochannels}. These scattering maps are of the form
\[
\wt \SSS^j (\la, L_0, \eta_0,\xi_0)= \begin{pmatrix} \la +\Delta_j(\eta_0,\xi_0)\\L_0\\ \SSS^j (\eta_0,\xi_0)\end{pmatrix}
\]
where $\SSS^j$ is independent of $\la$ but may depend\footnote{To simplify the notation we omit the dependence of $\SSS^j$ on $L_0$.
In fact, from now one we will restrict the scattering map to a level of energy.
Since the energy determines $L_0$ it can be treated as a fixed constant} on $L_0$ and is given by
%
%
%
\begin{equation}\label{def:ScatteringFormulasP}
 \begin{split}
\SSS^j(\eta_0,\xi_0)=
\begin{pmatrix}
\eta_0-i
\wt \nu \pi L_0^4 (\tte+\eta_0\xi_0)^{-3}
\left[ A_1 \eta_0+2 A_2 \eta_0^2\xi_0+A_3 \tte ^{-2}
\right]\\
\xi_0+i\wt \nu \pi L_0^4 (\tte+\eta_0\xi_0)^{-3}
\left[ A_1 \xi_0+2 A_2 \eta_0\xi_0^2+A_3 \tte ^{-2}
\right]\\
\end{pmatrix}
+\RRR^j(\eta_0,\xi_0)
\end{split}
\end{equation}
where $\tte=\Tet-L_0$,
\begin{equation}\label{def:A1A2A3}
A_1= \frac{3N_2}{8L_0},\quad A_2=-\frac{3 N_2}{16 L_0^2},\quad  A_3= - N_3\frac{15}{\sqrt{2} 64} L_0^\frac{3}{2}.
\end{equation}
(see \eqref{def:massterm}) and $\RRR^j$ satisfies
\[
\RRR^j(\eta_0,\xi_0)=\OO\left(\Tet^{-5},\Tet^{-4}|\eta_0|\right)
\]
Moreover,
\begin{itemize}
\item $\SSS^j$ is symplectic in the sense that it preserves the symplectic form $d\eta_0\wedge d\xi_0$.
\item Fix $N\geq 3$. Then,  the derivatives of $\RRR^j$ satisfy
\[
|D^k\RRR^j(z) |\leq C(k) \Theta^{-5}, \,\,k=1\ldots N
\]
%
for $z\in \widetilde \DD$, where $C(k)$ is a constant which may depend on $k$ but is independent of $\Theta$.
\item  There exists points
$z_0^j=(\eta_0^j,\xi_0^j)$, $j=1,2$, satisfying
\begin{equation}\label{def:formulafixedscattering}
\xi_0^j=\overline{\eta_0^j}\qquad \text{and}\qquad\eta_0^j=\frac{5N_3}{8N_2}\frac{L_0^3}{\sqrt{2L_0}}\tte^{-2}+\OO\left(\Tet^{-3}\log^2\Tet\right)
\end{equation}
where $N_2$ and $N_3$ are the constants introduced in \eqref{def:massterm},
such that $\SSS^j(\eta_0^j,\xi_0^j)=(\eta_0^j,\xi_0^j)$. Furthermore, the distance between these two points is exponentially small as
%
\begin{equation}\label{def:DistanceHomosteor}
  \eta_0^2-\eta_0^1= \xi_0^2-\xi_0^1=-\frac{4}{\sqrt{\pi}}L_0^{1/2} \tte^{9/2} \ex^{-\frac{\tilde \nu \tte^3}{3 L_0^3}} \left(1+\OO\left(\Tet^{-1}\ln^ 2\Tet\right)\right).
\end{equation}
\end{itemize}
\end{theorem}
This theorem is proven in Section \ref{sec:existencescattering}.

\begin{figure}[h]
\begin{center}
\includegraphics[height=4cm]{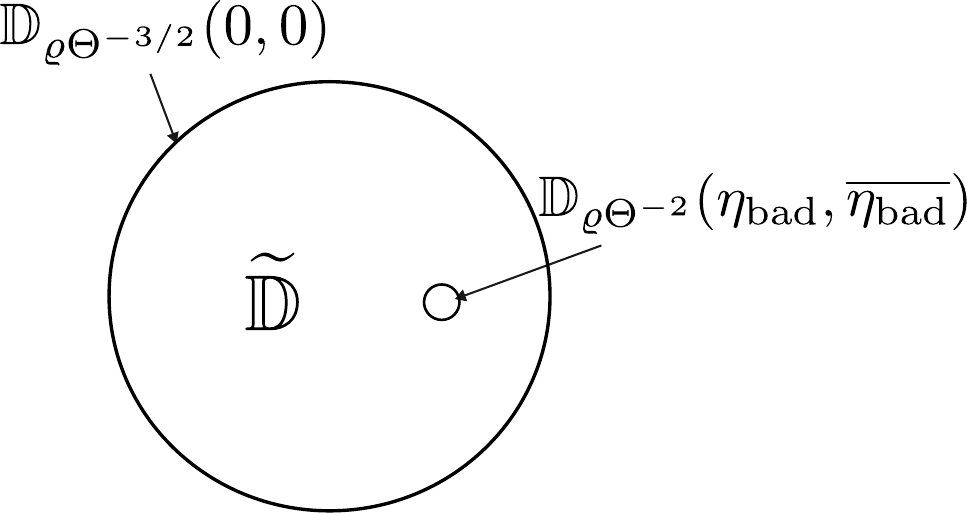}
\end{center}
\caption{The domain $\wt \DD$ in \eqref{def:domainscattering} of the scattering maps (see Theorem \ref{prop:scatteringmap}).}\label{fig:Discoscattmap}
\end{figure}

To analyze the return map from a neighborhood of infinity to itself along the homoclinic channels it is convenient to reduce the dimension of the model. To this end,  we apply the classical Poincar\'e-Cartan reduction. We fix the energy level $\wt \KK=-\frac{ \nu}{2L_0^2}$. Then, by the Implicit Function Theorem,  we know that there exists a function $\tL(\la,\eta,\xi,\wt r,\wt y;\Theta)$ satisfying
\[
 \wt \KK\left(\la,-\tL(\la,\eta,\xi,\wt r,\wt y),\eta,\xi,\wt r,\wt y\right) =-\frac{ \nu}{2L_0^2}.
\]
The function $\tL$ depends also on $\Theta$ and $L_0$ (which can be treated now as a parameter). We omit this dependence to simplify the notation.

The function $\tL$ can be seen as a Hamiltonian system of two and a half degrees of freedom with $\la$ as time. Then, it is well known that the trajectories of $\tL$  coincide with the trajectories of $\wt \KK$ at the energy level $\wt \KK=-\frac{ \nu}{2L_0^2}$ (up to time reparameterization).
%
%
From now on and, in particular, to analyze the return map from a neighborhood of infinity to itself we consider the flow given by the Hamiltonian $\tL$.

To this end, we need to compute the scattering maps associated  to the new Hamiltonian $\tL$ from the scattering maps  $\wt\SSS_j$ obtained in Theorem \ref{prop:scatteringmap} (restricted to the energy level $\wt \KK=-\frac{ \nu}{2L_0^2}$). Indeed, if we denote by
\[
\check{\SSS}^j:\, \TT\times \wt\DD \subset\TT\times\CC\to \TT\times\CC, \quad j=1,2
\]
the scattering maps associated to $\tL$, they are of the form
\[
\check{\SSS}^j (\la, \eta_0,\xi_0)= \begin{pmatrix}\la\\ \SSS^j (\eta_0,\xi_0)\end{pmatrix}.
\]
where $\SSS^j$ are the functions introduced in \eqref{def:ScatteringFormulasP}.
Note that the fact that $\la$  is now time, implies that the corresponding component in the scattering map is the identity. Nevertheless, as the $(\eta,\xi)$ coordinates of the periodic orbit in $\EE_\infty$ do not evolve in time, the associated components of  the scattering maps for the non-autonomous Hamiltonian $\tL$ and the original one $\wt\KK$ coincide.

%
%
To construct the invariant set for the return map, the first step is to construct an isolating block for a suitable (large) iterate of the scattering maps. To this end, the first step is to obtain properties of $\SSS^j$.
Theorem \ref{prop:scatteringmap}
implies the following proposition.

%

\begin{proposition}\label{prop:scatteringproperties}
Assume that $m_0\neq m_1$. The scattering maps $\SSS^j$ introduced in \eqref{def:ScatteringFormulasP} have the following properties.
\begin{itemize}
 \item
 Fix $N\geq 3$.
 For $j=1,2$, the expansion of $\SSS^j$ around the points $z_0^j$ obtained in Theorem \ref{prop:scatteringmap}. is of the form
 \[
  \SSS^j(z)=z_0^j+\mu_j (z-z_0^j)+\sum_{k=2}^NP_k\left(z-z_0^j\right)+\OO\left(z-z_0^j\right)^{N+1}
 \]
where $z=(\eta,\xi)$ and
%
%
\begin{equation}\label{def:VapScattering}
\mu_j=e^{i\omega_j}\qquad \text{ with }\quad \omega_j=\wt \nu \pi L_0^4 A_1 \tte^{-3}+\OO\left(\Tet^{-4}\right)
\end{equation}
%
and $P_k$ are homogeneous polynomials in
$\eta-\eta_0^j$ and
$\xi-\xi_0^j$
of degree $k$.
Moreover they satisfy
\[
 \begin{split}
 P_2(z)&=\sum_{i+j=2}b^2_{ij}\xi^i\eta^j\qquad \text{ with }\quad b^2_{ij}=\OO\left(\Tet^{-5}\right)\\
 P_3(z)&= \TTT\tte^{-3}|z|^2 z+\Tet^{-5}\OO\left(z^3\right)\qquad
  \end{split}
\]
where
\begin{equation}\label{def:TTT}
 \TTT=-4i\wt\nu\pi L_0^4A_2+\OO\left(\Theta^{-1}\right)
\end{equation}
(see \eqref{def:A1A2A3} and \eqref{def:massterm}) satisfies $\TTT\neq 0$ and the coefficients of $P_k$ for $k\geq 4$ are of order $\OO\left(\Tet^{-3}\right)$.
\end{itemize}
\end{proposition}

To construct symbolic dynamics in a neighborhood of the  infinity manifold and the associated homoclinic channels we rely on the usual isolating block construction to prove the existence of hyperbolic dynamics (see Figure \ref{fig:isolatingblockfigure}). To capture the hyperbolicity in the $(\eta,\xi)$-directions we must rely on the scattering maps. Indeed, in these directions the dynamics close to infinity is close to the identity up to higher order (see Theorem  \ref{prop:lambdalemma} and the heuristics in Section \ref{sec:Step2}) and, therefore, hyperbolicity can only be created through the dynamics along the invariant manifolds, which is encoded in the scattering maps. Note, however, that we have the scattering maps only defined in domains where the dynamics is close to a rotation around an elliptic fixed point (with a very small frequency). For this reason it is rather hard to obtain hyperbolicity for any one of them. Instead, we obtain it for a suitable high iterate of a combination of the scattering maps. This is stated in the next theorem, whose prove is deferred to Section \ref{sec:IsolatingBlockScattering}.

\begin{theorem}\label{thm:BlockScattering}
Assume that $m_0\neq m_1$, fix $L_0\in [1/2,2]$ and  take $\Theta\gg 1$ large enough.  Then, there exists $0<\wt\kk_0 \ll 1$  and a change of coordinates
\[
 \Upsilon: (-\wt\kk_0 ,\wt\kk_0)^2\to \widetilde\DD, \quad (\eta_0,\xi_0)=\Upsilon( \varphi,J)
\]
where $\widetilde\DD$ is the domain introduced in \eqref{def:domainscattering}, such that  the scattering maps
\[
 \wh \tS_i= \Upsilon^{-1}\circ \SSS_i\circ \Upsilon,
\]
where $\SSS_i$ have been introduced in Theorem \ref{prop:scatteringmap}, satisfy the following statements.
\begin{enumerate}
\item They are of the form
\begin{equation*}
 \wh  \tS^1(\varphi,J)=\begin{pmatrix}
                       \varphi+\wt B (J)+\OO\left(J^{2}\right)\\
                       J+\OO\left(J^{3}\right)
                      \end{pmatrix}
\qquad \text{and}\qquad
\wh\tS^2(\varphi,J)=\begin{pmatrix}\wh\tS^2_\varphi(\varphi,J)\\\wh\tS^2_J(\varphi,J)\end{pmatrix}
                      \end{equation*}
which satisfies
 \begin{equation*}
 \nu=\pa_\varphi\wh\tS^2_J(0,0)\neq 0 \qquad \text{and}\qquad \wh\tS^2_J(0,J)=0\quad \text{for}\quad J\in (-\wt\kk_0 ,\wt\kk_0).
\end{equation*}
\item  For any $0< \wt\kk\ll\wt\kk_0$, there exists $M=M(\wt\kk)$ such that the rectangle
\begin{equation}
\label{def:RRR}
 \RRR=\left\{ (\varphi,J):\, 0\leq \varphi\leq 2\nu\ii\tkk, \, 0\leq J\leq \tkk\right\}
\end{equation}
is an isolating block for  $\wh\tS=(\wh\tS^1)^M\circ \wh\tS^2$.
Namely, if one considers a $\CCC^1$ curve $J=\gamma(\varphi)$ with $\gamma:[0,2\nu\ii\tkk]\to \RR$ with $\ga(\varphi)\in [0, \tkk]$, then, its image $(\varphi_1(\varphi),J_1(\varphi))=\wh \tS(\varphi,\ga(\varphi))$ is a graph over its horizontal component and satisfies that
\[
 J_1(\varphi)\in (0,\tkk), \quad \varphi_1(0)<0\quad\text{and} \quad \varphi_1(2\nu\ii\tkk)>2\nu\ii\tkk.
\]
\item For  $z=(\varphi,J)\in\RRR$, the matrix $D\wh \tS(z)$ is hyperbolic with eigenvalues $\la_{\wh \tS}(z), \la_{\wh \tS}(z)\ii\in\RR$
with
\[
\la_{\wh \tS}(z)\gtrsim \tkk^{-1}
\]
Furthermore, there exist two vectors fields $V_j:\RRR\to T\RRR$ of the form
\[
 V_1=\begin{pmatrix}1\\ 0 \end{pmatrix},\qquad  V_2=\begin{pmatrix}V_{21}(z)\\ 1 \end{pmatrix}\quad \text{with}\quad|V_{21}(z)| \lesssim\tkk,
\]
which satisfy, for $z\in\RRR$,
\[
 \begin{split}
  D\wh \tS(z)V_1&=\la_{\wh \tS}(z)\left(V_1+\wh V_1(z)\right)\qquad \text{with} \qquad |\wh V_1(z)|\lesssim \tkk\\
 D\wh \tS(z)V_2(z)&=\la_{\wh \tS} (z)\ii\left(V_2(\wh \tS(z))+\wh V_2(z)\right)\qquad \text{with}  \qquad |\wh V_2(z)|\lesssim \tkk.
 \end{split}
\]
\end{enumerate}
\end{theorem}

\section{Local behavior close to infinity and a parabolic Lambda Lemma}\label{sec:LambdaLemma}

\subsection{McGehee coordinates}\label{sec:McGehee}

To study the behavior of Hamiltonian~\eqref{def:Ham:Poincarenr} close to the infinity manifold~$\EE_\infty$, we introduce the classical
McGehee coordinates $\wt r=2/x^2$. To simplify the notation, in this section we drop the tilde of $\wt y$. The  Hamiltonian $\wt\KK$ becomes
\[
\JJ(\la,L,\eta,\xi,x, y;\Tet)=-\frac{\nu}{2 L^2}+\frac{y^2}{2}+\frac{(\Theta-L+\eta\xi)^2}{2}\frac{x^4}{4}-\frac{x^2}{2}+ V (\la, L,\eta,\xi, x)
\]
%
where
\[
V(\la, L,\eta,\xi, x) = \wt W\left(\la, L,\eta,\xi, \frac{2}{x^2}\right) = \OO(x^6),
\]
and $\wt W$ is the potential in \eqref{def:potentialwtilde}, while the canonical symplectic form
$
d\la\wedge d L+id\eta\wedge d\xi+dr \wedge d y
$
is transformed into
\[
d\la\wedge d L+id\eta\wedge d\xi-\frac{4}{x^3} dx\wedge d y.
\]
Hence, the equations of motion associated to $\JJ$ are
\begin{equation*}
\begin{aligned}
\lambda' & =  \partial_L \JJ = \frac{\nu}{L^3} + \OO(x^4), &L' & = -  \partial_\la \JJ = \OO(x^6), \\
\eta' & = -i\partial_\xi \JJ= -\frac{i}{4}\left(\Theta-L+\eta\xi\right)  x^4 \eta + \OO(x^6), &
\xi' & =i\partial_\eta \JJ= \frac{i}{4}\left(\Theta-L+\eta\xi\right) x^4 \xi + \OO(x^6), \\
x' & = - \frac{x^3}{4}\partial_y \JJ = - \frac{1}{4} x^3 y, &
y' & = - \frac{x^3}{4}\left(-\partial_x \JJ\right) = - \frac{1}{4} x^4 +
\frac{(\Theta-L+\eta\xi)^2}{8}x^6 +\OO(x^8).
\end{aligned}
\end{equation*}
In the new variables, the periodic orbits $P_{\eta_0,\xi_0}$ in the energy level $\wt\KK=-\frac{\nu}{2L_0^2}$ become
\[
P_{\eta_0,\xi_0} = \{\lambda\in\TT, \;L=L_0,\; \eta= \eta_0,\; \xi = \xi_0,\, x = y =0\}
\]
for any $\eta_0,\xi_0 \in \CC$ with $|\eta_0|,|\xi_0|\geq L_0^{1/2}$ (see \eqref{def:PoincareVariablesnr}). To study the local behavior around $\EE_\infty$, we consider the new variables
\begin{equation}
\label{def:abinfinity}
a  = \eta \,e^{-i\left(\Theta-L+\eta\xi\right) y}
\qquad \text{ and }\qquad
b  = \xi \, e^{i\left(\Theta-L+\eta\xi\right) y}.
%
\end{equation}
The equations of motion become
\begin{equation*}
\label{eq:equationsofmotioninMcGeheestraightened}
\begin{aligned}
\lambda' & =\partial_L\JJ = \frac{\nu}{L^3} + \OO(x^4), &
L' & = -  \partial_\la \JJ = \OO(x^6), \\
a' & =  \OO(x^6), &
b' & =  \OO(x^6), \\
x' & = - \frac{1}{4} x^3 y, &
y' & = - \frac{1}{4} x^4 +
\frac{(\Theta-L+ab)^2}{8}x^6+\OO(x^8).
\end{aligned}
\end{equation*}
\begin{remark}\label{rmk:VariablesE}
Note that the change of coordinates \eqref{def:abinfinity} is the identity on $\EE_\infty$. Therefore, Theorem~\ref{prop:scatteringmap} is still valid in these coordinates.
\end{remark}

As we have done in Section~\ref{sec:varietatainfinit}, we restrict to the energy level $\JJ=-\frac{\nu}{2L_0^2}$ and express $L$ in terms of the rest of the variables in a neighborhood of $x=y=0$.
An immediate computation shows that
\begin{equation*}
L = L_0 +\frac{1}{2 \nu}(x^2-y^2) + \OO_2(x^2,y^2).
\end{equation*}
This function is even in $x$ and $y$.

Taking $\lambda$ as the new time and denoting the derivative with respect to this new time by a dot, we obtain the $2\pi$-periodic equation
\[
\begin{aligned}
\dot x & = - K_0 x^3 y \left( 1+B(x^2-y^2)+\OO_2(x^2,y^2)\right), \\
\dot y & =-K_0 x^4\left( 1-(4A(z)+B)x^2 -B y^2 + \OO_2(x^2,y^2)\right), \\
\dot z & = \OO(x^6),\\
\dot t&=1,
\end{aligned}
\]
where, abusing notation, we denote again by $t$ the new time $\lambda$,  $z = (a,b)$ belongs to a compact set and 
\begin{equation}
\label{def:AiBconstantsdeltermedegrau6}
\begin{aligned}
A(z) & = \frac{1}{8}\left(\Theta-L_0+ab\right)^2 =\frac{ \Theta^2}{8}\left(1-\frac{1}{\Theta}\left(L_0-ab\right)\right)^2, \\
B & = \frac{3}{2} \frac{1}{L_0  \nu},\qquad K_0 = \frac{L_0^3}{ 4\nu}.
\end{aligned}
\end{equation}
Observe that, since $(a,b)$ belong to a  compact set, taking $\Theta$ large enough, we can assume that
\begin{equation*}
A  = \frac{ \Theta^2}{8} \left( 1-\frac{1}{\Theta}\left(L_0-ab\right)\right)^2
> \frac{1}{16} \Theta^2 >  0.
\end{equation*}
Scaling $x$ and $y$ by $K_0^{1/3}$, one obtains the following non-autonomous $2\pi$-periodic in time equation
\begin{equation}
\label{def:systematinfinityold}
\begin{aligned}
\dot x & = -  x^3 y(1+ B x^2-B y^2 + R_1(x,y,z,t)), \\
\dot y & = -  x^4(1+ (B-4A)x^2- B y^2 + R_2(x,y,z,t)), \\
\dot { z} & =  R_3(x,y,z,t),\\
\dot t&=1,
\end{aligned}
\end{equation}
where we have kept the notation $x$, $y$ for the scaled variables and  $A$,$B$ for the scaled constants.
Moreover,
\begin{enumerate}
\item
the functions $R_i$, $i=1,2,3$, are even in $x$,
\item $R_3(x,y,z,t) = \OO_3(x^2)$ and
$
R_i (x,y,z,t) =   \OO_2(x^2,y^2),\quad i=1,2$.
\end{enumerate}


The periodic orbit $P_{\eta_0,\xi_0}$ becomes $P_{\eta_0,\xi_0} = \{x = y =0,\; b= \eta_0,\; a = \xi_0,\; t\in\TT\}$.

We now apply the change  of variables $(x,y,z) = (x,y,\Upsilon (\varphi,J))$, where $\Upsilon$ is given in Theorem~\ref{thm:BlockScattering}.
The transformed equation has the same form as in~\eqref{def:systematinfinityold} with statements 1 and 2 above. From now on, we will assume $z = (\varphi,J)$. In particular, the scattering maps associated to the infinity manifold $\{x=y=0\}$ will satisfy the properties of Theorem~\ref{thm:BlockScattering}.

\subsection{$C^1$ behavior close to infinity}

To study the local behavior of system~\eqref{def:systematinfinityold} close to $\EE_\infty$ we start by finding a suitable set of coordinates,
 provided by the next theorem, whose proof is deferred to Appendix~\ref{app:proofofprop:coordinatesatinfinity}.


\begin{theorem}\label{prop:coordinatesatinfinity}
Let $K\subset \RR^2$ be a compact set. For any $N\ge 1$, there exists a neighborhood $U$ of the subset of $\EE_\infty$, $\EE_\infty^K=\cup_{z_0\in K}  P_{z_0}$, in $\RR^2 \times \RR^2 \times \TT$ and a $\CCC^N$ change of variables
\[
\Phi: (x,y,z,t)\in U  \mapsto (q,p,\tilde z, t) = (A (x,y)^\top, z,t) + \OO_2(x,y)
\]
with $A$ a constant matrix,
that transforms system~\eqref{def:systematinfinityold} into
\begin{equation}
\label{eq:infinitymanifoldsstraightened}
\begin{aligned}
\dot q & = q ((q+p)^3+\OO_4(q,p)), & \dot{\tilde  z} & = q^N p^N \OO_{4}(q,p), \\
\dot p & = -p ((q+p)^3+\OO_4(q,p)), & \dot t & = 1.
\end{aligned}
\end{equation}
\end{theorem}

\begin{remark}
It is worth to remark that the change of variables in Theorem~\ref{prop:coordinatesatinfinity} is analytic
in some complex sectorial domain of $\RR^2\times\RR^2\times\TT$ with $\EE^K_\infty$ in its vertex. This claim is made precise in the proof of the Theorem \ref{prop:coordinatesatinfinity}.
\end{remark}

To simplify the notation, we drop the tilde from the new $z$ variable. Let $N>10$ be fixed.
We are interested in the behavior of system~\eqref{eq:infinitymanifoldsstraightened} in the region $\Phi(U) \cap\{q+p\ge 0\}$.
The  stable and unstable invariant manifolds
of $P_{z_0}$ in this region are, respectively,
$W^s(P_{z_0}) = \{q = 0, \; p >0,\;  z = z_0, \; t\in \TT\}$  and $W^u(P_{z_0}) = \{p = 0,\; q >0,\;  z = z_0, \; t\in \TT\}$.
%
Even if  the invariant manifold $\EE_\infty^K$ is not a normally hyperbolic invariant manifold it behaves as such and its  smooth stable and unstable invariant manifolds
\[
 W^\ast(\EE_\infty^K)=\bigcup_{z_0\in K}W^\ast(P_{z_0}),\quad \ast=u,s.
\]
Moreover, the invariant manifolds  $W^\ast(\EE_\infty^K)$ are foliated, as in the normally hyperbolic setting, by stable and unstable leaves $W_{w_0}^\ast$, $\ast=u,s$, which are defined as follows:
$w\in W_{w_0}^s$ if and only if
\[
 |\varphi_s(w)-\varphi_s(w_0)|\to 0\qquad \text{as}\qquad  s\to +\infty
\]
and analogously for the unstable leaves with backward time.


This allows us to  define the classical wave maps associated to the stable and unstable foliations, which we denote by $\Omega^s$ and $\Omega^u$, as
\begin{equation}\label{def:WaveMaps}
 \Omega^\ast(w)=w_0 \quad\text{if and only if}\quad w\in W_{w_0}^\ast\quad\text{ for }\ast=u,s.
\end{equation}
Observe that in the local coordinates given by Theorem \ref{prop:coordinatesatinfinity}, one has that locally
$W^s(\EE_\infty^K)=\{q=0\}$  and  $W^u(\EE_\infty^K)=\{p=0\}$. Moreover
\[
 \Omega^s(0,p,z,t)=(0,0,z,t)\qquad \text{and}\qquad  \Omega^s(q,0,z,t)=(0,0,z,t).
\]

The next step is to prove a Lambda Lemma that will describe the local dynamics close to $\EE_\infty^K$, in the coordinates given by Theorem \ref{prop:coordinatesatinfinity}.

Note that the particular form and invariance  of  $W^\ast(\EE_\infty^K)$,$\ast=u,s$, implies that the solution $\varphi_{s}(w_0)$ through any point $w_0 = (q_0,p_0,z_0,t_0) \in \Phi(U) \cap \{ q >0, \; p >0\}$ at $s=0$
satisfies  $\varphi_{s}(w_0) \in \Phi(U) \cap \{ q >0, \; p >0\}$ for all $s$ such that $\varphi_{s}(w_0) \in \Phi(U)$.
We define, then,
\begin{equation*}
V_\rho = \{(q,p)\mid |q|,|p| < \rho,\; q >0, \; p >0\}.
\end{equation*}
Let $W \subset \RR^2$ be an bounded open set.
Given $0<\delta < a \leq\rho$ and $\wt W \subset W$, we define the sections
\begin{equation}
\label{def:sectionsLambdaadeltapm}
\begin{aligned}
\Lambda_{a,\delta}^- (\wt W) & = \{(q,p,z,t) \in V_{\rho} \times \wt W \times \TT\mid p = a, \;0 < q < \delta\}, \\
\Lambda_{a,\delta}^+  (\wt W) & = \{(q,p,z,t) \in V_{\rho} \times \wt W \times \TT\mid q = a, \; 0 < p < \delta \}.
\end{aligned}
\end{equation}
and the associated Poincar\'e map
\begin{equation}\label{def:localmapLambda}
\Psi_\loc: \Lambda_{a,\delta}^- (\wt W)\longrightarrow\Lambda_{a,\delta}^+  (W)
\end{equation} induced by the flow of~\eqref{eq:infinitymanifoldsstraightened}, wherever it is well defined.

\begin{theorem}
\label{prop:lambdalemma}
Assume $N>10$ in system~\eqref{eq:infinitymanifoldsstraightened}.
Let $K \subset W$ be a compact set. There exists $0< \rho <1 $ and $C>0$, satisfying $C \rho <3/5$,   such that, for any $0<a\leq\rho$ and any $\delta\in (0,a/2)$, the Poincar\'e map $\Psi_\loc: \Lambda_{a,\delta}^- (K)\to \Lambda_{a,\delta^{1-C a}}^+  (W)$ associated to system~\eqref{eq:infinitymanifoldsstraightened} is well defined. Moreover, $\Psi_\loc$ satisfies the  following.
\begin{enumerate}
\item
There exist $\wt C_1, \wt C_2 >0$ such that, for any $(q,a,z_0,t_0) \in \Lambda_{a,\delta}^- (K)$,
$\Psi_\loc(q,a,z_0,t_0) = (a,p_1,z_1,t_1)$ satisfies
\begin{equation*}
\begin{aligned}
q^{1+Ca} \le  p_1 & \le q^{1-C a}, \\
 |z_1 - z_0 | & \le \frac{1}{2N} a^{N(1+C a)} q^{N(1-C a)}, \\
\wt C_1 q^{-3(1-C a)/2} \le t_1 -t_0 & \le \wt C_2 q^{-3(1+C a)/2}.
\end{aligned}
\end{equation*}
\item Fix any $M>0$. Then, there exists $\de_0$ and $\wt C_3>0$, such that for any $\de\in (0,\de_0)$ and $\gamma:[0,\delta) \to \overline{V_{\rho}} \times W \times \TT$, a $\CCC^1$ curve with $\gamma((0,\delta)) \subset \Lambda_{a,\delta}^- (K)$, of has the form $\gamma(q) = (q,a,z_0(q),t_0(q))$ and
satisfying $\|\ga\|_{\CCC^1}\leq M$ the folllowing is true. Its image $\Psi_\loc (\gamma(q)) = (a, p_1(q),z_1(q),t_1(q))$
 satisfies
 \begin{equation*}
\left|p_1'(q)\right|\leq \wt C_3,\qquad \left|\frac{p_1'(q)}{t_1'(q)}\right|\le \wt C_3 q^{1-Ca},\qquad  \left|z_1'(q)\right| \le \wt C_3, \qquad
|t_1'(q)|  \ge \wt C_3 \frac{1}{q^{3/5-Ca}}.
\end{equation*}
\item There exists $\wt C_4>0$ such that, if $\gamma:[0,1] \to \Lambda_{a,\delta}^- (K)$, a $\CCC^1$ curve,  has the form $\gamma(u) = (q_0(u),a,z_0(u),t_0(u))$,  then
$\Psi_\loc (\gamma(u)) = (a, p_1(u),z_1(u),t_1(u))$ satisfies, for all $u\in [0,1]$,
\[
|z_1'(u) - z_0'(u)| \le \wt C_4 \|\gamma'(u)\| q_0(u)^{N-10}.
\]
\item Fix any $M>0$. Then, there exists $\de_0$  such that for any $\de\in (0,\de_0)$,
 any $\CCC^1$ curve $(z_0(u),t_0(u)) \in K \times \TT$, $u \in [0,1]$, satisfying  $\|(z_0(u),t_0(u)) \|_{\CCC^1}\leq M$ and  any $\tilde q_0 \in (0,\delta)$,  there exists $q_0:[0,1] \to (0,\delta)$, $q_0(0) = \tilde q_0$, $|q_0'(0) |< \tilde q_0^{1/5}$,
such that $\Psi_\loc (q_0(u),a,z_0(u),t_0(u)) = (a,p_1(u),z_1(u),t_1(u))$ satisfies
\[
\begin{aligned}
|p_1'(0)| & \le  \wt C_5 \tilde q_0^{\frac{3}{5}-Ca}\|(q_0'(0),z_0'(0),t_0'(0))\|, \\
|z_1'(0)-z_0'(0)| & \le \wt C_5 \tilde q_0 \|(q_0'(0),z_0'(0),t_0'(0))\|, \\
t_1'(0) & = t_0'(0),
\end{aligned}
\]
for some $\wt C_5>0$ independent of the curve and $\delta$.
\end{enumerate}
\end{theorem}

The proof of Theorem~\ref{prop:lambdalemma} is deferred to Section \ref{sec:LambdaLemmaTechnical}.

\section{Construction of the hyperbolic set}\label{sec:IsolatingBlock}

The final step in constructing the Smale horseshoe for the 3 Body Problem given by the Hamiltonian \eqref{def:Ham:Poincarenr} is to analyze the dynamics in the vicinity of the disk  $\EE_\infty^K$ at infinity (see Theorem \ref{prop:coordinatesatinfinity}) and their invariant manifolds analyzed in Theorem \ref{thm:MainSplitting}.

\subsection{The return map}\label{sec:locandglobmaps}
We construct a return map in a suitable section transverse to the invariant manifolds of the form \eqref{def:sectionsLambdaadeltapm}. This map is built as a composition of the local map (close to infinity) studied in Theorem \ref{prop:lambdalemma} (see also Figure \ref{fig:LambdaLemma}), and a global map (along the invariant manifolds), which we analyze now. To build the horseshoe we will have to consider a suitable high iterate of the return map. To be more precise, we consider different return maps associated to two different homoclinic channels (and therefore, associated to two different scattering maps obtained in Theorem \ref{prop:scatteringmap}).

To define these return maps, we consider the sections $\Lambda _{a,\delta}^\pm$ given in \eqref{def:sectionsLambdaadeltapm}, which are transverse to the stable/unstable invariant manifolds, and we call
\begin{equation}\label{def:Sections:0}
\begin{split}
\Sigma_1&\equiv \Lambda _{a,\de}^-(K) =\{p=a, 0<q<\de , t \in\TT, z\in K\},\qquad \\
\Sigma_2&\equiv
\Lambda _{a,\de^{1-Ca}}^+(K) =\{q=a, 0<p< \de^{1-C a}, t \in\TT,z\in K\},
\end{split}
\end{equation}
where  $(q,p,\ga,z)$ are the coordinates defined by Theorem \ref{prop:coordinatesatinfinity}, $K \subset \RR^2$ is a compact set and we take
\[
 \de<\frac{a}{2}\qquad \text{and}\qquad a\leq \rr.
\]
Theorem~\ref{prop:lambdalemma}) ensures that there exists $C>0$ such that the local map
\[
\Psil: {\Sigma_1}  \to {\Sigma_2}
\]
(see \eqref{def:localmapLambda}).
is well defined.

The global maps will be defined from suitable open sets in $\Sigma_2$ to $\Sigma_1$.
 They are defined as the maps induced by Hamiltonian~\eqref{def:Ham:Poincarenr} expressed in the coordinates given by Proposition \ref{prop:coordinatesatinfinity}.
In fact, to construct them, we use slightly different coordinates which are defined in suitable coordinates in neighborhoods of the homoclinic channels at $\Sigma_2$ to $\Sigma_1$.
Let $W^{u,s}$ be the invariant unstable and stable manifolds of $q=p=0$. Theorem~\ref{prop:scatteringmap} ensures that $W^u$ and $W^s$ intersect transversely along two homoclinic channels, $\Gamma_i$, $i=1,2$. We use this transverse intersections to define the new system of coordinates.
\begin{enumerate}
\item
In the coordinates\footnote{Note that we have reordered the variables. The reason is that, in the section $\Sigma_2$, the variable $t$ will play a similar role as the variable $p$ whereas the variable $z$ is treated as a center variable.}  $(p,t,z)$ in $\Sigma_2$, $\Sigma_2 \cap W^u = \{p =0\}$.
Since we are in a perturbative setting (when $\Theta$ is large enough), we have that $W^s \cap \Sigma_2 = \{p = w^s(t,z)\}$.
Moreover, the intersection of the homoclinic channels  $\Gamma_i$ with $\Sigma_2$, $i=1,2$, given by Theorem~\ref{prop:scatteringmap} can be parametrized as  $\{(p,t,z) = (0, t_i^s(z),z)\}$.
In particular, the functions $t_i^s$ satisfy
\[
w^s(t_i^s(z),z) = 0, \qquad \partial_t w^s(t_i^s(z),z) \neq 0.
\]
Hence, the equation $p = w^s(t,z)$ defines in the neighborhood of each homoclinic channel two functions $\tilde w_i^s(p,z)$, satisfying $\tilde w_i^s(0,z) = 0$, such that
\[
p = w^s(t,z) \qquad \Longleftrightarrow \qquad t = t_i^s(z) +\tilde   w_i^s(p,z),
\]
in a neighborhood of $\Gamma_i \cap \Sigma_2$, $i=1,2$.
That is, $(p, t_i^s(z) +\tilde  w_i^s(p,z),z)$ parametrizes
$W^s \cap \Sigma_2$ in a neighborhood of $\Gamma_i\cap \Sigma_2$.
We define two new sets of coordinates in $\Sigma_2$,
defined in a neighborhood of $\Gamma_i \cap \Sigma_2$,
\begin{equation}
\label{def:coordenadeslocalsaSigma2}
(p,\tau,z) = A_i(p,t,z) = (p,t-t_i^s(z) - \tilde w_i^s(p,z),z), \qquad i =1,2.
\end{equation}
In these coordinates, $W^s \cap \Sigma_2$ in each of the neighborhoods of $\Gamma_i$ is given by $\tau = 0$.
\item
We proceed analogously in $\Sigma_1$. In the coordinates $(q,\gamma,z)$ in $\Sigma_1$, $\Sigma_1 \cap W^s = \{q =0\}$,  $W^u \cap \Sigma_1 = \{q = w^u(t,z)\}$ for some function $w^u$ and the intersection of the homoclinic channels $\Gamma_i$, $i=1,2$, with $\Sigma_1$ are given by $\{(q,t,z) = (0,t_i^u(z),z)\}$ for some functions $t_i^u$.
In particular,
\[
w^u(t_i^u(z),z) = 0, \qquad \partial_t w^u(t_i^u(z),z) \neq 0.
\]
Hence, the equation $q = w^u(t,z)$ can be inverted in the neighborhood of $\Gamma_i \cap \Sigma_1$, $i=1,2$, by defining  two functions $\tilde w_i^u(q,z)$ satisfying $\tilde w_i^u(0,z) = 0$, such that
\[
q = w^u(t,z) \qquad \Longleftrightarrow \qquad t = t_i^u(z) + \tilde w_i^u(q,z).
\]
That is, $(q, t_i^u(z) + \tilde w_i^u(q,z), z)$ parametrizes
$W^u \cap \Sigma_1$ in a neighborhood of $\Gamma_i\cap \Sigma_1$.
We define two new sets of coordinates in $\Sigma_1$,
defined in a neighborhood of $\Gamma_i \cap \Sigma_1$,
\begin{equation}
\label{def:coordenadeslocalsaSigma1}
(q,\sigma,z) = B_i(q,t,z) = (q,t-t_i^u(z) - \tilde w_i^u(q,z),z), \qquad i =1,2.
\end{equation}
In these coordinates, $W^u \cap \Sigma_1$ in each of the neighborhoods of $\Gamma_i$ is given by $\sigma = 0$.
\end{enumerate}

Let $\Psig{i}$, $i=1,2$, be the two \emph{global} maps from a neighborhood of $\Gamma_i \cap \Sigma_2$ in $\Sigma_2$, which we denote by $U_i^2$, to a neighborhood of $\Gamma_i \cap \Sigma_1$ in $\Sigma_1$, which we denote by $U_i^1$, defined by the flow.
Choosing the coordinates $(p,\tau,z)$ in $U_i^2$ and $(q,\sigma,z)$ in $U_i^1$, given by $A_i$ and $B_i$, respectively, we have that:

For points $(p,\tau,z)=(0,0,z) \in \Gamma_i\cap \Sigma_2$, the global map $\Psig{i}$ map is given by:
\begin{itemize}
\item
Compute $(0,0,\hat t, z)=\Omega^u(a,0,t_i^s(z),z)$, where $\Omega^u$ is the wave map introduced in \eqref{def:WaveMaps}.
\item
Compute $(0,0, \hat t, \wh \tS_i(z))$, where $\wh \tS_i$ is the scattering map introduced in Theorem \ref{prop:scatteringmap}
\item
Compute $(\Omega^s)^{-1}(0,0, \hat t, \wh \tS_i(z))=(0,a, \tilde t, \wh \tS_i (z))$ (see in \eqref{def:WaveMaps}).
\item
Finally
in coordinates $(q,\sigma,z)$ this last point becomes $(0,0,\wh \tS_i (z))$
\item
$\wtPsig{i}(0,0,z)=(0,0,\wh \tS_i (z))$
\end{itemize}
Using this fact and the changes of coordinates in \eqref{def:coordenadeslocalsaSigma2}, \eqref{def:coordenadeslocalsaSigma1} we have that
\begin{equation}
\label{def:globalmapiinlocalcoordinates}
\begin{pmatrix}
q_1 \\ \sigma_1 \\ z_1
\end{pmatrix} =
\wtPsig{i}(p,\tau,z) = B_i \circ \Psig{i} \circ A_i^{-1}(p,\tau,z) =
\begin{pmatrix}
\tau \nu_1^i(z) ( 1+ \OO_1(p,\tau)) \\
p \nu_2^i(z)(1+\OO_1(p,\tau)) \\
\wh \tS_i(z) + \OO_1(p,\tau)
\end{pmatrix},
\end{equation}
where $\nu_1^i(z) \nu_2^i(z) \neq 0$.
 Indeed,
the claim follows from the fact that $\Psig{i}(\{p=0\}\cap U_i^2) = \Psig{i} (W^u \cap \Sigma_2 \cap U_i^2) =
W^u \cap \Sigma_1 \cap U_i^1 = \{\sigma = 0\}\cap U_i^1$, that
$\Psig{i}(\{\tau=0\}\cap U_i^2) = \Psig{i} (W^s \cap \Sigma_2 \cap U_i^2) =
W^s \cap \Sigma_1 \cap U_i^1 = \{q = 0\}\cap U_i^1$ and expanding around $(0,0,z)$. The fact that $\nu_1^i(z) \nu_2^i(z) \neq 0$ follows immediately from the fact that $\Psig{i}$ are diffeomorphisms.
It is then immediate that
$\wtPsig{i}^{-1} : U_i^1 \to U_i^2$ has the form
\begin{equation}
\label{def:inverseglobalmapiinlocalcoordinates}
\wtPsig{i}^{-1}(q,\sigma,z) = A_i \circ \Psig{i}^{-1} \circ B_i^{-1}(q,\sigma,z) =
\begin{pmatrix}
\sigma \mu_1^i(z) ( 1+ \OO_1(q,\sigma)) \\
q \mu_2^i(z)(1+\OO_1(q,\sigma)) \\
\wh \tS_i ^{-1}(z) + \OO_1(q,\sigma)
\end{pmatrix},
\end{equation}
where $\mu_1^i(z) \mu_2^i(z) \neq 0$.

Notice that Theorem~\ref{prop:lambdalemma} implies that, for $1\le i,j\le 2$, $\Psil(U^1_i) \cap U^2_j\neq \emptyset$.
We will denote by $\Psiloc{i}{j} = {\Psil}_{\mid \Psil^{-1}(U^2_j)\cap U^1_i} : U^1_i \to U^2_j$ and its expression in coordinates
\[
\wtPsiloc{i}{j} = A_j \circ \Psiloc{i}{j} \circ B_i^{-1}.
\]
Observe that the map $\Psiloc{i}{j}$ does not depend on $i$ and $j$ and the dependence of $\wtPsiloc{i}{j}$ on $i$ and $j$ is only through the systems of coordinates $A_j$ and $B_i$.

The combination of the global maps along the homoclinic channels and the local map allows to define four different  maps
$\Psi_{ij} : U^2_i \to U^2_j$ by setting $\Psi_{i,j} = \Psiloc{i}{j} \circ \Psig{i}$. We will denote its expression in coordinates as
\begin{equation}
\label{def:wtPsiij}
\wt \Psi_{i,j} = \wtPsiloc{i}{j} \circ \wtPsig{i}.
\end{equation}

Let us specify the domains we will consider. Given $\de \in(0,a/2)$, let $\QQQ_\de^i \subset \Sigma_2$ be the set
\begin{equation}
\label{def:Qrho}
\QQQ_\de^i = \{0 < p < \de, \; 0 < \tau < \de, \; z \in \RRR\},
\end{equation}
where $(p,\tau,z)$ are the coordinates $A_i$ defined in the neighborhood of $\Gamma_i  \cap \Sigma_2$, called $U^2_i$,$i=1,2$,
and $\RRR$ was introduced in~\eqref{def:RRR}. We remark that the ``sides'' $\{p=0\}$ and $\{\tau = 0\}$ of $\QQQ_\de^i$ are $W^u \cap \Sigma_2$ and $W^s \cap \Sigma_2$, respectively, and
the ``vertex'' $\{p=\tau = 0\}$ is $\Gamma_i \cap \Sigma_2$.

Let $\Psi$ be the map formally defined as
\begin{equation}
\label{def:Psi}
\Psi = \Psi_{1,2} \circ \Psi_{1,1}^{M-1} \circ \Psi_{2,1},
\end{equation}
where $M$ is given by~\eqref{def:M}. We will denote by $\wt \Psi$ its expression in the $A_2$ coordinate system, that is,
\begin{equation}
\label{def:Psitilde}
\wt \Psi =\wt \Psi_{1,2} \circ \wt \Psi_{1,1}^{M-1} \circ \wt \Psi_{2,1}: \QQ_\de^2\longrightarrow \Sigma_2.
\end{equation}

\subsection{Symbolic dynamics: conjugation with the shift}

We consider in $\QQQ_\de^2$, defined in~\eqref{def:Qrho}, the set of coordinates $(p,\tau,\varphi,J)$ given by $A_2$ and
Theorem~\ref{thm:BlockScattering}. The coordinates have been chosen in such a way that $(\tau,\varphi)$ variables are ``expanding'' by $\wt \Psi$, while
the $(p,J)$ variables are ``contracting''. To formalize this idea, we introduce the classical concepts of \emph{vertical} and \emph{horizontal rectangles}  in our setting (see Figure \ref{fig:isolatingtotal}) as well as \emph{cone fields} (see~\cite{Smale65,Moser01,Wiggins90,KatokH95}).

We will say that $H \subset \QQQ_\de^2$ is a \emph{horizontal rectangle} if
\begin{equation}
\label{def:rectanglehoritzontal}
H = \{(p,\tau,\varphi,J) \in \QQQ_\de^2; \;h_1^-(\tau,\varphi) \le p \le h_1^+(\tau,\varphi),\;
h_2^-(\tau,\varphi) \le J \le h_2^+(\tau,\varphi)\},
\end{equation}
where $h_i^\pm : (0,\de) \times (0,\tkk) \to (0,\de) \times (0,\tkk)$, $i=1,2$, are $\ell_h$-Lipschitz. Analogously,
$V \subset \QQQ_\de^2$ is a \emph{vertical rectangle} if
\begin{equation}
\label{def:rectanglevertical}
V = \{(p,\tau,\varphi,J) \in \QQQ_\de^2; \;v_1^-(p,J) \le \tau \le v_1^+(p,J),\;
v_2^-(p,J) \le \varphi \le v_2^+(p,J)\},
\end{equation}
with $\ell_v$-Lipschitz functions $v_i^\pm : (0,\de) \times (0,\tkk) \to (0,\de) \times (0,\tkk)$, $i=1,2$.

If $H$ is the horizontal rectangle~\eqref{def:rectanglehoritzontal}, we split $\partial H = \partial_h H \cup \partial_v H$ as
\[
\begin{aligned}
\partial_h H & = \{\omega\in \QQQ_\de^2; \; (p,J) = (h_1^-,h_2^-)(\tau,\varphi) \text{ or } (p,J) = (h_1^+,h_2^+)(\tau,\varphi) \}, \\
\partial_v H & = \{\omega\in \QQQ_\de^2; \; (\tau,\varphi)=(0,0) \text{ or } (\tau,\varphi)=(\de,\tkk) \}
\end{aligned}
\]
and, analogously,
if $V$ is the vertical rectangle~\eqref{def:rectanglevertical}, we split $\partial V = \partial_h V \cup \partial_v V$ as
\[
\begin{aligned}
\partial_h V & = \{\omega\in \QQQ_\de^2; \; (p,J)=(0,0) \text{ or } (p,J)=(\de,\tkk) \}, \\
\partial_v V & = \{\omega\in \QQQ_\de^2; \; (\tau,\varphi) = (v_1^-,v_2^-)(p,J) \text{ or } (\tau,\varphi) = (v_1^+,v_2^+)(p,J) \}.
\end{aligned}
\]

Additionally, we define the \emph{stable and unstable cone fields} in the following way. For $\omega \in \QQQ_\de^2$, we consider in $T_\omega \QQQ_\de^2$ the basis given by the coordinates $(p,\tau,\varphi,J)$ and write $x \in T_\omega \QQQ_\de^2$ as $x = (x_u,x_s)$ meaning
$x = x_{s,p} \frac{\partial}{\partial p} + x_{u,\tau} \frac{\partial}{\partial \tau} + x_{u,\varphi} \frac{\partial}{\partial \varphi} +
x_{s,J} \frac{\partial}{\partial J}$. We define $\|x_u\| = \max\{|x_{u,\tau}|,|x_{u,\varphi}|\}$ and $\|x_s\| = \max\{|x_{s,p}|,|x_{s,J}|\}$. Then, a \emph{$\kappa_s$-stable cone} at $\omega \in \QQQ_\de^2$ is
\begin{equation*}
S^s_{\omega,\kappa_s} = \{x \in T_\omega \QQQ_\de^2; \; \|x_u\| \le \kappa_s \|x_s\| \}
\end{equation*}
and a \emph{$\kappa_u$-unstable cone} at $\omega \in \QQQ_\de^2$,
\begin{equation}
\label{def:coninestable}
S^u_{\omega,\kappa_u} = \{x \in T_\omega \QQQ_\de^2; \;\|x_s\| \le \kappa_u \|x_u\| \}.
\end{equation}

Having in mind~\cite{Wiggins90} (see also~\cite{Moser01}), we introduce the following hypotheses. Let $F:\QQQ_\de^2 \to \RR^4$ be a $C^1$
diffeomorphism onto its image.
\begin{enumerate}
\item[\fH]
There exists two families $\{H_n\}_{n\in \NN}$, $\{V_n\}_{n\in \NN}$ of horizontal and vertical rectangles in $\QQQ_\de^2$, with $\ell_h \ell_v < 1$, such that
$H_n \cap H_{n'} = \emptyset$, $V_n \cap V_{n'} = \emptyset$, $n\neq n'$, $H_n \to \{p =0\}$, $V_n \to \{\tau = 0\}$, when $n\to \infty$, in the sense of the Hausdorff distance, $F(V_n) = H_n$, homeomorphically, $F^{-1}(\partial_v V_n) \subset \partial_v H_n$, $n\in \NN$.
\item[\sH]
There exist $\kappa_u,\kappa_s,\mu >0$ satisfying $0 < \mu < 1 - \kappa_u \kappa_s$ such that if $\omega \in \cup_{n} V_n$, $DF(\omega) S^u_{\omega,\kappa_u} \subset S^u_{F(\omega),\kappa_u}$, if $\omega \in \cup_{n} H_n$, $DF^{-1}(\omega) S^s_{\omega,\kappa_s} \subset S^s_{F^{-1}(\omega),\kappa_s}$ and,
denoting $x^+ = DF(\omega)x$ and $x^- = DF^{-1}(\omega)x$, if $x \in S^u_{\omega,\kappa_u}$, $|x_u^+| \ge \mu^{-1} |x_u|$ and if $x \in S^s_{\omega,\kappa_s}$, $|x_s^-| \ge \mu^{-1} |x_s|$.
\end{enumerate}

Finally, we introduce symbolic dynamics in our context (see~\cite{Moser01} for a complete discussion). Consider the \emph{space of sequences} $\Sigma = \NN^{\ZZ}$, with the topology\footnote{This toplogy can be also defined by the distance $d(s,r)=\sum_{k\in\ZZ}4^{-|k|}\delta(s_k,r_k)$ where $\delta(n,m)=1$ if $n=m$ and $\delta(n,m)=0$ if $n\neq m$.} induced by the neighborhood basis of $s^* = (\dots,s_{-1}^*,s_0^*,s_1^*,\dots)$
\[
J_j = \{ s \in \Sigma; \; s_k = s_k^*, \; |k| < j\}, \qquad J_{j+1}\subset J_j
\]
and the \emph{shift map} $\sigma: \Sigma \to \Sigma$ defined by $\sigma(s)_j = s_{j+1}$. The map $\sigma$ is a homeomorphism.

We have the following theorem, which is a direct consequence of Theorems~2.3.3 and~2.3.5 of~\cite{Wiggins90}.

\begin{theorem}
\label{thm:ferradura}
Assume that $F:\QQQ_\de^2 \to \RR^4$, a $C^1$ diffeomorphism onto its image, satisfies {\fH} and \sH. Then there exists a subset $\XX \subset \QQQ_\de^2$ and a homeomorphism $h: \XX \to \Sigma$ such that $h \circ F_{\mid \XX}  = \sigma \circ h$.
\end{theorem}

\begin{remark}
Hypothesis {\sH} implies that the set $\XX$ given by Theorem~\ref{thm:ferradura} is hyperbolic.
\end{remark}

\begin{theorem}
\label{thm:lanostraferradura}
If $\tkk$ and $\de$ are small enough, $\wt \Psi$ satisfies {\fH} and \sH.
\end{theorem}

Theorem~\ref{thm:lanostraferradura} is an immediate consequence of the following two propositions.
Proposition \ref{prop:horizontalstrips} implies that $\wt \Psi$ indeed satisfies \fH and Proposition \ref{prop:condicionsdecons} implies that $\wt\Psi$ satisfies $\sH$.

\begin{proposition}
\label{prop:horizontalstrips}
If $\de$ is small enough, $\wt \Psi(\QQQ_\de^2) \cap \QQQ_\de^2$ has an infinite number of connected components.
More concretely, there exists  $0 < \tau_1 < \tau_2 < \de$ such that the set
\[
H_{\tau_1,\tau_2} = \{(p,\tau,\varphi,J) \in \QQQ_\de^2\mid \tau_1 \le \tau \le \tau_2 \}
\]
satisfies the following. There exists $\{H_n\}_{n\in \NN}$, a family of horizontal rectangles,
\[
H_n = \{(p,\tau,\varphi,J) \in \QQQ_\de^2\mid h_{1,n}^-(\tau,\varphi) \le p \le h_{1,n}^+(\tau,\varphi),\;
h_{2,n}^-(\tau,\varphi) \le J \le h_{2,n}^+(\tau,\varphi)\},
\]
with  $H_n \cap H_{n'} = \emptyset$, if $n\neq n'$, such that $h_{1,n}^-,h_{1,n}^+ \to 0$ uniformly
when $n\to \infty$, and
\[
\sup_n \Lip h_{2,n}^-, \sup_n \Lip h_{2,n}^+ \lesssim \OO(\tkk)+\OO(\de),
\]
with $\wt \Psi(H_{\tau_1,\tau_2}) \cap \QQQ_\de^2 = \cup_{n\in \NN} H_n$.

The analogous claim holds for vertical rectangles and $\wt \Psi^{-1}$, that is, there exist $0 < p_1 < p_2 < \de$ such that the set
\[
V_{p_1,p_2} = \{(p,\tau,\varphi,J) \in \QQQ_\de^2\mid p_1 \le p \le p_2 \}
\]
satisfies the following. There exists $\{V_n\}_{n\in \NN}$, a family of vertical rectangles,
\[
V_n = \{(p,\tau,\varphi,J) \in \QQQ_\de^2\mid v_{1,n}^-(p,J) \le \tau \le v_{1,n}^+(p,J),\;
v_{2,n}^-(p,J) \le \varphi \le v_{2,n}^+(p,J)\},
\]
with  $V_n \cap V_{n'} = \emptyset$, if $n\neq n'$, such that $v_{1,n}^-,v_{1,n}^+ \to 0$ uniformly
when $n\to \infty$, and
\[
\sup_n \Lip v_{2,n}^-, \sup_n \Lip v_{2,n}^+ \lesssim \OO(1),
\]
with $\wt \Psi^{-1}(V_{p_1,p_2}) \cap \QQQ_\de^2 = \cup_{n\in \NN} V_n$.

In particular, $\wt \Psi$ satisfies {\fH}.
\end{proposition}

The proof of this proposition is placed in Section~\ref{sec:proofofpropositionprop:horizontalstrips}.

\begin{proposition}
\label{prop:condicionsdecons} $\wt \Psi$ satisfies {\sH} with $\kappa_u = \OO(\de)+\OO(\tkk)$, $\kappa_s = \OO(1)$ and $\mu = \OO(\tkk)$.
\end{proposition}
The proof of this proposition is placed in Section~\ref{sec:provaproposiciodelscons}.

Propositions \ref{prop:horizontalstrips}  and \ref{prop:condicionsdecons} imply that the map $\wt\Psi$ satisfy hypotheses {\fH} and {\sH}. Therefore, one can apply Theorem \ref{thm:ferradura} to $\wt\Psi$. This completes the proof of Theorem \ref{thm:Main2}. Indeed, one can consider as Poincar\'e section $\Pi$ just the section $\Sigma_1$ introduced in \eqref{def:Sections:0} expressed in the original coordinates. Note that in Theorem \ref{thm:Main2} we are fixing the center of mass at the origin and therefore we are dealing with a four degree of freedom Hamiltonian system. The Poincar\'e map $\PP$ is just the return map $\wt\Psi$ written in the expressed in the original coordinates with fixed center of mass. On the contrary,  $\wt\Psi$ is the return map for the 3 body problem once has been reduced by rotations (see \eqref{def:rotationreduction}). Since Theorem \ref{thm:Main2} is stated for the 3 body problem without the rotation symplectic reduction, one has to add the dynamics of this ``extra angle''. Since this angle is cyclic, its dynamics only depends on the point in $\NN^\ZZ$.

Theorem~\ref{thm:Main1} is also a direct consequence of Theorem~\ref{thm:lanostraferradura}. Indeed, note that, the symbol in  $\NN$ keep track of the proximity of corresponding strip to each point in $\XX$ to the invariant manifolds of infinity. That is, the larger the symbol the closer the strip is to the invariant manifolds. This implies that these points get closer to infinity. For this reason, by construction, if one considers a bounded sequence in $\Sigma$, the corresponding orbit in $\XX$ is bounded. If one considers a sequence $\{\omega_k\}_{k\in\ZZ}$ which is unbounded both as $k\to\pm\infty$, the corresponding orbit belongs to \textit{OS$^-$}$\cap$\textit{OS$^+$}. Indeed, the orbit keeps visiting for all forward and backward times a fixed neighborhood of the homoclinic channel (which is ``uniformly far'' from infinity) but at the same time keeps getting closer and closer to infinity because the sequence is unbounded. By considering sequences which are bounded at one side and unbounded at the other, one can construct trajectories which belong to \textit{B}$^-\cap$\textit{OS}$^+$ and \textit{OS}$^-\cap $\textit{B}$^+$. The trajectories which are (in the future or in the past) parabolic-elliptic or hyperbolic-elliptic do not belong to $\XX$ but they can be built analogously. Indeed, as is done by Moser \cite{Moser01}, one can consider sequences of the form
\[
 \omega=(\ldots, \omega_{-1},\omega_0,\omega_1,\ldots,\omega_{M-1},\infty).
\]
That is points whose $M-1$ forward iterates come back to the section $\Sigma_1$ and then the trajectory goes to infinity. By the construction of the horizontal strips, one can built orbits which have these behavior since the strips get streched and therefore its image hit the invariant manifolds of infinity (which correspond to the motions \textit{PE$_3$}) and hit points ``at the other side'' of the invariant manifolds, which correspond to hyperbolic motions \textit{HE$_3$} (see Figure \ref{fig:isolatingtotal}). The same can be achieved for the inverse return map $\wt\Psi^{-1}$ and the vertical strips. For this reason one can combine future/past  \textit{PE$_3$} and  \textit{HE$_3$} with any other types of motion.

\section{Proof of Theorem \ref{thm:MainSplitting}: Parameterization of the invariant manifolds of infinity}
\label{sec:ParamInvManifolds}

Theorem \ref{thm:MainSplitting} gives an asymptotic formula for the distance between the unstable manifold of the periodic orbit  $P_{\eta_0,\xi_0}$ and the stable manifold of the periodic orbit $P_{\eta_0+\delta \eta,\xi_0+\delta \xi}$. In this section we carry out the first step of its proof. We consider suitable graph parameterizations of the invariant manifolds and we analyze their analytic extensions to certain complex domains. Later, in Section \ref{sec:Difference}, we use these analytic extensions to obtain asymptotic formulas for the difference between the parameterizations for real values of the parameters.

This section is structured as follows. First, in Section \ref{sec:adaptedcoordinates} we consider symplectic coordinates which are adapted to have graph paramerizations of the invariant manifolds. These coordinates justify the particular choice of paramerizations given in \eqref{def:sigmaparamnousnr}.
Then, in Section \ref{sec:ParamInvManifoldsGraph} we  analyze the analytic extension of this graph paramerizations to certain complex domains. Such analysis is performed by means of a fixed point argument in suitable Banach spaces of formal Fourier series. These graphs paramerizations are singular at a certain point (where the invariant manifolds cease to be graphs). To overcome this problem, in Section \ref{sec:InvManifold:ExtensionUnstable}, we consider a different type of parameterizations.

\subsection{An adapted system of coordinates}\label{sec:adaptedcoordinates}

To study the unstable manifold of the periodic orbit $P_{\eta_0,\xi_0}$ and the stable manifold of the periodic orbit $P_{\eta_0+\delta \eta,\xi_0+\delta\xi}$ with $|\delta \eta|,|\delta\xi|\ll1$ we perform a change of variables to the coordinates introduced in \eqref{def:PoincareVariablesnr}.
This transformation
\[
(\wt r, \wt y,\la,L,  \eta,\xi) \to ( u, Y,\gamma,\Lambda,\alpha, \beta)
\]
relies on the parameterization of the unperturbed separatrix associated to the periodic orbit  $P_{\eta_0,\xi_0}$ given
by Lemma \ref{lemma:homounperturbed} and is defined as
\begin{equation}\label{def:ChangeThroughHomo}
\begin{aligned}
\wt r =&\,G_0^2 \wh r_\h(u),&\qquad
\wt y =&\, \frac{\wh y_\h(u)}{G_0}+\frac{ Y}{G_0^2\wh y_\h(u)}+\frac{\Lambda-(\eta_0+\alpha)(\xi_0+\beta)+\eta_0\xi_0}{G_0^2\wh y_\h(u)(\wh r_\h(u))^2}\\
\la =&\, \gamma+\phi_\h(u),&\qquad 
L =&\, L_0+\Lambda\\
\eta =&\,e^{i\phi_\h(u)}(\eta_0+\alpha),&\qquad
\xi =&\,e^{-i\phi_\h(u)}(\xi_0+\beta)
\end{aligned}
\end{equation}
where
\begin{equation}\label{ThetaG0}
G_0=\Theta-L_0+\eta_0\xi_0.
\end{equation}
This change of coordinates is consistent with the particular form of the parameterization of the perturbed invariant manifolds given in \eqref{def:sigmaparamnousnr}. Indeed, we look for   parameterizations of the unstable manifold of the
periodic orbit $P_{\eta_0, \xi_0}$ and the stable manifold of the periodic orbit $P_{\eta_0+\delta \eta,\xi_0+\delta \xi}$ as graphs in $(u,\lo)$ as
\begin{equation*}
(u,\lo) \mapsto (Y,\Lambda,\alpha, \beta) =Z^*(u,\lo), \quad *=u,s.
\end{equation*}
It can be easily checked, using Lemma \ref{lemma:homounperturbed}, that the change of coordinates \eqref{def:ChangeThroughHomo} is symplectic in the sense that the pull back of the canonical symplectic form is just
$\omega= du\wedge d Y+ d\gamma \wedge d\Lambda+id\alpha\wedge d\beta$.
This fact will be strongly used later in Section \ref{sec:Average}.

To analyze the dynamics it is enough to express the Hamiltonian  $\wt\KK$  \eqref{def:Ham:Poincarenr} in terms of the new variables.
We also scale time as
\begin{equation}
\label{def:timescaling}
 t=G_0^3s
\end{equation}
to have parabolic motion of speed one coupled with a fast rotation.
Then, the Hamiltonian becomes
\begin{equation}\label{def:HamFinal}
\begin{split}
\PP(u,Y,\gamma,\Lambda, \alpha, \beta) =&   -\frac{G_0^3\nu}{2 (L_0+\Lambda)^2}+\frac{G_0}{2}\left(\wh y_\h(u)+\frac{ Y}{G_0\wh y_\h(u)}
+\frac{\Lambda-(\eta_0+\alpha)(\xi_0+ \beta)+\eta_0 \xi_0}{G_0 \wh y_\h(u)( \wh r_\h(u))^2}\right)^2\\
&+\frac{(\Theta-L_0-\Lambda+(\eta_0+\alpha)(\xi_0+ \beta))^2}{2G_0\wh  r_\h(u)^2}-\frac{G_0}{\wh r_\h(u)}\\
&+ G_0^3\wt W\left(\gamma+\phi_\h(u), L_0+
\Lambda,e^{i\phi_\h(u)}(\eta_0+\alpha),
e^{-i\phi_\h(u)}(\xi_0+\beta), G_0^2\wh  r_\h(u)\right).
\end{split}
\end{equation}
Observe that we do not write the dependence of $\PP$ on the parameters $L_0,\eta_0, \xi_0$, nor on $G_0$.
In a natural way we can write $\PP=\PP_0(u,Y,\gamma,\Lambda, \alpha, \beta)+\PP_1(u,\gamma,\Lambda, \alpha, \beta)$ where, using  \eqref{ThetaG0},
\begin{equation}\label{eq:P0}
\begin{split}
\PP_0(u,Y,\gamma,\Lambda, \alpha, \beta)=& -\frac{G_0^3\nu}{2 (L_0+\Lambda)^2}+\frac{G_0}{2}\left(\wh y_\h(u)+\frac{ Y}{G_0\wh y_\h(u)}
+\frac{\Lambda-(\alo+\alpha)(\beto+ \beta)+\alo \beto}{G_0\wh  y_\h(u)(\wh r_\h(u))^2}\right)^2\\
&+\frac{(\Theta-L_0-\Lambda+(\alo+\alpha)(\beto+ \beta))^2}{2G_0\wh  r_\h(u)^2}-\frac{G_0}{\wh r_\h(u)}\\
=&   -\frac{G_0^3\nu}{2 (L_0+\Lambda)^2}+Q_0 (u,  Y,\Lambda-(\alo+\alpha)(\beto+ \beta)+\alo\beto)
\end{split}
\end{equation}
where, taking into account \eqref{def:HamPendulum}, $Q_0$ can be written as
\begin{equation}\label{def:Q0}
\begin{split}
Q_0(u, Y,q)=&\,\frac{G_0}{2}\left(\wh y_\h(u)+\frac{Y}{G_0\wh y_\h(u)}
+\frac{q}{G_0\wh  y_\h(u)(\wh r_\h(u))^2}\right)^2+\frac{(G_0-q)^2}{2G_0\wh  r_\h(u)^2}-\frac{G_0}{\wh r_\h(u)}\\
=&\,Y+\frac{Y^2}{2G_0\wh y_\h^2(u)}+f_1(u) Yq+ f_2(u)\frac{q^2}{2}
\end{split}
\end{equation}
with
\begin{equation}\label{def:f1f2}
f_1(u) =\frac{1}{G_0 y_h^2(u)r_h^2(u)},\qquad
f_2(u) =\frac{2}{G_0 r_h^3(u)y_h^2(u)}
\end{equation}
and
\begin{equation}\label{def:P1}
\PP_1(u,\gamma,\Lambda, \alpha, \beta)=
G_0^3\wt W\left(\gamma+\phi_\h(u), L_0+
\Lambda,e^{i\phi_\h(u)}(\alo+\alpha),
e^{-i\phi_\h(u)}(\beto+\beta), G_0^2\wh  r_\h(u)\right).
\end{equation}
The periodic orbits at infinity $P_{\eta_0,\xi_0}$  and $P_{\eta_0+\delta \eta,\xi_0+\delta\xi}$ are now given by
\[
\begin{split}
{P}_{\alo, \beto}=&\{ (u,Y,\gamma,\Lambda, \alpha, \beta)=(\pm \infty,0,\gamma,0, 0,0), \ga \in \TT\}\\
{P}_{\alo+\delta \eta, \beto+\delta\xi}=&\{ (u,Y,\gamma,\Lambda, \alpha, \beta)=(\pm \infty,0,\gamma,0, \delta \eta,\delta\xi), \ga \in \TT\}
\end{split}
\]
The equations for the integrable system, which corresponds to $\PP_1=0$, are
\[
\begin{aligned}
\dot u &=\partial_Y \PP_0= \partial_Y Q_0=1+\frac{Y}{G_0\wh y_h^2}+f_1(u) q\\
\dot { Y} &=- \partial_u \PP_0=-\partial_u Q_0=\frac{\wh y_\h'(u)Y^2}{\wh y_\h^3(u)G_0}+f_1'(u)Yq+f_2'(u)\frac{q^2}{2}\\
\dot \gamma &= \partial_\La \PP_0=\frac{G_0^3 \nu}{(L_0+\Lambda)^3}+\partial_q Q_0=\frac{G_0^3 \nu}{(L_0+\Lambda)^3}+f_1(u)Y+f_2(u)q\\
\dot \Lambda &= - \partial_\gamma \PP_0= -\partial_\gamma Q_0=0 \\
\dot \alpha &= -i \partial_\beta \PP_0=i  \al\partial_q Q_0=i\al\left(f_1(u)Y+f_2(u)q\right)\\
\dot \beta &= i  \partial_\al \PP_0= -i \beta \partial_q Q_0=-i\bet\left(f_1(u)Y+f_2(u)q\right)
\end{aligned}
\]
where $q=\Lambda-(\alo+\al)(\beto+\bet)+\alo\beto$ and $f_1$ and $f_2$ are given in \eqref{def:f1f2}.

This system has a $2$-dimensional homoclinic manifold to the
periodic orbit  $P_{\alo, \beto}$ given by
\begin{equation}\label{def:homomanifoldugamma}
\{ u_\h=u, \, Y_\h=0, \, \gamma_\h= \gamma, \,   \Lambda_\h=0, \, \alpha_\h=0,\,  \beta_\h=0, \,   u \in \RR, \ga \in \TT\}
\end{equation}
whose  dynamics  is given by
\begin{equation*}
(\dot u,\dot {Y} , \dot \gamma, \dot \La,\dot \al,\dot \bet)=\left( 1,0,  \frac{G_0^3 \nu}{L_0^3}, 0,0,0\right).
\end{equation*}
(recall that we have scaled time as \eqref{def:timescaling}).


\subsection{Graph parameterizations of the perturbed invariant manifolds}
We  look for parameterizations of the unstable  manifold of the
periodic orbit ${P}_{\alo, \beto}$ and the stable manifold of the periodic orbit $P_{\eta_0+\delta \eta,\xi_0+\delta\xi}$ as perturbations of the same homoclinic manifold \eqref{def:homomanifoldugamma} which exists when one takes $\PP_1=0$. We look for them as graphs with respect to the variables $(u,\lo)$, that is
\begin{equation}
\label{def:Parameterizations}
(u,\lo) \mapsto (Y,\Lambda,\alpha, \beta) =Z(u,\lo)
\qquad \text{where}\qquad
Z (u,\lo) =
\begin{pmatrix}
Y(u,\lo)\\
\Lambda(u,\lo)\\
\alpha(u,\lo)\\
\beta(u,\lo)
\end{pmatrix}.
\end{equation}
Note that in the unperturbed case, $Z=0$ is a manifold homoclinic to  $P_{\eta_0,\xi_0}$.

The graph parameterizations \eqref{def:Parameterizations} are not defined in a neighborhood of $u=0$ since the symplectic transformation
\eqref{def:ChangeThroughHomo} is not well defined at $u=0$. For this reason, we shall use different parameterizations depending on the domain.
%
%
\begin{figure}[h]
\begin{center}
\includegraphics[height=4cm]{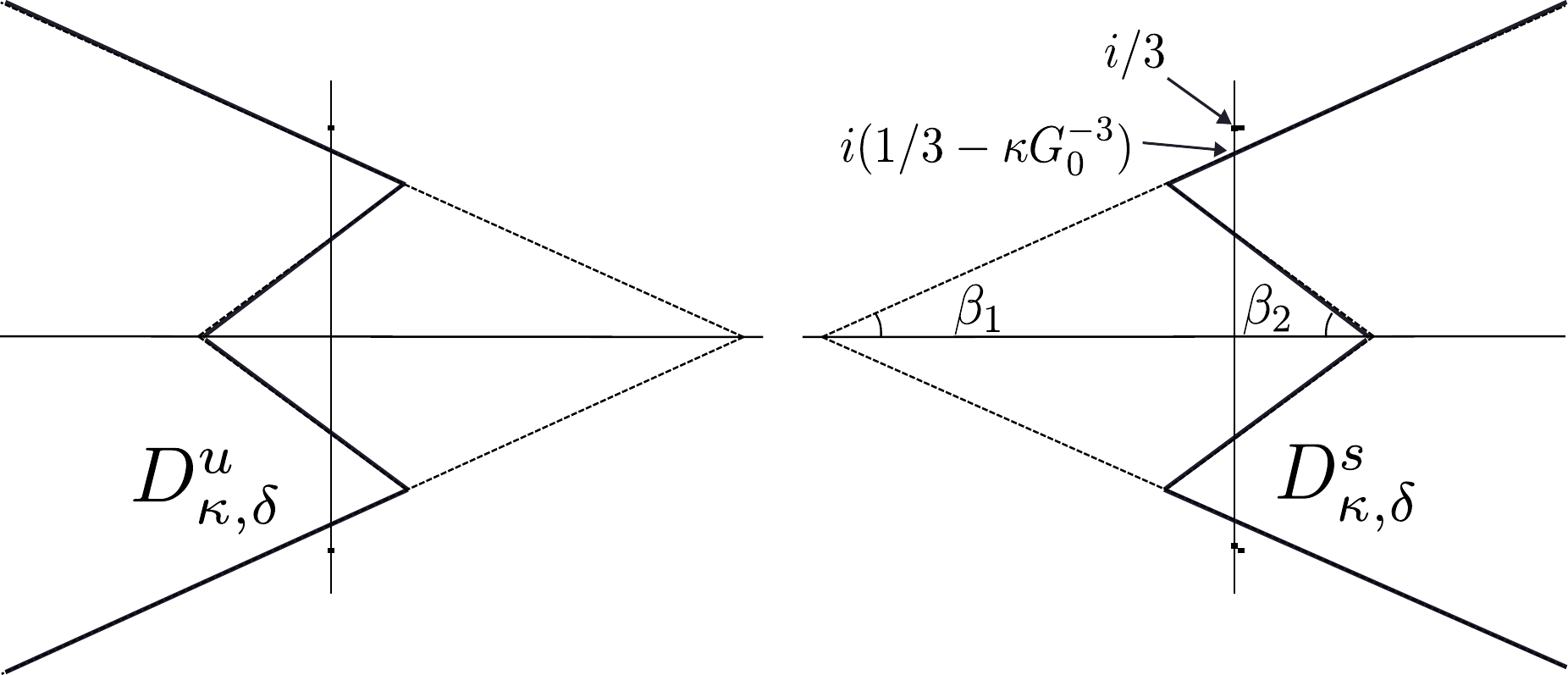}
\end{center}
\caption{The domains  $D^{u}_{\kk,\de}$ and $D^{s}_{\kk,\de}$ defined in
\eqref{def:DomainOuter}.}\label{fig:OuterDomains}
\end{figure}

First in Sections \ref{sec:ParamInvManifoldsGraph} and \ref{Sec:FixtPt}, we obtain graph parameterizations  \eqref{def:Parameterizations} in the  domains $D^{u,s}_{\kk,\de}\times \TT$,  where
\begin{equation}\label{def:DomainOuter}
\begin{split}
D^{s}_{\kk,\de}=\Big\{u\in\CC;& |\Im u|<\tan \beta_1\Re u+1/3-\kk
G_0^{-3},
|\Im u|>-\tan \beta_2\Re u+1/6-\de\Big\}\\
\dps D^{u}_{\kk,\de}=\{u\in\CC; &-u\in D^{s}_{\kk,\de}\},
\end{split}
\end{equation}
which do not contain $u=0$ (see Figure~\ref{fig:OuterDomains}). These are the same domains that were used in \cite{GuardiaMS16}.
However, the intersection  domain
 $D^{s}_{\kk,\de}\cap D^{u}_{\kk,\de}$ has empty intersection for real values of $u$ and therefore, to compare both manifolds one needs to extend the stable manifold to a domain which overlaps $D^{u}_{\kk,\de}\cap\RR$. This is done in Section \ref{sec:InvManifold:ExtensionUnstable}.
%

\subsection{The invariance equation for the graph parameterizations}\label{sec:ParamInvManifoldsGraph}
%
%

The graph parameterizations introduced in  \eqref{def:Parameterizations}
  satisfy  the invariance equation
\begin{equation}\label{eq:invarnova}
\left(\partial_u Z, \partial_\ga Z\right)(u,\gamma)\cdot\left(\begin{array}{c}
\partial_Y \PP\\
 \partial_\La \PP
\end{array}\right)(u,\gamma,Z(u,\gamma))=
\begin{pmatrix}
-\partial_u P\\
- \partial_\ga \PP\\
-i \partial_\bet \PP\\
i \partial_\al \PP
\end{pmatrix}(u,\gamma,Z(u,\gamma))
\end{equation}
Using  vector notation $Z=Z(x)$ is invariant if
\begin{equation}\label{eq:Invariance1}
DZ(x)X_x(x,Z(x))= X_Z(x,Z(x))
\end{equation}
where $X_x=\left(\partial_Y \PP, \partial_\La \PP\right)^\top$
and
$X_Z=\left(-\partial_u \PP,-
 \partial_\ga \PP,
-i \partial_\bet \PP,
i  \partial_\al \PP\right)^\top$.

Observe that $X=X_0+X_1$ where $X_i$ are the Hamiltonian vector fields associated to $\PP_i$. Of course, when $\PP_1=0$, we have that  the unperturbed homoclinic manifold $Z=0$ to the periodic orbit $P_{\eta_0,\xi_0}$
satisfies the invariance equation \eqref{eq:invarnova}.
In fact,
\[
\begin{split}
\left(\partial_Y \PP_0,\partial_\La \PP_0\right)^\top
(u,\gamma,0)&=\left(1,\nu \frac{G_0^3}{L_0^3} \right)^\top \\
X_Z^0(x, 0) = \left(-\partial_u \PP_0,
- \partial_\ga \PP_0,
-i \partial_\bet \PP_0,
i  \partial_\al \PP_0\right)^\top  (x,0)&=0.
\end{split}
\]

\begin{proposition}\label{prop:invariance}
The invariance equation \eqref{eq:invarnova} can be rewritten as 
\begin{equation}\label{eq:InvarianceVect}
 \LL Z=AZ+F(Z)\qquad \text{with}\qquad F(Z)=-\GG_1(Z)\pa_u Z-\GG_2(Z)\pa_\ga Z+\QQ(Z)
\end{equation}
where 
\begin{itemize}
\item $\LL$ is the operator 
\begin{equation}\label{def:OperadorL}
\mathcal{L}(Z)=
\pa_u Z+\nu \frac{G_0^3}{ L_0^{3}}\pa_\ga Z.
\end{equation}
\item The functions $\GG_1$ and $\GG_2$ are defined as
\begin{equation}\label{def:G1G2}
\begin{split}
\GG_1 (u,\lo,Y,\La,\al,\bet) =&\,
 \frac{ Y}{G_0\wh y_h^2(u)}+f_1(u)q\\
 \GG_2 (u,\lo,Y,\La,\al,\bet) =&\, \frac{G_0^3 \nu}{(L_0+\Lambda)^3}-\frac{G_0^3\nu}{L_0^3}+f_1(u)Y+f_2(u)q
 +\partial_\La \PP_1(u,\lo,\La,\al,\bet)
\end{split}
\end{equation}
where $q=\Lambda-\alo\beta-\beto\al-\al\bet$
\item The matrix $A$ is 
\begin{equation}\label{def:A}
A(u)=
\left(\begin{array}{cc}0&0\\ \mathcal{A}(u)& \mathcal{B}(u) \end{array}\right)
\end{equation}
with
\begin{equation}\label{eq:AB}
\mathcal{A}(u)= i
\left(\begin{array}{cc}f_1(u)\alo & f_2(u)\alo \\
-f_1(u) \beto & - f_2(u)\beto
\end{array}\right),
\qquad
\mathcal{B}(u)= i f_2(u)
\left(\begin{array}{cc}-\alo \beto & -\alo^2 \\
\beto^2 & \alo\beto
\end{array}\right)
\end{equation}
where $f_1$ and $f_2$ are defined in \eqref{def:f1f2}.
\item The function $\QQ$ is 
\begin{equation}\label{def:Q}
\begin{split}
 \QQ_1 (u,\lo,Y,\La,\al,\bet) =&\, \frac{\wh y_\h'(u) }{G_0\wh y_\h^3(u)}Y^2-f_1'(u)Yq-f_2'(u)\frac{q^2}{2}-\frac{\partial \PP_1}{\partial  u}(u,\lo,\La,\al,\bet)\\
 \QQ_2 (u,\lo,Y,\La,\al,\bet) =&\, -\frac{\partial \PP_1}{\partial  \ga}(u,\lo,\La,\al,\bet)\\
  \QQ_3 (u,\lo,Y,\La,\al,\bet)=&\,i  \al\left[f_1(u)Y+f_2(u)\left(\La-2\alo\beta-\al\beto-\al\beta\right)\right]- i \partial_\beta \PP_1(u,\lo,\La,\al,\bet)\\
 \QQ_4 (u,\lo,Y,\La,\al,\bet) =&\,-i  \bet\left[f_1(u)Y+f_2(u)\left(\La-\alo\beta-2\al\beto-\al\beta\right)\right]
 + i \partial_\al \PP_1(u,\lo,\La,\al,\bet).
\end{split}
\end{equation}
\end{itemize}
\end{proposition}


To solve the invariance equations we first  integrate the linear system
$\LL Z=A(u) Z$, which, writing $Z=(Z_{Y\La},Z_{\al\bet})$, reads
\begin{equation}\label{variacionals}
 \begin{split}
 \mathcal{L} Z_{Y\La} &=\, 0 \\
 \mathcal{L} Z_{\al\bet} &=\, \mathcal{A}(u) Z_{Y\La}+\mathcal{B}(u) Z_{\al\bet}
\end{split}
\end{equation}
where $\mathcal{A}(u)$ and  $\mathcal{B}(u)$ are given in \eqref{eq:AB}. The proof of the following lemma is straighforward.

\begin{lemma}
A fundamental matrix of the linear system \eqref{variacionals} is
\begin{equation}\label{def:FundamentalMatrix}
\Phi_A^\ast(u)=\left(\begin{array}{cccc}
1& 0& 0 &0 \\
0& 1 & 0& 0 \\
\alo g_1(u)&
\alo g^\ast_2(u)& 1-\alo \beto g^\ast_2(u) &-\alo^2 g^\ast_2(u)\\
 - \beto g_1(u)& - \beto g^\ast_2(u) &\beto^2g^\ast_2(u)&1+\alo \beto
g^\ast_2(u)
 \end{array}\right),\qquad \ast=u,s
\end{equation}
where $g'_1(u)=if_1(u)$ and $(g_2^*)'(u)=if_2(u)$, $\ast=u,s$. In particular we can choose the functions $g_2^*$ satisfying that $\lim_{\Re u\to+\infty} g_2^s(u)=0$ and $\lim_{\Re u\to-\infty} g_2^u(u)=0$.
\end{lemma}

We define two different  inverse operators of $\LL$,
\begin{equation}\label{def:Outer:IntegralOperator}
\begin{split}
\GG^u(h)(u,\ga)=\int_{-\infty}^0 h(u+s, \ga+\nu G_0^3L_0^{-3} s)ds\\
\GG^s(h)(u,\ga)=\int_{+\infty}^0 h(u+s, \ga+\nu G_0^3L_0^{-3} s)ds,
\end{split}
\end{equation}
We use them to prove the existence of the stable and unstable invariant manifolds.
Here we only deal with the stable manifold of $P_{\eta_0+\delta\eta,\xi_0+\delta_\xi}$ and we take $\GG=\GG^s$ and  $\Phi_A=\Phi_A^s$ (the unstable manifold of $P_{\eta_0,\xi_0}$ is obtained analogously from the particular case $\delta\eta=\delta\xi=0$).

We use this operator and the fundamental matrix $\Phi_A$ to derive an integral equation equivalent to the invariance equation \eqref{eq:InvarianceVect}.

\begin{lemma}\label{lemma:Invariance}
The parameterization of the stable manifold $Z^s$ of the periodic orbit $P_{\eta_0+\delta\eta,\xi_0+\delta_\xi}$  is  a fixed point of the operator
\begin{equation}\label{def:OperatorInfinity}
\FF(Z)= \Phi_A\delta z+ \Phi_A\GG\left(\Phi_A\ii F(Z)\right)=\Phi_A \delta z+\GG_A\circ F(Z)
\end{equation}
where $\delta z=(0,0,\delta\eta,\delta\xi)^\top$ and
\begin{equation}\label{def:OperadorGA}
\GG_A(h)=\begin{pmatrix}
\GG(h_1)\\
\GG(h_2)\\
\GG(h_3) +\alo \GG\left(f_1 \GG(h_1)\right)+ \alo\GG\left(f_2\GG(h_2)\right)
-\alo \GG\left(f_2\GG\left(\beto h_3+\alo h_4\right)\right)\\
  \GG(h_4)-\beto \GG\left(f_1 \GG(h_1)\right)- \beto\GG\left(f_2\GG(h_2)\right)
+\beto\GG\left(f_2\GG\left(\beto h_3+\alo h_4\right)\right).
\end{pmatrix}
\end{equation}
\end{lemma}

\begin{proof}
Using the fundamental matrix $\Phi_A(u)$ and the variation of constants formula give the first equality in \eqref{def:OperatorInfinity}. We point out that $\lim_{u\to +\infty}\Phi_A(u)\delta z=\delta z$. For the second one we write $\Phi_A\GG\left(\Phi_A\ii F(Z)\right)$ in components as
\[
\Phi_A\GG\left(\Phi_A\ii F(Z_1)\right)=\left(
\begin{array}{l}
 \GG(F_1(Z_1))\\
 \GG(F_2(Z_1))\\
 \GG(F_3(Z_1)) +\alo g_1 \GG(F_1(Z_1))+ \alo g_2\GG(F_2(Z_1))\\
\,\,+ \alo \GG\left(-g_1 F_1(Z_1)-g_2 F_2(Z_1)+\beto g_2F_3(Z_1)+\alo g_2F_4(Z_1)\right)\\
\,\,-\alo g_2\GG\left(\beto F_3(Z_1)+\alo F_4(Z_1)\right)\\
 \GG(F_4(Z_1))-\beto  g_1 \GG(F_1(Z_1))- \beto g_2\GG(F_2(Z_1))\\
\,\, -\beto \GG\left(-g_1 F_1(Z_1)-g_2 F_2(Z_1)-\beto
g_2F_3(Z_1)-\alo g_2F_4(Z_1)\right)\\
\,\,+\beto g_2\GG\left(\beto F_3(Z_1)+\alo F_4(Z_1)\right)\end{array}\right).
\]
Then, it only suffices to note that the  terms of the form $g_i\GG(F_j)-\GG(g_i F_j)$ can be  rewritten as
\begin{equation*}
\label{def:ByParts}
g_i\GG(F_j)-\GG(g_i F_j)=\GG(\pa_u g_i\GG(F_j))=\GG( f_i\GG(F_j)).
\end{equation*}
Indeed, it is enough to apply the operator $\LL$ to both sides.
%
\end{proof}

\subsection{A fixed point of the operator $\FF$}\label{Sec:FixtPt}
To obtain a fixed point of the operator $\FF$ in \eqref{def:OperatorInfinity}, for $(u,\gamma)\in D^{s}_{\kk,\de}\times\TT_\sigma$, we introduce the functional setting we work with.
We consider functions of the form $Z=( Y,L,\al,\bet)^\top$.
Take $h$ any of these components and define its  Fourier  series
\[
 h(u,\lo)=\sum_{q\in \ZZ}h^{[q]}(u)e^{iq\ga}
\]
Denote by $f$ any of the Fourier coefficients, which are only functions of $u$, and take $\rr>0$. We
consider the following norm, which captures the behavior as $\Re u\to \infty$
and also the behavior ``close'' to the singularities of the unperturbed
homoclinic (see Lemma \ref{lem:HomoInfinity}),
\[
 \|f\|_{n,m,q}=\sup_{D^{s}_{\kappa,\de} \cap \{\Re u\geq \rr\}}\left|
u^ n f(u)\right|
+\sup_{D^s_{\kappa,\de} \cap \{\Re u\leq
\rr\}}\left|
\left(u-\frac{i}{3}\right)^{m}\left(u+\frac{i}{3}\right)^{m} e^{-iq\phi_\h(u)}
f(u)\right|.
\]
Now, for a fixed $\sigma>0$, we define the norm for $h$  as
\[
 \|h\|_{n,m}=\sum_{k\in\ZZ}\|h^{[q]}\|_{n,m,q}e^{|q|\sigma}.
\]
We denote the corresponding Banach space by $\YY_{n,m}$. Note that such norms do not define functions in $D^s_{\kappa,\de} \times\TT_\sigma$. Indeed, the
Fourier series may be divergent for complex $u$ due to the term
$e^{-iq\phi_\h(u)}$ which grows exponentially as $|q|\to\infty$. Still, since
$|e^{-iq\phi_\h(u)}|=1$ for real values of $u$, the Fourier series define
actual functions for real values of $u$.

To prove the existence of the invariant manifolds we need to keep control of
the first derivatives for sequences of Fourier
coefficients $h\in\YY_{n,m}$. The derivatives of
sequences  are defined in the natural way
\begin{equation}\label{def:DerivativeFourierSeries}
   \pa_u h(u,\ga)=\sum_{q\in \ZZ}\pa_u  h^{[q]}(u) e^{iq\ga},\qquad
 \pa_\ga h(u,\ga)=\sum_{q\in \ZZ}(i\ell) h^{[q]}(u) e^{iq\ga}.
\end{equation}
Then, we also consider the norm
\[
 \lln h\rrn_{n,m}=\|h\|_{n,m}+\|\pa_u h\|_{n+1,m+1}+G_0^3\|\pa_\gamma
h\|_{n+1,m+1}
\]
and denote by  $\XX_{n,m}$ the associated Banach space.
Since each component has different behavior, we need to consider weighted norms for $Z=(Y,L,\al,\bet)$. We define
\begin{equation}\label{def:Norm}
\begin{split}
\|Z\|_{n,m, \vect}=&\,\|Y\|_{n+1,m+1}+ \|\La\|_{n,m}+\|\al e^{i\phi_\h(u)}\|_{n,m}+ \|\bet e^{-i\phi_\h(u)}\|_{n,m}\\
 \lln Z\rrn_{n,m, \vect}=&\,\|Z\|_{n,m, \vect}+\|\pa_u Z\|_{n+1,m+1,
\vect}+G_0^3\|\pa_\ga Z\|_{n+1,m+1,
\vect}.
 \end{split}
\end{equation}
We denote by $\YY_{n,m,\vect}$ and $\XX_{n,m,\vect}$ the associated Banach
spaces.

Since the Banach space $\XX_{n,m,\vect}$ is a space of formal Fourier series, the terms $\pa_z\PP_1(u,\ga,Z)$, $z=u,\ga,\Lambda,\al,\bet$, which appear in Proposition \ref{prop:invariance} for $Z =(Y,\al,\La,\bet) \in \XX_{n,n,\vect}$ are understood 
 formally by the formal Taylor expansion\footnote{Note that the function $\PP_1$ in \eqref{eq:InvarianceVect} does not depend on $Y$}
\begin{equation}\label{def:composition}
\pa_z\PP_1(u,\gamma,Z) = \sum_{\ell\ge 0} \frac{1}{\ell !} \sum_{i_1=0}^{\ell} \sum_{i_2 = 0}^{\ell-i_1}
\frac{\ell !}{i_1 !i_2! (\ell-i_1-i_1)!} \frac{\partial^{\ell} (\pa_z\PP_1)}{\partial^{i_1}\al \partial^{i_2}\bet \partial^{\ell-i_1-i_2} \La} (u,\gamma,0,0,0) \al^{i_1} \bet^{i_2} \La^{\ell-i_1-i_2}.
\end{equation}
where $z=u,\ga,\Lambda,\al,\bet$. In Lemma \ref{lem:CompositionManifolds} below we give conditions on $Z$ which make this formal composition meaningful.

Finally, we define 
\begin{equation*}
\label{eq:Ztilde}
\widetilde{Z}=Z-\delta z\qquad \text{where}\qquad \delta z=(0,0,\delta\eta,\delta\xi)^\top
\end{equation*}
and introduce the operator
\begin{equation}\label{eq:Ftilde}
\mathcal{\widetilde{F}}(\widetilde{Z})=\mathcal{F}(\widetilde{Z}+\delta z)-\delta z
\end{equation}
where $\FF$  is defined in~\eqref{def:OperatorInfinity}. It is clear that $Z$ is a fixed point of $\mathcal{F}$ if and only if $\widetilde{Z}$ is a fixed point of  $\widetilde{\mathcal{F}}$.

\begin{theorem}\label{thm:ExistenceManiFixedPt}
Let $\delta z=(0,0,\delta \eta,\delta\xi)$ and denote by $B_{\rho}$
the ball of radius $\rho$ in $\XX_{1/3,1/2,\vect}$.
There exists $b_0 >0$ such that if  $G_0 \gg 1$, $|\delta\eta|,|\delta\xi|\lesssim G_0^{-3}$ and
\begin{equation}\label{cond:eccentricity}
|\alo| G_0^{3/2} \ll 1,
\end{equation}
then, the operator  $\mathcal{\widetilde{F}}$  defined in~\eqref{eq:Ftilde}
has the following properties.
\begin{enumerate}
\item $\widetilde{\mathcal{F}}:B_{b_0 G_0^{-3}\ln G_0} \to  B_{b_0 G_0^{-3}\ln G_0}$,
\item It is Lipschitz  in $ B_{b_0 G_0^{-3}\ln G_0}\subset\XX_{1/3,1/2,\vect}$ with
Lipschitz constant
$\Lip(\mathcal{\widetilde{F}}) \simeq   G_0^{-3/2}\ln^2 G_0$.
\end{enumerate}
Therefore, $\widetilde{\FF}$ has a fixed point $\widetilde{Z}^s$ in $B_{b_0 G_0^{-3}\ln G_0}$. Denoting by $Z^s=\delta z+\widetilde{Z}^s$ and by $\mathcal{F}$ the operator ~\eqref{def:OperatorInfinity} we have that $Z^s-\mathcal{F}(0)\in B_{b_0 G_0^{-3}\ln G_0}$ and 
\[
\lln  Z^s-\FF(0)  \rrn_{1/3,1/2,\vect} \lesssim G_0^{-9/2}\ln^3 G_0.
\]
Moreover, $\widetilde{\La}^s$ satisfies
\begin{equation}\label{def:LambdaImprovedEstimate}
\lln\widetilde{\La}^s\rrn_{1/3,1}\lesssim G_0^{-9/2}.
\end{equation}
\end{theorem}
To have estimates of the derivatives of the scattering map, we need also estimates for the derivatives of the invariant manifolds parameterizations for real values of $(u,\ga)$. To this end we state the following proposition.

\begin{proposition}\label{prop:InvManReal}
The parameterization  $Z^s=\delta z+\widetilde{Z}^s$
 obtained in Theorem \ref{thm:ExistenceManiFixedPt} is also defined in the domain
\begin{equation}\label{def:domainsInvReal}
 u\in D^s_{\kk,\de}\cap\RR, \ga\in\TT, |\alo|\leq \frac{1}{2}, |\beto|\leq \frac{1}{2}.
\end{equation}
Moreover, in this domain the functions $Z^s=(Y,\varLambda,\alpha,\beta)$ satisfy that for $N\geq 0$,
\begin{equation}\label{def:NderivManifolds}
|D^N( Z^s-\FF(0))|\lesssim G_0^{-6}.
\end{equation}
where $D^N$ denotes the differential of order $N$ with respect to the variables $(u,\ga,\alo,\beto)$ and $C$ is a constant which may depend on $N$ but independent of $G_0$.
\end{proposition}

Note that the condition \eqref{cond:eccentricity} is not required in Proposition \ref{prop:InvManReal}. Indeed this condition is needed to extend the Fourier coefficients of the invariant manifolds parameterizations into points of $D^s_{\kk,\de}$ wich are $G_0^{-3}$-close to the singularities $u=\pm i/3$. The extension to the disk $|\alo|\leq \frac{1}{2}, |\beto|\leq \frac{1}{2}$ is needed to apply Cauchy estimates to obtain \eqref{def:NderivManifolds}, which is needed (jointly with the analogous estimate for the parameterization of the unstable manifold) to obtain the estimates for the difference between the invariant manifolds given in \eqref{def:DerivMelnikovAltres}.

We devote the rest of this section to proof Theorem \ref{thm:ExistenceManiFixedPt}.
First in Section \ref{sec:ExistManifoldsLemmas} we state several lemmas which give properties of the norms and the functional setting. Then, in Section \ref{sec:proofexistencemani} we give the fix point argument which  proves Theorem \ref{thm:ExistenceManiFixedPt}.
Finally, in Section \ref{sec:Proof:InvReal}, we explain how to adapt the proof of Theorem \ref{thm:ExistenceManiFixedPt} to prove Proposition \ref{prop:InvManReal}.

\subsubsection{Technical lemmas}\label{sec:ExistManifoldsLemmas}
We devote this section to state several lemmas which are needed to prove Theorem \ref{thm:ExistenceManiFixedPt}. The first one, whose prove is straighforward, gives properties of the Banach spaces $\YY_{n,m,\kk,\de,\sigma}$.
\begin{lemma}\label{lemma:banach:AlgebraProps}
The spaces $\YY_{n,m,\kk,\de,\sigma}$ satisfy the following properties:
\begin{itemize}
 \item If $h\in \YY_{n,m,\kk,\de,\sigma}$ and $g\in
\YY_{n',m',\kk,\de,\sigma}$, then the formal product of Fourier series
$hg$ defined as usual by
\[
 (hg)^{[\ell]}(v)=\sum_{k\in\ZZ}h^{[k]} g^{[\ell-k]}
\]
satisfies that $hg\in \YY_{n+n',m+m',\kk,\de,\sigma}$ and
$ \|hg\|_{n+n',m+m'\sigma}\leq \|h\|_{n,m,\sigma}\|g\|_{n',m',\sigma}$.
\item If $h\in \YY_{n,m,\kk,\de,\sigma}$, then $h\in
\YY_{n-\eta,m,\kk,\de,\sigma}$ with $\eta>0$ and
 $\|h\|_{n-\eta,m,\sigma}\leq K\|h\|_{n,\sigma}$.
\item If $h\in \YY_{n,m,\kk,\de,\sigma}$, then $h\in
\YY_{n,m+\eta,\kk,\de,\sigma}$ with $\eta>0$ and
 $\|h\|_{n,m+\eta,\sigma}\leq K\|h\|_{n,\sigma}$.
\item If $h\in \YY_{n,m,\kk,\de,\sigma}$, then $h\in
\YY_{n,m-\eta,\kk,\de,\sigma}$ with $\eta>0$ and
 $\|h\|_{n,m-\eta,\sigma}\leq K G_0^{3\eta}\|h\|_{n,m,\sigma}$.
\end{itemize}
\end{lemma}

We are going to find a fixed point of the operator \eqref{def:OperatorInfinity} in a suitable space $\XX_{n,m,\vect}$, which is rather a space of \emph{formal} Fourier series than a space of functions. More concretely, although the elements of $\XX_{n,m,\vect}$ define functions for real values of $u$ and $\gamma$, these series are just formal in some  open sets of $\CC^2$. The previous lemma ensures that $\XX_{n,m,\vect}$ is an algebra with respect to the usual product, but we need to give a meaning to the composition.

\begin{lemma}\label{lem:CompositionManifolds}
Consider $Z\in\XX_{1,1/2,\vect}$ satisfying $\|Z\|_{1,1/2,\vect} \ll G_0^{-3/2}$. Then, 
the formal compositions 
$\pa_z\PP_1(u,\gamma,Z(u,\ga))$, $z=u,\ga,\Lambda,\al,\bet$, defined in \eqref{def:composition} satisfy
\[
 (\pa_u\PP_1(\cdot,\cdot,Z),\pa_\ga\PP_1(\cdot,\cdot,Z), \pa_\beta\PP_1(\cdot,\cdot,Z), \pa_\al\PP_1(\cdot,\cdot,Z))\in\XX_{2,3/2,\vect} 
\]
and
\[
\left\|(\pa_u\PP_1(\cdot,\cdot,Z),\pa_\ga\PP_1(\cdot,\cdot,Z), \pa_\beta\PP_1(\cdot,\cdot,Z), \pa_\al\PP_1(\cdot,\cdot,Z))\right\|_{2,3/2,\vect} \lesssim G_0^{-3}
\]
Moreover, if one defines
\[
\begin{split}
 \Delta \PP_1(Z,\wt Z)= &(\pa_u\PP_1(\cdot,\cdot,Z),\pa_\ga\PP_1(\cdot,\cdot,Z), \pa_\beta\PP_1(\cdot,\cdot,Z), \pa_\al\PP_1(\cdot,\cdot,Z))\\&-  (\pa_u\PP_1(\cdot,\cdot,\wt Z),\pa_\ga\PP_1(\cdot,\cdot,\wt Z), \pa_\beta\PP_1(\cdot,\cdot,\wt Z), \pa_\al\PP_1(\cdot,\cdot,\wt Z)),
\end{split}
 \]
then, for $Z,Z'\in\XX_{1,1/2,\vect}$,
\[
 \left\|
 \Delta \PP_1(Z,\wt Z) \right\|_{3,2,\vect}\lesssim G_0^{-3}\|Z-\wt Z\|_{1,1/2,\vect}.
\]
%
\end{lemma}
The proof of this lemma is a straighforward computation.

We also need a precise knowledge of the behavior of the paramerization of the unperturbed homoclinic introduced in Lemma \ref{lemma:homounperturbed} as $u\to\pm \infty$ and close to its complex singularities. They are given in the next two lemmas.

\begin{lemma}\label{lem:HomoInfinity}
The homoclinic~\eqref{def:homoclinic} with
initial conditions~\eqref{def:CondicioInicialHomo}
 behaves as follows:
\begin{itemize} 
\item As $|u|\rightarrow +\infty$,
\[\wh r_\h (u)\sim  u^{2/3}, \quad
\wh y_\h (u)\sim   u^{-1/3}\quad \text{and}\quad 
 \phi_\h (u)-\pi\sim   u^{-1/3}\,\,(\text{mod } 2\pi).
\]
%
\item As $u\longrightarrow \pm i/3$,
\[
\wh r_\h(u)\sim \left( u\mp\frac{i}{3}\right)^{1/2},\quad \wh
y_\h(u)\sim \left( u\mp \frac{i}{3}\right)^{-1/2}, \quad
%
e^{i \phi_\h (u)}\sim \left(\frac{u+\frac{i}{3}}{u-\frac{i}{3}}\right)^{1/2}.
\]
\end{itemize}
\end{lemma}

The proof of this lemma is given in \cite{GuardiaMS16}.
%
From Lemma \ref{lem:HomoInfinity}, one can derive also properties for the functions  $f_1$ and $f_2$ introduced in \eqref{def:f1f2}.
%
%
\begin{lemma}\label{lemma:f1f2}
The functions $f_1$ and $f_2$ introduced in \eqref{def:f1f2} satisfy $f_1\in\XX_{2/3,0}$ and $f_2\in\XX_{4/3,1/2}$. Moreover,
\[
 \lln f_1\rrn_{2/3,0}\lesssim G_0^{-1}\quad\text{ and }\quad \lln f_2\rrn_{4/3,1/2}\lesssim G_0^{-1}.
\]
\end{lemma}
Finally we give properties of the operators introduced in \eqref{def:Outer:IntegralOperator} and \eqref{def:OperadorGA}.
\begin{lemma}\label{lemma:Operator:1}
The operator $\GG=\GG^ s$ in \eqref{def:Outer:IntegralOperator}, when
considered acting on the spaces $\XX_{n,m}$ and $\YY_{n,m}$ has the following
properties.
\begin{enumerate}
\item For any $n> 1$ and $m\geq 1$, $\GG:\YY_{n,m}\longrightarrow
\YY_{n-1,m-1}$ is well defined and linear continuous.
Moreover $\LL\circ\GG=\mathrm{Id}$.
\item If $h\in\YY_{n,m}$ for some $n> 1$ and $m>1$, $\GG(h)\in \YY_{n-1,m-1}$
and
\[
 \| \GG\left(h\right)\|_{n-1,m-1}\lesssim  \| h\|_{n,m}.
\]
 \item If $h\in\YY_{n,1}$ for some $n>1$, $\GG\left(h\right)\in
\YY_{n-1, 0}$ and
\[
 \|\GG\left(h\right)\|_{n-1, 0}\lesssim  \ln G_0\| h\|_{n, 1}
\]
\item If $h\in\YY_{n,m}$ for some $n\geq 1$ and $m\geq 1$ satisfies $\langle
h\rangle_\ga=0$, $\GG(h)\in
\YY_{n,m}$
and
\[
 \| \GG\left(h\right)\|_{n,m}\lesssim G_0^{-3}  \| h\|_{n,m}.
\]

\item If $h\in\YY_{n,m}$ for some $n\geq 1$ and $m\geq 1$, $\pa_u\GG(h),\pa_\ga\GG(h)\in
\YY_{n,m}$
and
\[
\begin{split}
 \| \pa_u\GG\left(h\right)\|_{n,m}&\lesssim  \| h\|_{n,m}\\
 \|\pa_\ga \GG\left(h\right)\|_{n,m}&\lesssim G_0^{-3} \| h\|_{n,m}.
\end{split}
 \]
\item From the previous statements, one can conclude that if $h\in\YY_{n,m}$
for some $n>1$ and $m\geq 1$, then $\GG(h)\in \XX_{n-1,m-1}$ and
\[
 \begin{aligned}
  \lln \GG(h)\rrn_{n-1,m-1}&\lesssim \|h\|_{n,m} \qquad\qquad\,\,\,\, \  &&\text{
if }m>1\\
 \lln \GG(h)\rrn_{n-1,m-1}&\lesssim \ln G_0\|h\|_{n,m} \qquad &&\text{ if }m=
1.
 \end{aligned}
 \]
 \end{enumerate}
 Additionally,
\begin{enumerate}
\item[(7)]
if $h\in\YY_{n,m}$ for $n>1$, $m>3/2$ then
\[
\| e^{\pm i \phi_h} \GG( e^{\mp i \phi_h} h)\|_{n-1,m-1} \lesssim  \|h \|_{n,m},
\]
\item[(8)]
if $h\in\YY_{n,3/2}$ for $n>1$, then
\[
\| e^{\pm i \phi_h} \GG( e^{\mp i \phi_h} h)\|_{n-1,1/2} \lesssim \ln G_0 \|h \|_{n,3/2}.
\]
\end{enumerate}
\end{lemma}
Claims 1 to 6 in this lemma are proved for $m>1$ in
\cite{GuardiaMS16}. The case $m=1$ can be proven analogously.  Claims 7 and 8 can be deduced analogously
taking into account the expression of $e^{\mp i \phi_h}$ given in Lemma~\ref{lem:HomoInfinity}

From this
result we can deduce the following lemma, which is a direct consequence of Lemmas \ref{lemma:Operator:1} and \ref{lemma:f1f2}.


\begin{lemma}\label{lemma:Operator:2}
Consider $h\in\YY_{n,m,\vect}$ for $n>1$ and  $m\geq 3/2$. Then, the operator $\GG_A$ introduced in \eqref{def:OperadorGA} satisfies the following.
\begin{itemize}
 \item If $m>3/2$, $\GG_A\left(h\right)\in \XX_{n-1,m-1,\vect}$
and
$ \lln \GG_A\left(h\right)\rrn_{n-1,m-1,\vect}\lesssim  \|
h\|_{n,m,\vect}$,
\item If $m=3/2$, $\GG_A\left(h\right)\in \XX_{n-1,1/2,\vect}$
and $\lln \GG_A\left(h\right)\rrn_{n-1,1/2,\vect}\lesssim  \ln G_0\|
h\|_{n,3/2,\vect}$.
\end{itemize}
\end{lemma}

\subsubsection{The fixed point argument: Proof of Theorem \ref{thm:ExistenceManiFixedPt}}\label{sec:proofexistencemani}
To prove the existence of a fixed point of the operator $\widetilde{\FF}$ defined in \eqref{eq:Ftilde}, we start by analyzing $\widetilde{\FF}(0)= (\Phi_A -\mathrm{Id})\delta z+\GG_A\circ F(\delta z)$ (see \eqref{eq:InvarianceVect} and \eqref{def:OperadorGA}) where $\delta z =(0,0,\delta_\eta,\delta_\xi)$. Given $\widetilde{Z}\in\mathcal{X}_{n,m,\vect}$ we denote by $Z=\delta z+ \widetilde{Z}$. Since $F$ has several terms we split $F(Z)=-\GG_1(Z)\pa_uZ-\GG_2(Z)\pa_\ga Z+\QQ(Z)$ (see \eqref{eq:InvarianceVect})  as $F=F^1+F^2+F^3$ where
\begin{align}
F^1(Z)&=\dps-\GG_1(Z)\pa_uZ-\GG_2(Z)\pa_\ga Z\label{def:Exist:RHS:1}\\
F^2(Z)&=\dps\begin{pmatrix}\frac{\wh y_\h'(u) }{G_0\wh y_\h^3(u)}Y^2-f_1'(u)Yq-f_2'(u)\frac{q^2}{2}\\
 0 \\
 i \al\left(f_1(u)Y+f_2(u)\La-2f_2(u)\alo\beta-f_2(u)\al\beto-f_2(u)\al\beta\right)\\
 -i \bet\left(f_1(u)Y+f_2(u)\La-f_2(u)\alo\beta-2f_2(u)\al\beto-f_2(u)\al\beta\right)\end{pmatrix}\label{def:Exist:RHS:2}\\
F^3(Z)&=\begin{pmatrix}-\partial_u \PP_1(u,\lo,\La,\al,\bet)\\-\partial_\ga \PP_1(u,\lo,\La,\al,\bet)\\
- i\partial_\beta \PP_1(u,\lo,\La,\al,\bet)\\ i\partial_\al \PP_1(u,\lo,\La,\al,\bet)\end{pmatrix}.\label{def:Exist:RHS:3}
\end{align}
and $q=\La-\alo\beta-\beto\al-\al\beta$. First we notice that $F^1(\delta z)=0$. Denoting by $|\delta z|=|\delta \eta|+|\delta\xi|$ it is straightforward to check, using Lemma \ref{lemma:f1f2}, that
\[
\left\lVert F^2(\delta z)\right\rVert_{4/3,1/2}\lesssim G_0^{-1}|\delta z|^2.
\]
On the other hand, by the bounds of $\PP_1$ in Lemma \ref{lem:P1k} (see the estimates \eqref{eq:boundsP1q}) and the estimates for $\wh r_\h$ and $\wh y_\h$ in Lemma \ref{lem:HomoInfinity}, one has  that $F^3(\delta z )\in \YY_{2,3/2,\vect}$ and
\[
 \left\|F^3(\delta z)\right\|_{2,3/2,\vect}\lesssim G_0^{-3}.
\]
Then, applying Lemma \ref{lemma:Operator:2} one obtains $\widetilde{\FF}(0)\in\XX_{1/3,1/2,\vect}$ and that there exists $b_0>0$ such that
\begin{equation}\label{def:Ftitlla0}
\lln \widetilde{\FF}(0)\rrn_{1/3,1/2,\vect}\leq \frac{b_0}{4}(G_0^{-2}+|\delta z|+|\delta z|^2)G_0^{-1}\ln G_0\leq \frac{b_0}{2} G_0^{-3}\ln G_0 .
\end{equation}
where we have used that 
\[
\left\lVert (\Phi_A-\mathrm{Id})\delta z \right\rVert_{1/3,1/2,\mathrm{vec}}\lesssim G_0^{-1} |\delta z|
\]
and the hypothesis $|\delta z|\lesssim G_0^{-3}$.
Next step is to prove that $\widetilde{\FF}$ is contractive in the ball
$B(b_0G_0^{-3}\ln G_0)\subset\XX_{1/3,1/2,\vect}$.
For that we compute separately the Lipschitz constant of each of the terms $F^{i}(Z)$ for  $i=1,2,3$.

\begin{notation}
In the statements of the forthcoming lemmas, given an element $\widetilde{Z}\in \mathcal{X}_{1/3,1/2}$ we write $Z=\widetilde{Z}+\delta z$. 
\end{notation}
We also assume without mentioning that $|\delta\eta|,|\delta\xi|\lesssim G_0^{-3}$ and $\eta_0 G_0^{3/2}\ll1$.

\begin{lemma}\label{lemma:PropsP}
Consider $\widetilde{Z},\widetilde{Z}'\in \XX_{1/3,1/2,\vect}$ with $\lln \widetilde{Z}\rrn_{1/3,1/2,\vect}, \lln
\widetilde{Z}'\rrn_{1/3,1/2,\vect}\lesssim G_0^{-3}\ln G_0$. Then, the functions $\GG_1$ and $\GG_2$ introduced in \eqref{def:G1G2} satisfy that
\[
\begin{split}
 \|\GG_1(Z)\|_{2/3,1/2}, \|\GG_1(Z')\|_{2/3,1/2}&\lesssim G_0^{-4}\ln G_0\\
\|\GG_2(Z)\|_{1/3,1/2}, \|\GG_2(Z')\|_{1/3,1/2}&\lesssim \ln G_0
\end{split}
\]
and
\[
\begin{split}
 \| \GG_1(Z)-\GG_1(Z')\|_{2/3,1/2}&\lesssim G_0^{-1}\lln \widetilde{Z}-\widetilde{Z}'\rrn_{1/3,1/2,\vect}\\
\| \GG_2(Z)-\GG_2(Z')\|_{1/3,1/2}&\lesssim G_0^{3}\lln \widetilde{Z}-\widetilde{Z}'\rrn_{1/3,1/2,\vect}.
\end{split}
 \]
%
 \end{lemma}
This lemma is a direct consequence of the definition of $\GG_1$ and $\GG_2$ in
\eqref{def:G1G2} and Lemmas \ref{lem:HomoInfinity}
and \ref{lemma:banach:AlgebraProps}.
We use this lemma to compute the Lipschitz constant of $F^1$.

\begin{lemma}\label{lemma:LipschitzVarietats:F1}
Consider $\widetilde{Z},\widetilde{Z}'\in \XX_{1/3,1/2,\vect}$ with $\lln \widetilde{Z}\rrn_{1/3,1/2,\vect}, \lln
\widetilde{Z}'\rrn_{1/3,1/2,\vect}\leq b_0 G_0^{-3}\ln G_0$. Then, the function $F^1$ introduced in
\eqref{def:Exist:RHS:1} satisfies
\[
\begin{split}
\| F^1(Z)\|_{5/3,2,\vect}&\lesssim G_0^{-6}\ln^2 G_0\\ 
\| F^1(Z)- F^1(Z')\|_{5/3,2,\vect}&\lesssim G_0^{-3}\ln G_0\lln
\widetilde{Z}-\widetilde{Z}'\rrn_{1/3,1/2,\vect}.
\end{split}
\]
\end{lemma}

\begin{proof}
The first part plainly follows from the definition of $F^1$ and the estimates in Lemma \ref{lemma:PropsP}. We now obtain the result for the difference. For the second component can be written as
\[
\begin{split}
F^1_2(Z)-F^1_2(Z')=&
\,\left(\pa_u\La-\pa_u\La'\right)\GG_1(Z)+
\pa_u\La'\left(\GG_1(Z)-\GG_1(Z')\right)\\
  &\,
\left(\pa_\ga\La-\pa_\ga\La'\right)\GG_2(Z)+
\pa_\ga\La'\left(\GG_2(Z)-\GG_2(Z')\right).
 \end{split}
\]
Then, the estimate for the second component is a consequence of  Lemmas \ref{lemma:banach:AlgebraProps} and  \ref{lemma:PropsP} and the fact that $\lln \widetilde{Z}\rrn_{1/3,1/2,\vect}, \lln
\widetilde{Z}'\rrn_{1/3,1/2,\vect}\leq b_0 G_0^{-3}\ln G_0$. The other components can be estimated analogously.
\end{proof}

\begin{lemma}\label{lemma:LipschitzVarietats:F2}
Consider $\widetilde{Z},\widetilde{Z}'\in \XX_{1/3,1/2,\vect}$ with $\lln \widetilde{Z}\rrn_{1/3,1/2,\vect}$, $\lln
\widetilde{Z}'\rrn_{1/3,1/2,\vect}\leq b_0 G_0^{-3}\ln G_0$. Then,
the function $F^2$ introduced in \eqref{def:Exist:RHS:2} satisfies
%
\[
\begin{split}
\| F^2(Z) \|_{4/3,2, \vect}&\lesssim G_0^{-7}\ln^2 G_0\\
\| F^2(Z)-F^2(Z')\|_{4/3,2, \vect}&\lesssim G_0^{-4}\ln G_0 \lln
\widetilde{Z}-\widetilde{Z}'\rrn_{1/3,1/2,\vect}.
\end{split}
\]
\end{lemma}
\begin{proof}

We recall that $F^2$ was defined in~\eqref{def:Exist:RHS:2}.
For the first component $F_1^2$ we obtain
\[
\| F_1^2(Z) \|_{7/3,3}\lesssim G_0^{-1} \left( \| Y\|^2_{4/3,3/2}+\| Y\|_{2/3,5/2} \|q\|_{0,1/2}+\|q\|^2_{0,1/2}\right) \lesssim G_0^{-7}\ln^2G_0,
\]
where we have used that $\|q \|_{0,1/2}\lesssim (|\delta z|+\| \tilde{Z}\|_{0,1/2})\lesssim G_0^{-3}\ln G_0$. On the other hand, for the difference
\[
\left\| F^2_1(Z)-F^2_1(Z')\right\|_{7/3,3}\lesssim G_0^{-4}  \ln G_0 \|
\widetilde{Z}-\widetilde{Z}'\|_{1/3,1/2,\vect}.
\]
%
Similar computations lead to the following estimate for the third component
\[
 \left\| F^2_3(Z)\right\|_{4/3,2}\lesssim G_0^{-7}\ln G_0
\]
and to the bound for the difference
\[
 \left\| F^2_3(Z)-F^2_3(Z')\right\|_{4/3,2}\lesssim G_0^{-4}\ln G_0\|
\widetilde{Z}-\widetilde{Z}'\|_{1/3,1/2,\vect}
\]
Proceeding analogously one obtains the same estimate for $F^2_4$. Since $F_2^2=0$,
the claim follows.
\end{proof}

\begin{lemma}\label{lemma:LipschitzVarietats:F3}
Consider $\widetilde{Z},\widetilde{Z}'\in \XX_{1/3,1/2,\vect}$ with $\lln \widetilde{Z}\rrn_{1/3,1/2,\vect}, \lln
\widetilde{Z}'\rrn_{1/3,1/2,\vect}\leq b_0 G_0^{-3}\ln G_0$. Then, the function $F^3$ introduced in
\eqref{def:Exist:RHS:3} satisfies
\[
\begin{split}
\|F^3(Z)-F^3(0)\|_{2,2,\vect}&\lesssim G_0^{-6}\ln^2 G_0\\
\|F^3(Z)-F^3(Z')\|_{2,2,\vect}&\lesssim G_0^{-3}\lln
\widetilde{Z}-\widetilde{Z}'\rrn_{1/3,1/2,\vect}.
\end{split}
\]
\end{lemma}
\begin{proof}
To prove the first statement, one can write 
\[
 F^3(Z)-F^3(0)=(F^3(\de z+\wt Z)-F^3(\de z))+(F^3(\de z)-F^3(0))
\]
For the first term in the right hand side, it  is enough to apply Lemma \ref{lem:CompositionManifolds}. To estimate the second term, one can use the mean value theorem and the estimates for the derivatives of $\PP_1$ given in Lemma \ref{lem:P1k} (and Cauchy estimates). The second statement in Lemma \ref{lemma:LipschitzVarietats:F3} is a direct consequence of Lemma \ref{lem:CompositionManifolds}.
\end{proof}

Now we are ready to prove Theorem \ref{thm:ExistenceManiFixedPt}.
\begin{proof}[Proof of Theorem \ref{thm:ExistenceManiFixedPt}]
Lemmas  \ref{lemma:LipschitzVarietats:F1}, \ref{lemma:LipschitzVarietats:F2} and \ref{lemma:LipschitzVarietats:F3} imply that
\[
\|F(Z)-F(Z')\|_{4/3,2,\vect}\lesssim G_0^{-3}\ln G_0 \lln
\widetilde{Z}-\widetilde{Z}'\rrn_{1/3,1/2,\vect}.
\]
for $\widetilde{Z},\widetilde{Z}'\in B(b_0 G_0^{-3})\subset \YY_{1/3,1/2,\vect}$. Therefore, applying Lemma \ref{lemma:Operator:2}, one has
\begin{equation}\label{def:LipschitzImproved}
\lln \widetilde{\FF}(\widetilde{Z})-\widetilde{\FF}(\widetilde{Z}')\rrn_{1/3,1,\vect}\lesssim G_0^{-3}\ln^2 G_0\lln
\widetilde{Z}-\widetilde{Z}'\rrn_{1/3,1/2,\vect},
\end{equation}
which implies
\[
\lln \widetilde{\FF}(\widetilde{Z})-\widetilde{\FF}(\widetilde{Z}')\rrn_{1/3,1/2,\vect}\lesssim G_0^{-3/2}\ln^2 G_0\lln
\widetilde{Z}-\widetilde{Z}'\rrn_{1/3,1/2,\vect}.
\]
Thus, for $G_0$ large enough, $\widetilde{\FF}$ is contractive from $B(b_0 G_0^{-3}\ln G_0)\subset \YY_{1/3,1/2,\vect}$ to itself with Lipschitz constant of size $\Lip\lesssim G_0^{-3/2}\ln^2 G_0$  and it has a unique fixed point $\widetilde{Z}^s$. Denote now by $Z^s=\delta z+\widetilde{Z}^s$. By definition of the operator $\mathcal{F}$ we have that 
\[
Z^s-\mathcal{F}(0)=\mathcal{G}_A(F(Z^s)-F(0))=\mathcal{G}_A F_1(Z^s)+\mathcal{G}_A F_2(Z^s)+ \mathcal{G}_A( F_3(Z^s)- F_3(0))
\]
so it follows from Lemmas  \ref{lemma:LipschitzVarietats:F1}, \ref{lemma:LipschitzVarietats:F2} and \ref{lemma:LipschitzVarietats:F3} that $Z^s-\mathcal{F}(0)\in \mathcal{X}_{1/3,1/2,\vect}$ and 
\[
\begin{split}
\| Z^s-\mathcal{F}(0)\|_{1/3,1/2,\vect} \lesssim& G_0^{3/2}\ln G_0 \|   F(Z^s)-F(0)\|_{1/3,1,\vect}\\
\lesssim & G_0^{3/2} \lln \mathcal{G}_A F_1(Z^s)+\mathcal{G}_A F_2(Z^s)+ \mathcal{G}_A( F_3(Z^s)- F_3(0))\lln_{1/3,1,\vect}\\
\lesssim&  G_0^{3/2} ( G_0^{-6}\ln^2 G_0+ G_0^{-7}\ln^2 G_0 + G_0^{-6}\ln^3 G_0)  \lesssim  G_0^{-9/2}\ln^3 G_0
\end{split}
\]
as was to be shown. Now it only remains to obtain the improved estimates for $\widetilde{\Lambda}^s$. Since it is a fixed point of $\widetilde{\FF}$, it can be written as
\[
 \widetilde{\La}^s= \widetilde{\FF}_2(0)+\left(\widetilde{\FF}_2(\widetilde{Z}^s)-\widetilde{\FF}_2(0)\right).
\]
For the second term, we use \eqref{def:LipschitzImproved} to obtain
\[
\lln \widetilde{\FF}_2(\widetilde{Z}^s)-\widetilde{\FF}_2(0)\rrn_{1/3,1}\lesssim  G_0^{-3}\ln^2 G_0 \lln \widetilde{Z}^s\rrn_{1/3,1/2}\lesssim G_0^{-6}\ln^3 G_0.
\]
Using \eqref{def:OperatorInfinity}, we write the first term  as $\widetilde{\FF}_2(0)=\GG(F_2(\delta z))$ where $F_2(\delta z)=-\partial_\ga \PP_1(u,\lo,0,\delta \eta,\delta\xi)$. Since $\langle F_2(\delta z)\rangle_\ga=0$ and satisfies $\|F_2(\delta z)\|_{2,3/2}\lesssim G_0^{-3}$ (see Lemma \ref{lem:P1k}), one can apply item 4 of Lemma \ref{lemma:Operator:1} one obtains
\[
 \lln\widetilde{\FF}_2(0)\rrn_{1,1}\lesssim G_0^{3/2} \lln\widetilde{\FF}_2(0)\rrn_{2,3/2}\lesssim  G_0^{-3/2}\left\|F_2(\delta z)\right\|_{2,3/2}\lesssim G_0^{-9/2}.
\]

\end{proof}

\subsubsection{Proof of Proposition \ref{prop:InvManReal}}\label{sec:Proof:InvReal}
The proof of Theorem \ref{thm:ExistenceManiFixedPt} can be carried out in the same way in the smaller domain
\[
\wt{D}^{u}=\Big\{u\in\CC; |\Im u|<-\tan \beta_1\Re v+1/4,
|\Im u|>\tan \beta_2\Re u+1/6-\de\Big\}
\]
where all the points at a uniform distance from the singularities $u=\pm i/3$. In this case the perturbing potential $\PP_1$ can be easily bounded of order $G_0^{-3}$ with the norm \eqref{def:Norm} with $m=0$ for any $(\alo,\beto)$ satisfying $|\alo|,|\beto|\leq 1/2$. That is, without imposing condition \eqref{cond:eccentricity}. Note that now the weight at the singularities $u=\pm i/3$, measured by $m$, is harmless since the points  $\wt{D}^{u}$ are $\OO(1)$-far from them.
%

This  gives the estimates for the invariant manifolds. One can obtain the improved estimate for  $\La$  as has been done in the proof of Theorem \ref{thm:ExistenceManiFixedPt}. To obtain the estimates for the derivatives it is enough to apply Cauchy estimates. Note that in all the variables one can apply these estimates in disks of radius independent of $G_0$.

\subsection{Extension of the parametrization of the unstable manifold by the flow}\label{sec:InvManifold:ExtensionUnstable}
Theorem~\ref{thm:ExistenceManiFixedPt} gives a graph
parameterization $Z^s$ of the form  \eqref{def:Parameterizations} of the  stable invariant manifold of the periodic orbit $P_{\eta_0+\delta\eta,\xi_0+\delta\xi}$ as  a formal Fourier
series with analytic Fourier coefficients defined  in the domain
$D^{s}_{\rr_1,\kk_0,\de_0}\times \TT$.

To compute the difference betweent the stable and unstable manifold, it is necessary to have the parameterizations
of both manifolds defined in a common (real) domain. However, since
$\wh y_\h(0) = 0$, it is not possible to extend these
parameterizations to a common domain containing a real interval (see \eqref{def:ChangeThroughHomo}).  Therefore, to compare them,  we extend the unstable manifold using a
different parametrization. We proceed analogously as in \cite{GuardiaMS16}.

Theorem \ref{thm:ExistenceManiFixedPt} gives a paramerization $Z^s$ of the unstable invariant manifold of the periodic orbit $P_{\alo+\delta\eta,\beto+\delta\xi}$ associated to Hamiltonian \eqref{def:HamFinal} as a graph of the form
\begin{equation}\label{def:sigmaparam}
Z^s(u,\ga)=
 \left(u, Y(u,\ga), \ga, \La(u,\ga), \al(u,\ga), \bet(u,\ga)\right).
\end{equation}
The first step is to look for  a
change of variables of the form
\begin{equation}\label{def:ChangeToFlow}
\Id +g:(v,\xi)\mapsto
(u,\ga)=(v+g_1(v,\xi),\xi+g_2(v,\xi)),\,\,
\end{equation}
 in
such a way that, applied to~\eqref{def:sigmaparam},
$ \tZ^s = Z^s \circ (\Id +g)$
satisfies the invariance equation
\begin{equation}\label{invariancebytheflow21}
 \Phi_t\left( \tZ^s (v,\xi)\right) = 
\tZ^s \left(v+t,\xi+\frac{\nu G_0^3}{L_0^3} t\right),
\end{equation}
where $\Phi_t$ is the flow associated to the
Hamiltonian system~\eqref{def:HamFinal}. Note that the composition is understood as  formal composition of
formal Fourier series
\[
 h\circ (\Id +g) (v,\xi) = h(v+g_1(v,\xi),\xi+g_2(v,\xi))
=\sum_{m=0}^\infty \frac{1}{m!}\sum_{n=0}^m \binom{m}{n}
 \pa_v^{m-n}  \pa_\xi^{n} h(v,\xi) g_1^{m-n}(v,\xi) g_2^{n}(v,\xi).
\]
Denoting by $X$ the
associated vector field to Hamiltonian~\eqref{def:HamFinal} 
equation \eqref{invariancebytheflow21} is
equivalent to
\begin{equation}\label{invariancebytheflow2}
\LL( \tZ^s) = X \circ  \tZ^s,
\end{equation}
where the operator $\LL$ is defined in~\eqref{def:OperadorL}.

We want to obtain a parameterization of the unstable manifold of the form  $\wh \tZ^s$ in the domain $D
^\fl_{\kk,\de}\times \TT$,  where
\begin{equation}\label{def:DominiParamFlux}
D^\fl_{\kk,\de}=
\left\{v\in\CC; |\Im v|<\tan\beta_1\Re
v+1/3-\kk G_0^{-3},\,|\Im v|<-\tan\beta_2\Re v+1/6+\de\right\},
\end{equation}
which can be seen in Figure~\ref{fig:DomFlux}.

We will relate the two types of the parameterization in the overlapping domain
\begin{equation}\label{def:DomainOverlap}
D^\ovr_{\kk,\de}=D^\fl_{\kk,\de}\cap D^s_{\kk,\de}.
\end{equation}
Proceeding as in \cite{GuardiaMS16}, one can obtain in this domain the change of coordinates \eqref{def:ChangeToFlow}. Abusing notation, we use the Banach space $\YY_{n,m}$ introduced in Section \ref{Sec:FixtPt}. Recall that the index $n$ referred to the decay at infinity and therefore it does not give any information in the compact domain $D^\ovr_{\kk,\de}$ and therefore we can just take $n=0$.

\begin{lemma}\label{lemma:GraphToFlow}
Let $\de_0$,  $\kk_0$ and  $\sigma_0$  be the constants
fixed in the statement of
Theorem~\ref{thm:ExistenceManiFixedPt}. Let $\sigma_1<\sigma_0$, $\de_1<\de_0$ and $\kk_1>\kk_0$  such that $(\log
\kk_1-\log \kk_0)/2 < \sigma_0 -\sigma_1$ be fixed. Then, for $G_0$
big enough, there exists a (not necessarily convergent) Fourier
series $g =(g_1,g_2) \in \YY_{0,\rr_2,\kk_1,\de_1,\sigma_1}\times
\YY_{0,\rr_2,\kk_1,\de_1,\sigma_1}$ satisfying
\[
\|g_1\|_{0,0}\leq b_2 G_0^{-4}, \qquad \|g_2\|_{0,0}\leq b_2 G_0^{-3/2},
\]
where $b_2>0$ is a constant independent of  $G_0$, such
that
$\tZ^s = Z^s \circ (\Id + g)$,
satisfies~\eqref{invariancebytheflow2}.
\end{lemma}

Once we have obtained a parameterization $\tZ^s$ which satisfies \eqref{invariancebytheflow2} in the overlapping domain $D^\ovr_{\kk,\de}$ (see \eqref{def:DomainOverlap}), next step is to extend this parameterization to the domain $D^\fl_{\kk,\de}$ in \eqref{def:DominiParamFlux}. This extension is done through a fixed point argument and follows the same lines as the flow extension of \cite{GuardiaMS16} (Section 5.5.2). We  write the parameterization $ \tZ^s$ as $\tZ^s(v,\xi)=  \tZ^s_0(v,\xi)+ \tZ^s_1(v,\xi)$ with
\begin{equation}\label{def:FlowParam}
  \tZ^s_0= \begin{pmatrix} v\\0\\\xi\\0\\ 0 \\ 0\end{pmatrix},\qquad\tZ^s_1= \begin{pmatrix} G_0^{-1}\tU(v,\xi)\\ \tY (v,\xi)\\ G_0^3 \tGa(v,\xi)\\\tLa(v,\xi)\\ \tA(v,\xi) \\ \tB(v,\xi)\end{pmatrix}.
\end{equation}
The $G_0$--factors in the $u$ and  $\ga$ component is just to normalize the sizes. In the statement of the following lemma, abusing notation, we also use the Banach space $\YY_{n,m}$ introduced in Section \ref{Sec:FixtPt} referred to the domain $D^\fl_{\kk,\de}$. Since the domain $D^\fl_{\kk,\de}$ is compact and all points are at a distance independent of $G_0$ from the singularities $v=\pm i/3$, we can just take $n=m=0$ (all norms $\|\cdot\|_{n,m}$ are equivalent).

\begin{lemma}\label{lemma:FlowExtension}
Let $\kk_1$ , $\de_1$, $\sigma_1$ be the constants considered in Lemma \ref{lemma:GraphToFlow}. Then, there exists a solution of equation \eqref{invariancebytheflow2} of the form \eqref{def:FlowParam} (as a formal Fourier series) for $(v,\xi)\in D^\fl_{\kk_1,\de_1}\times\TT_{\sigma_1}$, whose Fourier coefficients are analytic continuation of those obtained in Lemma \ref{lemma:GraphToFlow}. Moreover, they satisfy $ \tZ^s_1\in \YY_{0,0}^6$ and
\[
\| \tU\|_{0,0}, \|\tY\|_{0,0}, \|  \tGa\|_{0,0}, \| \tLa\|_{0,0}, \| \tA\|_{0,0}, \|  \tB\|_{0,0}\lesssim G_0^{-3}.
\]
\end{lemma}
The proof of this lemma is through a fixed point argument analogous to the one considered in the proof of Proposition 5.20 in \cite{GuardiaMS16}. This is a standard fixed point argument in the sense that the domain $D^\fl_{\kk,\de}$ is ``far'' from the singularities $v=\pm i/3$ (the distance to these points is independent of $G_0$). The only issue that one has to keep in mind that we are dealing with formal Fourier series and therefore  all equations (and the compositions of functions involved in them) have to be understood in terms of formal Fourier series.

Once we have obtained this flow parameterizations in $D^\fl_{\kk_1,\de_1}$, the last step is to switch back to the graph parameterization \eqref{def:Parameterizations}. We want the graph parameterization to be defined in the  following domain where we can compare the graph parameterizations of the stable and unstable invariant manifolds.
\begin{equation}\label{def:DominisRaros}
\begin{split}
D_{\kk,\de}=
&\left\{v\in\CC;\right.\left. |\Im v|<\tan\beta_1\Re
v+1/3-\kk G_0^{-3},
|\Im v|<-\tan\beta_1\Re v+1/3-\kk G_0^{-3},\right.\\
& \left.|\Im v|>-\tan\beta_2\Re v+1/6-\delta\right\},
\end{split}
\end{equation}
where  $\kappa\in (0,1/3)$, $\de\in (0,1/12)$  and
$\beta_1,\beta_2\in (0,\pi/2)$  are fixed  independently of~$G_0$ (see Figure \ref{fig:DomRaro}).
Therefore, this domain is not empty provided $G_0>1$.

\begin{figure}[H]
\begin{center}
\includegraphics[height=5cm]{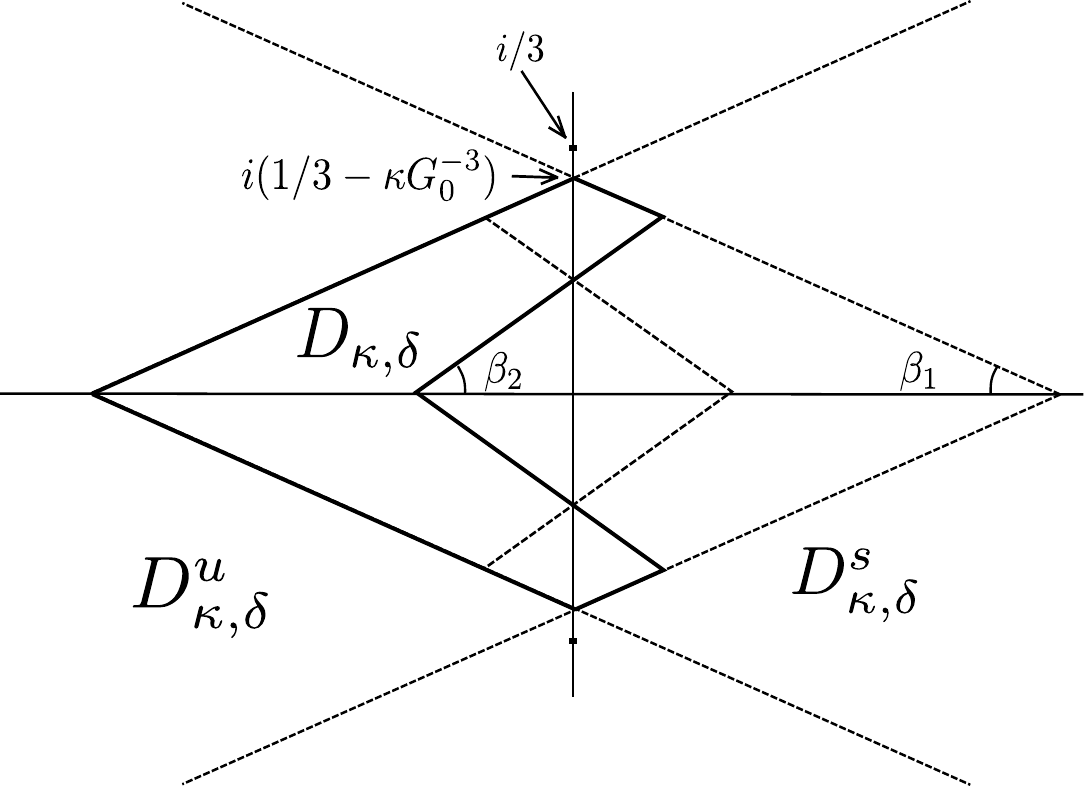}
\end{center}
\caption{The domains  $D_{\kk,\de}$ defined in
\eqref{def:DominisRaros}.}\label{fig:DomRaro}
\end{figure}

Note that,  Theorem \ref{thm:ExistenceManiFixedPt} gives already the graph parameterization $Z^s$ in the domain $D_{\kk,\de}\cap D^s_{\kk,\de}$. Now it only remains to show that they are also defined in the domain 
\begin{equation}\label{def:Domain:UnstableLastPiece}
\begin{split}
\wt D_{\kk,\de}=
\Big\{v\in\CC;\, &|\Im v|<\tan\beta_1\Re
v+1/3-\kk G_0^{-3}, |\Im v|>\tan\beta_2\Re v+1/6-\delta\\
& |\Im v|<\tan\beta_2\Re v+1/6+\delta\Big\}.
\end{split}
\end{equation}
(see Figure~\ref{fig:DomExtFinal}).
\begin{figure}[H]
\begin{center}
\includegraphics[height=5cm]{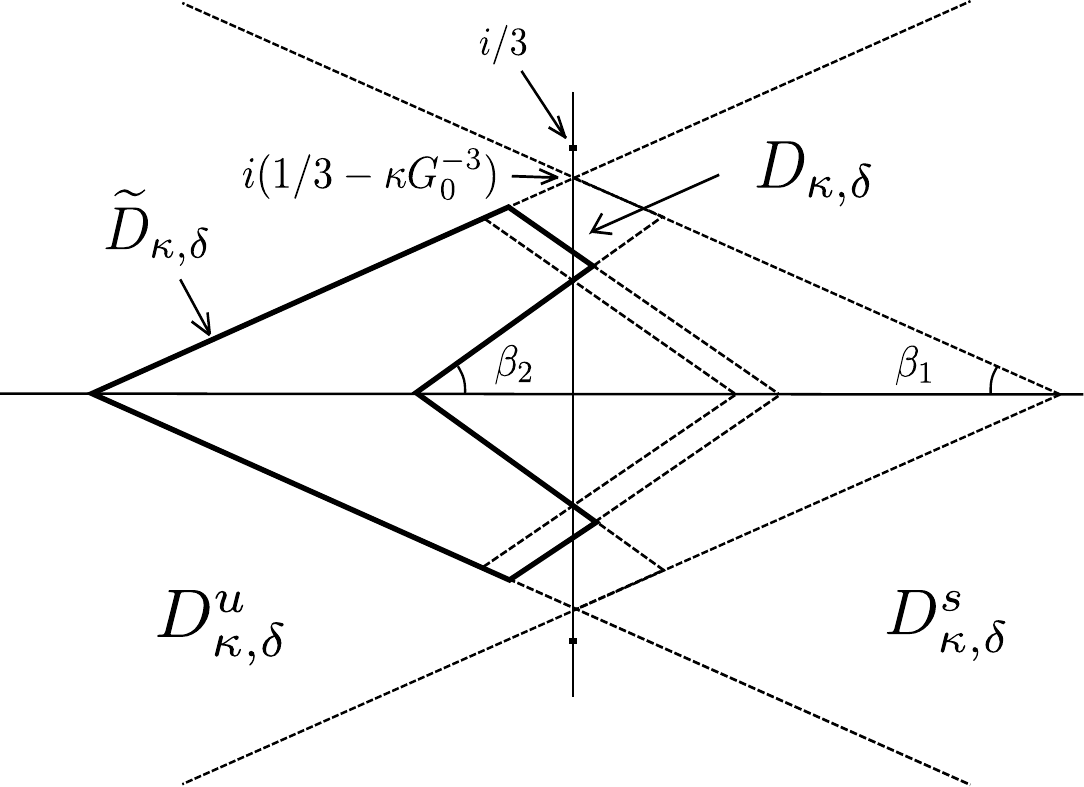}
\end{center}
\caption{The domain  $\wt D_{\kk,\de}$ defined in
\eqref{def:Domain:UnstableLastPiece}.}\label{fig:DomExtFinal}
\end{figure}
Indeed, it is easy to see that
\begin{equation*}
 D_{\kk,\de}\subset D^s_{\kk,\de}\cup\wt D_{\kk,\de}.
\end{equation*}
To this end, we look for a change of coordinates which transforms the flow--parameterization obtained in Lemma \ref{lemma:FlowExtension} to the graph parameterization \eqref{def:Parameterizations}. Note that this change is just the inverse of the change obtained in Lemma \ref{lemma:GraphToFlow} (in the domain where both are defined). The change we look for is just the inverse of
\begin{equation*}
 \begin{split}
 u&=v+G_0^{-1}\tU(v,\xi)\\
 \ga&=\xi+G_0^3 \tGa(v,\xi)
 \end{split}
\end{equation*}

\begin{lemma}\label{lemma:FromParamToGraph}
Consider  the constants $\kk_1$, $\de_1$ and $\sigma_1$ considered in
Lemma~\ref{lemma:FlowExtension} and any $\kk_2>\kk_1$,
$\de_2>\de_1$ and $\sigma_2<\sigma_1$. Then,
\begin{itemize}
 \item  
 There exists a function $h=(h_1,h_2) \in \YY_{0,0}$ with
\[
\|h_1\|_{0,0} \le b_4 \mu G_0^{-4},\qquad \|h_2\|_{0,0} \le b_4 \mu G_0^{-1}.
\]
such that the change of coordinates $\Id+h$ is the inverse of the restriction of
the change given by Lemma \ref{lemma:GraphToFlow} to
the domain~$D^{s}_{\kk_1,\de_1} \cap \wt D_{\kk_2,\de_2}$.
\item 
Moreover,
\[
Z^s= \tZ^s \circ (\mathrm{Id}+h)
\]
defines a formal Fourier series which gives a parameterization of the stable invariant manifold as a graph, that is of the form \eqref{def:Parameterizations}. 
Then, in the domain $D_{\kk_2,\de_2}\times \TT_{\sigma_2}$ this parameterization satisfies
\[
\begin{aligned}
 \|Y\|_{0,3/2}&\lesssim G_0^{-3}\ln G_0, &\,\|\La\|_{0,3/2}&\lesssim& G_0^{-9/2}, \\
 \|\al e^{i\phi(u)}\|_{0,1/2}&\lesssim G_0^{-3}\ln G_0,&\, \|\bet e^{-i\phi(u)}\|_{0,1/2}&\lesssim& G_0^{-3}\ln G_0.
 \end{aligned}
\]
\end{itemize}
\end{lemma}

%

\begin{figure}[H]
\begin{center}
\includegraphics[height=5cm]{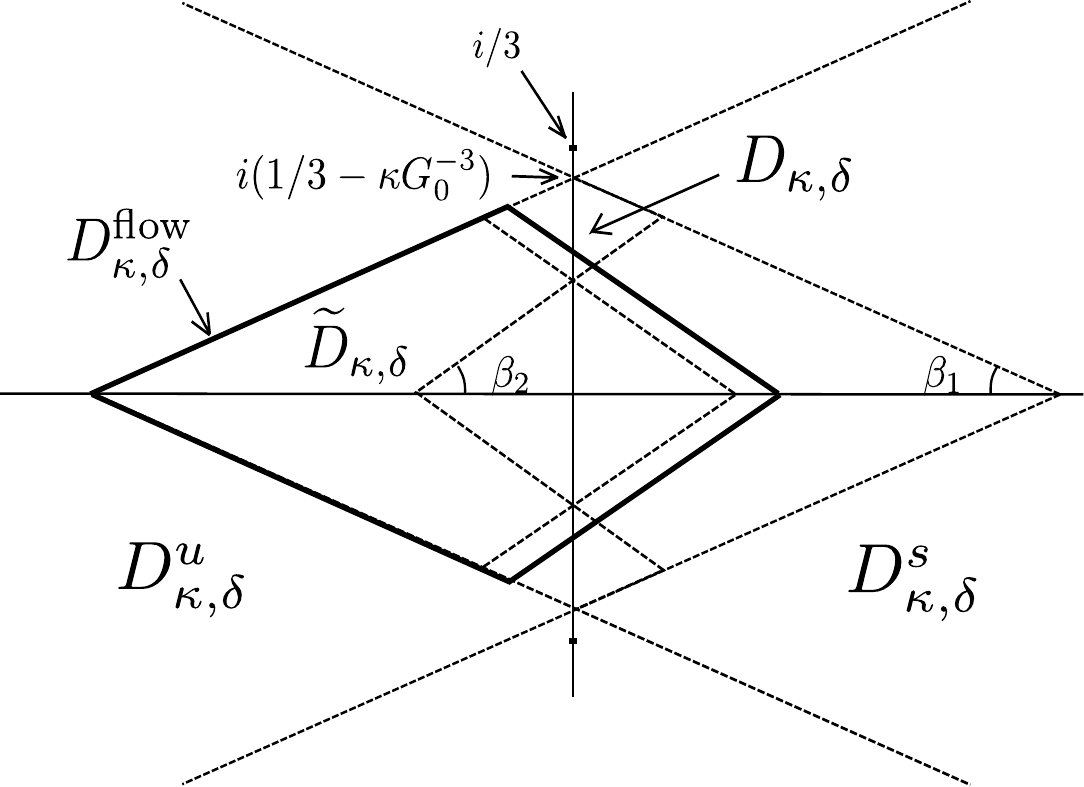}
\end{center}
\caption{The domain  $ D^\fl_{\kk,\de}$ defined in
\eqref{def:Domain:UnstableLastPiece}.}\label{fig:DomFlux}
\end{figure}

Taking $\delta z=0$ and proceeding analgously one obtains  analogous result to Theorems \ref{thm:ExistenceManiFixedPt} and Proposition \ref{prop:InvManReal} for the unstable manifold of the periodic orbit $P_{\eta_0,\xi_0}$. For the sake of clarity, we sum up the properties of the graph parameterizations of the unstable manifold of $P_{\eta_0,\xi_0}$ and the stable manifold of $P_{\eta_0+\delta\eta,\xi_0+\delta\xi}$ in the following Theorem.

\begin{theorem}\label{thm:parametrcommdomain}
Let $\delta z=(0,0,\delta\eta,\delta_\xi)$. Then, if $G_0\gg1$, $|\delta\eta|,|\delta\xi|\lesssim G_0^{-3}$ and 
\[
|\eta_0|G_0^{3/2}\ll 1, 
\]
the unstable manifold of the periodic orbit $P_{\eta_0,\xi_0}$ and the stable manifold of the periodic orbit $P_{\eta_0+\delta\eta,\xi_0+\delta\xi}$ admit graph parameterizations $Z^{u,s}:D_{\kappa,\delta}\to \mathbb{C}^4$ of the form \eqref{def:Parameterizations} which satisfy 
\[
\| Z^{s}-\mathcal{F}^{s}(0)\|_{0,1/2},\| Z^{u}-\mathcal{F}^{u}(0)\|_{0,1/2}\lesssim G_0^{-9/2}\ln^3 G_0
\]
where $\mathcal{F}^{s}$ is the operator defined in \eqref{def:OperatorInfinity} and $\mathcal{F}^{u}$ is defined analogously but taking $\delta z=0$. In particular, the estimates
\[
\begin{aligned}
 \|Y^*\|_{0,3/2}&\lesssim G_0^{-3}\ln G_0, &\,\|\La^*\|_{0,3/2}&\lesssim& G_0^{-9/2}, \\
 \|\al^* e^{i\phi(u)}\|_{0,1/2}&\lesssim G_0^{-3}\ln G_0,&\, \|\bet^* e^{-i\phi(u)}\|_{0,1/2}&\lesssim& G_0^{-3}\ln G_0.
 \end{aligned}
\]
and 
\[
 \| \La^*\|_{0,1}\lesssim G_0^{-9/2}
\]
hold for $*=u,s$.

Moreover these parameterizations admit an analytic continuation to the 
and for $N\geq 0$,
\[
|D^N( Z^s-\FF^s(0))|,|D^N( Z^u-\FF^u(0))|\lesssim G_0^{-6},
\]
where  $D^N$ denotes the differential of order $N$ with respect to the variables $(u,\ga,\alo,\beto)$ and $C$ is a constant which may depend on $N$ but independent of $G_0$.
\end{theorem}
%

\section{Proof of Theorem \ref{thm:MainSplitting}: The difference between the invariant manifolds of infinity}
\label{sec:Difference}
This section is devoted to prove Theorem \ref{thm:MainSplitting}.
Once we have obtained the parametrization of the  invariant
manifolds (as formal Fourier series) up to  points $\OO(G_0^{-3})$
close to the singularities $u=\pm i/3$ in Theorem \ref{thm:parametrcommdomain},
the next step is to study
their difference. We fix $(L_0, \alo, \beto)$ and
consider the parameterization $Z^u$ of the unstable manifold of the periodic orbit $P_{\alo,\beto}$ and the parameterization $Z^s$ of the stable manifold of the periodic orbit $P_{\alo+\delta\eta,\beto+\delta\xi}$.

We define then the difference vector $\Delta=\Delta(u,\ga)$ as
\begin{equation}\label{def:difference}
 \Delta=\left(Y^u- Y^s, \La^u-
\La^s, \al^u- \al^s,\bet^u-
\bet^s\right)^T.
\end{equation}
The
Fourier coefficients of  $\Delta$ are defined in the domain $D_{\kk,\de}$ introduced in \eqref{def:DominisRaros}.

Using the equations for $(Y^*, \La^*, \al^*, \bet^*)$ for $*=u,s$
in
\eqref{eq:InvarianceVect}, we have that $\Delta$ satisfies an equation of the
form
\begin{equation}\label{eq:Diffeq}
\wt\LL\Delta=A \Delta+B\Delta +R\Delta,
\end{equation}
where $\wt\LL$ is the linear operator
\begin{equation}\label{def:operatorLtilde}
\wt\LL=\LL+\GG_1(Z^u)\pa_u+\GG_2(Z^u)\pa_\ga,
\end{equation}
$\GG_1, \GG_2$ are the operators defined in \eqref{def:G1G2}, $A$ is the matrix introduced in
\eqref{def:A} and $B$ and $R$ are matrices which depend on $Z^u$ and
$Z^s$ and its derivatives and are expected to be small compared  to $A$. 
The matrix $B$ has only one non-zero term,
\begin{equation}\label{def:MatrixB}
\begin{split}
B_{21}&=-\frac{\pa_u \La^s}{G_0\wh y_\h^2}+f_1(u)\pa_\ga\Lambda^s\\
B_{ij}&=0\qquad \text{ otherwise}.
\end{split}
\end{equation}
where $f_1$ is the function introduced in \eqref{def:f1f2}. The matrix $R$ is defined as follows
\begin{equation}\label{def:matrixR}
\begin{aligned}
 R(u,\ga)=&\,\int_0^1 D_Z \QQ\left(u,\ga,
sZ^u(u,\ga)+(1-s)Z^s(u,\ga)\right)ds\\
 &\,-\pa_u Z^s(u,\ga)\int_0^1 D_Z \GG_1\left(u,\ga,
sZ^u(u,\ga)+(1-s)Z^s(u,\ga)\right)ds\\
&\,-\pa_\ga Z^s(u,\ga)\int_0^1 D_Z \GG_2\left(u,\ga,
sZ^u(u,\ga)+(1-s)Z^s(u,\ga)\right)ds-B,
\end{aligned}
 \end{equation}
where $\QQ$ is the function introduced in \eqref{def:Q}.  
The reason for defining the matrix $B$ and not putting all terms together in $R$ will be clear later. 
Note that $R$ satisfies
\[
 R_{21}=0.
\]
Roughly speaking, the first order of such equation is
$\LL\Delta=A\Delta$. To give an heuristic idea of the proof let us assume that $\Delta$ is a solution of this equation
instead of \eqref{eq:Diffeq}. 
Then, one can easily check that $\Delta$ must be
of the form
\[
 \Delta=\Phi_A C
\]
where $\Phi_A $ is the fundamental matrix introduced in
\eqref{def:FundamentalMatrix} (actually a suitable modification of it) and $C(u,\ga)$ is a vector whose $\ga$-- Fourier
coefficients are defined (and bounded) in $D_{\kk,\de}$ and satisfying $\LL
C=0$. Then, one can  show that for real values of the parameters the function
$C$ (minus its average with respect to $\ga$) is exponentially small (see Lemma
\ref{lemma:Lazutkin}).
%

Now, $\Delta$ is a solution of~\eqref{eq:Diffeq} instead of
$\LL\Delta=A\Delta$. 
Thus, to apply Lemma~\ref{lemma:Lazutkin} we adapt these
ideas. 
We do this in several steps. 
First, in Section~\ref{sec:diff:banach} we
describe the functional setting. 
Then, in Section~\ref{sec:change:straighteningoperator} we perform a symplectic change of
coordinates to straighten the operator in the left hand side of~\eqref{eq:Diffeq}. 
Then, in Section~\ref{sec:FundamentalMatrix}, we look for a fundamental solution of the transformed linear partial differential equation. 
Finally, in Section~\ref{sec:Lazutkin},
we deduce the asymptotic formula of the distance between the invariant manifolds and in Section~\ref{sec:Average} we obtain more refined estimates for the average of the difference of the $\Lambda$ component.

\subsection{Weighted Fourier norms and Banach spaces}\label{sec:diff:banach}
We define the Banach spaces for Fourier series
with coefficients defined in $D_{\kk,\de}$.
First, we define the
Banach spaces for the Fourier coefficients as
\[
\PP_{m,q,\kk,\de}=\left\{h:D_{\kk,\de}\rightarrow \CC: \text{analytic},
\|h\|_{m,q}<\infty  \right\},
\]
where
\[
 \|h\|_{m,q}=\sup_{u\in
D_{\kk,\de}}\left||u-i/3|^m|u+i/3|^{m}
e^{iq\phi_\h(u)}h(u)\right|.
\]
for any $m,q\in\RR$. 
Note that these definitions are the same
as in Section \ref{Sec:FixtPt} but for
functions defined in $D_{\kk,\de}$ instead of $D^s_{\kk,\de}$.
Now we define the Banach space for Fourier series
\[
\QQQ_{m,\kk,\de,\sigma}=\left\{h(u,\ga)=\sum_{\ell\in
\ZZ}h^{[\ell]}(u) e^{i\ell\ga}: h^{[\ell]}\in
\PP_{m,\ell,\kk,\de},
\|h\|_{m,\sigma}<\infty\right\},
\]
where
\[
\|h\|_{m,\sigma}=\sum_{\ell\in\ZZ}\left\|h^{[\ell]}\right\|_{m,
\ell} e^{|\ell|\sigma}.
\]
The Banach space $\QQQ_{m,\kk,\de,\sigma}$ satisfies the algebra properties stated in Lemma~\ref{lemma:banach:AlgebraProps}. 
From now on in this section, we will refer to this lemma understanding the properties stated in it as properties referred to elements of $\QQQ_{m,\kk,\de,\sigma}$ instead of elements of $\YY_{m,\kk,\de,\sigma}$.

Now, we need to define  vector and matrix norms associated to the just introduced norms. 
Those norms inherit the structure of the norms
considered in Section \ref{sec:ParamInvManifolds}.
We consider
\[
\QQQ_{m,\kk,\de,\sigma,\vect}= \QQQ_{m+1,\kk,\de,\sigma}\times
\QQQ_{m,\kk,\de,\sigma}^3
\]
with the norm
\begin{equation}\label{def:VectNorm:Diff}
\|Z\|_{m,\sigma, \vect}= \|Y\|_{m+1,\sigma}+ \|\La\|_{m,\sigma}+\|e^{i\phi_h}\al\|_{m,\sigma}+ \||e^{-i\phi_h}\bet\|_{m,\sigma}.
\end{equation}
Analogously, we consider the Banach space $\QQQ_{\nu,\kk,\de,\sigma,\matrx}$ of $4\times 4$ matrices  with the associated  norm
\begin{equation}\label{def:MatrixNorm}
\begin{split}
\|\Psi\|_{m, \sigma,\matrx}=\\
\max\Bigg\{&
\|\Psi_{11}\|_{m,\sigma}+ \|\Psi_{21}\|_{m-1,\sigma}+ \|e^{i\phi_\h(u)}\Psi_{31}\|_{m-1,\sigma}+
\|e^{-i\phi_\h(u)}\Psi_{41}\|_{m-1,\sigma},\\
&
\|\Psi_{12}\|_{m+1,\sigma}+ \|\Psi_{22}\|_{m,\sigma}+ \|e^{i\phi_\h(u)}\Psi_{32}\|_{m,\sigma}+
\|e^{-i\phi_\h(u)}\Psi_{42}\|_{m,\sigma},\\
&
\|e^{-i\phi_\h(u)}\Psi_{13}\|_{m+1,\sigma}+ \|e^{-i\phi_\h(u)}\Psi_{23}\|_{m,\sigma}+ \|\Psi_{33}\|_{m,\sigma}+
\|e^{-2i\phi_\h(u)}\Psi_{43}\|_{m,\sigma},\\
&
\|e^{i\phi_\h(u)}\Psi_{14}\|_{m+1,\sigma}+ \|e^{i\phi_\h(u)}\Psi_{23}\|_{m,\sigma}+ \|e^{2i\phi_\h(u)}\Psi_{34}\|_{m,\sigma}+
\|\Psi_{44}\|_{m,\sigma}\Bigg\}.
%
%
%
\end{split}
\end{equation}

\begin{lemma}\label{lemma:NormMatrix}
The norms $\|\cdot\|_{m, \sigma,\vect}$ and $\|\cdot\|_{m, \sigma,\matrx}$ introduced in \eqref{def:VectNorm:Diff} and \eqref{def:MatrixNorm} respectively have the following properties
\begin{itemize}
\item Consider $Z\in
\QQQ_{\nu,\kk,\de,\sigma,\vect}$ and a matrix $\Psi\in
\QQQ_{\eta,\kk,\de,\sigma,\matrx}$. Then, $\Psi Z\in
\QQQ_{\nu+\eta,\kk,\de,\sigma,\vect}$ and
\[
 \|\Psi Z\|_{\nu+\eta,\sigma, \vect}\lesssim \|\Psi\|_{\eta,\sigma,
\matrx}\|Z\|_{\nu,\sigma, \vect}.
\]
\item Consider matrices $\Psi\in
\QQQ_{\eta,\kk,\de,\sigma,\matrx}$ and $\Psi'\in
\QQQ_{\nu,\kk,\de,\sigma,\matrx}$. Then, $\Psi \Psi'\in
\QQQ_{\nu+\eta,\kk,\de,\sigma,\matrx}$ and
\begin{equation*}
 \|\Psi \Psi'\|_{\nu+\eta,\sigma, \matrx}\lesssim \|\Psi\|_{\eta,\sigma,
\matrx}\|\Psi'\|_{\nu,\sigma, \matrx}.
\end{equation*}
\end{itemize}
\end{lemma}

In the present section we will need to take derivatives of and
compose Fourier series. 
%
\begin{lemma}\label{lemma:banach:CauchyEstimates:diff}
Fix constants $\sigma'<\sigma$, $\kk'>\kk$ and $\de'>\de$ and take  $h\in \QQQ_{m,\kk,\de,\sigma}$. 
Its derivatives, as defined 
in~\eqref{def:DerivativeFourierSeries}, satisfy
\begin{itemize}
 \item $\pa_v^n h\in \QQQ_{m,\kk',\de',\sigma'}$ and
\[
 \|\pa_v^n h\|_{\nu,\sigma'}\leq \left(\frac{\kappa'}{\kappa}\right)^m
\frac{G_0^{3n}n!}{(\kappa'-\kappa)^n}
 \|h\|_{\nu,\sigma}.
\]
 \item $\pa_\xi h\in \QQQ_{m,\kk',\de',\sigma'}$ and
\[
 \|\pa_\xi h\|_{\nu,\sigma'}\leq \frac{1}{\sigma-\sigma'}\|h\|_{\nu,\sigma}.
\]
\end{itemize}
\end{lemma}

\begin{lemma}\label{lemma:banach:Composition}
We define the formal composition of formal Fourier series
\[
 h(v+g(v,\xi),\xi)=\sum_{n=0}^\infty \frac{1}{n!}\pa_v^nh(v,\xi) g^n(v,\xi).
\]
Fix constants $\sigma'<\sigma$, $\kk'>\kk$ and $\de'>\de$. Let
$\kk'-\kk>\eta>0$. Then,
\begin{itemize}
\item 
If $h\in \QQQ_{m,\kk,\de,\sigma}$, $g\in \QQQ_{0,\kk',\de',\sigma'}$ and
$\|g\|_{0,\sigma'}\leq \eta G_0^{-3}$
we have that
 $X(v,\xi)=h(v+g(v,\xi),\xi)$ satisfies  $X\in \YY_{m,\kk',\de',\sigma'}$ and
\[
 \|X\|_{m,\sigma'}\leq \left(\frac{\kk'}{\kk}\right)^m
\left(1-\frac{\eta}{\kk'-\kk}\right)^{-1}
 \|h\|_{m,\sigma}.
\]
Moreover, if $\|g_1\|_{0,\sigma},\|g_2\|_{0,\sigma}\leq \eta
G_0^{-3}$, then $Y(v,\xi)=h(v+g_2(v,\xi),\xi)-h(v+g_1(v,\xi),\xi)$
satisfies
\[
 \|Y\|_{m,\sigma'}\leq \frac{G_0^{3}}{\kk'-\kk}
 \left(\frac{\kk'}{\kk}\right)^m \left(1-\frac{\eta}{\kk'-\kk}\right)^{-2}
 \|h\|_{m,\sigma}\|g_2-g_1\|_{0,\sigma}.
\]
\item 
If $\pa_v h\in \QQQ_{m,\kk,\de,\sigma}$, $g_1, g_2 \in
\QQQ_{0,\kk',\de',\sigma'}$ and
$\|g_1\|_{0,\sigma'},\|g_2\|_{0,\sigma'}\leq \eta G_0^{-3}$ we
have that
  $Y\in \YY_{m,\kk',\de',\sigma'}$ and
\[
 \|Y\|_{m,\sigma'}\leq \left(\frac{\kk'}{\kk}\right)^{m}
 \frac{1}{1-\frac{\eta}{\kk'-\kk}}
 \|\pa_v h\|_{m,\sigma}\|g_2-g_1\|_{0,\sigma'}.
\]
\end{itemize}
\end{lemma}
Finally we give estimates for the matrices appearing in the right hand side of \eqref{eq:Diffeq}.
\begin{lemma}\label{lemma:boundsR}
The matrices $B$ and $R$ in \eqref{def:MatrixB} and  \eqref{def:matrixR}
satisfy the following
\begin{itemize}
\item 
$B_{21}$ satisfies $\|B_{21}\|_{1,\sigma}\lesssim G_0^{-11/2}$. Therefore  $\|B\|_{2,\sigma,\matrx}\lesssim G_0^{-11/2}$.
\item 
$R\in
\QQQ_{3/2,\kk,\de,\sigma,\matrx}$ and $\|R\|_{3/2,\sigma,\matrx}\lesssim G_0^{-3}.$
\end{itemize}
\end{lemma}
\begin{proof}
For $B_{21}$ one needs the improved bounds for $\La^s$ given in \eqref{def:LambdaImprovedEstimate} and the estimates in Lemma \ref{lem:HomoInfinity}.
The estimates for $R$ are obtained through an easy but tedious computation using the definitions of $\QQ$, $\GG_1$ and $\GG_2$ given in \eqref{def:Q} and \eqref{def:G1G2}, Lemma \ref{lem:HomoInfinity}, Lemma \ref{lem:P1k} and the estimates for the $Z^u$, $Z^s$ given in Theorem \ref{thm:parametrcommdomain}. 
Note that since we are dealing with formal Fourier series the compositions are understood as in Lemma \ref{lem:CompositionManifolds}.
\end{proof}

%
%
%
%

\subsection{Straightening the differential
operator}\label{sec:change:straighteningoperator}
First step is to perform a symplectic change of coordinates \emph{in phase
space} so that one transforms the operator $\wt\LL$ in \eqref{def:operatorLtilde} into $\LL$. Namely, to  remove the term
$\GG_1(Z^u)\pa_u\Delta+\GG_2(Z^u)\pa_\ga\Delta$ from the left hand side
of  equation \eqref{eq:Diffeq}.

\begin{theorem}\label{thm:DifferenceSymplecticStraightening}
 Let $\sigma_2$, $\kk_2$ and $\de_2$ be the constants considered in Lemma \ref{lemma:FromParamToGraph}. Let
$\sigma_3<\sigma_2$, $\kk_3>\kk_2$ and $\de_3>\de_2$ be fixed. Then, for $G_0$
big enough and $|\alo| G_0^{3/2}$ small enough, there exists a symplectic transformation given by
a (not necessarily convergent) Fourier series
\[
 (u,Y,\lo, \La, \al,\beta)=\Phi  (v,\wt Y,\lo, \wt \La, \al,\beta)
\]
of the form
\begin{equation}\label{def:SymplecticStraight}
 \Phi  (v,\wt Y,\lo, \wt \La, \al,\beta)=\left(v+\CCC(v,\lo),
\frac{1}{1+\pa_v\CCC(v,\lo)}\wt Y,\lo,
\wt \La-\frac{\pa_\ga\CCC(v,\lo)}{1+\pa_v\CCC(v,\lo)}\wt Y, \al,\beta\right)
\end{equation}
where $\CCC\in\QQQ_{0,\kk_3,\de_3,\sigma_3}$ satisfying
\[
\|\CCC\|_{0,\sigma_3}\leq b_6G_0^{-4}\ln G_0,
\]
with $b_6>0$ a constant independent of $G_0$, such
that
\begin{equation}\label{def:difference:reparameterized}
 \wt
\Delta(v,\gamma)=\begin{pmatrix}\frac{1}{1+\pa_v\CCC(v,\lo)}&0&0&0\\-\frac{\pa_\ga\CCC(v,\lo)}{1+\pa_v\CCC(v,\lo)}&1&0&0\\0&0&1&0\\0&0&0&1\end{pmatrix}\Delta\left(v+\CCC(v,\gamma),\gamma\right),
\end{equation}
where $\Delta$ is the function defined
in~\eqref{def:difference}, is well defined and satisfies the equation
\begin{equation}\label{eq:Diffeq:transf}
\LL\wt \Delta=A\wt \Delta+\BB\wt\Delta+\RRR\wt \Delta
\end{equation}
where $A$ is the matrix introduced in \eqref{def:A} and the matrices $\BB$ and $\RRR$ satisfy
\begin{itemize}
\item $\|\BB_{21}\|_{1}\lesssim G_0^{-11/2}$ and $\BB_{ij}=0$ otherwise. Therefore $\|\BB\|_{2,\sigma_3,\matrx}\lesssim G_0^{-11/2}$
\item $\RRR\in
\QQQ_{3/2,\kk_3,\de_3,\sigma_3,\matrx}$, $\RRR_{21}=0$, $\|\RRR\|_{3/2,\sigma_3,\matrx}\lesssim G_0^{-3}$.
\end{itemize}
\end{theorem}

We devote the rest of this section to prove this theorem.

\subsubsection{Proof of Theorem \ref{thm:DifferenceSymplecticStraightening}}
The first step is to perform a change of coordinates
\begin{equation}\label{def:change:straighteningoperator}
\Phi_0:(v,\ga)\mapsto
(u,\ga)=(v+\CCC(v,\gamma),\gamma),
\end{equation}
to straighten the operator $\wt\LL$. Then, we will apply a change of coordinates to the conjugate variables to make the change symplectic.

To straighten the operator we proceed as in \cite{GuardiaMS16}.
Consider an operator of the form
\[
 \wt \LL =(1+Q_1(u,\ga))\pa_u+\frac{\nu
G_0^3}{L_0^3}(1+Q_2(u,\ga))\pa_\ga.
\]
and  consider a change of coordinates of the form
\eqref{def:change:straighteningoperator} where satisfies
\begin{equation}\label{eq:Straightening}
 \LL g=\frac{Q_1\circ\Phi_0-Q_2\circ\Phi_0}{1+Q_2\circ\Phi_0}.
\end{equation}
Then, if $h$ solves the equation $\wt \LL h=D$ for some
$D$, the transformed $\wt h=h\circ\Phi_0$ satisfies
the equation
\[
\LL \wt h=\wt D\qquad \text{ where }\quad \wt D=\frac{D\circ
\Phi_0}{1+Q_2\circ\Phi_0}.
\]
Note that all these equations and transformations have to make sense for formal
Fourier series. In particular, the compositions are understood as in  Lemma
\ref{lemma:banach:Composition} and the fraction as
\[
 \frac{1}{1+Q_2(u,\ga)}=\sum_{q\geq 0}(-Q_2(u,\ga))^q.
\]
\begin{proposition}\label{prop:straighteninLhat}
Let $\sigma_3$, $\kk_3$ and $\de_3$ be the constants considered in Theorem \ref{thm:DifferenceSymplecticStraightening}.
Then, for $G_0$ big enough and $|\alo| G_0^{3/2}\ll 1$, there exists
a (not necessarily convergent) Fourier series
$\CCC\in\QQQ_{0,\kk_3,\de_3,\sigma_3}$ satisfying
\[
\|\CCC\|_{0,\sigma_3}\leq b_6G_0^{-4}\ln G_0,\qquad \|\pa_v\CCC\|_{1/2,\sigma_3}\leq b_6G_0^{-3},\qquad \|\pa_\ga\CCC\|_{1/2,\sigma_3}\leq b_6G_0^{-6}
\]
with $b_6>0$ a constant independent of $G_0$, such
that
\begin{equation}\label{def:difference:reparameterized}
\Delta^\ast(v,\gamma)=\Delta\left(v+\CCC(v,\gamma),\gamma\right),
\end{equation}
where $\Delta$ is the function defined
in~\eqref{def:difference}, is well defined and satisfies the equation
\begin{equation}\label{eq:Diffeq:transf}
\LL\wt \Delta=A \Delta^\ast+\wt B\Delta^\ast+\wt R \Delta^\ast
\end{equation}
where
\[
\wt B=\frac{B\circ\Phi_0}{1+\GG_2(Z^u)\circ \Phi_0},\qquad
 \wt R=\frac{A\circ\Phi_0-A+R\circ\Phi_0}{1+\GG_2(Z^u)\circ \Phi_0}
\]
with $\Phi_0(v,\ga)=\left(v+\CCC(v,\gamma),\gamma\right)$. Morever the matrix $\wt B$
satisfies $\|\wt B_{21}\|_{1,\sigma_3}\lesssim G_0^{-11/2}$, $B_{ij}=0$ otherwise, which imply $\|\wt B\|_{2,\sigma_3,\matrx}\lesssim G_0^{-11/2}$. 
The matrix $\wt R\in
\QQQ_{3/2,\kk_3,\de_3,\sigma_3,\matrx}$, $\wt R_{21}=0$ and
\[
 \|\wt R\|_{3/2,\sigma_3,\matrx}\lesssim G_0^{-3}.
\]
%
\end{proposition}
Using, the definition of $\wt\LL$ in \eqref{def:operatorLtilde}, to prove this
proposition,
we look for a function $\CCC$ satisfying  equation \eqref{eq:Straightening}
with
\begin{equation}\label{def:Q1Q2}
 Q_1(u,\ga)=\GG_1(Z^u)(u,\ga),\qquad Q_2(u,\ga)=\frac{L_0^3}{\nu
G_0^3}\GG_2(Z^u)(u,\ga).
\end{equation}
%
The next lemma gives estimates for these functions.
\begin{lemma}
The functions $Q_1$ and $Q_2$ satisfy
\begin{equation}\label{def:Q1Q2Estimates}
\begin{aligned}
\|Q_1\|_{1/2,\sigma_2}&\lesssim G_0^{-4}\ln G_0,&  \|Q_2\|_{1,\sigma_2}&\lesssim
G_0^{-9/2}\\
\|\pa_u Q_1\|_{3/2,\sigma_2}&\lesssim G_0^{-4}\ln G_0,&
\|\pa_u Q_2\|_{2,\sigma_2}&\lesssim
G_0^{-9/2}\\
\|\pa_\ga Q_1\|_{3/2,\sigma_2}&\lesssim G_0^{-7}\ln G_0,&
\|\pa_\ga Q_2\|_{1,\sigma_2}&\lesssim
G_0^{-9/2}.
\end{aligned}
\end{equation}
\end{lemma}
\begin{proof}
Lemma \ref{lemma:PropsP} gives the estimate for $Q_1$. Analogous estimates can be obtained for its derivatives, differentiating \eqref{def:G1G2}  and using the estimates for $Z^u$ and its derivatives in Theorem \ref{thm:parametrcommdomain} and Lemma \ref{lem:HomoInfinity}. To estimate $Q_2$ and its derivatives one can proceed analogously taking into account the improved estimates for $\La^u$ in Theorem \ref{thm:parametrcommdomain}.
\end{proof}

%
We obtain a solution of equation \eqref{eq:Straightening}  by considering a
left inverse of the operator $\LL$ in the space $\QQ_{1/2,\kk,\de,\sigma}$ and
setting up a fixed point argument.

We define the following operator acting on the Fourier
coefficients as
\begin{equation}\label{def:canvi:operadorG}
\wt\GG(h)(u,\ga)=\sum_{q\in\ZZ}\wt\GG(h)^{[q]}(u)e^{iq\gamma},
\end{equation}
where its Fourier coefficients are given by
\begin{align*}
\dps\wt\GG(h)^{[q]}(u)&= \int_{\ol{u}_2}^u e^{iq \nu G_0^3L_0^{-3}
(t-u)}h^{[q]}(t)\,dt& \text{ for }q< 0\\
\dps\wt\GG(h)^{[0]}(u)&=\int_{u^\ast}^u
h^{[0]}(t)\,dt
&
\\
\dps\wt\GG (h)^{[q]}(u)&=\int_{u_2}^u e^{iq \nu G_0^3L_0^{-3}
(t-u)}h^{[q]}(t)\,dt& \text{ for }q>0.
\end{align*}
Here  $u_2=i(1/3-\kk G_0^{-3})$ is the  top vertex of the
domain $D_{\kk,\de}$, $\ol u_2$ is its conjugate, which corresponds to the
bottom vertex of the domain  $D_{\kk,\de}$ and  $u^\ast$ is the left endpoint
of $D_{\kk,\de}\cap\RR$.

\begin{lemma}\label{lemma:diff:OperadorG}
The operator $\wt\GG$  in~\eqref{def:canvi:operadorG} satisfies that
\begin{itemize}
\item If $h\in \QQQ_{\nu,\kk,\de,\sigma}$ for some $\nu\in (0,1)$, then
$\wt\GG(h)\in \QQQ_{0,\kk,\de,\sigma}$  and $\left\|\wt \GG(h)\right\|_{0,\sigma}\leq K\|h\|_{\nu,\sigma}$.
\item If $h\in \QQQ_{1,\kk,\de,\sigma}$, then
$\wt\GG(h)\in \QQQ_{0,\kk,\de,\sigma}$  and
$
\left\|\wt \GG(h)\right\|_{0,\sigma}\leq K\ln G_0\|h\|_{1,\sigma}.
$
\item If $h\in \QQQ_{\nu,\kk,\de,\sigma}$ for some $\nu>1$, then
$\wt\GG(h)\in \QQQ_{\nu-1,\kk,\de,\sigma}$  and
$
\left\|\wt \GG(h)\right\|_{\nu-1,\sigma}\leq K\|h\|_{\nu,\sigma}.
$
\item If $h\in \QQQ_{\nu,\kk,\de,\sigma}$ for some $\nu>0$, then
$\pa_v\wt\GG(h)\in \QQQ_{\nu,\kk,\de,\sigma}$  and
$
\left\|\pa_v\wt \GG(h)\right\|_{\nu,\sigma}\leq K\|h\|_{\nu,\sigma}.
$
\item If $h\in \QQQ_{\nu,\kk,\de,\sigma}$ for some $\nu>0$ and $\langle
h\rangle=0$, then
$\wt\GG(h)\in \QQQ_{\nu,\kk,\de,\sigma}$  and
\[
\left\|\wt \GG(h)\right\|_{\nu,\sigma}\leq KG_0^{-3}\|h\|_{\nu,\sigma}.
\]
\end{itemize}
Moreover,
if $h$ is a real-analytic Fourier series, that is $h^{[q]}(\ol u)=\ol{h^{[-q]}(u)}$, then so is $\wt \GG(h)$.
\end{lemma}
This lemma can be  proven as Lemma~8.3
of~\cite{GuardiaOS10}.

\begin{proof}[Proof of Proposition \ref{prop:straighteninLhat}]
We prove Proposition \ref{prop:straighteninLhat} by looking for a
fixed point of the operator
\begin{equation}\label{def:Straightening:operator}
\wt\KK=\wt\GG\circ\KK,\qquad \KK(\CCC)(v,\ga)= \left.
\frac{Q_1(u,\ga)-Q_2(u,\ga)}{1+Q_2(u,\ga)}\right|_{u=v+\CCC(v,\ga)}
\end{equation}
where $\wt\GG$ is the operator introduced in \eqref{def:canvi:operadorG} and
$Q_1$, $Q_2$ are the formal Fourier series in \eqref{def:Q1Q2}.

We write $\KK(0)$ as
\[
 \KK(0)=\frac{Q_1-Q_2}{1+Q_2}=Q_1-\frac{Q_2(1+Q_1)}{1+Q_2}.
\]
Note that, by \eqref{def:Q1Q2Estimates}, the second term satisfies
\[
 \left\|\frac{Q_2(1+Q_1)}{1+Q_2}\right\|_{1,\sigma_2}\lesssim G_0^{-9/2}.
\]
Now, by Lemmas
\ref{lemma:diff:OperadorG} and  \ref{lemma:banach:AlgebraProps},
there exists a constant $b_6>0$ independent of $G_0$, such that
\[
\begin{split}
\left\|\wt \KK(0)\right\|_{0,\sigma_3}=\left\|\wt\GG\circ \KK(0)\right\|_{0,\sigma_3}&\,\leq\left\|\wt \GG(Q_1)\right\|_{0,\sigma_3}+\left\|\wt\GG\left(\frac{Q_2(1+Q_1)}{1+Q_2}\right)\right\|_{0,\sigma_3}\\
 &\,\leq\|Q_1\|_{1/2,\sigma_3}+\ln G_0\left\|\frac{Q_2(1+Q_1)}{1+Q_2}\right\|_{1,\sigma_3}\\
&\,\leq \frac{b_6}{2}G_0^{-4}\ln G_0.
\end{split}
\]
 Now we prove that $\wt\KK$ is
a Lipschitz operator in the ball $B(b_6 G_0^{-4}\ln G_0)\subset
\QQQ_{0,\kk_3,\de_3,\sigma_3}$.
Take $g_1,g_2\in B(b_6 G_0^{-4}\ln G_0)\subset
\QQQ_{0,\kk_3,\de_3,\sigma_3}$.  By Lemma \ref{lemma:banach:Composition} and
estimates \eqref{def:Q1Q2Estimates},
\[
 \left\|\KK(g_2)-\KK(g_1)\right\|_{3/2, \sigma_3}\lesssim
\left\|\pa_u\left[\frac{Q_1(u,\ga)-Q_2(u,\ga)}{1+Q_2(u,\ga)}
\right]\right\|_{3/2,\sigma_2}\|g_2-g_1\|_{0,\sigma_3}\lesssim G_0^{-3}\|g_2-g_1\|_{0,\sigma_3}.
\]
Then, by Lemma \ref{lemma:banach:AlgebraProps} and \ref{lemma:diff:OperadorG},
\[
\left\|\wt\KK(g_2)-\wt \KK(g_1)\right\|_{0, \sigma_3}\lesssim
G_0^{3/2}\left\|\wt\KK(g_2)-\wt \KK(g_1)\right\|_{1/2, \sigma_3}
\lesssim G_0^{3/2}\left\|\KK(g_2)-\KK(g_1)\right\|_{3/2, \sigma_3}\lesssim G_0^{-3/2}\|g_2-g_1\|_{0,\sigma_3}.
\]
Thus, taking $G_0$ large enough, the operator $\wt\KK$ is a contractive
operator $B(b_6 G_0^{-4}\ln G_0)\subset
\QQQ_{0,\kk_3,\de_3,\sigma_3}$. 
The fix point of the operator gives the change
of coordinates provided in Proposition \ref{prop:straighteninLhat}.

To obtain the estimates for $\pa_v\CCC$ it is enough to use that we have seen that for $\CCC\in B(b_6 G_0^{-4}\ln G_0)\subset
\QQQ_{0,\kk_3,\de_3,\sigma_3}$, $\KK(\CCC)$ satisfies $\KK(\CCC)\in
\QQQ_{1/2,\kk_3,\de_3,\sigma_3}$ and $\|\KK(\CCC)\|_{1/2,\sigma_3}\lesssim G_0^{-3}$. Then, by Lemma \ref{lemma:diff:OperadorG}, $\pa_v\CCC=\pa_v\wt\GG\circ\KK(\CCC)\in \QQQ_{1/2,\kk_3,\de_3,\sigma_3}$ and satisfies $\|\pa_v\wt\GG\circ\KK(\CCC)\|_{1/2,\sigma_3}\lesssim G_0^{-3}$. The estimates for  $\pa_\ga\CCC$ are obtained through the identity
\[
 \pa_\ga \CCC=\frac{L_0^3}{\nu G_0^3}\left(\KK(\CCC)-\pa_v\CCC\right).
\]
Finally, the estimates for $\wt B$ and  $\wt R$ are a direct consequence of the estimates for $\CCC$
just obtained, the estimate of $R$ in Lemma \ref{lemma:boundsR}, the identity \eqref{def:Q1Q2}, estimates \eqref{def:Q1Q2Estimates}, the
definition of $A$  in \eqref{def:A}, the estimates of the functions $f_1$ and $f_2$ given in Lemma \ref{lemma:f1f2} and the condition $|\alo| G_0^{3/2}\ll 1$.
\end{proof}

Now, we are ready to prove Theorem \ref{thm:DifferenceSymplecticStraightening}.

\begin{proof}[Proof of Theorem \ref{thm:DifferenceSymplecticStraightening}]
It is straightforward to check that the transformation \eqref{def:SymplecticStraight} is symplectic. It only remains to obtain the estimates for $\BB$ and $\RRR$. To this end, it is enough to apply the transformation
\[
Y= \frac{1}{1+\pa_v\CCC(v,\ga)}\wt Y,\qquad
\La=\wt \La-\frac{\pa_\ga\CCC(v,\ga)}{1+\pa_v\CCC(v,\ga)}\wt Y
\]
to equation \eqref{eq:Diffeq:transf} to obtain the  formulas for the
coefficients  $(\BB+\RRR)_{ij}$. To this end, to a $4\times 4$ matrix $M$
whose entries $M_{ij}$ are functions of $(v,\ga)$ we define the following
$4\times 4$ matrix $\JJ(M)$ whose coefficients $\JJ(M)_{ij}$ are defined a
\[
 \begin{split}
  \JJ(M)_{11}&=
M_{11}+\frac{\pa_v\KK(\CCC)}{1+\pa_v\CCC}-M_{12}\pa_\ga\CCC,\quad
\JJ(M)_{1j}=\left(1+\pa_v\CCC\right) M_{1j},\quad j=2,3,4,\\
  \JJ(M)_{21}&=\frac{M_{21}+\pa_\ga\CCC M_{11}+\pa_\ga\KK(\CCC)-
M_{22}\pa_\ga\CCC-\left(\pa_\ga\CCC\right)^2 M_{12}}{1+\pa_v\CCC} \\
  \JJ(M)_{2j}&=M_{2j}+\pa_\ga\CCC M_{1j},\quad j=2,3,4\\
  \JJ(M)_{i1}&=\frac{M_{i1}+M_{i2}\pa_\ga\CCC}{1+\pa_v\CCC},\quad
\JJ(M)_{ij}=M_{ij},\quad i=3,4, j=2,3,4.
 \end{split}
\]

We split $\BB$ and $\RRR$ as before. That is, $\BB_{ij}=0$ for $ij\neq 21$ and
$\RRR_{21}=0$. Then, the coefficients of the matrix $\BB$ and  $\RRR$ in
Theorem \ref{thm:DifferenceSymplecticStraightening} are defined as
\[
\BB_{21}=\frac{\wt B_{21}+\pa_\ga\CCC \wt R_{11}+\pa_\ga\KK(\CCC)-
\wt R_{22}\pa_\ga\CCC-\left(\pa_\ga\CCC\right)^2 \wt R_{12}}{1+\pa_v\CCC}
\]
and
\[
 \RRR=\JJ(A)-A+\JJ(\wt R)-(\BB-\wt B)
\]
where $A$ and $\wt R$ are the matrices defined in
\eqref{def:A} and Proposition \ref{prop:straighteninLhat} respectively. This
implies that $\RRR_{ij}=\JJ(\wt R)_{ij}$ for all coefficients except
$\RRR_{21}=0$ and
\[
 \RRR_{i1}=\frac{-A_{i1}\pa_v\CCC+\wt R_{i1}+\pa_\ga\CCC\left(A_{i2}+\wt R_{i2}\right)}{1+\pa_v\CCC}
\]

Then, one can obtain the estimates for the coefficients of $\RRR$ using these definitions, the estimates for $\wt R$ and $\CCC$ in Proposition \ref{prop:straighteninLhat}, the estimates for the matrix $A$ given in Lemma
\ref{lemma:f1f2} (see the definition of $A$ in \eqref{def:A}) and the condition $|\alo| G_0^{3/2}\ll 1$. For the bounds of $\pa_v\KK(\CCC)$ and $\pa_\ga\KK(\CCC)$ one has to use the definition of $\KK(\CCC)$ in \eqref{def:Straightening:operator} to obtain
\[
\begin{split}
\pa_v\KK(\CCC)(v,\ga)=&\, \pa_u\left[\left.
\frac{Q_1(u,\ga)-Q_2(u,\ga)}{1+Q_2(u,\ga)}\right]\right|_{u=v+\CCC(v,\ga)} \left(1+\pa_v\CCC(v,\ga)\right)\\
\pa_\ga\KK(\CCC)(v,\ga)=&\, \pa_u\left[\left.
\frac{Q_1(u,\ga)-Q_2(u,\ga)}{1+Q_2(u,\ga)}\right]\right|_{u=v+\CCC(v,\ga)} \pa_\ga\CCC(v,\ga)\\
&\,+ \pa_\ga\left[\left.
\frac{Q_1(u,\ga)-Q_2(u,\ga)}{1+Q_2(u,\ga)}\right]\right|_{u=v+\CCC(v,\ga)}
\end{split}
\]
Then, using the estimates in \eqref{def:Q1Q2Estimates} and Lemma \ref{lemma:banach:Composition}, one has
\[
 \|\pa_v\KK(\CCC)\|_{3/2,\sigma_3}\lesssim G_0^{-3},\qquad  \|\pa_\ga\KK(\CCC)\|_{1,\sigma_3}\lesssim G_0^{-11/2}.
\]
\end{proof}

\subsection{The general solution for the straightened linear
system}\label{sec:FundamentalMatrix}

Now, we solve the linear equation \eqref{eq:Diffeq:transf} by looking for a fundamental matrix $\Psi$ satisfying
\begin{equation}\label{def:LinearPDE}
 \LL\Psi=(A+\BB+\RRR)\Psi.
\end{equation}
Note that  in \eqref{def:FundamentalMatrix} we have obtained a fundamental matrix $\Phi_A$ of the linear equation $\LL\Psi=A\Psi$. 
However, it can be easily seen that this matrix does not have good estimates with respect to the norm introduced in \eqref{def:MatrixNorm}. 
Thus, we modify it slightly. 
Let us introduce the notation $\Phi_A= (V_1, \ldots, V_4)$ where $V_i$ are the columns of the matrix. Then, we define the new fundamental matrix 
\begin{equation}\label{def:FundamentalMatrixModified}
\wt\Phi_A=(\wt V_1, \ldots \wt V_4)\qquad \text{  defined as }\qquad
\begin{aligned}
\wt V_1&=V_1-\alo g_1\left(\frac{i}{3}\right)V_3+\beto g_1\left(-\frac{i}{3}\right)V_4\\ \wt V_2&=V_2-\alo g_2\left(\frac{i}{3}\right)V_3+\beto g_2\left(-\frac{i}{3}\right)V_4\\ \wt V_j&=V_j, j=3,4.
\end{aligned}
\end{equation}

\begin{lemma}\label{lemma:FundMatEstimates}
Assume $|\alo| G_0^{3/2}\ll1$. The fundamental matrix $\wt\Phi_A$ and its inverse  $\wt\Phi_A\ii$ satisfy $\wt\Phi_A,\wt\Phi_A\ii\in \QQ_{0,\kk_3,\de_3,\sigma_3}$ and
\[
 \left\|\wt\Phi_A\right\|_{0,\sigma_3,\matrx}\lesssim 1,\qquad  \left\|\wt\Phi_A\ii\right\|_{0,\sigma_3,\matrx}\lesssim 1.
\]
Moreover, the matrices $\Phi_A$ in \eqref{def:FundamentalMatrix} and $\wt \Phi_A$ in \eqref{def:FundamentalMatrixModified} are related as $ \Phi_A=\wt \Phi_A \JJ$ where $\JJ$ is a constant matrix which satisfies
\[
\JJ=\Id+\OO(|\alo|).
\]
Moreover,  the  $\OO(|\alo|)$ terms are only present in the third and fourth row of the matrix.
\end{lemma}
The proof of this lemma is a direct consequence of the definition of $\wt\Phi_A$ in \eqref{def:FundamentalMatrixModified} and Lemma \ref{lemma:f1f2}.

Now we look for a fundamental matrix \eqref{def:LinearPDE} as $\Psi=\wt\Phi_A(\Id+\wt\Psi)$. Then, $\wt\Psi$
must satisfy
\begin{equation}\label{eq:PDEFundMatrix}
 \LL\wt \Psi=\wt\Phi_A\ii(\BB+\RRR)\wt\Phi_A(\Id+\wt\Psi).
\end{equation}
Next theorem gives a matrix solution of this equation.
\begin{theorem}\label{thm:fundamentalmatrix}
Let $\sigma_3$, $\kk_3$ and $\de_3$ be the constants considered in Theorem \ref{thm:DifferenceSymplecticStraightening}. Then, for $G_0$
big enough and $|\alo| G_0^{3/2}$ small enough, there exists a matrix $\wt\Psi\in\QQQ_{1/2,\kk_3,\de_3,\sigma_3,\matrx}$ wich is a solution of  equation \eqref{eq:PDEFundMatrix} and satisfies
\[\left\|\wt\Psi\right\|_{1/2,\sigma_3, \matrx}\lesssim G_0^{-3}\ln G_0.\]
Moreover, for
\[\left\|\wt\Psi_{21}\right\|_{0,\sigma_3, \matrx}\lesssim G_0^{-9/2}\ln G_0.\]
\end{theorem}

We devote the rest of this section to prove this theorem. We solve equation \eqref{eq:PDEFundMatrix} through a fixed point argument by setting up an integral
equation.

The first step is to invert the operator $\LL$. To this end, we need to use \emph{different} integral operators depending on
the components. The reason is the significantly different behavior of the components close to the singularities of the unperturbed separatrix. That is, besides the operator $\wt \GG$ in
\eqref{def:canvi:operadorG}, we define the  operators
\begin{equation}\label{def:OperatorGPlusMinus}
\GG_\pm(h)(u,\ga)=\GG_\pm(h)^{[0]}(u)+\sum_{q\in\ZZ\setminus\{0\}}\wt\GG(h)^{[q]}(u)e^{iq\gamma},\qquad \GG_\pm(h)^{[0]}(u)=\int_{u_\pm}^{u}
h^{[0]}(t)\,dt
\end{equation}
where $u_+=u_2$ and $u_-=\bar u_2$, where $u_2$ has been introduced in
\eqref{def:canvi:operadorG}.
Note that equation \eqref{eq:PDEFundMatrix} has many solutions which arise from the fact that the operator $\LL$ has many left inverse operators. We choose just one solution which is convenient for us.

\begin{lemma}\label{lemma:diff:OperadorGpm}
The operators $\GG_\pm$  introduced in~\eqref{def:OperatorGPlusMinus} satisfy the following.
Assume $h\in \QQQ_{\nu,\kk,\de,\sigma,\matrx}$ with $\nu\geq 1/2$. Then $e^{\pm i\phi_\h(u)}\GG_\pm(e^{\mp i\phi_\h(u)}h)\in \QQQ_{\nu-1,\kk,\de,\sigma,\matrx}$ and
\[
\begin{split}
\left\|e^{\pm i\phi_\h(u)}\GG_\pm\left(e^{\mp i\phi_\h(u)}h\right)\right\|_{\nu-1,\sigma}&\lesssim \left\|h\right\|_{\nu,\sigma}\qquad \text{ for }\nu>1/2\\
\left\|e^{\pm i\phi_\h(u)}\GG_\pm\left(e^{\mp i\phi_\h(u)}h\right)\right\|_{-1/2,\sigma}&\lesssim \ln G_0\left\|h\right\|_{1/2,\sigma}.
\end{split}
\]
\end{lemma}

We define an integral operator $\GG_\matrx$ acting on matrices in
$\QQQ_{\nu,\kk,\de,\sigma,\matrx}$ linearly on the coefficents as follows. For
$M\in \QQQ_{\nu,\kk,\de,\sigma,\matrx}$, we define $\GG_\matrx(M)$ as
\begin{equation}\label{def:OperatorGMatrix}
 \begin{aligned}
\GG_\matrx(M)_{ij}&=\wt  \GG(M_{ij})\qquad&& \text{ for }i=1,2, j=1,2,3,4\\
\GG_\matrx(M)_{3j}&=  \GG_+(M_{ij})\qquad&& \text{ for }j=1,2,3,4\\
\GG_\matrx(M)_{4j}&=  \GG_-(M_{ij})\qquad &&\text{ for }j=1,2,3,4.
\end{aligned}
\end{equation}

\begin{lemma}\label{lemma:diff:MatrixOperadorG}
The operator $\GG_\matrx$  in~\eqref{def:OperatorGMatrix} has the following properties.
\begin{itemize}
\item Assume $M\in \QQQ_{\nu,\kk,\de,\sigma,\matrx}$ with $\nu\geq 2$. Then $\GG_\matrx(M)\in \QQQ_{\nu-1,\kk,\de,\sigma,\matrx}$ and
\[
\begin{split}
\left\|\GG_\matrx(M)\right\|_{\nu-1,\sigma}&\lesssim \left\|M\right\|_{\nu,\sigma}\qquad \text{ for }\nu>2\\
\left\|\GG_\matrx(M)\right\|_{1,\sigma}&\lesssim \ln G_0\left\|M\right\|_{2,\sigma}.
\end{split}
\]
\item Assume $M\in \QQQ_{\nu,\kk,\de,\sigma,\matrx}$ with $\nu\geq 3/2$ and $M_{21}=0$. Then $\GG_\matrx(M)\in \QQQ_{\nu-1,\kk,\de,\sigma,\matrx}$ and
\[
\begin{split}
\left\|\GG_\matrx(M)\right\|_{\nu-1,\sigma}&\lesssim \left\|M\right\|_{\nu,\sigma}\qquad \text{ for }\nu>3/2\\
\left\|\GG_\matrx(M)\right\|_{1/2,\sigma}&\lesssim \ln G_0\left\|M\right\|_{3/2,\sigma}.
\end{split}
\]
\end{itemize}
\end{lemma}

We use the operator $\GG_\matrx$ to look for solutions of \eqref{eq:PDEFundMatrix} through an integral equation.  We define the
operator
\begin{equation*}
 \wt\SSS(\Psi)=\GG_\matrx\circ\SSS(\Psi)\qquad\text{
with
}\qquad\SSS(\Psi)=\wt\Phi_A\ii(\BB+\RRR)\wt\Phi_A(\Id+\Psi).
\end{equation*}
%
%

\begin{lemma}\label{lemma:Fund:Lip}
The linear operator $\wt\SSS: \QQQ_{1/2,\kk_3,\de_3,\sigma_3,\matrx}\to \QQQ_{1/2,\kk_3,\de_3,\sigma_3,\matrx}$ is Lipschitz and satisfies
\[
 \left\|\wt\SSS(\Psi)-\wt\SSS(\Psi')\right\|_{1/2,\sigma_3,\matrx}\lesssim G_0^{-3/2} \ln G_0\left\|\Psi-\Psi'\right\|_{1/2,\sigma_3,\matrx}
\]
for any $\Psi,\Psi'\in \QQQ_{1/2,\kk_3,\de_3,\sigma_3,\matrx}$.
\end{lemma}
\begin{proof}
To compute the Lipschitz constant, we write
\[
 \SSS(\Psi)- \SSS(\Psi')=\wt\Phi_A\ii(\BB+\RRR)\wt\Phi_A(\Psi-\Psi').
\]
The properties of $\BB$ and $\RRR$ in Theorem \ref{thm:DifferenceSymplecticStraightening} imply that
$  \left\|\BB+\RRR\right\|_{1,\sigma_3,\matrx}\lesssim G_0^{-3/2}$.
Then, using also Lemmas \ref{lemma:FundMatEstimates} and Lemma \ref{lemma:NormMatrix},
\[
\begin{split}
\left\| \SSS(\Psi)- \SSS(\Psi')\right\|_{3/2,\sigma_3,\matrx}\lesssim &\, \left\| \wt\Phi_A\ii\right\|_{0,\sigma_3,\matrx}\left\| \BB+\RRR\right\|_{1,\sigma_3,\matrx}\left\| \wt\Phi_A\right\|_{0,\sigma_3,\matrx}\left\| \Psi-\Psi'\right\|_{1/2,\sigma_3,\matrx}\\
\lesssim & \,G_0^{-3/2}\left\| \Psi-\Psi'\right\|_{1/2,\sigma_3,\matrx}.
\end{split}
\]
%
Thus, applying Lemma \ref{lemma:diff:MatrixOperadorG},
\[
\left\| \wt\SSS(\Psi)- \wt \SSS(\Psi')\right\|_{1/2,\sigma_3,\matrx}\lesssim  \ln G_0 \left\| \SSS(\Psi)- \SSS(\Psi')\right\|_{3/2,\sigma_3,\matrx}\\
\lesssim G_0^{-3/2}\ln G_0\left\| \Psi-\Psi'\right\|_{1/2,\sigma_3,\matrx}.
\]

\end{proof}

Then, to finish the proof of Theorem \ref{thm:fundamentalmatrix}, it is enough to use Lemmas \ref{lemma:diff:OperadorG} and \ref{lemma:diff:MatrixOperadorG} and to use the estimates for $\BB$ and $\RRR$ in Theorem \ref{thm:DifferenceSymplecticStraightening} to see that
\[
\wt\SSS(0)=\GG_\matrx\left[\wt\Phi_A\ii(\BB+\RRR)\wt\Phi_A\right]
\]
satisfies
\[
\begin{split}
  \left\|\wt\SSS(0)\right\|_{1/2,\sigma_3,\matrx}\lesssim &\, G_0^{3/2} \left\|\GG_\matrx\left[\wt\Phi_A\ii\BB\wt\Phi_A\right]\right\|_{1,\sigma_3,\matrx}+\left\|\GG_\matrx\left[\wt\Phi_A\ii\RRR\wt\Phi_A\right]\right\|_{1/2,\sigma_3,\matrx}\\
  \lesssim &\, G_0^{3/2}\ln G_0 \left\|\wt\Phi_A\ii\BB\wt\Phi_A\right\|_{2,\sigma_3,\matrx}+\ln G_0\left\|\wt\Phi_A\ii\RRR\wt\Phi_A\right\|_{3/2,\sigma_3,\matrx}\\
  \lesssim &\, G_0^{3/2}\ln G_0 \left\|\BB\right\|_{2,\sigma_3,\matrx}+\ln G_0\left\|\RRR\right\|_{3/2,\sigma_3,\matrx}\\
    \lesssim &\,  G_0^{-3}\ln G_0.
\end{split}
  \]
Therefore, together with Lemma \ref{lemma:Fund:Lip}, one has that the operator $\wt\SSS$ has a unique fixed point $\wt\Psi$ which satisfies $\left\|\wt\Psi\right\|_{1/2,\sigma_3,\matrx}\lesssim G_0^{-3}\ln G_0$.

For the estimates for $\wt\Psi_{21}$ it is enough to write $\wt\Psi_{21}$ as
\[
 \wt\Psi_{21}=\wt\SSS(0)_{21}+\left[\wt\SSS(\Psi)-\wt\SSS(\Psi')\right]_{21}.
\]
For the first term, by Theorem \ref{thm:DifferenceSymplecticStraightening} and Lemma \ref{lemma:diff:OperadorG}, one has that
\[
 \left\|\wt\SSS(0)_{21}\right\|_{0,\sigma_3}\lesssim G_0^{-11/2}\ln G_0.
\]
For the second term it is enough to use Lemma \ref{lemma:Fund:Lip} to obtain
\[
\left\| \left[\wt\SSS(\Psi)-\wt\SSS(\Psi')\right]_{21}\right\|_{0,\sigma_3}\leq \left\| \left[\wt\SSS(\Psi)-\wt\SSS(\Psi')\right]_{21}\right\|_{-1/2,\sigma_3}\lesssim  G_0^{-3/2}\ln G_0\left\|\wt\Psi\right\|_{1/2,\sigma_3,\matrx}\lesssim G_0^{-9/2}\ln G_0.
\]

\subsection{Exponentially small estimates of the difference between the invariant manifolds}\label{sec:Lazutkin}

Last step is to obtain exponentially small bounds of the difference between invariant manifolds and its first order. Using that $\Psi=\wt\Phi_A(\Id+\wt\Psi)$ with $\wt\Psi$ obtained in Theorem \ref{thm:fundamentalmatrix} is a fundamental matrix of the  equation \eqref{def:LinearPDE},  we know that $\wt\Delta$ (which also satisfies \eqref{def:LinearPDE})  is of the form
\begin{equation}\label{def:DiffKernel}
\wt\Delta=\wt\Phi_A(\Id+\wt\Psi)\wh\Delta\quad\text{ where }\wh\Delta\quad\text{ satisfies }\LL\wh\Delta=0,
\end{equation}
and $\wt \Phi_A$ is given in \eqref{def:FundamentalMatrixModified} and $\wt\Psi$ is the matrix obtained in Theorem \ref{thm:fundamentalmatrix}.

To bound the function $\wh \Delta$,  we use the following lemma, proven in \cite{GuardiaMS16}.

\begin{lemma}\label{lemma:Lazutkin}
Fix $\kk>0$, $\de>0$ and $\sigma>0$. Let us consider a formal Fourier series
$\Upsilon\in \QQQ_{0,\kk,\de,\sigma}$ such that $\Upsilon\in\mathrm{Ker}\LL$. Define its average
\[
 \langle\Upsilon\rangle_\ga=\frac{1}{2\pi}\int_0^{2\pi}\Upsilon(v,\ga)d\ga.
\]
Then, the Fourier series $\Upsilon(v,\ga)$ satisfies the following.
\begin{itemize}
\item 
Is of the form
\[
\Upsilon(v,\ga)=\sum_{\ell\in\ZZ}\Upsilon^{[\ell]}(v)e^{i\ell\ga}=\sum_{\ell\in\ZZ}
\wt\Upsilon^{[\ell]}e^{i\ell\left(G_0^3v+\ga\right)}
\]
for certain constants $\wt\Upsilon^{[\ell]}\in\CC$ and its average $\langle\Upsilon\rangle_\ga$ is independent of $v$.
\item 
It defines a function for $v\in D_{\kk,\de}\cap \RR$ and $\ga\in\TT$, whose Fourier coefficients satisfy that
\[
\begin{split}
 \left|\Upsilon^{[\ell]}(v)\right|&\leq
 \sup_{v\in D_{\kk,\de}}\left|\Upsilon^{[\ell]}(v)\right|K^{|\ell|}
e^{-\dps\tfrac{|\ell|G_0^{3}}{3}}\lesssim \left\|\Upsilon\right\|_{0,\sigma}(KG_0^{3/2})^{|\ell|}
e^{-\dps\tfrac{|\ell|G_0^{3}}{3}}\\
 \left|\pa_v\Upsilon^{[\ell]}(v)\right|&\leq  \sup_{v\in D_{\kk,\de}}\left|\Upsilon^{[\ell]}(v)\right|K^{|\ell|}G_0^3e^{-\dps\tfrac{
|\ell|G_0^{3}}{3}}\lesssim G_0^3 \left\|\Upsilon\right\|_{0,\sigma}(KG_0^{3/2})^{|\ell|}
e^{-\dps\tfrac{|\ell|G_0^{3}}{3}}.
\end{split}
\]
\end{itemize}
\end{lemma}
Note nevertheless, that we do not want to bound $\wh\Delta$ but its difference with respect to its first order. 
The first order is defined through the
operators $\FF^{u,s}$ in \eqref{def:OperatorInfinity} (see Theorem \ref{thm:parametrcommdomain}) and is given by
\[
\wh
\Delta_0=\wt\Phi_A^{-1}\left(\FF^u(0)-\FF^s(0)\right).
\]
Using  \eqref{def:OperatorInfinity} and the relation $\wt\Phi_A^{-1}\Phi_A=\mathcal{J}$ given in Lemma \ref{lemma:FundMatEstimates},
\begin{equation}\label{def:Deltazerohat}
\wh
\Delta_0=\wt\Phi_A^{-1}\Phi_A\left[\delta z +\GG^u\left(\Phi_A^{-1}F(0)\right)-\GG^s\left(\Phi_A^{-1}F(0)\right)\right]=\mathcal{J}\delta z +\GG^u\left(\wt\Phi_A^{-1}F(0)\right)-\GG^s\left(\wt\Phi_A^{-1}F(0)\right).
\end{equation}
Since $\mathcal{J}\delta z$ is a constant term and  $\GG^{u,s}$ are both inverses of $\LL$, $\wh\Delta_0$ satisfies  $\LL\wh \Delta_0=0$.

We write then, $\wh\Delta$ as
\[
 \wh \Delta=\wh \Delta_0+\EE.
\]
The next two lemmas give estimates for these functions. Recall that $\Theta$ and $G_0$ are related through
\[
\omega=\frac{\nu G_0^3}{L_0^{3}}, \quad \mbox{with}\quad G_0=\Theta-L_0+\eta_0\xi_0.
\]
(see \eqref{eq:omega}). The estimates for $\EE$ are a consequence of  Lemma \ref{lemma:Lazutkin}.
\begin{lemma}\label{lemma:Delta0Hat}
The function $\wh\Delta_0$ in \eqref{def:Deltazerohat} satisfies that, for $v\in D_{\kk_3,\de_3}\cap\RR$ and $\ga\in\TT$,
\[
\begin{split}
\wh{\Delta Y}_0(v,\ga)&=\omega\pa_\sigma\LL(\omega v-\ga,\alo,\beto)\\
\wh{\Delta \La}_0(v,\ga)&=-\pa_\sigma\LL(\omega v-\ga,\alo,\beto)\\
\wh{\Delta \al}_0(v,\ga)&=\delta\eta-i\pa_{\beto}\LL(\omega v-\ga,\alo,\beto)+G_0^{-3}\mathcal{O}(\eta_0,G_0^{-1}\xi_0)\\
\wh{\Delta \bet}_0(v,\ga)&=\delta\xi+i\pa_{\alo}\LL(\omega v-\ga,\alo,\beto)+G_0^{-3}\mathcal{O}(\eta_0,G_0^{-1}\xi_0)
\end{split}
\]
where $\LL$ is the Melnikov potential introduced in Proposition \ref{prop:MelnikovPotential}.
\end{lemma}
\begin{lemma}\label{lemma:ErrorEstimate}
The function $\EE$ satisfies that,  for $v\in D_{\kk_3,\de_3}\cap\RR$ and $\ga\in\TT$,
\begin{equation}\label{def:ExpSmallErrors}
 \begin{aligned}
 |\EE_Y-\langle \EE_Y\rangle|&\lesssim e^{-G_0^3/3}G_0^2
 \ln^2G_0,&  |\EE_\La-\langle \EE_\La\rangle|&\lesssim e^{-G_0^3/3}G_0^{-1}
 \ln^2G_0 \\
 |\EE_\al-\langle \EE_\al\rangle|&\lesssim e^{-G_0^3/3}G_0^{1/2}
 \ln^2 G_0, &  |\EE_\bet-\langle \EE_\bet\rangle|&\lesssim e^{-G_0^3/3}G_0^{1/2}
 \ln^2 G_0
  \end{aligned}
  \end{equation}
and
\begin{equation*}
 \begin{aligned}
 |\langle \EE_Y\rangle|+|\langle \EE_\La\rangle|+|\langle \EE_\al\rangle|+ |\langle \EE_\bet\rangle|\lesssim G_0^{-6}|\ln G_0|^3.
 \end{aligned}
\end{equation*}
 \end{lemma}
Note that this lemma gives an expression of the difference between the paramerizations of the invariant manifolds as
\begin{equation}\label{def:DeltaTildeAsymp}
 \wt\Delta=\wt\Phi_A \left(\Id+\wt\Psi\right)\left(\wh \Delta_0+\EE\right)
\end{equation}
and the Fourier coefficients of $\EE$ (except its averages) have exponentially small bounds. In the next section we  improve the estimates for a certain average associated to the $\wt\Lambda$ component of $\EE$.

We finish this section proving Lemmas \ref{lemma:Delta0Hat} and \ref{lemma:ErrorEstimate}. 
\begin{proof}[Proof of Lemma \ref{lemma:Delta0Hat}]
For the $Y$ component, using the definition of $\wt\Phi_A$ in \eqref{def:FundamentalMatrixModified} and the properties of the matrix $\mathcal{J}$ stated in Lemma \ref{lemma:FundMatEstimates}, one has ($\pi_Y$ denotes the projection on the $Y$ component)
\[
\wh {\Delta Y}_0=\pi_Y (\mathcal{J}\delta z)+\GG^u(F_1(0))-\GG^s(F_1(0))=\GG^u(F_1(0))-\GG^s(F_1(0)).
\]
Then, recalling the definition of the operators $\GG^{u,s}$ in \eqref{def:Outer:IntegralOperator} and of $F(0)$, one can check that $\wh {\Delta Y}_0=\omega\pa_\sigma\LL$ as was to be shown. Proceeding analogously, one can prove that  $\wh {\Delta \La}_0=-\pa_\sigma\LL$.

For the components $\al$ and $\bet$ we use that, for real values of $v$, the matrix $\wt\Phi_A$ satisfies $\wt\Phi_A-\Id=\OO\left(G_0^{-1}\alo+G_0^{-1}\beto\right)$. 
Then, using Theorem \ref{thm:parametrcommdomain}, for real values of $(v,\ga)$,
\[
\begin{split}
\wh {\Delta \al}_0=&\,\pi_\alpha(\mathcal{J}\delta z)+\GG^u(F_3(0))-\GG^s(F_3(0))+\OO\left(G_0^{-4}\alo+G_0^{-4}\beto\right)\\
=&\,\delta\eta-i\pa_{\beto}\LL+\OO\left(G_0^{-4}\alo+G_0^{-4}\beto\right)+\OO\left(G_0^{-3}\eta_0\right)
\end{split}
\]
and analgously for $\bet$.
\end{proof}

\begin{proof}[Proof of Lemma \ref{lemma:ErrorEstimate}]
We split the function $\EE$ as  $\EE=\EE^1+\EE^2+\EE^3$ with
\begin{equation}\label{def:ErrorE}
\begin{split}
 \EE^1&=(\Id+\wt\Psi)\ii\wt\Phi_A\ii\wt \Delta-\wt\Phi_A\ii\wt \Delta=\left[(\Id+\wt\Psi)\ii-\Id\right]\wt\Phi_A\ii\wt\Delta\\
 \EE^2&=\wt\Phi_A\ii \left(\wt\Delta-\Delta\right)\\
 \EE^3&=\wt\Phi_A\ii \left(\Delta-\left(\FF(0)^u-\FF^s(0)\right)\right)
\end{split}
\end{equation}
where  $\wt\Psi$ is the matrix obtained in Theorem \ref{thm:fundamentalmatrix}.

We bound each term separately. For the first
term, one can write the matrix as
\[
(\Id+\wt\Psi)\ii-\Id =\sum_{k\geq 1}\left(-\wt\Psi\right)^k.
\]
Therefore, using the estimates for $\wt\Phi_A\ii$ and $\wt\Psi$ in Lemma \ref{lemma:FundMatEstimates} and Theorem \ref{thm:fundamentalmatrix} respectively and the properties of the matrix norm given in Lemma \ref{lemma:NormMatrix},
\[
 \left\|\left[(\Id+\wt\Psi)\ii-\Id\right]\wt\Phi_A\ii\right\|_{1/2,\sigma_3,\matrx}\lesssim G_0^{-3}\ln G_0.
\]
Then, using also Theorems \ref{thm:parametrcommdomain} and
\ref{thm:DifferenceSymplecticStraightening} and Lemma \ref{lemma:banach:Composition}, one has $\|\wt\Delta\|_{1/2,\sigma_3,\vect}\lesssim G_0^{-3}\ln G_0$. Thus,
\[
\left\| \EE^1\right\|_{1,\sigma_3,\vect}\lesssim  G_0^{-3}\ln G_0\|\wt \Delta\|_{0,\sigma_3,\vect}\lesssim G_0^{-6}\ln^2 G_0.
\]
For $\EE_2$, we use the  definition of $\wt\Delta$ in
\eqref{def:difference:reparameterized}. Theorems  \ref{thm:parametrcommdomain} and
\ref{thm:DifferenceSymplecticStraightening} and Lemma \ref{lemma:banach:Composition} imply that
\[
 \|\wt\Delta-\Delta\|_{1/3,\sigma_3,\vect}\lesssim G_0^{-7}\ln^2G_0.
\]
Then, using this estimate and Lemmas \ref{lemma:banach:AlgebraProps} and \ref{lemma:FundMatEstimates},
\[
\left\|\EE^2\right\|_{3/2,\sigma_3,\vect}\lesssim \left\|\wt\Phi_A\ii \right\|_{0,\sigma_3,\matrx} \|\wt\Delta-\Delta\|_{3/2,\sigma_3,\vect}\lesssim G_0^{-7}\ln^2 G_0.
\]
Finally, using that the paramerizations of the invariant manifolds $Z^\ast$, $\ast=u,s$  obtained in Theorem  \ref{thm:parametrcommdomain} are fixed points of the operators $\FF^\ast$, $\ast=u,s$, respectively,
we write $\EE^3$ as
\[
\EE^3= \wt\Phi_A\ii\left(\FF^u(Z^u)-\FF^s(Z^s)\right)-\wt\Phi_A\ii\left(\FF^u(0)-\FF^s(0)\right).
\]
Now, by \eqref{def:LipschitzImproved},
\[
 \left\|\FF^*(Z^*)-\FF^*(0)\right\|_{1,\vect,\sigma_3}\lesssim  G_0^{-6}|\ln G_0|^3 \qquad \ast=u,s.
\]
Therefore, using this estimate and the estimate for  $\wt \Phi_A$ in Lemma \ref{lemma:FundMatEstimates}, we obtain
\[
\left\|\EE^3\right\|_{1,\sigma_3,\vect}\lesssim \left\|\wt\Phi_A\ii \right\|_{0,\sigma_3,\matrx}\left(\left\|\FF^u(Z^u)-\FF^u(0)\right\|_{1,\sigma_3,\vect}+\left\|\FF^s(Z^s)-\FF^s(0)\right\|_{1,\sigma_3,\vect}\right)\lesssim    G_0^{-6}|\ln G_0|^3.
\]
Now, for $v\in D_{\kk_3,\de_3}\cap \RR$ and $\ga\in\TT$,
\[
|\langle\EE\rangle|\lesssim |\EE|\lesssim \left\|\EE^1\right\|_{1,\sigma_3,\vect}+\left\|\EE^2\right\|_{3/2,\sigma_3,\vect}+\left\|\EE^3\right\|_{1,\sigma_3,\vect}\lesssim   G_0^{-6}|\ln G_0|^3.
\]
For the other harmonics, one can use that  $\left\|\EE\right\|_{0,\sigma_3,\vect}\lesssim G_0^{-5/2}\ln^2G_0$. Then, by the definition of the norm,

\[
 \begin{aligned}
 \|\EE_Y\|_{0,\sigma_3}&\lesssim G_0^{1/2}\ln^2G_0,\qquad  &  \|\EE_\La\|_{0,\sigma_3}&\lesssim G_0^{-5/2}\ln^2G_0 \\
 \|\EE_\al\|_{0,\sigma_3}&\lesssim G_0^{-1}\ln^2 G_0, \qquad&  \|\EE_\bet\|_{0,\sigma_3}&\lesssim G_0^{-1}\ln^2 G_0.
  \end{aligned}
\]
which together with  Lemma \ref{lemma:Lazutkin} gives the estimates of the lemma.
\end{proof}

\subsection{Improved estimates for the averaged term}\label{sec:Average}

The last step in the analysis of the difference between the invariant manifolds is to obtain improved estimates for the averaged term for $\EE_Y$ and $\EE_\La$. 
Note that for these two components the first order given by the Melnikov potential  in Theorem \ref{thm:MainSplitting} is exponentially small (see Proposition \ref{prop:MelnikovPotential}). 
Therefore, to prove that $\wh \Delta_0$ is bigger than the error one has to show that the averages $\langle \EE_Y\rangle$ and $\langle \EE_\La\rangle$ are smaller than the corresponding Melnikov components $\omega\pa_\sigma\LL$, $-\pa_\sigma\LL$ (and thus exponentially small). 
We prove this fact by using the Poincar\'e invariant (for $\Lambda$) and the conservation of energy (for $Y$).

Consider the autonomous Hamiltonian system associated to \eqref{def:HamFinal}. 
Then, it has the Poincar\'e invariant
acting on closed loops of the phase space defined as follows. Consider a loop $\Gamma$  in phase space, then
\[
 I(\Ga)=\int_\Ga \left(Y du+\La d\ga +\beta d\al\right)
\]
is invariant under the flow associated to the Hamiltonian system. 
We use this invariant to improve the estimates of the $\ga$-averaged terms in $\wh{\Delta Y}$ and $\wh{\Delta\La}$.

In Theorem \ref{thm:parametrcommdomain}, we have obtained parameterizations of the unstable manifold of the
periodic orbit $P_{\alo,\beto}$ and the stable manifold of the periodic orbit $P_{\eta_0+\delta\eta,\xi_0+\delta\xi}$, $Z^u$ and $Z^s$ (as graphs).
Take any loop $\Ga^u$ contained in $W^u (P_{\alo,\beto})$, homotopic to a loop of the form 
$W^u ( P_{\alo,\beto})\cap \{u=u_0\}$ 
(and thus homotopic to $P_{\alo,\beto}$) and any loop $\Ga^s$ contained in $W^s (P_{\alo+\delta\eta,\beto+\delta\xi})$, homotopic to a loop of the form 
$W^s ( P_{\alo+\delta\eta,\beto+\delta\xi})\cap \{u=u_0\}$ 
(and thus homotopic to $P_{\alo+\delta\eta,\beto+\delta\xi}$). 
More concretely, in the case of the stable manifold, take a $\CCC^1$ function $f:\TT\to D^u\cap \RR$ and define an associated loop parameterized as
\begin{equation}\label{def:loops}
 \Ga^s=\{(u,\ga,\al,Y,\La, \bet)=(f(\ga), \ga, \al^s(f(\ga),\ga), Y^s(f(\ga),\ga), \La^s(f(\ga),\ga), \bet^s(f(\ga),\ga) \}
\end{equation}
where $Z^s=(Y^s, \La ^s,\al^s,  \bet^s)$ is the parameterization of the invariant manifold obtained in Theorem  \ref{thm:parametrcommdomain}.
\begin{lemma}\label{lemma:Poincare}
The loops of the form \eqref{def:loops} satisfy $I(\Ga^*)= 0$, $*=u,s$.
\end{lemma}
\begin{proof}
Call $\Phi_t$ the flow associated to the Hamiltonian $\PP$ in \eqref{def:HamFinal} and take a loop $\Ga^s$. Then, since the Poincar\'e invariant is invariant  under this flow
\[
 I(\Ga^s)=\lim_{t\to \infty} I(\Phi_t(\Ga^s)).
\]
Then, using that $\lim_{t\to\infty} \pi_\alpha (\Phi_t( \Ga^s))=\delta\eta$ and the estimates of the parameterizations of $Z^s$ as $u\to\infty$ in Theorem \ref{thm:parametrcommdomain}, it is clear that $\lim_{t\to\infty} I(\Phi_t(\Ga^s))=0$.
\end{proof}
Taking any of the loops $\Ga^s$ and $\Ga^u$ considered in Lemma \ref{lemma:Poincare},
we have $I(\Ga^s)- I(\Ga^u)=0$

Now consider the difference $\Delta$ between the parameterization of the invariant
manifolds  introduced in \eqref{def:difference}, which satisfies equation \eqref{eq:Diffeq}.

Let us consider any loop $\Ga^\ast$ of those considered in Lemma \ref{lemma:Poincare} and apply to them the symplectic exact transformation $\Phi$ obtained in Theorem \ref{thm:DifferenceSymplecticStraightening}.
That is, $\wt\Ga^*=\Phi\ii (\Ga^*)$. Since the Poincar\'e invariant is invariant under exact symplectic transformation, we have that
\[
 I(\wt \Ga^*)=\int_{\wt\Ga^*} \left(\wt Y dv+\wt \La d\ga +\wt\beta d\wt\al\right)=I(\Ga^*)=0.
\]
%
Now, let us fix a section $v=v_0$ and study the loops
$\wt \Ga^u=W^u(P_{\alo,\beto})\cap \{v=v_0\}$ and $\wt \Ga^s=W^s(P_{\alo+\delta\eta,\beto+\delta\xi})\cap \{v=v_0\}$. Note that these loops satisfy
\[
 I\left(\wt\Ga^*\right)=\int_0^{2\pi} \left(\wt \La^*(v_0,\ga)  +\wt\beta^*(v_0,\ga)
\partial_\ga\wt\al^* (v_0,\ga)\right)d\ga=0,\quad \ast=u,s
\]
where $(\wt Y^*, \wt \La^*,\wt\al^*,\wt \bet^*)$, $*=u,s$, are the parameterizations obtained in Theorem \ref{thm:DifferenceSymplecticStraightening}.

We subtract the equation for the stable and unstable loops. Integrating by parts, we obtain the
relation
\begin{equation}\label{def:PoincareInvariantEquality}
0=I(\wt\Ga^u)-I(\wt\Ga^s)= \int_0^{2\pi} \left(\Delta \wt \La(v_0,\ga)
+h_1(v_0,\ga)\Delta\wt\al(v_0,\ga)
+h_2(v_0,\ga)\Delta \wt\beta (v_0,\ga)\right)d\ga
\end{equation}
where
\[
h_1(v_0,\ga)=-\frac{1}{2}\left(\partial_\ga\wt\beta^s(v_0,
\ga)+\partial_\ga\wt\beta^u(v_0, \ga)\right)\quad\text{and}\quad
h_2(v_0,\ga)=\frac{1}{2}\left(\partial_\ga\wt\al^s(v_0,\ga)+\partial_\ga\wt\al^u(v_0,
\ga)\right).
\]
Note that the functions $h_i$ satisfy $\langle h_i\rangle_\ga=0$ and, by Theorem \ref{thm:parametrcommdomain}, $|h_i(v,\ga)|\lesssim G_0^{-6}\ln G_0$ for $i=1,2$ and real values of $(v,\ga)$.

Next step is to make the transformation
\begin{equation}\label{def:ChangeFundamental}
 (\wt
Y,\wt \Lambda,\wt \al,\wt \beta)^\top=\Psi(v,\tau) (\wh
Y,\wh \Lambda, \wh\al, \wh\beta)^\top
\end{equation}
where $\Psi=\wt\Phi_A(\Id+\wt\Psi)$  and $\wt\Psi$ is the matrix obtained in Theorem \ref{thm:fundamentalmatrix}.
%
%

The difference between the invariant manifolds in these new coordinates is the vector $\wh \Delta$ introduced in \eqref{def:DiffKernel}, which satisfies $ \LL\wh\Delta=0$.
%

We analyze the Poincar\'e invariant relation
 \eqref{def:PoincareInvariantEquality} after performing the change of
 coordinates \eqref{def:ChangeFundamental}.
Note that this change  is not
symplectic and therefore the Liouville form is not preserved. Thus, we apply the
change of coordinates to \eqref{def:PoincareInvariantEquality} directly
and we obtain
\begin{equation}\label{def:InvPoincTransf}
 \int_0^{2\pi} \left((1+  \wh h_0(v_0,\ga))\wh \Delta \La(v_0,\ga)
+\wh h_1(v_0,\ga)\wh \Delta\al(v_0,\ga)
+\wh h_2(v_0,\ga)\wh \Delta \beta (v_0,\ga)+\wh
h_3(v_0,\ga)\wh \Delta Y\right)d\ga=0
\end{equation}
for some functions $\wh h_i$. By the definition of the fundamental matrix $\wt\Phi_A$ in \eqref{def:FundamentalMatrixModified} and the estimates of the matrix $\wt\Psi$ obtained in Theorem \ref{thm:fundamentalmatrix}, one can easily check that,  for real values of $(v,\ga)$ and assuming $|\alo| G_0^{3/2}\ll1$, the functions $\wh h_i$ satisfy
\begin{equation}\label{def:EstimatesFunctionshs:1}
|\wh h_i|\lesssim G_0^{-3}\ln G_0, \quad i=0\ldots 2\qquad \text{ and }\qquad|\wh h_3|\lesssim G_0^{-9/2}\ln G_0.
\end{equation}
%
%
We would like to use \eqref{def:InvPoincTransf} to obtain more accurate estimates of $\langle \wh \Delta \Lambda\rangle$.  Assume for a moment that  $\langle \wh h_i\rangle=0$ for $i=1,2,3$ and let us introduce the
following notation
\[
\left\{ f\right\}_\ga(v,\ga)= f(v,\ga)-\langle
f\rangle_\ga(v).
\]
It certainly satisfies $\langle\left\{f(v,\ga)\right\}_\ga\rangle_\ga=0$.

By \eqref{def:ExpSmallErrors}, the four components of $\left\{\wh\Delta\right\}_\ga$ are
exponentially small.
Under the assumption  $\langle \wh h_i\rangle_\ga=0$ for $i=1,2,3$ then
\eqref{def:InvPoincTransf} becomes
\[
\langle \wh \Delta\Lambda \rangle_\ga= \frac{1}{2\pi (1+\langle
\wh h_0\rangle_\ga)}\int_0^{2\pi}F(\ga)d\ga
\]
with
\[
 F(\ga)=\left\{ \wh h_0 \right\}_\ga\left\{ \wh \Delta\Lambda
\right\}_\ga+\left\{ \wh h_1 \right\}_\ga\left\{ \wh \Delta\al
\right\}_\ga+\left\{ \wh h_2 \right\}_\ga\left\{ \wh \Delta\beta
\right\}_\ga+\left\{ \wh h_3 \right\}_\ga\left\{ \wh \Delta Y \right\}_\ga.
\]
Now, using the estimates given in Proposition \ref{prop:MelnikovPotential} and \eqref{def:ExpSmallErrors}, one would obtain
\[
 |F(\ga)|\lesssim G_0^{-5/2} e^{-G_0^3/3}\ln G_0
\]
therefore $\langle \wh \Delta\Lambda
\rangle_\ga$ would have the same estimate.

Now, this argument does not work because $\langle \wh h_i\rangle_\ga
\neq 0$ for $i=1,2,3$. Therefore, we have to perform the close to the identity change of coordinates which depends on $v$ but not on $\ga$.
\[
\check{\Delta} \La=\wh\Delta \La+\frac{\langle \wh h_1\rangle_\ga}{1+\langle \wh  h_0\rangle_\ga}\wh
\Delta \al+\frac{\langle \wh  h_2\rangle_\ga}{1+\langle \wh  h_0\rangle_\ga}\wh \Delta
\bet+\frac{\langle \wh  h_3\rangle_\ga}{1+\langle \wh  h_0\rangle_\ga}\wh \Delta Y
\]
and the other variables remain unchanged. Note that the functions $h_i$ are
small by \eqref{def:EstimatesFunctionshs:1} and therefore this change is close to the identity. Moreover, the functions $\langle \wh  h_i\rangle_\ga$ are independent of $\ga$ and therefore
\[
\langle\check{\Delta} \La\rangle_\ga=\langle\wh\Delta \La\rangle_\ga+\wh  C_1\langle\wh
\Delta \al\rangle_\ga+\wh  C_2\langle\wh \Delta
\bet\rangle_\ga+C_3\langle\wh \Delta Y\rangle_\ga.
\]
where
\begin{equation}\label{def:Ci}
\wh  C_i=\frac{\langle \wh h_i\rangle_\ga}{1+\langle \wh  h_0\rangle_\ga}.
\end{equation}
Thus
\[
\left|\left\{\check{\Delta} \La\right\}_\ga\right| \lesssim G_0^{-3/2}  e^{-G_0^3/3}.
\]
%
Now the relation \eqref{def:InvPoincTransf} becomes
\[
  \int_0^{2\pi} \left((1+\wh h_0(v_0,\ga))\check \Delta \La(v_0,\ga)
+\check  h_1(v_0,\ga)\wh \Delta\al(v_0,\ga)
+\check h_2(v_0,\ga)\wh \Delta \beta (v_0,\ga)+\check
h_3(v_0,\ga)\wh \Delta Y\right)d\ga=0
\]
where
\[
 \check h_i=\wh h_i-\frac{1+\wh h_0}{1+\langle \wh h_0\rangle_\ga}\langle \wh h_i\rangle_\ga,\qquad i=1,2,3,
\]
and therefore satisfy $\langle\check h_i\rangle_\ga=0$ and $|\check h_i|\lesssim G_0^{-3}\ln G_0$ for $i=1,2,3$. Therefore, the argument done
previously works and one can deduce that
\[
\left\langle \check \Delta \La(v_0,\ga)\right\rangle_\ga \lesssim G_0^{-5/2} e^{-G_0^3/3}\ln G_0.
\]

The just obtained results are summarized  in the following lemma.
\begin{lemma}\label{lemma:AverageLambda}
The function $\wh\Delta$ introduced in \eqref{def:DiffKernel} satisfies that $\wh\Delta(v,\ga) =N(v)\check\Delta(v,\ga)$ where $N$ is the matrix
\[
N(v)=\begin{pmatrix}
      1&0&0&0\\
      -\wh C_3&1&-\wh C_1&-\wh C_2\\
      0&0&1&0\\
      0&0&0&1.
     \end{pmatrix}
\]
which satisfies
\[
N=\Id+\OO\left(G_0^{-3}\ln G_0\right)\qquad \text{ and }\qquad N_{21}=\OO\left(G_0^{-9/2}\ln G_0\right).
\]
Moreover, for real values of $v\in D_{\kk_3,\de_3}\cap\RR$, $\left\langle \check \Delta \La(v,\ga)\right\rangle_\ga \lesssim G_0^{-5/2} e^{-G_0^3/3}\ln G_0$.
\end{lemma}

The next two lemmas complete the estimates of the errors in Theorem \ref{thm:MainSplitting}.

\begin{lemma}\label{lem:DeltaTildeFinal}
The function $\wt\Delta(v,\ga)$ can be written as
\[
 \wt\Delta(v,\ga)=\wt\NNN(v,\ga)\left(\wh\Delta_0(v,\ga)+\wt \EE(v,\ga)\right)
\]
where $\wh\Delta_0$ is the function introduced in \eqref{def:Deltazerohat}, $\wt\NNN$ is an invertible matrix satisfying
\[
\wt\NNN=\Id +\OO(G_0^{-3}\ln G_0)\qquad \text{and}\qquad \wt\NNN_{21}=\OO\left(G_0^{-9/2}\ln G_0\right),\]
and $\wt\EE$ satisfies
\[
 |\wt\EE_Y|\lesssim G_0^{-6}\ln^2 G_0,\qquad  |\wt\EE_\La|\lesssim e^{-G_0^3/3}G_0^{-5/2}\ln^2G_0,\qquad 
 |\wt\EE_\al|,  |\wt\EE_\bet|\lesssim G_0^{-6}\ln^2 G_0.
\]
Moreover, $\wt\EE_Y,\wt\EE_\al,\wt\EE_\bet\in\mathrm{Ker}\LL$ and 
\[
 \begin{aligned}
 |\wt\EE_Y-\langle\wt\EE_Y\rangle_\gamma|&\lesssim e^{-G_0^3/3}G_0^{2}
 \ln^2G_0 \\
 |\wt\EE_\al-\langle\wt\EE_\al\rangle_\gamma|,|\wt\EE_\bet-\langle\wt\EE_\bet\rangle_\gamma|&\lesssim e^{-G_0^3/3}G_0^{1/2}
 \ln^2G_0.
 \end{aligned}
\]
\end{lemma}
\begin{proof}
By  \eqref{def:DiffKernel} and   \eqref{lemma:AverageLambda}, one has that $\wt\Delta=\NNN\check\Delta$ with $\NNN=\wt\Phi_A (\Id+\wt\Psi)N$. Then, it can be written as
\[
 \wt\Delta=\NNN\left(\wh\Delta_0+\wt\EE\right)\qquad \text{ with }\qquad \wt\EE=N\ii\EE-(\Id-N\ii)\wh\Delta_0=\check\Delta-\wh\Delta_0.
\]
We estimate $\langle  \wt\EE\rangle_\ga$ and $\{ \wt\EE\}_\ga$ separately. For the $\Lambda$ average component we use that \[\langle  \wt\EE_\La\rangle_\ga=\langle \check\Delta \La\rangle_\ga=\OO\left(G_0^{-5/2} e^{-G_0^3/3}\ln G_0\right).\]
For the other averages one has to use that $N-\Id=\OO(G_0^{-3}\log G_0)$ and the estimates for $\EE_\al$ and $\EE_\beta$ in Lemma \ref{lemma:ErrorEstimate} and for $\wh\Delta_0$ given by Lemma \ref{lemma:Delta0Hat} and Proposition \ref{prop:MelnikovPotential}.
One can proceed analogously for  the no average terms $\{\EE\}_\ga$
\end{proof}
Now it only remains to express the difference between the invariant manifolds in the variables $(u,\ga)$ (see Theorem \ref{thm:DifferenceSymplecticStraightening}).
\begin{lemma}\label{lem:DeltaFinal}
The function $\Delta(u,\ga)$ introduced in \eqref{def:difference} can be written as
\[
\Delta(u,\ga)=\NNN(u,\ga)\left(\wh\Delta_0(u,\ga)+\EE(u,\ga)\right)
\]
where $\wh\Delta_0$ is the function introduced in \eqref{def:Deltazerohat}, $\NNN$ is an invertible matrix satisfying
\[
\NNN=\Id +\OO(G_0^{-3}\log G_0)\qquad \text{ and }\qquad \NNN_{21}=\OO\left(G_0^{-9/2}\ln G_0\right)
\]
 and $\EE$ satisfies
\[
 |\EE_Y|\lesssim G_0^{-6}\ln^2 G_0,\qquad   |\EE_\La|\lesssim e^{-G_0^3/3}G_0^{-5/2}\ln^2G_0,\qquad 
 |\EE_\al|, |\EE_\bet|\lesssim G_0^{-6}\ln^2 G_0.
\]
Moreover, $\EE_Y,\EE_\al,\EE_\bet$ and 
\[
 \begin{aligned}
 |\EE_Y-\langle\EE_Y\rangle_\gamma|&\lesssim e^{-G_0^3/3}G_0^2
 \ln^2G_0 \\
 |\EE_\al-\langle\EE_\al\rangle_\gamma|, |\EE_\bet-\langle\EE_\bet\rangle_\gamma|&\lesssim e^{-G_0^3/3}G_0^{1/2}
 \ln^2G_0.
 \end{aligned}
\]
\end{lemma}
The proof of this lemma is straighforward applying the inverse of change of coordinates obtained in Theorem \ref{thm:DifferenceSymplecticStraightening}.

\subsection{End of the proof of  Theorem \ref{thm:MainSplitting}}

Lemma \ref{lem:DeltaFinal}, recalling the expression of $\wh\Delta_0$ given in Lemma \ref{lemma:Delta0Hat},  completes the proof of  formulas \eqref{def:SplittingFormula}, \eqref{def:matriuN}, \eqref{def:MelnikvoWithErrors} in  Theorem \ref{thm:MainSplitting} (recall the relation between $G_0$ and $\Theta$ given in \eqref{eq:omega}). Note, however that it gives slightly worse estimates compared to those in Theorem \ref{thm:MainSplitting}. Indeed, Lemma \ref{lem:DeltaFinal} implies that  $\NNN$  is of the form 
\[
\NNN=\Id+\OO(\Tet^{-3}\log \Tet)
\]
and $\MM_\al$ and $\MM_\bet$ satisfy
\[
\begin{pmatrix}
\MM_\al(u,\ga,  z_0,\de z)\\
\MM_\bet(u,\ga,  z_0,\de z)
\end{pmatrix}
=
\begin{pmatrix}
-i\pa_{\beto}\LL(\ga-\omega u,  z_0)
+\OO\left(\Tet^{-6}\ln^2 \Tet\right)\\
i\pa_{\alo}\LL(\ga-\omega u,  z_0)
+\OO\left(\Tet^{-6}\ln^2 \Tet\right)
\end{pmatrix}.
\]
To obtain the estimates for the derivatives given in \eqref{def:DerivMelnikovLambda}, it is enough to apply Cauchy estimates. Indeed, the formulas \eqref{def:SplittingFormula}, \eqref{def:matriuN}, \eqref{def:MelnikvoWithErrors} are valid in a domain such that $|\eta_0|\Theta^{3/2}\ll 1$ and therefore applying Cauchy estimates one loses $\Theta^{3/2}$ at each derivative.

Then, it only remains to prove the estimates in \eqref{def:DerivadesManifolds} and \eqref{def:DerivMelnikovAltres}. The first ones are a direct consequence of Theorem \ref{thm:parametrcommdomain}. To obtain the estimates for the derivatives of $\MM$ in \eqref{def:DerivMelnikovAltres} we proceed as follows.

First note that, by Theorem \ref{thm:parametrcommdomain}, the paramerizations of the invariant manifolds admit an analytic continuation to the domain
\begin{equation}\label{def:realdomain}
u\in D_{\kappa,\delta}\cap \mathbb{R},\ \gamma\in\mathbb{T},\ |\eta_0|\leq\frac{1}{2},\  |\xi_0|\leq\frac{1}{2}.
\end{equation}
Proceeding analgously, one can also extend analytically to this domain the change of variable $\Phi$ obtained in Theorem \ref{thm:DifferenceSymplecticStraightening}. Then, one can easily check that in such domain, the associated function $\CCC$  satisfies
\[
 |\CCC|\lesssim \Theta^{-4}.
\]
(recall \eqref{eq:omega}). Similarly, one can extend the  matrix $\wt\Psi$ given by Theorem \ref{thm:fundamentalmatrix} to the same domain, where it satisfies
\[
 |\wt\Psi|\lesssim \Theta^{-3}.
\]
Then, one can conclude that the matrix $\NNN$ appearing in Theorem \ref{thm:MainSplitting} can be also analytically extended to \eqref{def:realdomain} where it satisfies
\[
 \NNN=\Id+\OO( \Theta^{-3}).
\]
(see \eqref{def:matriuN}). This analysis also gives the improved estimates for $\MM_\al$ and $\MM_\bet$  in  \eqref{def:MelnikvoWithErrors}.

Note that  the derivatives of $\CCC$, $\wt\Psi$ and $\NNN$ with respect to $(\eta_0.\xi_0)$ have the same estimates in the domain 
\[|\eta_0|\leq\frac{1}{4},\  |\xi_0|\leq\frac{1}{4}.\]
Indeed, it is enough to apply Cauchy estimates. Using these estimates and the estimates of the derivatives of the parameterizations of the invariant manifolds given by Theorem \ref{thm:parametrcommdomain}, one can easily deduce formulas \eqref{def:DerivMelnikovAltres}.

\section{The homoclinic channels and the associated scattering maps}\label{sec:scatteringglobal}
We devote this section to prove the results on the scattering maps stated in Section \ref{sec:scatteringstatements}.

First, in Section \ref{sec:existencescattering} we prove Theorem \ref{prop:scatteringmap}. That is we prove the existence of two homoclinic channels and we obtain formulas for the associated scattering maps (in suitable domains).
%

Finally, in Section \ref{sec:IsolatingBlockScattering} we prove 
Theorem \ref{thm:BlockScattering} which provides the existence of an isolating block for a suitable high iterate of a combination of the two scattering maps obtained in  Theorem \ref{prop:scatteringmap}.


\subsection{The scattering maps: Proof of Theorem \ref{prop:scatteringmap}}\label{sec:existencescattering}
We devote this section to prove the existence and derive formulas for the scattering maps given by Theorem \ref{prop:scatteringmap}.

Consider two periodic orbits $P_{\alo^-,\beto^-}, P_{\alo^+,\beto^+}\in\EE_\infty$ and denote
\[
 \de\eta=\alo^+-\alo^-,\quad  \de\xi=\beto^+-\beto^-.
\]
From now on we denote 
\[
 (\alo^-,\beto^-)=(\eta_0,\xi_0)\qquad \text{and}\qquad (\alo^+,\beto^+)=(\eta_0+ \de\eta,\xi_0+ \de\xi).
\]
We fix a section $u=u^*\in(u_1,u_2)$ (which is transverse to the flow) and we analyze the intersection of the unstable manifold of $P_{\alo,\beto}$ and the stable manifold of $P_{\alo+\de\eta,\beto+\de\xi}$ in this section. By the expression \eqref{def:sigmaparamnousnr} and Theorem \ref{thm:MainSplitting}, these invariant manifolds intersect along an heteroclinic orbit if there exists $\ga\in\TT$ such that
\begin{equation}\label{def:differencezero}
\begin{split}
Y^u(u^*,\ga,z_0) -Y^s(u^*,\ga, z_0,\de z)&=0\\
\La^u(u^*,\ga,z_0) -\La^s(u^*,\ga,z_0,\de z)&=0\\
\al^u(u^*,\ga,z_0) -\al^s(u^*,\ga, z_0,\de z)&=0\\
\bet^u(u^*,\ga, z_0)
-\bet^s(u^*,\ga, z_0,\de z)&=0
\end{split}
\end{equation}
where $z_0=(\alo,\beto)\in\wt\DD$ (see \eqref{def:domainscattering}) and $\delta z=( \de\eta, \de\xi)$.

Using \eqref{def:SplittingFormula} in Theorem \ref{thm:MainSplitting} and the fact that the matrix $\NNN$ is invertible, obtaining a zero $(\ga,\de z)$ of \eqref{def:differencezero} is equivalent to obtain a zero of
\[
 \begin{split}
\MM_Y(u^*,\ga, z_0,\de z)&=0\\
\MM_\La(u^*,\ga,  z_0,\de z)&=0\\
\de\eta+ \MM_\al(u^*,\ga,  z_0,\de z)&=0\\
\de\xi+ \MM_\bet(u^*,\ga,  z_0,\de z)&=0.
\end{split}
\]
Since we are dealing with an autonomous Hamiltonian, by the conservation of energy, if three of the components above vanish, the fourth one does too. For this reason we can get rid of the first equation and just look for zeros $(\ga,\de z)$ of
\begin{equation}\label{def:diffvarhetero}
 \begin{split}
 \MM_\La(u^*,\ga,  z_0,\de z)&=0\\
\de\eta+ \MM_\al(u^*,\ga,  z_0,\de z)&=0\\
\de\xi+ \MM_\bet(u^*,\ga,  z_0,\de z)&=0.
\end{split}
\end{equation}
We emphasize that by zeros  we mean that for a given $z_0$ and $u^*$ there exists $\ga$ and $\de z$ which solve these three equations.

We first analyze the third and fourth equations, that is 
\[
 \de\eta+ \MM_\al(u^*,\ga,  z_0,\de z)=0,\qquad 
\de\xi+ \MM_\bet(u^*,\ga,  z_0,\de z)=0.
\]
Using the asymptotic expansions for $ \MM_\al$ and $ \MM_\bet$ given in Theorem \ref{thm:MainSplitting}, one can obtain $(\de\eta,\de\xi)$ in terms of $(\alo,\beto)$ and $\ga$ as follows,
%
\begin{equation}\label{eq:melnikovseccio9}
  \begin{split}
\de\eta&=\de\eta(u^*,\ga,z_0)=-i\pa_{\beto}\LL(\ga-\omega u^*, \alo,\beto)+P_1(u^*,\ga,\alo,\beto)\\
\de\xi&=\de\xi(u^*,\ga,z_0)=i\pa_{\alo}\LL(\ga-\omega u^*, \alo,\beto)+P_2(u^*,\ga,\alo,\beto)
\end{split}
\end{equation}
%
%
where by Proposition \ref{prop:MelnikovPotential} and
the estimates of $\MM$ in Theorem \ref{thm:MainSplitting}, implicit derivation and Cauchy estimates, the functions $P_1$ and $P_2$ satisfy
\begin{equation}\label{def:EstimatesPDerivs}
\begin{aligned}
|\pa_{\alo}^i\pa_{\beto}^j P_i|&\leq C(i,j) \Theta^{-6}\qquad &\text{for}\quad & i,j\geq 0\\
|\pa_{\alo}^i\pa_{\beto}^j\pa^{k}_\ga P_i|&\leq C(i,j,k) \Tet^{1/2-3(i+j)/2}\ex^{-\frac{\tilde \nu \tte^3}{3 L_0^3}}  &\text{for}\quad & i,j\geq 0, k\geq 1
\end{aligned}
\end{equation}
for some constants $C(i,j)$ and $C(i,j,k)$.

Now we solve the equation for the $\Lambda$ component  evaluated at \eqref{eq:melnikovseccio9}, that is
\[
\MM_\La(u^*,\ga,  z_0,\de z(u^*,\ga,z_0))=0,
\]
which gives  two solutions $\ga^1,\ga^2$. 
Note that in the domain $\wt\DD$ introduced in \eqref{def:domainscattering}, one has that 
\[
\left| \LL^{[-1]}\right|\gtrsim \Theta^{-\frac{1}{2}}.
\]
Then, dividing this equation  by the factor $(-2\LL^{[-1]})$, one obtains an equation of the form
\[
 \sin (\omega u^*-\ga)+\OO(\Theta^{-1/2})=0
\]
which has two solutions
\begin{equation}\label{def:gammaj}
 \ga^j=\ga^j(u^*,z_0)=
 \omega u^*+(j-1)\pi+\OO\left(\Tet^{-1/2}\right),\qquad j=1,2.
\end{equation}
Moreover, by Cauchy estimates on the domain $\wt\DD$ introduced in \eqref{def:domainscattering} (reducing the constant $\varrho$ used in the definition of $\wt\DD$), one can see that  for $i,k\geq 0$,
\begin{equation}\label{def:derivativeschannels}
 |\pa_{\alo}^i\pa_{\beto}^k \ga^j|\leq C(i,k)\Theta^{-1/2-2(i+k)}
\end{equation}
for some constant $C(i,k)$ independent of $\Theta$.

Now we obtain asymptotic formulas for the scattering maps.

Observe that the values $(u^*,\ga^j(u^*,z_0),z_0,\de z(u^*,\ga^j(u^*,z_0),z_0))$, $u^*\in(u_1,u_2)$ and $z_0\in\wt\DD$ (see  \eqref{def:domainscattering}) solving equations \eqref{def:diffvarhetero}
give rise to heteroclinic points 
\[
z_{\het}^j=(\la^j_{\het}(u^*,z_0),w^j_{\het}(u^*,z_0))\in W^u(P_{z_0})\cap W^s(P_{z_0+\de z_0^j})
\]
with $\la^j_{\het}=\ga^j(u^*,z_0) +\phi_\h(u^*)$.
Consequently, there exist $\la^j_\pm$ such that
\begin{align}
&\psi\left(t, z_{\het}^j\right)-\psi\left(t, \la^j_- ,w_0\right)\to 0 \quad \mbox{as}\quad t\to -\infty \\
&\psi\left(t, z_{\het}^j\right)-\psi\left(t, \la^j_+ ,w_0+\de w_0^j\right)\to 0 \quad \mbox{as}\quad t\to +\infty,
\end{align}
where $w_0=(L_0,\infty,0,\alo,\beto)$ and $\delta w_0^j$ is given by \eqref{eq:melnikovseccio9} with $\ga=\ga^j(u^*,z_0)$, in the sense that the asymptotic condition in the $\tilde r$ component of $\psi\left(t, z_{\het}^j\right)$ means that it becomes unbounded.

An important observation is  that,  
using that the system \eqref{def:Poincarenr} is autonomous,  we have that, for any $r\in \RR$
\[
\begin{split}
\psi\left(t+r, z_{\het}^j \right)-\psi\left(t+r, \la^j_- ,w_0\right)\to 0 \quad \mbox{as}\quad t\to - \infty,\\
\psi\left(t+r, z_{\het}^j \right)-\psi\left(t+r, \la^j_+ ,w_0+\de w_0^j\right)\to 0 \quad \mbox{as}\quad t\to +\infty.
\end{split}
\]
Now, using that
\begin{equation}\label{lambdes}
\begin{split}
\psi\left(t+r, z_{\het}^j\right) &= \psi\left(t,\psi(r, z_{\het}^j)\right), \nonumber \\
\psi\left(t+r, \la^j_- ,w_0 \right)&
= \psi\left(t, \psi\left(r,  \la^j_-,w_0\right)\right)
=(\la^j_- +\frac{\nu}{L_0^3}r,w_0),
\end{split}
\end{equation}
and analgously for the other periodic orbit.

Calling $z_{\het}^j(r)=\psi\left(r, z_{\het}^j\right)$, and 
$\la_\pm^j (r)=\la_\pm^j +\frac{\nu}{L_0^3}r$, we have that
\[
\begin{split}
\psi\left(t, z_{\het}^j (r)\right)-\psi\left(t, \la^j_- (r),w_0\right)\to 0 \quad \mbox{as}\quad t\to- \infty\\
\psi\left(t, z_{\het}^j (r)\right)-\psi\left(t, \la^j_+ (r),w_0+\de w_0^j\right)\to 0 \quad \mbox{as}\quad t\to + \infty.
\end{split}
\]
Therefore, the orbit through $z_{\het}^j(r)$ is an heteroclinic orbit between the points $(\la^j_-(r),w_0)$ and $(\la^j_+(r),w_0+\de w_0^j)$,
for any $r\in\RR$.

Analogously, given any $\la \in \TT$, we can choose $r$ through
$\la=\la^j_-+\frac{\nu}{L_0^3}r$ and, abusing notation, calling again the heteroclinic point $z^j _{\het}(\la):=z^j _{\het}(r)$, we have that
\begin{equation}
\begin{split}
&\psi(t, z_{\het}^j (\la)-\psi(t, \la, w_0)\to 0 \quad \mbox{as}\quad t\to -\infty \\
&\psi(t, z_{\het}^j (\la))-\psi(t, \la +\Delta ^j  ,w_0+\de w_0^j)\to 0 \quad \mbox{as}\quad t\to + \infty
\end{split}
\end{equation}
where $\Delta^j =\la^j_+-\la^j_-$.

Consequently, the scattering maps are of the form
\begin{equation*}
\wt \SSS^j:
\begin{pmatrix}L_0 \\ \la\\ \eta_0\\ \xi_0\end{pmatrix}\to \begin{pmatrix}L_0 \\ \la+\Delta^j\\ \eta_0+\de \eta_0^j\\ \xi_0+\de \xi_0^j\end{pmatrix}
\end{equation*}
where $(\alo+\de\eta^j,\beto+\de\xi^j)=\SSS^j(\alo,\beto)$   are independent of $\la$ and $L_0$ is preserved by the conservation of the energy (recall that we are omitting the dependence on $L_0$ of all the functions).

Observe that $z_{\het}^j (\la)=z_{\het}^j (\la,\alo,\beto)$, with $(\alo, \beto)\in\wt\DD$  and $\la \in\TT$, gives a different parameterization of the homoclinic chanel introduced in \eqref{def:homochannels}.

Finally, note that to obtain formulas for  $\SSS^j$ one has just to evaluate \eqref{eq:melnikovseccio9} at the solutions  $\ga^j$ in \eqref{def:gammaj}, to obtain
%
\begin{equation}\label{def:ScatteringFormulas}
 \begin{split}
\SSS^j \begin{pmatrix}\alo\\ \beto\end{pmatrix}=
\begin{pmatrix}
\alo-i\pa_{\beto}\LL (j\pi,\alo,\beto)+
\OO\left(\Theta^{-6}\right)\\
\beto+i\pa_{\alo}\LL  (j\pi,\alo,\beto) +\OO\left(\Theta^{-6}\right)
\end{pmatrix}.
\end{split}
\end{equation}
Then, it is enough to use the formulas of the Melnikov potential $\LL$ given in 
Proposition \ref{prop:MelnikovPotential} to obtain the formulas for $\SSS^j$ in Proposition \ref{prop:scatteringmap}.

To obtain estimates for the derivatives of $\SSS^j$ it is enough to use the estimates for the derivatives of the Melnikov potential in Proposition \ref{prop:MelnikovPotential} and the estimates \eqref{def:EstimatesPDerivs} and \eqref{def:derivativeschannels}.

Finally, Theorem 8 in  \cite{DelshamsLS06a} implies that the map $\wt\SSS^j$ is symplectic. Then, using the particular form of $\wt\SSS^j$ one can easily see that $\SSS^j$ is symplectic in the sense that fixing $L=L_0$ it preserves the symplectic form $d\eta_0\wedge d\xi_0$.

Now it only remains to analyze the fixed points $(\eta_0^j,\xi_0^j)$ of the scattering maps $\SSS^j$. The particular form of the fixed points given in \eqref{def:formulafixedscattering} is just a consequence of \eqref{def:ScatteringFormulas} and the asymptotic expansions of the Melnikov potential given in Proposition \ref{prop:MelnikovPotential}. Note that $(\eta_0^j,\xi_0^j)\in \wt\DD$ where $\wt\DD$ is the domain introduced in \eqref{def:domainscattering}.

To prove the  asymptotic formula \eqref{def:DistanceHomosteor} for the difference between the two fixed points one cannot use \eqref{def:ScatteringFormulas} but has to go back to equations \eqref{def:diffvarhetero} and analyze them when $\de\eta=\de\xi=0$.
In particular, we know that
\[
\MM_\al(u^*,\ga^j, \alo^j,\beto^j)=0,\quad \MM_\bet(u^*,\ga^j, \alo^j,\beto^j)=0.
\]
We subtract the equalities for $j=1$ and $j=2$ to obtain
\[
\MM_\al(u^*,\ga^2, \alo^2,\beto^2)-\MM_{\al}(u^*,\ga^1, \alo^1,\beto^1)=0\qquad \text{and}\qquad
\MM_\bet(u^*,\ga^2, \alo^2,\beto^2)-\MM_\bet(u^*,\ga^1, \alo^1,\beto^1)=0.
\]
Taylor expanding, defining $\Delta\alo=\alo^ 2-\alo^ 1$, $\Delta\beto=\beto^ 2-\beto^ 1$ and using the estimates in Proposition \ref{prop:MelnikovPotential} and  Theorem \ref{thm:MainSplitting},
we have
\[
 \begin{split}
E_\al  +\pa_{\alo}\wt \MM_\al(u^*,\ga^2, \alo^1,\beto^1)\Delta\alo+\pa_{\beto}\wt \MM_\al(u^*,\ga^2, \alo^1,\beto^1)\Delta\beto+\Tet^{-3}\OO_2\left(\Delta\alo, \Delta\beto\right)&=0\\
E_\bet  +\pa_{\alo}\wt \MM_\bet(u^*,\ga^2, \alo^1,\beto^1)\Delta\alo+\pa_{\beto}\wt \MM_\bet(u^*,\ga^2, \alo^1,\beto^1)\Delta\beto+\Tet^{-3}\OO_2\left(\Delta\alo, \Delta\beto\right)&=0
 \end{split}
\]
where
\[
 \begin{split}
E_\al&= \MM_\al(u^*,\ga^2, \alo^1,\beto^1)- \MM_\al(u^*,\ga^1, \alo^1,\beto^1)=-\frac{3i}{2}\wt\nu N_2\sqrt{\pi}L_0^{7/2}\tte^{3/2}\ex^{-\frac{\tilde \nu \Tet^3}{3 L_0^3}}\left(1+\OO\left(\Tet\ii\right)\right)\\
E_\bet&=\MM_\bet(u^*,\ga^2, \alo^1,\beto^1)-\MM_\bet(u^*,\ga^1, \alo^1,\beto^1)=\frac{3i}{2}\wt\nu N_2\sqrt{\pi}L_0^{7/2}\tte^{3/2}\ex^{-\frac{\tilde \nu \Tet^3}{3 L_0^3}}\left(1+\OO\left(\Tet\ii\right)\right).
 \end{split}
\]
Moreover, using again  the estimates in Proposition \ref{prop:MelnikovPotential} and  Theorem \ref{thm:MainSplitting},
\[
\begin{split}
 \pa_{\alo}\MM_\al(u^*,\ga^2, \alo^1,\beto^1)&=-
 \frac{3i}{8}\wt\nu \pi L_0^3\tte^{-3}N_2+\OO\left(\Tet^{-5}\right)\\
 \pa_{\beto} \MM_\bet(u^*,\ga^2, \alo^1,\beto^1)&=\frac{3i}{8}\wt\nu \pi L_0^3\tte^{-3}N_2+\OO\left(\Tet^{-5}\right)
 \end{split}
 \]
and
\[
    \pa_{\beto}\MM_\al(u^*,\ga^2, \alo^1,\beto^1)=\pa_{\alo}\MM_\bet(u^*,\ga^2, \alo^1,\beto^1)=\OO\left(\Tet^{-5}\right).
\]
Then, it is enough to apply the Implicit Function Theorem.

\subsection{An isolating block for an iterate of the  scattering map: Proof of Theorem \ref{thm:BlockScattering}}\label{sec:IsolatingBlockScattering}
We devote this section to construct an isolating block of a suitable iterate of the scattering map. That is, for the map $\wt \SSS=(\SSS^1)^M\circ\SSS^2$ for a suitable large $M$  which depends on the size of the block. This will be a good approximation of the projection of an iterate of the return map into the central variables. To construct the block we need a ``good'' system of coordinates. We rely on the properties of the  scattering maps obtained in Proposition \ref{prop:scatteringmap}.

The steps to build the isolating block are
\begin{enumerate}
 \item Prove the existence of a KAM invariant curve $\TT_*$  for $\SSS^1$.
To apply KAM Theory we first have to do a finite number of steps of Birkhoff Normal Form around the elliptic point of $\SSS^1$ and consider action-angle coordinates.
 \item Prove that the preimage of  $\TT_*$ by the other scattering map, that is $(\SSS^2)^{-1}(\TT_*)$, intersects transversally  $\TT_*$.
 \item Pick a small ``rectangle'' whose lower boundary is a piece of this torus and one adjacent side is a piece of $(\SSS^2)^{-1}(\TT_*)$. Show that such rectangle has the isolating block property for $\wt \SSS=(\SSS^1)^M\circ\SSS^2$.
\end{enumerate}

We start with Step 1. We follow, without mentioning explicitly, all the notation used in Proposition \ref{prop:scatteringmap}. However, we restrict the scattering maps to much narrower domains (see Lemma \ref{lemma:BNF} below). In fact, we consider domains which are balls centered at the fixed points of the scattering maps obtained in Proposition \ref{prop:scatteringmap} and exponentially small radius. We use the notation for disks introduced in \eqref{def:disk}. These exponentially small  domains are enough to build the isolating block.

Since in this section we need to perform several symplectic transformations  to the scattering maps, we denote them by $\Phi_i$, $i=1,2,3,4$.

\begin{lemma}\label{lemma:BNF}
Fix 
$N_1\in\NN$, $N_1\geq 3$. For $\Theta>0$ large enough, there exists a symplectic change of coordinates
 \[
\Phi_1: \DD_{\rr/2}\left(z_0^1 \right)\to \DD_\rr\left(z_0^1\right)\qquad \text{with}\qquad \rr=\tth^{11/2}  \ex^{-\frac{\tilde \nu \tth^3}{3 L_0^3}}
\]
such that $\wt\SSS^1=\Phi_1\ii\circ\SSS^1\circ\Phi_1$ is of the form
 \begin{equation}\label{def:BNFofS1}
 \wt  \SSS^1(z)=z_0^1+e^{i\left(\omega_1+C_1|z-z_0^1|^{2}+C_2|z-z_0^1|^{4}\right)} (z-z_0^1) +\OO\left(z-z_0^1\right)^{7}
 \end{equation}
where
 the $\omega_1$ has been introduced in \eqref{def:VapScattering}, the constant $C_1=\TTT\tth^{-3}+\OO\left(\tth^{-5}\right)$ with $\TTT$ as introduced in \eqref{def:TTT} in Proposition \ref{prop:scatteringmap}, which satisfies $C_1\neq0$ and    $C_2$ such that $C_2=\OO\left(\tth^{-3}\right)$.

Morover, $\Phi_1$ satisfies
\[
 \Phi_1(z)=z+\tth^{-2}\OO\left(z^2\right)+\OO\left(z^3\right).
\]
\end{lemma}
\begin{proof}
 The proof of this lemma is through the classical method of Birkhoff Normal Form by (for instance) generating functions. Fix $N>0$. Note that then the small divisors which arise in the process are of the form $|k\tth^{-3}-1|$ for $k=1\ldots 2N-2$. Then, taking $\Theta$ big enough, they satisfy
 \[
  |k\tth^{-3}-1|\gtrsim \Theta^{-3}.
 \]
With such estimate and the estimates of the Taylor coefficients of the scattering map $\SSS^1$ given in Proposition \ref{prop:scatteringmap}, one can easily complete the proof of the lemma.
\end{proof}

Next step is the application of KAM Theorem. To this end, we consider action angle coordinates for $\wt  \SSS^1$ (centered at the fixed point). Note that the first order of $\wt\SSS^1$ in \eqref{def:BNFofS1} is integrable and therefore it only depends on the action.
\begin{lemma}\label{lemma:ActionAngleCoords}
Fix a parameter $\rr\in \left(0,\frac{1}{2}\tth^{11/2}  \ex^{-\frac{\tilde \nu \tth^3}{3 L_0^3}} \right)$ and any $\ell\geq 4$.  Consider the change of coordinates
\[
(\alo,\beto)= \Phi_2(\theta,I)=(\alo^1+\rr\sqrt{I}e^{i\theta},\beto^1+\rr\sqrt{I}e^{-i\theta}).
\]
Then, the map $\wh\SSS^1=\Phi_2\ii\circ\wt\SSS^1\circ\Phi_2$ is symplectic with respect to the canonical form $d\theta\wedge dI$ and it is of the form
 \begin{equation}\label{def:S1ActionAngle}
\wh  \SSS^1(\theta,I)=\begin{pmatrix}
                       \theta+B(I)+\OO_{\CCC^\ell}\left(\rr^{6}\right)\\
                       I+\OO_{\CCC^\ell}\left(\rr^{6}\right).
                      \end{pmatrix}
 \end{equation}
Moreover, there exists two constants $C_1,C_2\neq 0$ such that  for $I\in[1,2]$, the function $B$ satisfies
\[
\begin{split}
 B(I)&=C_1\tth^{-3}+\OO\left(\tth^{-4}\right)\\
\pa_I B(I)&= \rr^2\left(C_2\tth^{-3}+\OO\left(\tth^{-4}\right)\right)\\
\pa_I^2 B(I)&=\OO(\rr^4).
\end{split}
 \]
 \end{lemma}

We use the following KAM Theorem from \cite{Herman83c, Herman86b} (we use the simplified version already stated in \cite{DelshamsLS00})

\begin{theorem}\label{thm:KAM}
Let $f:\TT^ 1\times[0,1]\to \TT^ 1\times[0,1]$ be an exact symplectic $\CCC^\ell$ map, $\ell\geq 4$. Assume that $f=f_0+\de f_1$, where $f(I,\theta)=(\theta+A(I),I)$, $A$ is $\CCC^\ell$ and satisfies $|\pa_I A(I)|\geq M$ and $\|f_1\|_{\CCC^ \ell}\leq 1$.

Then, if $\sigma=\de^{1/2}M^{-1}$ is sufficiently small, for a set of Diophantine numbers with $\tau=5/4$, we can find invariant tori which are the graph of $\CCC^{\ell-3}$ functions $u_\omega$, the motion on them is $\CCC^{\ell-3}$ conjugate to a rotation by $\omega$ and $\|u_\omega\|_{\CCC^{\ell-3}}\lesssim\de^{1/2}$ and the tori cover the whole annulus except a set of measure smaller than $\mathrm{ constant}M\ii\de^{1/2}$.
\end{theorem}

Note that the map \eqref{def:S1ActionAngle} in $[1,2]\times\TT$ satisfies the properties of Theorem \ref{thm:KAM} with $M\gtrsim \Theta^{-3}\rr$ and $\de=\rr^6$ for any regularity $\CCC^\ell$ (the scattering map is actually analytic). This theorem then gives, in particular, a torus $\TT_*$ which is invariant by  $\wh  \SSS^1$ and is parameterized as  graphs as
\begin{equation}\label{def:InvTori}
\TT_*=\left\{I=\Psi(\theta)=I_*+\OO_{\CCC^1}(\rr^{2}), \theta\in\TT\right\},\qquad j=1,2,
\end{equation}
where $I_*$ satisfies $I_*\in[1,2]$.
Note that the $\CCC^1$ in the error refers to derivatives with respect to $\theta$.
We can apply several steps of Birkhoff Normal Form around the torus $\TT_*$.

\begin{lemma}\label{lemma:BNFTorus}
There exists a symplectic change of coordinates $\Phi_3$ satisfying
\[
(I,\theta)=\Phi_3(J,\psi)=\left(\psi+\OO_{\CCC^1}\left(\rr^{2}\right),J+I_*+\OO_{\CCC^1}\left(\rr^{2}\right)\right),
\]
such that $\{J=0\}=\Phi_3^{-1}(\TT_*)$ and  the map $\wh\SSS^1$ becomes
\begin{equation}\label{def:S1tilde}
 \tS^1(\psi,J)=\begin{pmatrix}
                       \psi+\wt B (J)+\OO\left(J^{2}\right)\\
                       J+\OO\left(J^{3}\right)
                      \end{pmatrix}
\end{equation}
where
\begin{equation}\label{def:b1}
 \wt B(J)=b_0+b_1J \quad \text{with} \quad b_0=C_1 \tth^{-3}+\OO\left(\Theta^{-4}\right),\qquad C_1\neq 0.
\end{equation}
%
and $b_1\in\RR$ which  satisfies $b_1\neq 0$ for $\Theta$ small enough.
\end{lemma}
Now we express the scattering map $\SSS^2$ also in $(\psi,J)$ coordinates to compare them. To this end, we
take
\begin{equation}\label{def:ChoiceRho}
 \rr=\tth^{7/2} \ex^{-\frac{\tilde \nu \tth^3}{3 L_0^3}}
\end{equation}
The exponent $7/2$ is not crucial and one could take any other exponent in the interval $(3/2, 11/2)$.

\begin{lemma}\label{lemma:ScatteringS2ActionAngle}
Take $\rr$ of the form \eqref{def:ChoiceRho} and $\Theta$ large enough. Then, the scattering map  $\SSS^2$ expressed in the coordinates $(\psi,J)$ obtained in Lemma \ref{lemma:BNFTorus} is of the form
\[
\tS^2(\psi,J)=\begin{pmatrix}\tS^2_\psi(\psi,J)\\\tS^2_J(\psi,J)\end{pmatrix}=\begin{pmatrix}\psi+f_1(\psi)+\OO\left(J\right)\\
J+g_1(\psi)+\OO\left(J\right)\end{pmatrix}
\]
where
\[
f_1(\psi)=\OO_{\CCC^1}\left(\Theta^{-2}\right)\qquad \text{and}\qquad
g_1(\psi)=C_2\tth^{-2}\cos\psi+\OO_{\CCC^1}\left(\Theta^{-3}\ln^2 \Theta\right),
\]
with some constant $C_2\neq 0$.
\end{lemma}
\begin{proof}
We need to apply to $\SSS^2$ the changes of coordinates given in Lemmas \ref{lemma:BNF} and \ref{lemma:ActionAngleCoords}.  We first apply the symplectic transformation $\Phi_1$ in Lemma \ref{lemma:BNF}. Then, we obtain
that $\wt  \SSS^2=\Phi_1\ii\circ\wt\SSS^2\circ\Phi_1$ is of the form
 \[
\wt\SSS^2(z)=\Phi\ii\circ\SSS^2\circ\Phi=z_0^2+\la_2 (z-z_0^2)+\wt P_2\left(z-z_0^2\right)
 \]
for some function $\wt P_2$ which satisfies
\[
 |\wt P_2\left(z-z_0^2\right)|\leq C_3 \left|z-z_0^2\right|^2
\]
for some constant $C_3$ independent of $\tth$.

Now we apply the (scaled) Action-Angle transformation considered in \ref{lemma:ActionAngleCoords}. To this end, taking into account  \eqref{def:DistanceHomosteor}, we define
\begin{equation}\label{def:EstimateDelta}
\Delta=z_0^1-z_0^2 =-\frac{6i}{\sqrt{\pi}}L_0^{1/2} \tth^{9/2} \ex^{-\frac{\tilde \nu \tth^3}{3 L_0^3}} \left(1+\OO\left(\Theta^{-1}\ln^ 2\Theta\right)\right).
\end{equation}
Denoting by $(\theta_1,I_1)$ the image of $(\theta,I)$ by $\wt\SSS^2$,
we obtain (in complex notation)
\[
\begin{split}
\rr \sqrt{I_1}e^{i\theta_1}&=-\Delta+\la_2\left(\Delta+\rr \sqrt{I}e^{i\theta} \right)+\sum_{k=2}^N\wt P_k\left(\Delta+\rr \sqrt{I}e^{i\theta}\right)+\OO\left(|\Delta+\rr \sqrt{I}e^{i\theta}|^2\right)^{N+1}\\
&=(\la_2-1)\Delta +\la_2 \rr \sqrt{I}e^{i\theta}+\OO(\rr+\Delta)^2.
\end{split}
\]
Therefore,
\[
 \sqrt{I_1}e^{i\theta_1}=(\la_2-1)\frac{\Delta}{\rr} +\la_2 \sqrt{I}e^{i\theta}+\frac{1}{\rr}\OO(\rr+\Delta)^2.
\]
Now, condition \eqref{def:ChoiceRho} and the fact that $\lambda_2 = e^{i\omega_2}$, with $\omega_2 = \omega_0 \tth^{-3} + \OO(\Theta^{-4})$, implies that
\[
 (\la_2-1)\frac{\Delta}{\rr}\lesssim \Theta^{-2}\ll 1.
\]
Therefore, using also \eqref{def:EstimateDelta}, \eqref{def:VapScattering} and the fact $|\la_2|=1$ (see Proposition \ref{prop:scatteringmap}), we obtain
\[
 I_1=I\left|1+\frac{\la_2-1}{\la_2\sqrt{I}}\frac{\Delta}{\rr}e^{i\theta}+\frac{1}{\rr}\OO_{\CCC^1}(\rr+\Delta)^2\right|^2=I\left(1+\frac{C}{\sqrt{I}} \tth^{-2}\cos\theta+\OO_{\CCC^1}\left(\Theta^{-3}\ln^2 \Theta\right)\right)
 \]
for some constant $C\neq 0$ independent of $\Theta$. The notation $\OO_{\CCC^1}$ refers to derivatives with respect to $(\theta,I)$.


The formulas for $\theta_1$ can be obtained analogously to obtain the following expansion
\[
\wh\SSS^2(\theta,I)=\begin{pmatrix}\theta+\OO_{\CCC^1}\left(\Theta^{-2}\right)\\
I+C_2\tth^{-2}\sqrt{I}\cos\theta+\OO_{\CCC^1}\left(\Theta^{-3}\ln^2 \Theta\right)\end{pmatrix}
\]
for some constant $C_2\neq 0$.

Now it only remains to apply the change of coordinates $\Phi_3$ obtained in Lemma \ref{lemma:BNFTorus}.
\end{proof}

We use the expressions of $\tS^1$ and $\tS^2$ given in Lemmas \ref{lemma:BNFTorus} and \ref{lemma:ScatteringS2ActionAngle} to build the isolating block. Consider the torus $\TT_*=\{J=0\}$ which is invariant by $\tS^1$. Then,  we define
\[
 \TT_-=\left(\tS^2\right)^{-1}(\TT_*)
\]
and we denote by $Z_{*}=(\psi_{*},0)$  the intersections between $\TT_*$ and $\TT_-$ which satisfies
\begin{equation}\label{def:theta11}
 \psi_{*}=\frac{\pi}{2}+\OO\left(\Theta^{-1}\ln^2\Theta\right).
\end{equation}
This point will be one of the vertices of the block and ``segments'' within $\TT_*$ and $\TT_-$ will be two of the edges of the block.

Lemma \ref{lemma:ScatteringS2ActionAngle} implies that
\[
 \pa_\psi \tS_J^2(\psi_*,0)=C_2 \tth^{-2}+\OO\left(\Theta^{-3}\ln^2 \Theta\right)\neq 0.
\]
Therefore, in a neighborhood of $(\psi_*,0)$, the torus $\TT_-$ can be parameterized as
\begin{equation}\label{def:functionh}
\psi=h(J)  \qquad \text{for}\qquad |J|\ll 1\qquad \text{and}\qquad h(0)=\psi_*.
\end{equation}
In other words, there exists a function $h$ satisfying $\tS_J^2(h(J),J)=0$.

To analyze such block we perform a last change of coordinates so that the segment of $\TT_-$  becomes vertical.

\begin{lemma}\label{lemma:BlockScatteringStraight}
 The symplectic change of coordinates
 \[
 \Phi_4:(\psi,J)=(\varphi+h(J),J) \qquad\text{for}\quad \varphi\in\TT, \,\,|J|\ll 1,
 \]
transforms the scattering maps given in Lemmas \ref{lemma:BNFTorus} and \ref{lemma:ScatteringS2ActionAngle}  into
\begin{equation}\label{def:S1tilde}
 \wh  \tS^1(\varphi,J)=\begin{pmatrix}
                       \varphi+\wt B (J)+\OO\left(J^{2}\right)\\
                       J+\OO\left(J^{3}\right)
                      \end{pmatrix}
\qquad \text{and}\qquad
\wh\tS^2(\varphi,J)=\begin{pmatrix}\wh\tS^2_\varphi(\varphi,J)\\\wh\tS^2_J(\varphi,J)\end{pmatrix}
                      \end{equation}
which satisfies $\wh\tS^2_J(0,J)=0$ for $|J|\ll 1$.
\end{lemma}

Note that, since $\wh\tS^2$ its a diffeomorphism, it satisfies
\begin{equation}\label{def:Scatt2Deriv}
 \nu=\pa_\varphi\wh\tS^2_J(0,0)\neq 0.
\end{equation}
We use Lemma \ref{lemma:BlockScatteringStraight} and this fact to build the isolating block.

Note that now the point $Z_{*}=(\psi_*, 0)$ has become $\wh Z_{*}=(0,0)$ and the chosen two sides of the block are $J=0$ and $\varphi=0$. We consider the block $\RRR$ defined as
\[
\RRR=\left\{ (\varphi,J):\, 0\leq \varphi\leq 2\nu\ii\tkk, \, 0\leq J\leq \tkk\right\} \quad\text{for some}\quad \tkk\ll 1.
\]
The choice of $\varphi=2\nu\ii\tkk$ is for the following reason. It implies that, for $\tkk$ small enough,
\[
 \wh\tS^2_J(2\nu\ii\tkk,J)\geq \tkk \qquad \text{for}\qquad J\in (0,\tkk).
\]
Then,
\[
 \RRR'= \wh\tS^2(\RRR)\cap \{0\leq J\leq \tkk\}
\]
is a ``rectangle'' bounded by the segments $J=0$, $J=\tkk$ and two others of the form $\varphi=h_i(J)$, $i=1,2$ which satisfy
\[
| h_2(J)-h_1(J)|\lesssim \tkk\qquad \text{for}\qquad 0\leq J\leq \tkk.
\]
Now, we show that for a suitable $M\gg 1$, $\RRR$ is an isolating block for $(\wh\tS^1)^M\circ \wh\tS^2$. To this end, we must analyze $(\wh\tS^1)^M(\RRR')$. Note that $M$ will depend on $\tkk$.

Consider the vertices of the rectangle $\RRR'$, $Z_{ij}$, $i,i=1,2$, with
\[
 Z_{i1}=(\varphi_{i1},0),\quad  Z_{i2}=(\varphi_{i2},\tkk)\qquad \text{with} \quad \varphi_{1j}\leq \varphi_{2j},\qquad j=1,2.
\]
Note that they satisfy
\[
 |\varphi_{ij}-\varphi_{i'j'}|\lesssim \tkk,\qquad i,j,i',j'=1,2.
\]
We define
\[
 Z_{ij}^M=(\varphi_{ij}^M,J_{ij}^M)=\left(\wh\tS^1\right)^M( Z_{ij}),
\]
which by Lemma \ref{lemma:BlockScatteringStraight} satisfy
\[
J_{i1}^M=0\qquad\text{ and }\qquad J_{i2}^M=\tkk+M\OO(\tkk^{3}).
\]
Choosing a suitable $M$ satisfying
\begin{equation}
\label{def:M}
\frac{1}{4b_1\tkk}\leq M\leq \frac{1}{2b_1\tkk},
\end{equation}
where $b_1$ is the constant introduced in Lemma \ref{lemma:BNFTorus}, we show that
\begin{equation}\label{def:IsoBlockEstimates}
 \varphi_{12}^M-\varphi_{21}^M\gtrsim \frac{1}{8} \qquad \text{and}\qquad -\frac{1}{16}\leq \varphi_{21}^M\leq 0.
\end{equation}
Indeed, for the first one, note that
\[
 \varphi_{12}^M-\varphi_{21}^M= \varphi_{12}-\varphi_{21}+M b_1\tkk+M\OO(\tkk^2)=\frac{1}{4}+\OO(\tkk)\gtrsim \frac{1}{8}.
\]
For the second estimate in \eqref{def:IsoBlockEstimates} it is enough to choose a suitable $M$ by using the particular form  of the first component of $\wh \tS_1$ in Lemma   \ref{lemma:BlockScatteringStraight} and the definition of $b_0$ in \eqref{def:b1}.

The estimates in \eqref{def:IsoBlockEstimates} implies that $\RRR$ is an isolating block.

Now, we compute $D\wh \tS=D[(\wh\tS^1)^ M\circ\wh\tS^2]$.
\begin{lemma}\label{def:DifferentialScattering}
For  $z=(\varphi,J)\in\RRR$, the matrix $D\wh \tS(z)$ is hyperbolic with eigenvalues $\la(z), \la(z)\ii\in\RR$
with
\[
\la_{\wh \tS}(z)\gtrsim \tkk
\]
Moreover, there exist two vectors fields $V_j:\RRR\to T\RRR$ of the form
\[
 V_1=\begin{pmatrix}1\\ 0 \end{pmatrix},\qquad  V_2=\begin{pmatrix}V_{21}(z)\\ 1 \end{pmatrix}\quad \text{with}\quad|V_{21}(z)| \lesssim\tkk,
\]
which satisfy, for $z\in\RRR$,
\[
\begin{aligned}
  D\wh \tS(z)V_1&=\la_{\wh \tS}(z)\left(V_1+\wh V_1(z)\right)\qquad &&\text{with} \qquad |\wh V_1(z)|\lesssim \tkk\\
 D\wh \tS(z)V_2(z)&=\la_{\wh \tS} (z)\ii\left(V_2(\wh \tS(z))+\wh V_2(z)\right)\qquad&& \text{with}  \qquad |\wh V_2(z)|\lesssim \tkk.
 \end{aligned}
\]
\end{lemma}

\begin{proof}
Note that
\[
 D\wh\tS^1=\begin{pmatrix} 1 &(b_1+\OO(\tkk))\\ 0& 1\end{pmatrix}+\OO\left(\tkk^{2}\right)
\]
and therefore, since $M\sim\tkk\ii$,
\[
 D(\wh\tS^1)^M=\begin{pmatrix} 1 &Mb_1+\OO(1)\\ 0& 1\end{pmatrix}+M\OO\left(\tkk^{2}\right)=\begin{pmatrix} 1 &Mb_1+\OO(1)\\ 0& 1\end{pmatrix}+\OO\left(\tkk\right)
\]
On the other hand, by Lemma \ref{lemma:BlockScatteringStraight}, \eqref{def:Scatt2Deriv} and taking into account that $\wh\tS^2$ is symplectic,
\[
 D\wh\tS^2=\begin{pmatrix} \eta&-\nu\ii \\ \nu& 0\end{pmatrix}+\OO(\tkk)
\]
for some  $\eta\in\RR$.
Then,
\begin{equation}\label{def:DiffScattering}
 D\wh\tS=\begin{pmatrix}M\nu b_1 &-\nu\ii\\ \nu&0\end{pmatrix}+\OO\left(1\right)
\end{equation}
Since this matrix is simplectic (the scattering maps are, see \cite{DelshamsLS08}), to prove hyperbolicity it is enough to check that the trace is bigger than 2. Indeed, for $(\varphi,J)\in \RRR$ and $\tkk>0$ small enough,
\[
\begin{split}
 \mathrm{tr} D\wh\tS&= M\nu b_1+\OO(1)\gtrsim \tkk\ii.
\end{split}
\]
The statements for $V_1$ are straightforward considering the form of $D\wh\tS$ in \eqref{def:DiffScattering}. To obtain those for $V_2$ it is enough to invert the matrix $D\wh\tS$ and compute the eigenvector of large eigenvalue.
\end{proof}

\section{A parabolic normal form: Proof of Theorem~\ref{prop:coordinatesatinfinity}}
\label{app:proofofprop:coordinatesatinfinity}

The Theorem will be an immediate consequence of the Lemmas~\ref{lem:normalformatinfinitystep1},
\ref{lem:2ndstepnormalform} and~\ref{lem:thirdstepofnormalformqkpk} below.

\subsection{First step of normal form}

The first step of the normal form transforms the ``center'' variables $z$ so that its dynamics becomes much closer to the identity in a neighborhood of infinity.
\begin{lemma}
\label{lem:normalformatinfinitystep1}
For any $N\ge 0$, there exist an analytic change of variables of the form
\[
\tilde z = z + Z(x,y,z,t),
\]
where $Z$ is a polynomial in $(x,y)$ of order at least $3$, even in $x$, such that equation~\eqref{def:systematinfinityold} becomes
in the new variables
\begin{equation}
\label{def:systematinfinity1}
\begin{aligned}
\dot x & = -  x^3 y(1+ B x^2-B y^2 + \OO_4(x,y)), \\
\dot y & = -  x^4(1+ (B-4A)x^2- B y^2 + \OO_4(x,y)), \\
\dot{ \tilde z} & =  x^6 \OO_N(x,y).
\end{aligned}
\end{equation}
where the $\OO_N(x,y)$ terms in the equations of $\dot{\tilde z}$ as well as the $\OO_4(x,y)$ are even functions of $x$. The constants
$A$ and $B$ were introduced in~\eqref{def:AiBconstantsdeltermedegrau6}.
\end{lemma}

\begin{proof}
Equation~\eqref{def:systematinfinityold} is in the claimed form for $N=0$. We proceed by induction. Assume the claim is true for $N$, that is,
that the equation is
\begin{equation}
\label{def:inductionNaveraging}
\begin{aligned}
\dot x & = -  x^3 y(1+ B x^2-B y^2 + \OO_4(x,y)), \\
\dot y & = -  x^4(1+ (B-4A)x^2- B y^2 + \OO_4(x,y)), \\
\dot{ z} & =  x^6 p_N(x,y,a,b,t)+ x^6 \OO_{N+1}(x,y),
\end{aligned}
\end{equation}
where $p_N$ is a homogeneous polynomial in $(x,y)$, even in $x$, of degree $N$ with coefficients depending on $(z,t)$.

First, with an averaging step, we can assume that $p_N$  does not depende on $t$. Indeed, given $U_N$ such that
$\partial_t U_{N}  = p_N -\tilde p_N$, where $\tilde p_N = \langle p_N \rangle_{t}$, the change
\[
\tilde z = z + x^6 U_{N}(x,y,z,t),
\]
transforms equation~\eqref{def:inductionNaveraging} into
\begin{equation}
\label{def:inductionN}
\begin{aligned}
\dot x & = -  x^3 y(1+ B x^2-B y^2 + \OO_4(x,y)), \\
\dot y & = -  x^4(1+ (B-4A)x^2- B y^2 + \OO_4(x,y)), \\
\dot{ \tilde z} & =  x^6  \tilde p_N(x,y,\tilde z) +  x^6 \OO_{N+1}(x,y).
\end{aligned}
\end{equation}
Clearly, since $p_N$ is even in $x$, so is $U_n$ and then the parity on $x$ of the equation remains the same.

Second, we consider the change
\[
\hat z = \tilde z + Z_{N+3}(x,y,\tilde z),
\]
where $Z_{N+3}$ is a homogeneous polynomial in $(x,y)$ of degree $N+3$, even in $x$, with coefficients depending on $z$. It transforms equation~\eqref{def:inductionN} into
\[
\begin{aligned}
\dot x & = -  x^3 y(1+ B x^2-B y^2 + \OO_4(x,y)), \\
\dot y & = -  x^4 (1+ (B-4A)x^2- B y^2 + \OO_4(x,y)), \\
\dot{ \hat z} & =  x^3  \left(x^3\tilde p_N(x,y,\hat z) - y\partial_x Z_{N+3}(x,y,\hat z) - x\partial_y Z_{N+3}(x,y,\hat z) \right)+  x^6 \OO_{N+1}(x,y).
\end{aligned}
\]
Clearly, since $Z_{N+3}$ is even in $x$, the parity in $x$ of the equation is preserved.

Since $x^3\tilde p_N(x,y,\hat z)$ is an odd polynomial in $x$, it is in the range
of the operator $L:C_{N+3} \mapsto y \partial_x C_{N+3} + x \partial_y C_{N+3}$, acting on homogeneous polynomials of degree $N+3$, even in $x$, the claim follows. Indeed, for any $j,\ell\ge 0$ such that $2j+\ell+1 = N+3$,
\[
L \left( \sum_{i=0}^j a_i x^{2(j-i)} y^{\ell+2i+1} \right) = x^{2j+1} y^{\ell},
\]
where
\[
 \begin{split}
 a_0&=\frac{1}{\ell+1},\\
 a_i&=(-1)^i\frac{(2j)\cdots (2j-2i+2)}{(\ell+1) \cdots (\ell+1+2i)},\qquad i\geq 1.
 \end{split}
\]
\end{proof}

\subsection{Second step of normal form: straightening the invariant manifolds of infinity}

Here we use the invariant manifolds of the periodic orbits $P_{z_0}$, $z_0 \in \RR^2$, to find coordinates in which
these manifolds are the coordinate planes.

Let $K \subset \RR^2$ be fixed. 
Given $\rho>0$, we denote by $K_{\CC}^\rho$, a neighborhood of $K$ in $\CC^2$ such that $\Re z \in K$ and $|\Im z | < \rho$, for all $z \in K_{\CC}^\rho$. 
Given $\delta, \sigma, \rho >0$, we consider the domain
\begin{multline}
\label{def:Urho}
U_{\delta, \rho} = \{(q,p,z,t)\in \CC \times \CC \times \CC^2 \times \CC \mid |\Im q |<\delta \Re q,\,
|\Im p |< \delta \Re p, \\
 \|(q,p)\|<\rho,\, z \in K_{\CC}^\rho, \, |\Im t| < \sigma\}.
\end{multline}

\begin{lemma}
\label{lem:2ndstepnormalform}
There exists a $\CCC^{\infty}$ change of variables of the form
$
(q,p,z) \mapsto (q,p,z) + \OO_2(q,p)$,
analytic in a domain of the form~\eqref{def:Urho}, that transforms equation~\eqref{def:systematinfinity1} into
\begin{equation}
\label{eq:infinitymanifoldsstraightened2}
\begin{aligned}
\dot q & = q ((q+p)^3+\OO_4(q,p)),  & \dot z & = q p \OO_{N+4}(q,p),\\
\dot p & = -p ((q+p)^3+\OO_4(q,p)), & \dot t & = 1.
\end{aligned}
\end{equation}
Equation~\eqref{eq:infinitymanifoldsstraightened2} is analytic in $U_{\delta,\rho}$,
for some $\delta,\rho>0$, and $\CCC^{\infty}$ at $(q,p)=(0,0)$.
\end{lemma}

\begin{proof}
We start by straightening the tangent directions of the topological saddle in~\eqref{def:systematinfinity1}. We introduce
\[
(q,p,z) = ((x-y)/2,(x+y)/2,z).
\]
In these variables, equation~\eqref{def:systematinfinity1} becomes
\begin{equation}
\label{def:straightenedsaddle}
\begin{aligned}
\dot q & = (q+p)^3\left(q   -2A q^3-6\left(A-\frac{2B}{3}\right)q^2 p - 6A q p^2-2Ap^3 + \OO((q+p)^5)\right), \\
\dot p & = - (q+p)^3\left(p   +2A q^3+6\left(A-\frac{2B}{3}\right)q^2 p + 6A q p^2+2Ap^3 + \OO((q+p)^5)\right) \\
\dot{ z} & =  (q+p)^6 \OO_{N}(q,p).
\end{aligned}
\end{equation}
This equation is analytic and $2\pi$-periodic in $t$ in a neighborhood $U$ of $\{q=p=0, \, z \in \RR^2, \, t \in \RR\}$ in $\CC \times \CC \times \CC^2 \times \CC$.

Let $K \subset \RR^2$ be a fixed compact set. Let $K_{\CC}$ be a complex neighborhood of $K$ in $\CC^2$, $\sigma>0$ and $\rho>0$ be such that $\{\|(q,p)\|< \rho\}\times K_{\CC} \times \{|\Im t | < \sigma \} \subset U$.

Let $z_0 \in K$. 
We claim that that the periodic orbit $P_{z_0}$ of~\eqref{def:straightenedsaddle} has invariant stable and unstable manifolds which admit parametrizations $\gamma^{s,u}(\cdot,t;z_0)$ of the form
\begin{equation}
\label{eq:paraminvmanperiodicorbits}
\begin{aligned}
(p,z) & = \gamma^{u}(q,t;z_0) = \left(\frac{A}{2} q^3+\OO(q^4), z_0+\OO(q^{N+3})\right), \\
(q,z) & = \gamma^{s}(p,t;z_0) = \left(\frac{A}{2} p^3 +\OO(p^4), z_0+\OO(p^{N+3})\right),
\end{aligned}
\end{equation}
where $\gamma^{s,u}$ are analytic in
\[
\begin{aligned}
V^{u}_{\delta,\rho} & = \{(q,t,z_0) \in \CC^3\mid |\Im q | < \delta |\Re q|, \, |q| < \rho, \, |\Im t| < \sigma, \,z_0 \in K_{\CC}^{\rho}\}, \\
V^{s}_{\delta,\rho} & = \{(p,t,z_0) \in \CC^3\mid |\Im p | < \delta |\Re p|, \, |p| < \rho, \, |\Im t| < \sigma, \,z_0 \in K_{\CC}^{\rho}\},
\end{aligned}
 \]
for some $\delta,\rho,\sigma >0$ and are of class $\CCC^{\infty}$ at $q=0$ and $p=0$,
respectively.

We prove the claim for $\gamma^u$, being the one for $\gamma^s$  analogous.

First we remark that, if $\gamma^u$ exists and is $\CCC^{\infty}$ at $q=0$, substituting in the vector field an imposing the invariance condition,
one obtains that it must be of the form given by~\eqref{eq:paraminvmanperiodicorbits}.

We change the sign of time in~\eqref{def:straightenedsaddle}.  With the introduction of the new variables
\[
\tilde z = \frac{1}{q+p}(z-z_0),
\]
equation~\eqref{def:straightenedsaddle} becomes
\begin{equation}
\label{eq:systemtoapplyinvariantmanifoldtheorem}
\begin{aligned}
\dot q & = -q  (q+p)^3  + \OO((q+p)^6), \\
\dot p & = p (q+p)^3  + \OO((q+p)^6), \\
\dot{ \tilde z} & =  \tilde z (q-p)(q+p)^2 + \tilde z \OO_5(q,p)+ \OO_{N+5}(q,p)
\end{aligned}
\end{equation}
where the vector field is analytic in $U$.

By Theorem~2.1 in \cite{BaldomaFdlLM07} (see also~\cite{GuardiaSMS17}, where the dependence on parameters is established), the periodic orbit $(q,p,\tilde z) = (0,0,0)$ of~\eqref{eq:systemtoapplyinvariantmanifoldtheorem},
which is parabolic, has an invariant stable manifold, parametrized by
\begin{equation}
\label{eq:tildegammaparabolicmanifold}
(q,p,\tilde z) = \tilde \gamma (u,t,z_0) = (u + \OO(u^2), \OO(u^2), \OO(u^2))
\end{equation}
with $\tilde \gamma$  analytic in $\wt V_{\delta,\rho} = \{|\Im u | < \delta |\Re u|, \, |u| < \rho, \, |\Im t| < \sigma, \,z_0 \in K_{\CC}^\rho\}$, for some $\delta,\rho,\sigma >0$ and  of class $\CCC^{\infty}$ at $u=0$. Since~\eqref{eq:systemtoapplyinvariantmanifoldtheorem} has the time reversed, $\tilde \gamma$ corresponds to the unstable manifold of~\eqref{eq:paraminvmanperiodicorbits}. 
We can invert the first component of~\eqref{eq:tildegammaparabolicmanifold}, $q = \pi_q \tilde \gamma(u,t,z_0)$, to obtain $u = q+U(q,t,z_0)$, defined and analytic in
\[
V^{u}_{\delta',\rho'}  = \{(q,t,z_0) \in \CC^3\mid |\Im q | < \delta' |\Re q|, \, |q| < \rho', \, |\Im t| < \sigma, \,z_0 \in K_{\CC}^{\rho'}\},
\]
for any $0< \delta' < \delta$ and some $0 < \rho' < \rho$, and  is $\CCC^{\infty}$ in $V^{u}_{\delta',\rho'} \cup \{0\}$\footnote{This claim can be proven as follows. $U$ is trivially $\CCC^\infty$ at $u=0$. Observe that the function $U$ is the solution of the fixed point equation $U= F(U)$, with $F(U)(q,t,z_0)= - \varphi(q+U(q,t,z_0),t,z_0)$. Since $\varphi(u,t,z_0) = \OO(u^2)$ and is analytic in $\wt V_{\delta_0,\rho_0}$, we have that, for any $0 < \delta' < \delta$ and $0<\rho'<\rho$, $\partial_u \varphi(u,t,z_0) = \OO(u)$ in $\wt V_{\delta',\rho'}$. Using this fact, it is immediate to see that $F$ is a contraction  with the norm $\|\varphi\|_2 = \sup_{(q,t,z_0) \in V^u_{\delta',\rho''}}|q^{-2} U(q,t,z)|$, if $\rho''$ is small enough.}.
Hence, the stable manifold of $(q,p,\tilde z) = (0,0,0)$ of~\eqref{eq:systemtoapplyinvariantmanifoldtheorem}, and therefore the unstable one of ~\eqref{eq:paraminvmanperiodicorbits},  can be written as a graph as
\[
(p,\tilde z) = \hat \gamma(q,t,z_0) = \pi_{p,\tilde z} \tilde \gamma(q+U(q,t,z_0),t,z_0) = \OO(q^2).
\]
We claim that $\pi_{\tilde z} \hat \gamma (q,t,z_0) = \OO(q^{N+2})$. 

Indeed, assume that $\pi_{\tilde z} \hat \gamma (q,t,z_0) = a_L q^{L}+\OO(q^{L+1})$. It is clear that $L\ge 2$. But, denoting by $X$ the vector field in~\eqref{eq:systemtoapplyinvariantmanifoldtheorem}, since the graph of $\hat \gamma$ is invariant,
it satisfies
\begin{multline*}
-L a_L q^{L+3} +\OO(q^{L+4}) = \frac{\partial}{\partial q} \pi_{\tilde z} \hat \gamma (q,t,z_0) X_{q} (q,\tilde \gamma(q),t) \\ =  X_{\tilde z} (q,\tilde \gamma(q),t) = a_L q^{L+3} + a_L \OO(q^{L+5}) + \OO(q^{N+5}),
\end{multline*}
from which the claim follows. 

Going back to $z = z_0+(q+p) \tilde z$ we obtain
that $(p,z) = \gamma^u(q,t,z_0) = (\OO(q^2),z_0+\OO(q^{N+3})$ is a parametrization of the unstable manifold of the periodic orbit $z=z_0$. Substituting this expression into~\eqref{def:straightenedsaddle}, one obtains that $\OO(q^2) = \frac{A}{2} q^3+\OO(q^4)$, which proves the claim for $\gamma^u$ in~\eqref{eq:paraminvmanperiodicorbits}.

Now we straighten the invariant manifolds using the functions $\gamma^u$ and $\gamma^s$ in~\eqref{eq:paraminvmanperiodicorbits}. We claim that there exist variables $(q,p,z)$ in which equation~\eqref{def:straightenedsaddle} becomes
\begin{equation}
\begin{aligned}
\dot q & = q ((q+p)^3+\OO_4), \\
\dot p & = -p ((q+p)^3+\OO_4), \\
\dot z & = q p\OO_{N+4}(q,p),
\end{aligned}
\end{equation}
being defined and analytic in a domain of the form~\eqref{def:Urho} and is $\CCC^{\infty}$ at $(q,p) = (0,0)$. We will apply two consecutives changes of variables, each of them straightening one invariant manifold.

Let $z_0^u(q,z,t)= z + \OO_{N+3}(q)$ be such that
\[
z = \pi_z \gamma^u(q,t;z_0^u(q,z,t)),
\]
which is also analytic on $V_{\delta,\rho}^u$, $\CCC^{\infty}$ at $q=0$. 
We define the new variables $(\tilde q, \tilde p, \tilde z) =
\Psi_1^{-1}(q,p,z,t)$ by
\begin{equation}
\label{def:primer_canvi}
\begin{aligned}
\tilde q & = q, \\
\tilde p & = p - \pi_p \gamma^u(q,t;z_0^u(q,z,t)) = p - \frac{A}{2} q^3 + \OO_4(q), \\
\tilde z & = z_0^u(q,z,t) = z+ \OO_{N+3}(q).
\end{aligned}
\end{equation}
Again, it is easy to see that the map
\[
\Psi_1(\tilde q,\tilde p,\tilde z, t) =
\begin{pmatrix}
\tilde q \\
\tilde p + \psi_1^{\tilde p}(\tilde q, \tilde z, t) \\
\tilde z + \psi_1^{\tilde z}(\tilde q, \tilde z, t)
\end{pmatrix}=
\begin{pmatrix}
\tilde q \\
\tilde p + \frac{A}{2} \tilde{q}^3 + \OO_4(\tilde{q}) \\
\tilde z + \OO_{N+3}(\tilde q)
\end{pmatrix}
\]
is analytic in the domain
\begin{equation}
\label{def:Wdeltarhou}
W_{\delta,\rho}^u = \{(\tilde q,\tilde p,\tilde z,t)\in \CC^4 \mid |\tilde q|,|\tilde p| \le \rho, \; \Re \tilde p \ge 0,\; |\Im \tilde q | < \delta \Re \tilde q, \; \tilde z \in K_{\CC}^{\rho},\; t \in \TT_\sigma \}.
\end{equation}
for some $\gamma, \rho,\sigma >0$. 
Moreover, since $\Psi_{1 \mid \tilde q = 0}$ is the identity, it is  analytic in $\{|\tilde p | < \rho,\; \tilde z \in K_{\CC}^{\rho},\; t \in \TT_\sigma \}$. Observe that we are considering $\tilde q$ with $\Re \tilde q \ge 0$.

Observe that,  in these variables, the unstable manifold is given by $\tilde p = 0$.
We claim that, in these variables, equation~\eqref{def:straightenedsaddle} becomes
\begin{equation}
\begin{aligned}
\dot{\tilde q} & = \tilde q (\tilde q+\tilde p)^3+\OO_4, \\
\dot{\tilde p} & = -\tilde p ((\tilde q+\tilde p)^3+\OO_4), \\
\dot{\tilde z} & = \tilde p \OO_{N+5}(\tilde q,\tilde p),
\end{aligned}
\end{equation}
Indeed, the claim for $\dot{\tilde q}$ is an immediate substitution. To see the claim for $\tilde z$, we observe
that, on the unstable manifold, $z_0$ is constant. Hence
\[
\dot{\tilde z} = \frac{d}{dt} z_0^u(q,z,t) = \partial_q z_0^u(q,z,t) \dot q +\partial_z z_0^u(q,z,t) \dot z + \partial_t z_0^u(q,z,t) = \OO_{N+6}(q,p)+ \partial_t z_0^u(q,z,t).
\]
Since $\dot{\tilde z}_{\mid \tilde p = 0} = 0$, $\partial_t z_0^u(q,z,t) = \OO_{N+6}(q)$ and the claim follows. Then the claim for $\dot{\tilde p}$ is an immediate substitution.

Observe that the composition $\gamma^s \circ \Psi_1$, where $\gamma^s(p,t;z_0)$ is the function in~\eqref{eq:paraminvmanperiodicorbits}, is well defined and analytic in $W^u_{\delta',\rho'}$, for $0< \delta' < \delta/3$ and $\rho'$ small enough. Indeed, if $(\tilde q, \tilde p, \tilde z, t) \in W^u_{\delta',\rho'}$, using the the function $A$ in~\eqref{def:AiBconstantsdeltermedegrau6} is real analytic and positive for real values of $z \in K_{\CC}^{\rho}$,
\[
\pi_{\tilde p} \Psi_1(\tilde q, 0, \tilde z, t) = \frac{A}{2} \tilde q^3 + \OO(\tilde q^4) \in V^{s}_{\delta,\rho},
\]
and $\Re \pi_{\tilde p} \Psi_1(\tilde q, 0, \tilde z, t) >0$, which implies that, for any $\tilde p \in V^{s}_{\delta,\rho}$ with $\Re \tilde p \ge 0$, $\tilde p + \frac{A}{2} \tilde q^3 + \OO(\tilde q^4) \in V^{s}_{\delta,\rho}$.

It can be seen with the same type of fixed point argument that the stable manifold $(q,z) = \gamma^s(p, t, z_0)$ in~\eqref{eq:paraminvmanperiodicorbits} can be written in the variables $(\tilde q, \tilde p, \tilde z)$ as
\begin{equation}
\label{def:stablemanifoldintildevariables}
(\tilde q,\tilde z)  = \tilde \gamma^{s}(\tilde p,t;z_0) = \left(\frac{A}{2} \tilde p^3 +\OO(\tilde p^4), z_0+\OO(\tilde p^{N+3})\right),
\end{equation}
analytic in
\[
V^{s}_{\delta'',\rho''}  = \{(\tilde p,t,z_0) \in \CC^3\mid |\Im \tilde p | < \delta'' |\Re \tilde p|, \, |\tilde p| < \rho'', \, |\Im t| < \sigma, \,z_0 \in K_{\CC}^{\rho''}\},
\]
for some $0 <\delta'' < \delta'$, $0 < \rho'' < \rho'$.

Now, repeating the same arguments, we obtain a change of variables $(\hat q, \hat p, \hat z) =
\Psi_2^{-1}(\tilde q,\tilde p,\tilde z,t)$ such that
\[
\Psi_2(\hat q,\hat p,\hat z, t) =
\begin{pmatrix}
\hat q + \psi_2^{\hat q}(\hat p, \hat z, t)\\
\hat p  \\
\hat z + \psi_2^{\hat z}(\hat p, \hat z, t)
\end{pmatrix}=
\begin{pmatrix}
\hat q + \frac{A}{2} \hat{p}^3 + \OO_4(\hat{p})\\
\hat p  \\
\hat z + \OO_{N+3}(\hat p)
\end{pmatrix}
\]
is analytic in the domain
\begin{equation}
\label{def:Wdeltarhos}
W_{\delta,\rho}^s = \{(\hat q,\hat p,\hat z,t)\in \CC^4 \mid |\hat q|,|\hat p| \le \rho, \; \Re \hat p \ge 0,\; |\Im \hat p | < \delta \Re \hat p, \; \hat z \in K_{\CC}^{\rho},\; t \in \TT_\sigma \}.
\end{equation}
for some $\gamma, \rho,\sigma >0$ (smaller than $\delta''$ and $\rho''$). Moreover, $\Psi_{2 \mid \tilde q = 0}$ is analytic in $\{|\tilde p | < \rho,\; \tilde z \in K_{\CC}^{\rho},\; t \in \TT_\sigma \}$.
This change is the identity on the unstable manifold.
The previous arguments show that equation~\eqref{def:systematinfinity1} in the $(\hat q, \hat p, \hat z)$ variables has the form~\eqref{eq:infinitymanifoldsstraightened2}.

The change of variables $ \Psi_1 \circ \Psi_2$
is then analytic in $U_{\delta, \rho}$, defined in~\eqref{def:Urho}, for some $\delta,\rho >0$.

\end{proof}

\subsection{Third step of normal form}

Next lemma provides a better control
of the dynamics of $z=(a,b)$ close to $(q,p) = (0,0)$.

\begin{lemma}
\label{lem:thirdstepofnormalformqkpk}
Let $N \ge 2$ be fixed. For any $1 \le k < (N+1)/2$, there exists a  change of variables
\[
\Phi( q,  p,  z ,t) = (q,p,z + \OO_{N+3}(q,p),t),
\]
analytic in a domain of the form~\eqref{def:Urho} and of class $C^{N+2}$ at $(q,p) = (0,0)$ such that equation~\eqref{def:systematinfinityold} becomes, in the new
variables,
\begin{equation}
\label{eq:infinitymanifoldsstraightened3}
\begin{aligned}
\dot q & = q ((q+p)^3+\OO_4), \\
\dot p & = -p ((q+p)^3+\OO_4), \\
\dot z & = q^k p^k \OO_{N+6-2k}, \\
\dot t & = 1,
\end{aligned}
\end{equation}
where $z = (a,b)$.
\end{lemma}

\begin{proof}
We prove the claim by induction. The case $k=1$ is given by Lemma~\ref{lem:2ndstepnormalform}.
We observe that, since equation~\eqref{eq:infinitymanifoldsstraightened2} is analytic in $U_{\delta,\rho}$ and $\CCC^{\infty}$ at $(0,0)$, the $z$ component of the vector field, $qp \OO_{N+4}(q,p)$, can be written as
\[
qp \OO_{N+4}(q,p) = \sum_{0 \le j < \lfloor \frac{N+2}{2} \rfloor}
(qp)^{j+1} (Q_{N+4-2j}(q,z,t)+P_{N+4-2j}(p,z,t)) + (qp)^{\lfloor \frac{N+4}{2}\rfloor} \OO_2(q,p),
\]
where the functions $Q_{N+4-2j}(u,z,t), P_{N+4-2j}(u,z,t) = \OO(u^{N+4-2j})$, $0 \le j < \lfloor \frac{N+4}{2} \rfloor$,  are analytic in a domain of the form
\[
V_{\delta,\rho}  = \{(u,z,t) \in \CC^3\mid |\Im u | < \delta \Re u, \, |u| < \rho, \, |\Im t| < \sigma, \,z \in K_{\CC}^{\rho}\},
\]
for some $\delta, \rho >0$.

Assume that the equation is in the form
\begin{equation}
\label{eq:infinitymanifoldsstraightenedinduction}
\begin{aligned}
\dot q & = q ((q+p)^3+\OO_4), \\
\dot p & = -p ((q+p)^3+\OO_4), \\
\dot z & = q^k p^k \OO_{M}(q,p),
\end{aligned}
\end{equation}
where $M=N+6-2k>5$ and is analytic in $U_{\delta,\rho}$.
We write the terms $\OO_{M}(q,p)= R_{M}(q,p,z,t)$ as
\[
R_{M}(q,p,z,t) = Q_{M}(q,t,z)+P_{M}(p,t,z)+\wt R_{M}(q,p,t,z),
\]
where
\[
Q_{M}(q,t,z) = R_{M}(q,0,z,t) = \OO_{M}(q), \quad P_{M}(p,t,z) = R_{M}(0,p,z,t) = \OO_{M}(p).
\]
It is clear that $\wt R_{M}(q,p,t,z) =  qp \OO_{M-2}(q,p)$.

We perform a change a variables to get rid of the term $Q_{M}$ and $P_{M}$ of the form
\begin{equation}
\label{eq:lastchangeofnormalform}
\tilde z = z + q^k p^k(A(q,z,t)+B(p,z,t)).
\end{equation}
The equation for $\tilde z$ becomes
\[
\dot{\tilde z} = q^{k+1} p^{k+1} \OO_{M-2}(q,p)
\]
if
\begin{align}
\label{eq:homologicaleqforqjpjA}
Q_{M} +  h_4 A + f \partial_q A+ \partial_t A & = 0,\\
\label{eq:homologicaleqforqjpjB}
P_{M} + \tilde h_4 B + \tilde f \partial_q B+ \partial_t B & = 0,
\end{align}
where, from~\eqref{eq:infinitymanifoldsstraightenedinduction},
\begin{equation}
\label{def:h4}
k(\dot q p + q \dot p) = qp (h_4(q,z,t)+\tilde h_4(p,z,t)+ qp\hat h_2(q,p,z,t))
\end{equation}
with $h_4(q,z,t) =\OO_4(q)$, $\tilde h_4(p,z,t) = \OO_4(p)$ and $\hat h_2(q,p,z,t) = \OO_2(q,p)$,
and
\begin{equation}
\label{def:ftildef}
\begin{aligned}
f(q,z,t) & = \dot q_{\mid p =0} = q^4 + \OO_5(q), \\
\tilde f(p,z,t) & = \dot p_{\mid q =0} = -p^4 + \OO_5(p).
\end{aligned}
\end{equation}
The functions $f,  h_4, Q_{M}$ and  $\tilde f,  \tilde h_4, P_{M}$  are defined respectively in the sectors
\[
\begin{aligned}
V & = \{|\Im q | < \delta \Re q, \, |q| < \rho, \, |\Im t| < \sigma, \,z \in K_{\CC}^{\rho}\}, \\
\wt V & = \{|\Im p | < \delta \Re p, \, |p| < \rho, \, |\Im t| < \sigma, \,z \in K_{\CC}^{\rho}\}.
\end{aligned}
\]
\begin{lemma}
\label{lem:homologicaleqforqjpj}
If $\rho$ is small enough, equations~\eqref{eq:homologicaleqforqjpjA} and~\eqref{eq:homologicaleqforqjpjB}  admit analytic solutions $A$, $B$, defined in $V$ and $\wt V$, such that
\[
\sup_{(q,z,t) \in V}|q^{-(M-3)}A(q,z,t)|, \sup_{(p,z,t) \in \wt V}|p^{-(M-3)}B(p,z,t)| < \infty.
\]
respectively,
\end{lemma}

\begin{proof}[Proof of Lemma~\ref{lem:homologicaleqforqjpj}]
We prove the claim for~\eqref{eq:homologicaleqforqjpjA}, being the proof for~\eqref{eq:homologicaleqforqjpjB} analogous.


We consider the change of variables $q = q(u)$, where $q(u)$ satisfies
\[
\frac{dq}{du} = f(q,z,t).
\]
Since $f(q,z,t) = q^4 + \OO_5(q)$, we have that
\[
q(u) = - \frac{1}{(3 u)^{1/3}} ( 1+ \OO(u^{-1/3})).
\]
It transforms~\eqref{eq:homologicaleqforqjpjA} into
\begin{equation}
\label{eq:homologicalequationuvariable}
\LL  A = -\wh Q_M - \hat h_4 A ,
\end{equation}
where
\begin{equation}
\LL A = \partial_u A + \partial_t A
\end{equation}
and $\wh Q_M(u, z,t) = Q_M(q(u),z,t)$ and $\hat h_4(u,z,t) = h_4(q(u),z,t)$ are defined in
\[
\wh V = \{ \Re u < -1/(3 \rho^3),\, |\Im u | < 3 \arctan \gamma |\Re u |,\, |\Im t| < \sigma, \,z \in K_{\CC}\}.
\]

We introduce the Banach spaces
\[
\XX_\kappa = \{ \alpha: \wh V \to \CC \mid \|\alpha\|_\kappa < \infty\}
\]
where, writing $\alpha(u,z,t) = \sum_{\ell \in \ZZ} \alpha^{[\ell ]} e^{i\ell t}$,
\[
\|\alpha\|_\kappa = \sum_{\ell \in \ZZ} \sup_{(u,z)\in U}| u^{\kappa} \alpha^{[\ell]}(u,z)| e^{-\sigma |\ell|},
\]
and
\[
U = \{ \Re u < - 1/(3 \rho^3),\, |\Im u | < 3 \arctan \gamma |\Re u |, \,z \in K_{\CC}\}.
\]
The following properties of the spaces are $\XX_\kappa$ are immediate:
\begin{enumerate}
\item
if $\alpha \in \XX_\kappa$ and $\tilde \alpha \in \XX_{\tilde \kappa}$, $\alpha \tilde \alpha \in \XX_{\kappa +\tilde \kappa}$ and
\[
\|\alpha \tilde \alpha\|_{\kappa + \tilde \kappa} \le \|\alpha\|_\kappa \|\tilde \alpha\|_{\tilde \kappa},
\]
\item
if $\alpha \in \XX_\kappa$, for all $\tilde \kappa \le \kappa$, $\alpha \in \XX_{\tilde \kappa}$ and
\[
\|\alpha\|_{\tilde \kappa} \le 3^{\kappa-\tilde \kappa} \rho^{3(\kappa-\tilde \kappa)}\|\alpha\|_\kappa.
\]
\end{enumerate}

It is clear that $\wh Q_M \in \XX_{M/3}$ and $\hat h_4 \in \XX_{4/3}$.
The following lemma, whose proof omit, is a simplified version of Lemma~\ref{lemma:Operator:1}.

\begin{lemma}
\label{lem:inverseofLappendix}
Let $\beta \in \XX_\kappa$, with $\kappa >1$. 
The equation $\LL \alpha = \beta$ has a solution
$\GG(\beta) \in \XX_{k-1}$ with $\|\GG(\beta)\|_{k-1} \le \|\beta\|_k$.
\end{lemma}

Using Lemma~\ref{lem:inverseofLappendix},  we can rewrite equation~\eqref{eq:homologicalequationuvariable} as the fixed point equation
\[
A = \GG(-\wh Q_M - \hat h_4 A).
\]
But, since $M>5$, the right hand side above is a contraction in $\XX_{M/3-1}$ if $\rho$ is small enough, since, by 1. and~2.,
\begin{multline*}
\| \GG(-\wh Q_M - \hat h_4 A)- \GG(-\wh Q_M - \hat h_4 \tilde A)\|_{M/3-1} \le \|  \hat h_4 (A- \tilde A)\|_{M/3} \\
\le \|  \hat h_4 \|_1 \|(A- \tilde A)\|_{M/3-1} \le 3^{1/3} \rho \|  \hat h_4 \|_{4/3}\|(A- \tilde A)\|_{M/3-1}.
\end{multline*}
Finally, since $u(q) = 1/(3q^3)+\OO(q^{-2})$, the lemma follows.
\end{proof}

Now we can finish the proof of Lemma~\ref{lem:thirdstepofnormalformqkpk}. With the choice of $A$ and $B$ given by
Lemma~\ref{lem:homologicaleqforqjpj}, the change of variables~\eqref{eq:lastchangeofnormalform} transforms~\eqref{eq:infinitymanifoldsstraightenedinduction} into a equation of the same form with $k$ replaced by $k+1$ and $M$ by $M-2$. Notice that this change of variables, since it is analytic in a sectorial domain of the form~\eqref{def:Urho}
and is of order $\OO_{M+2k-3}(q,p) = \OO_{N+3}(q,p)$, it is of class $C^{N-2}$ at $(q,p)=(0,0)$. The composition of the vector field with the change is well defined in a sectorial domain of the form~\eqref{def:Urho}, with $\gamma$ replaced with any $0 < \gamma' < \gamma$, if $\rho$ is small enough. Hence the lemma is proven.
\end{proof}

\section{The parabolic Lambda Lemma: Proof of Theorem \ref{prop:lambdalemma}}\label{sec:LambdaLemmaTechnical}

The proof of Theorem~\ref{prop:lambdalemma} will be a consequence of the following  technical lemmas and is deferred to the end of this section. To simplify the notation in this section we denote $\Psi=\Psi_\loc$.

Let $K>0$ be such that the terms $\OO_k$ in~\eqref{eq:infinitymanifoldsstraightened} satisfy
\[
\|\OO_k\| \le K \|(q,p)\|^k, \qquad (q,p,z,t) \in V_{\rho} \times W \times \TT.
\]
This bound is also true  for $(q,p,z,t) \in B_{\rho} \times W \times \TT$,
being $B_\rho = \{(q,p)\mid |q|,|p| < \rho\}$.

We express system~\eqref{eq:infinitymanifoldsstraightened} in a new time in which the topological saddle is a true saddle. Indeed, since the solutions of~\eqref{eq:infinitymanifoldsstraightened} with initial condition $(q_0,p_0,z_0,t_0) \in B_{\rho} \times W \times \TT$ with $q_0,p_0 \ge 0$ satisfy $q+p >0$ while in $B_{\rho} \times W \times \TT$, we can write
equations~\eqref{eq:infinitymanifoldsstraightened} in the new time~$s$ such that $dt/ds = (q+p)^{-3}$. System~\eqref{eq:infinitymanifoldsstraightened} becomes
\begin{equation}
\label{eq:infinitymanifoldsstraightenedintimes}
\begin{aligned}
q' & = q (1+\OO_1(q,p)), \\
p' & = -p (1+\OO_1(q,p)), \\
z' & = q^N p^N\OO_1(q,p), \\
t' & = \frac{1}{(q+p)^3},
\end{aligned}
\end{equation}
where $\mbox{}'$ denotes $d/ds$. 
The $\OO_1(q,p)$ terms are uniformly bounded in terms of $(q,p)$ in $V_{\rho} \times W \times \TT$.

Given $K \subset W$, for $w_0 = (q_0,p_0,z_0,t_0) \in V_{\rho} \times K \times \TT$, we define
\begin{equation}
\label{def:sw0}
s_{w_0} = \sup \,\{s>0\mid w(\tilde s) \in V_{\rho} \times W\times \TT,\; \forall \tilde s \in [0,s) \},
\end{equation}
where  $w$ is the solution of~\eqref{eq:infinitymanifoldsstraightened} with initial condition $w_0$.

Next lemma implies Item 1 of Theorem~\ref{prop:lambdalemma}.

\begin{lemma}
\label{lem:topological_lambda_lemma}
Let $K \subset W$ be a compact set. 
There exists $\rho$ and $C>1$ such that the solution $w= (q,p,z,t)$ of~\eqref{eq:infinitymanifoldsstraightenedintimes} with initial condition $w_0 = (q_0,p_0,z_0,t_0) \in V_{\rho} \times K \times \TT$ satisfies
\[
\log \left( \left( \frac{\rho}{q_0}\right)^{\frac{1}{1+C\rho}} \right) \le s_{w_0} \le \log \left( \left( \frac{\rho}{q_0}\right)^{\frac{1}{1-C\rho}} \right).
\]
%
Moreover, for any $0 < a \leq \rho $
and $0 < \delta < a/2$,
the Poincar\'e map 
\[
\Psi: \Lambda_{a,\delta}^- (K) \to \Lambda_{a,\delta^{1-C a}}^+  (W),
\]
where the sets $\Lambda_{a,\delta}^\pm (K)$ are defined in \eqref{def:sectionsLambdaadeltapm}, is well defined and, if $w = (q_0,a,z_0,t_0) \in \Lambda_{a,\delta}^- (K)$ and $\Psi(w) = (a,p_1,z_1,t_1)$, then
\begin{equation}
\label{bounds:C0bounds}
\begin{aligned}
q_0^{1+C a} \le  p_1 & \le q_0^{1-C a}, \\
 |z_1 - z_0 |  \le &  \frac{1}{2N} a^{N(1+3C a)} q_0^{N(1-3C a)}, \\
\wt C_1 q_0^{-3(1-C a)/2} \le  t_1 -t_0 & \le \wt C_2 q_0^{-3(1+C a)/2}
\end{aligned}
\end{equation}
for some constants $\wt C_1, \wt C_2 >0$ depending only on $a$.
\end{lemma}

\begin{proof}
 Let $w_0 = (q_0,p_0,z_0,t_0) \in V_{\rho} \times K \times \TT$ and $w= (q,p,z,t)$ the solution of~\eqref{eq:infinitymanifoldsstraightenedintimes} with initial condition $w_0$ at $s=0$. Since $q_0,p_0 >0$, while
$w \in V_{\rho} \times W \times \TT$, $q,p >0$. Hence, there exists $C>0$, depending only on $\rho$ and $W$, such that
\begin{equation}
\label{bound:infinitymanifoldsstraightenedintimes}
\begin{aligned}
(1-C \rho)q & \le q'  \le   (1+C \rho)q, \\
 -(1+C\rho)p & \le p'  \le  -(1-C\rho)p , \\
- C \rho q^N p^N & \le z_i'  \le  C \rho q^N p^N, \qquad i=1,2.
\end{aligned}
\end{equation}
Since $p$ is decreasing and $\{p=0\}$ is invariant, $w$ can only leave $V_{\rho} \times W \times \TT$ if $q=\rho$ or $z$ leaves~$W$. From~\eqref{bound:infinitymanifoldsstraightenedintimes}, we have that for all $s$ such $w(\tilde s) \in V_{\rho} \times W \times \TT$ for all $\tilde s \in [0,s)$,
\begin{equation}
\label{bound:solutions_integrated_qp}
\begin{aligned}
q_0 e^{(1-C \rho)s} & \le q(s)   \le   q_0 e^{(1+C \rho)s}, \\
p_0 e^{-(1+C \rho)s} & \le p(s)  \le  p_0 e^{-(1-C \rho)s}
\end{aligned}
\end{equation}
and
\begin{equation}
\label{bound:solutions_integrated_z}
| z_i(s)-z_i(0)|  \le  \frac{1}{2N} q_0^N p_0^N \left(e^{2N C \rho s}-1\right), \qquad i=1,2.
\end{equation}
In particular, the time $s_{q_0,q}$ to reach $q$ from $q_0$ is bounded by
\begin{equation}
\label{bound:timesfromq0toq}
\log \left( \frac{q}{q_0}\right)^{\frac{1}{1+C\rho}} \le s_{q_0,q} \le \log \left( \frac{q}{q_0}\right)^{\frac{1}{1-C\rho}} ,
\end{equation}
but, up to this time, for $i=1,2$, since $0 < q_0,p_0,q < \rho<1$,
%
\begin{equation}
\label{bound:zdifferenceattimesqq0}
|z_i(s_{q_0,q})-z_i(0)| \le \frac{1}{2N} q_0^N p_0^N  \left( \frac{q}{q_0}\right)^{\frac{2N C \rho}{1-C\rho}}
= \frac{1}{2N} q_0^{N\frac{1-3C \rho}{1-C\rho}} p_0^N  q^{\frac{2 N C \rho}{1-C\rho}} \le
\frac{1}{2N} \rho^{2N}.
\end{equation}
Hence, taking $\rho$ small enough depending on $K$, the solution through $w_0$ remains in $V_{\rho} \times W \times \TT$
for all $s$ such that 
\[
0 < s<s_{w_0} < \log  \left( \frac{\rho}{q_0}\right)^{\frac{1}{1-C\rho}}.
\]
Moreover, \eqref{bound:zdifferenceattimesqq0} ensures that the solution leaves $V_{\rho} \times W \times \TT$ through $q=\rr$.

The time $s_{q_0,q}$ to reach $q$ from $q_0$ is bounded by below by
\begin{equation}
\label{bound:timetoreachrhobybelow}
s_{q_0,q} \ge \log  \left( \frac{q}{q_0}\right)^{\frac{1}{1+C\rho}}, 
\end{equation}
putting $q=\rho$ in \eqref{bound:timetoreachrhobybelow} we obtain the lower bound for  $s_{w_0}$. 

Consequently, for any $0 < \delta < a < \rho$, if $w_0 \in \Lambda_{a,\delta}^- (K)$, the solution through $w_0$
satisfies $q = a$ at a time $s_*$ bounded by
\begin{equation}
\label{bound:timetoreachrhobybelow_from_section_to_section}
\log  \left( \frac{a}{q_0}\right)^{\frac{1}{1+C a}} \le s_* \le \log  \left( \frac{a}{q_0}\right)^{\frac{1}{1-C a}}.
\end{equation}
From~\eqref{bound:solutions_integrated_qp} with an analogous argument replacing $\rr$ by $a$, this solution satisfies
\[
 q_0^{1+2C a} \le p_0 \left( \frac{a}{q_0}\right)^{\frac{-1-C a}{1-C a}} \le p_0 e^{-(1+ C a)s_*} \le p(s_*) \le p_0 e^{-(1-C a)s_*} \le p_0 \left( \frac{a}{q_0}\right)^{\frac{-1+C a}{1+C a}}
\le  q_0^{1-2C a}
\]
and, for $i=1,2$,
\[
| z_i(s_*)-z_i(0)|  
\le \frac{1}{2N} a^{N\frac{1+Ca}{1-Ca}} q_0^{N(1-3C a)}.
\]
It only remains to estimate $t(s_*)-t_0$. Since
\[
t(s_*)-t_0 = \int_0^{s_*} \frac{1}{(q+p)^3}ds,
\]
$0 < q_0 < \delta$, $p_0 = a$ and $2 \delta < a$, we have that
\[
\begin{aligned}
t(s_*)-t_0 & \ge \int_0^{s_*} \frac{1}{(q_0 e^{(1+C a)s}+ p_0 e^{-(1-C a)s})^3}ds \\
& = \int_0^{s_*} \frac{e^{-3 C a s}}{(q_0 e^{s}+ p_0 e^{-s})^3}ds \\
& \ge \left( \frac{q_0}{a}\right)^{\frac{3C a }{1+C a}} \int_0^{s_*} \frac{1}{(q_0 e^{s}+ p_0 e^{-s})^3}ds \\
& = \left( \frac{q_0}{a}\right)^{\frac{3C a }{1+C a}} \frac{1}{(q_0p_0)^{3/2}} \int_{-\log(p_0/q_0)^{1/2}}^{s_*-\log(p_0/q_0)^{1/2}}
\frac{1}{( e^{\sigma}+  e^{-\sigma})^3}d\sigma \\
& \ge \frac{1}{a^{3(1+C a)/2}} \frac{1}{q_0^{3(1-C a)/2}} \int_{-(\log 2)/2}^{(1-Ca)(\log 2)/(2(1+C)a)}
\frac{1}{( e^{\sigma}+  e^{-\sigma})^3}d\sigma .
\end{aligned}
\]
With an analogous argument one obtains the upper bound of $t(s_*)-t_0$.
\end{proof}

Let $w = (q,p,z,t)$ be a solution of~\eqref{eq:infinitymanifoldsstraightened} with initial condition $w_0
\in V_{\rho} \times K \times \TT$. 
We define
\begin{equation}
\label{def:tau}
\tau = p/q
\end{equation}
and, from now on, abusing notation we denote $\OO_i=\OO_i(q,p)$.

Clearly, $0<\tau(s) <\infty$, for all $s$ such that $w\in V_{\rho} \times K \times \TT$. It is immediate from~\eqref{eq:infinitymanifoldsstraightenedintimes} that
\begin{equation}
\label{def:tauedo}
\frac{d \tau}{d s} = -(2+\OO_1) \tau.
\end{equation}
The variational equations around a solution of system~\eqref{eq:infinitymanifoldsstraightened} are
\begin{equation}
\label{eq:variationalequationsatinfinity}
\begin{pmatrix}
\dot Q \\ \dot P \\ \dot Z \\ \dot T
\end{pmatrix}
=
\begin{pmatrix}
(q+p)^2 (4q+p+\OO_2) & (q+p)^2 q (3+\OO_1) & q \OO_4 &  q \OO_4\\
-(q+p)^2p(3+\OO_1) & -(q+p)^2 (q+4p+\OO_2) & p \OO_4  & p \OO_4\\
q^{N-1}p^N \OO_4 & q^N p^{N-1} \OO_4 & q^N p^N \OO_4 & q^N p^N \OO_4 \\
0 & 0 & 0& 0
\end{pmatrix}
\begin{pmatrix}
 Q \\ P \\  Z \\ T
\end{pmatrix},
\end{equation}
where $Z = (Z_1,Z_2)$.

In order to prove item 2) of Theorem~\ref{prop:lambdalemma}, we will have to study the behaviour of the solutions of~\eqref{eq:variationalequationsatinfinity} with initial
condition $Q = Q_0 \neq 0$ along solutions $\varphi_{t}(w_0)$ of~\eqref{eq:infinitymanifoldsstraightened}, where the initial condition
$w_0 = (q_0,p_0,z_0,t_0) \in U \cap \{ q >0, \; p >0\}$ will be taken with $p_0$ small but fixed and $q_0$ arbitrarily close to $0$.

Equations~\eqref{eq:variationalequationsatinfinity} become, in the time $s$ in which~\eqref{eq:infinitymanifoldsstraightenedintimes} are written and using $\tau$ in~\eqref{def:tau} ,
\begin{equation}
\label{eq:variationalequationsatinfinitystau}
\begin{pmatrix}
 Q' \\ P' \\  Z' \\  T'
\end{pmatrix}
=
\begin{pmatrix}
\frac{4+\tau+\OO_1}{1+\tau} & \frac{3+\OO_1}{1+\tau} & q \OO_1 &  q \OO_1\\
-\frac{(3+\OO_1)\tau}{1+\tau} & -\frac{1+4\tau+\OO_1}{1+\tau} & p \OO_1  & p \OO_1\\
q^{N-1}p^N \OO_1 & q^N p^{N-1} \OO_1 & q^N p^N \OO_1 & q^N p^N \OO_1 \\
0 & 0 & 0& 0
\end{pmatrix}
\begin{pmatrix}
 Q \\ P \\  Z \\ T
\end{pmatrix}.
\end{equation}
It will be convenient to perform a linear change of variables to~\eqref{eq:variationalequationsatinfinitystau}.

\begin{proposition}
\label{lem:straighteningvariationalequations} There exists $\alpha^*$, with $0 < \alpha^* < 6/5$ such that for any  $\rho>0$, small enough,  any $w = (q,p,z,t)$, solution of~\eqref{eq:infinitymanifoldsstraightenedintimes} with $w_{\mid s = 0} = w_0 =(q_0,p_0,z_0,t_0)\in V_{\rho} \times K \times \TT$ and any $\alpha_0^*\in [0,\alpha^*]$, there exists a $C^N$ function $\alpha:[0,s_{w_0}] \to \RR$,
where $s_{w_0}$ was introduced in~\eqref{def:sw0}, with $\alpha(0) = \alpha_0^*$, such that, in the new variables
\[
\wt P = P + \alpha Q,
\]
system~\eqref{eq:variationalequationsatinfinity} becomes
\begin{equation}
\label{eq:variationalequationsatinfinitystausimplifyied}
\begin{pmatrix}
 Q' \\ \wt P' \\  Z' \\  T'
\end{pmatrix}
=
\begin{pmatrix}
\frac{4+\tau+\OO_1}{1+\tau} - \alpha\frac{3+\OO_1}{1+\tau} & \frac{3+\OO_1}{1+\tau} & q \OO_1 &  q \OO_1\\
0 & -\frac{1+4\tau+\OO_1}{1+\tau} + \alpha\frac{3+\OO_1}{1+\tau} &  p\OO_1  &  p\OO_1\\
q^{N-1}p^N \OO_1 & q^N p^{N-1} \OO_1 & q^N p^N \OO_1 & q^N p^N \OO_1 \\
0 & 0 & 0& 0
\end{pmatrix}
\begin{pmatrix}
 Q \\ \wt P \\  Z \\ T
\end{pmatrix}.
\end{equation}
Furthermore, for $s \in (0,s_{w_0}]$,
\begin{equation}
\label{bound:alphachange}
0 < \alpha(s) < \frac{2\tau(s)}{1+\tau(s)}.
\end{equation}
\end{proposition}

\begin{proof} Given $\alpha$ and $\wt P = P + \alpha Q$, since $\tau >0$, the equation for $\wt P$ is
\begin{multline}
\label{eq:tildePprime}
\wt P' = \left( -\frac{(3+\OO_1)\tau}{1+\tau} + \alpha' + (5+\OO_1)  \alpha - \alpha^2 \frac{3+\OO_1}{1+\tau}\right) Q \\
+ \left( -\frac{1+4\tau+\OO_1}{1+\tau} + \alpha\frac{3+\OO_1}{1+\tau}\right) \wt P
+ (p+\alpha q)\OO_1 Z + (p+\alpha q)\OO_1 T,
\end{multline}
The claim will follow finding an appropriate solution of
\begin{equation}
\label{eq:alphavariational}
 \alpha'   = \nu_0+ \nu_1 \alpha +\nu_2 \alpha^2 ,
\end{equation}
where
\begin{equation}
\label{def:nu0nu2}
\nu_0 = \frac{(3+\OO_1)\tau}{1+\tau}, \qquad \nu_1 (s) = -5+\OO_1, \qquad \nu_2 = \frac{3+\OO_1}{1+\tau}.
\end{equation}
Let $f(w,\alpha) = \nu_0+ \nu_1 \alpha +\nu_2 \alpha^2$ be the right hand side of~\eqref{eq:alphavariational}, where we have omitted the dependence of $\nu_i$, $i=1,2,3$ on $w$.
We introduce $\alpha_0$ and $\alpha_1$, the nullclines of~\eqref{eq:alphavariational},  by
\[
f(w,\alpha) = \nu_2 (\alpha-\alpha_0(\tau))(\alpha-\alpha_1(\tau)),
\]
and $R$, where
\begin{equation}
\label{def:alpha0}
\begin{aligned}
\alpha_0 (\tau) & = -\frac{\nu_1}{2 \nu_2}\left( 1-\left(1-4 \frac{\nu_0 \nu_2}{\nu_1^2}\right)^{1/2}\right) \\
& = \left(\frac{5}{6} +\OO_1\right)\left(1+\tau-\left((1+\tau)^2-\left(\frac{36}{25}+\OO_1 \right)
\tau \right)^{1/2}\right)
\\
& = \left(\frac{5}{6} +\OO_1\right)\left(1+\tau-\sqrt{R(\tau)}\right).
\end{aligned}
\end{equation}

To complete the proof of Proposition \ref{lem:straighteningvariationalequations}, we need the following two auxiliary lemmas.
\begin{lemma}
\label{lem:alpha0properties}
The function $\alpha_0$ has the following properties. 
For $(q,p) \in V_\rho$ (that is, $0< \tau < \infty$),
\begin{enumerate}
\item $4/5+\OO_1 \le \sqrt{R(\tau)}/(1+\tau) <1$,
\item $\lim_{\tau\to \infty} \alpha_0(\tau) = 3/5+\OO_1$,
\item $\lim_{\tau\to 0} \alpha_0(\tau)/\tau = 3/5+\OO_1$,
\item
\[
\frac{d}{ds} \alpha_0 = -(1+\OO_1) \frac{\sqrt{R}-\tau+1+\OO_1}{\sqrt{R}} \alpha_0,
\]
\item
\[
- \frac{(2+\OO_1)\alpha_0}{\sqrt{R}}  \le \frac{d}{ds} \alpha_0 \le
- \frac{(32/25+\OO_1)\alpha_0}{\sqrt{R}}
\]
\item and $\lim_{\tau\to 0} (d\alpha_0/ds)/\alpha_0 = -2+\OO_1$.
\end{enumerate}
Furthermore,
\begin{equation}
\label{bound:alpha0bytau}
0 < \alpha_0(\tau) < \frac{\tau}{1+\tau}.
\end{equation}
\end{lemma}

\begin{proof}
Items 1 to 6 are proven in~\cite{GorodetskiK12}. 
The rather crude bound~\eqref{bound:alpha0bytau} is a straightforward computation.
\end{proof}

Next lemma provides solutions of~\eqref{eq:alphavariational} close to the nullcline~$\alpha_0$.

\begin{lemma}
\label{lem:attractingnullcline}
For any $0 < \rho < 1$, small enough,  the following is true.
For any solution $w = (q,p,z,t)$ of~\eqref{eq:infinitymanifoldsstraightened} with initial condition $w_0
\in V_{\rho} \times K \times \TT$,
if $\alpha$ is a solution of~\eqref{eq:alphavariational} with $0 \le \alpha(s_0) \le 2 \alpha_0(\tau(s_0))$ for some $0 < s_0 < s_{w_0}$, then $0 < \alpha(s) < 2 \alpha_0(\tau(s))$ for all $s \in [s_0, s_{w_0}]$.
\end{lemma}

\begin{proof}[Proof of Lemma~\ref{lem:attractingnullcline}]
We only need to proof that $\alpha$ satisfies
\[
\mathrm{(i)} \; \;  \frac{d \alpha}{ds}_{\mid \tiny
\alpha  = 0 }
 >  0,   \qquad
 \mathrm{(ii)} \; \;
\frac{d \alpha}{ds}_{\mid \tiny
\alpha  = 2 \alpha_0 }  <  2 \frac{d \alpha_0}{ds}.
\]

Item (i) follows from
\[
\frac{d \alpha}{ds}_{\mid \tiny
\alpha  = 0 } = \nu_0 >0.
\]

Now we prove (ii). Using 5 of Lemma~\ref{lem:alpha0properties}, (ii) is implied by
\[
\begin{aligned}
\frac{d \alpha}{ds}_{\mid \tiny
\alpha  = 2 \alpha_0 }
& =  \frac{3+\OO_1}{1+\tau} \alpha_0(2\alpha_0-\alpha_1) \\
& =   \frac{3+\OO_1}{1+\tau} \alpha_0
\left(\alpha_0-\left(\frac{5}{3}+\OO_1\right) \sqrt{R}\right) \\
& < - 2\frac{2+\OO_1}{\sqrt{R}} \alpha_0,
\end{aligned}
\]
which, since $\alpha_0>0$, is equivalent to
\begin{equation}
\label{ineq:impliesiin}
\frac{3+\OO_1}{1+\tau} \left(\alpha_0-\left(\frac{5}{3}+\OO_1\right) \sqrt{R}\right)
<  -\frac{4+\OO_1}{\sqrt{R}},
\end{equation}
for $0<\tau$. Taking into account the definition of $\alpha_0$,
\eqref{ineq:impliesiin} is equivalent to prove
\[
\left(\frac{15}{2}+\OO_1\right) R-(4+\OO_1) (1+\tau) > \left(\frac{5}{2}+\OO_1\right) (1+\tau) \sqrt{R},
\]
for $0<\tau$, which is equivalent to
\[
\left(\left(\frac{15}{2}+\OO_1\right) R-(4+\OO_1) (1+\tau)\right)^2-
\left(\frac{5}{2}+\OO_1\right)^2 (1+\tau)^2 R >0.
\]

If we disregard the $\OO_1$ terms, which are small if $\rho$ is small, the above inequality simply reads
\[
50 \tau^4 -13 \tau^3 + \frac{826}{25} \tau^2 - \frac{75}{3} \tau +6
= \tau^2\left( 50 \tau^2-13 \tau +\frac{76}{25}\right) + 30 \tau^2 - \frac{75}{3} \tau +6>0.
\]
But the above inequality holds, since $50 \tau^2-13 \tau +\frac{76}{25}>0$ and $30 \tau^2 - \frac{75}{3} \tau +6>0$.
\end{proof}

Let $\alpha$ be any solution of~\eqref{eq:alphavariational} with $\alpha(0) \in [0,2\alpha_0(\tau(0))]$, which, by the definition of $\tau$ and Item 2 in Lemma~\ref{lem:alpha0properties}, is a nonempty interval (recall that $\tau(0)\gg 1$).
 By Lemma~\ref{lem:attractingnullcline}, $\alpha$ is well defined for $s\in [0,s_{w_0}]$ and
$0 < \alpha(s) < 2\alpha_0(\tau(s))$.  Then, bound~\eqref{bound:alpha0bytau} implies~\eqref{bound:alphachange}. System~\eqref{eq:variationalequationsatinfinitystausimplifyied} is obtained by a straightforward computation. Observe that, by~\eqref{bound:alpha0bytau}, the terms $(p+\alpha q)\OO_1$ in~\eqref{eq:tildePprime} are indeed $p\OO_1$.
\end{proof}


\begin{lemma}
\label{lem:solvingstraightenedvariationalequations}
Choose $N>10$ in Theorem~\ref{prop:coordinatesatinfinity}. 
Let $W$ and $K$ be the sets considered in Theorem~\ref{prop:lambdalemma}.
Let $w = (q,p,z,t)$ be a solution of~\eqref{eq:infinitymanifoldsstraightened} with initial condition, at $s=0$, $w_0
\in \Lambda_{a,\delta}^-(K)$. Let $\tilde s_{w_0}$ be such that $w(\tilde s_{w_0}) \in \Lambda_{a,\delta^{1-Ca}}^+(W)$. Let $W=(Q,\wt P,Z,T)$ be a solution of~\eqref{eq:variationalequationsatinfinitystausimplifyied} with initial condition, at $s=0$, $W_0=(Q_0,\wt P_0,Z_0,T_0)$.
\begin{enumerate}
\item
For all $s\in[0,s_{w_0}]$,
\[
\|(Z,T)-(Z_0,T_0)\| \le K q_0^{N-10} \|W_0\|.
\]
\item
For $s = \tilde  s_{w_0}$,
\[
|\wt P(\tilde s_{w_0})| \le  C q_0^{\frac{3}{5}+C\rho}(|\wt P_0|+ \OO_1  \|W_0\|).
\]
\item
Assume that $W_0$ satisfies $Q_0 \neq 0$. Then, there exists $\delta$ such that for any $w_0 \in \Lambda_{a,\delta}^-(K)$,
\[
|Q(\tilde s_{w_0})| \ge C \left(|Q_0|- C q_0^{\frac{1}{5}+\OO_1(\rho)} \|W_0\|\right) q_0^{-(\frac{3}{5}+\OO_1(\rho))}.
\]
\item For any $\wt P_0$, there exists a linear map $\wt Q(Z_0,T_0)$ satisfying $|\wt Q(Z_0,T_0)| \le C q_0^{\frac{1}{5}+\OO_1(\rho)} \|(Z_0,T_0)\|$, such that the solution $W$ of~\eqref{eq:variationalequationsatinfinitystausimplifyied} with initial condition $W_0 = (\wt Q(Z_0,T_0), \wt P_0, Z_0,T_0)$
    satisfies $Q(\tilde s_{w_0}) = 0$. 
\end{enumerate}
\end{lemma}

\begin{proof}
Here $\|\cdot\|$ will denote the $\sup$-norm. Let $\alpha$ be any of the functions given by Lemma~\ref{lem:attractingnullcline}. Using~\eqref{bound:alpha0bytau},
a direct computation shows that the spectral radius of the matrix defining system~\eqref{eq:variationalequationsatinfinitystausimplifyied} is bounded by $7+\OO_1(\rho)$. Hence, if $\rho$ is small enough,
\begin{equation}
\label{bound:roughboundvariational}
\|W(s) \| \le e^{8 s}\|W_0\|, \qquad 0 \le s \le \tilde s_{w_0},
\end{equation}
where, by~\eqref{bound:timesfromq0toq}, the time $\tilde s_{w_0}$  is
bounded by above by
\begin{equation}
\label{bound:tildesw0}
\tilde s_{w_0} \le \log \left( \frac{a}{q_0}\right)^{\frac{1}{1-C\rho}}.
\end{equation}
Let $W_{Q,P} = (Q,P)$ and $W_{Z,T} = (Z,T)$. The vector $W_{Z,T}$  satisfies
\[
W_{Z,T}' = q^N p^N \OO_1 W_{Z,T} + q^{N-1} p^{N-1} \OO_1 W_{Q,P}.
\]
Since, by~\eqref{bound:solutions_integrated_qp},
\[
\|q^{N-1} p^{N-1} \OO_1 W_{Q,P}\| \le (q_0 p_0)^{N-1} e^{2NC \rho s} \|W_{Q,P}\| \le (q_0 p_0)^{N-1}
e^{(8+2NC\rho) s} \|W_0\|
\]
and $N>10$, we have that, for $0 \le s \le \tilde s_{w_0}$, if $\rho$ is small enough,
\[
\|W_{Z,T}(s)-W_{Z,T}(0)\| \le K q_0^{N-10} \|W_0\|,
\]
for some constant $K$. This proves 1.

Now we prove 2. The equation for $\tilde P$ is
\[
\wt P ' = A \wt P + p \OO_1 W_{Z,T},
\]
where
\[
A = -\frac{1+4\tau+\OO_1}{1+\tau} + \alpha\frac{3+\OO_1}{1+\tau}.
\]
Using the bounds on $\alpha$ given by Lemma~\ref{lem:attractingnullcline} and~\eqref{bound:alphachange}  a straightforward computation shows that, if $\rho$ is small enough, for all $\tau>0$,  $A < -3/5$.
Hence, using again~\eqref{bound:solutions_integrated_qp},
$|\wt P(s) | \le  (|\wt P_0|+ \OO_1  \|W_0\|) e^{-\frac{3}{5}s}$, which, taking into account the bound of $\tilde s_{w_0}$, implies Item 2.

The equation for $Q$ is
\[
Q' = \wt A Q + \wt B \wt P + q \OO_1 W_{Z,T},
\]
where
\[
\wt A = \frac{4+\tau+\OO_1}{1+\tau} - \alpha\frac{3+\OO_1}{1+\tau}, \qquad \wt B= \frac{3+\OO_1}{1+\tau}.
\]
Again, a straightforward computation shows that, if $\rho$ is small enough, $\wt A > 3/5$.
Defining $u(s) = \exp \int_0^s \wt A (\sigma)\,d\sigma$, we have that
\begin{equation}
\label{eq:Qs}
Q(s) = u (s) \left[Q_0 + \int_0^s u(-\sigma)(\wt B(\sigma) \wt P (\sigma)+ q (\sigma)\OO_1 W_{Z,T}(\sigma))\,d\sigma  \right]
\end{equation}
We bound the terms in the integral in the following way. First we observe that, using~\eqref{def:tauedo},
\[
0 \le \wt B(\sigma) = \frac{3+\OO_1}{1+\tau} \le \frac{q_0}{p_0} \frac{3+\OO_1}{\frac{q_0}{p_0}+e^{-(2+\OO_1(\rho))s}}
< \frac{q_0}{p_0} (3+\OO_1) e^{(2+\OO_1(\rho))s}.
\]
Hence, by the previous bound on $\wt P$, for some constant $K>0$,
\[
\begin{aligned}
\left| \int_0^s u(-\sigma)\wt B(\sigma) \wt P (\sigma)\,d\sigma \right|
& \le (|\wt P_0|+ \OO_1  \|W_0\|) (3+\OO_1) \frac{q_0}{p_0}\int_0^s e^{(\frac{4}{5}+\OO_1(\rho))\sigma} \,d\sigma \\
& \le K (|\wt P_0|+ \OO_1  \|W_0\|)  q_0 e^{(\frac{4}{5}+\OO_1(\rho))s}.
\end{aligned}
\]
Using~\eqref{bound:solutions_integrated_qp} to bound $q(s)$, the other term in the integral can be bounded as
\[
\begin{aligned}
\left| \int_0^s u(-\sigma)q (\sigma)\OO_1 W_{Z,T}(\sigma)\,d\sigma \right|
& \le \OO_1(\rho)\|W_0\| q_0 \int_0^s e^{(\frac{2}{5}+\OO_1({\rho}))\sigma} \,d\sigma \\
& \le \OO_1(\rho)\|W_0\| q_0  e^{(\frac{2}{5}+\OO_1({\rho}))s}.
\end{aligned}
\]
That is, since $0 < s < \tilde s_{w_0} $ and using~\eqref{bound:tildesw0},
\[
\left|Q_0 + \int_0^s u(-\sigma)(\wt B(\sigma) \wt P (\sigma)+ q (\sigma)\OO_1 W_{Z,T}(\sigma))\,d\sigma \right|
\ge |Q_0|- K q_0^{\frac{1}{5}+\OO_1(\rho)} \|W_0\|.
\]
Since $0 < q_0 < \delta$, substituting this bound into~\eqref{eq:Qs} and evaluating  at $s = \tilde s_{w_0} $, we obtain 3.

Item 4 follows immediately from the bounds of the terms inside the brackets in~\eqref{eq:Qs}.
\end{proof}

\begin{proof}[Proof of Theorem~\ref{prop:lambdalemma}]
Lemma~\ref{lem:topological_lambda_lemma} proves that the Poincar\'e map $\Psi:\Lambda_{a,\delta}^- (K) \to
\Lambda_{a,\delta^{1-Ca}}^+ (W)$ is well defined for any compact $K\subset W$ if $0<2\delta<a$ are small enough
and also implies the estimates of Item  1 of Theorem~\ref{prop:lambdalemma}.

Given $I \subset \RR$, an interval, let $\gamma(u) = (q_0(u),a,z_0(u),t_0(u))$, $u \in I$ be a $C^1$ curve with
$0 < q(u) < \delta$,
and $\tilde \gamma = \Psi \circ \gamma = (a,p_1,z_1,t_1)$,
which is well defined if $\delta$ is small enough. 
Along this proof we  choose different curves $\gamma(u)$.

Let us compute $\tilde \gamma'(u)$. Let $X = (X_q,X_p,X_{z},X_t)$ denote the vector field in~\eqref{eq:infinitymanifoldsstraightened} and $w=(q,p,z,t)$. Since
$\tilde \gamma(u) = \varphi_{t_1(u)-t_0(u)}(\gamma(u))$,
we have that
\begin{multline}
\label{eq:gammatildeprime}
\tilde \gamma'(u) =
\begin{pmatrix}
0 \\ p_1'(u) \\ z_1'(u) \\ t_1'(u)
\end{pmatrix}
=
D_w \varphi_{t_1(u)-t_0(u)}(\gamma(u))\gamma'(u)+ X(\tilde \gamma(u)) (t_1'(u)-t_0'(u)) \\
=
\begin{pmatrix}
Q_{\mid t_1(u)-t_0(u)} + X_q(\tilde \gamma(u))(t_1'(u)-t_0'(u)) \\
P_{\mid t_1(u)-t_0(u)} + X_p(\tilde \gamma(u))(t_1'(u)-t_0'(u)) \\
Z_{\mid t_1(u)-t_0(u)} + X_z(\tilde \gamma(u))(t_1'(u)-t_0'(u)) \\
T_{\mid t_1(u)-t_0(u)} + X_t(\tilde \gamma(u))(t_1'(u)-t_0'(u)) \\
\end{pmatrix} ,
\end{multline}
where $(Q,P,Z,T)$ is the solution of~\eqref{eq:variationalequationsatinfinity}
along $\varphi_{t-t_0(u)}(\gamma(u))$ with initial condition $\gamma'(u)$.

From the first component of~\eqref{eq:gammatildeprime},
\begin{equation}
\label{eq:t1pmt0p}
t_1'(u)-t_0'(u) = -\frac{Q_{\mid t_1(u)-t_0(u)}}{X_q(\tilde \gamma(u))}.
\end{equation}
We observe that $X_q(\tilde \gamma(u)) = a((a+\OO(q_0(u)^{1-\OO_1(a)})^3+\OO(a^4)))$.

We choose $\alpha$ in Lemma~\ref{lem:attractingnullcline} such that $\alpha(0) = 0$. 
We apply the change of variables
of Proposition~\ref{lem:straighteningvariationalequations} and consider $(Q,\tilde P,Z,T)$, the corresponding solution
of~\eqref{eq:variationalequationsatinfinitystausimplifyied}. 
By the choice of $\alpha$,
$(Q,P,Z,T)_{\mid s=0} = (Q,\tilde P,Z,T)_{\mid s=0}$.

Now we prove 2 of Theorem~\ref{prop:lambdalemma}. 
Assume $q(u) = u$, $0<u<\delta$. 
Let $W_u^0 = (Q_u^0,P_u^0,Z_u^0,T_u^0)= \gamma'(u)$. 
Let $(Q_u,P_u,Z_u,T_u)$ be the solution of~\eqref{eq:variationalequationsatinfinitystau} with initial condition $W_u^0$ and $(Q_u,\wt P_u,Z_u,T_u)$,
the solution of~\eqref{eq:variationalequationsatinfinitystausimplifyied} with the same initial condition.
If $\delta$ is small enough, $\sup_{0<u<\delta} \|\gamma'(u)\| = \sup_{0<u<\delta} \|W_u^0\|  < 2 \|W_0^0\|$. 
Hence, if $\delta$ is small enough, by item 3 of Lemma~\ref{lem:solvingstraightenedvariationalequations},
\[
|Q_u(\tilde s_{w_0})|  \ge C \left(|Q_u^0|- C u^{\frac{1}{5}+\OO_1(\rho)} \|W_0^0\|\right) u^{-(\frac{3}{5}+\OO_1(\rho))}
\ge \wt C |Q_u^0| u^{-(\frac{3}{5}+\OO_1(\rho))}.
\]
In the case we are considering, $Q_u^0 = 1$.
This inequality, combined with~\eqref{eq:t1pmt0p}, implies
\begin{equation}
\label{bound:t1pminust0p}
|t_1'(u)-t_0'(u)| \ge C  u^{-(\frac{3}{5}+\OO_1(\rho))}.
\end{equation}
Hence, by item 2 of Lemma~\ref{lem:solvingstraightenedvariationalequations}, the bound of $\alpha$ given by Lemma~\ref{lem:attractingnullcline}, bound~\eqref{bound:alpha0bytau}, \eqref{eq:t1pmt0p} and the facts that $T'=0$, $T=t_0'(u)$,
$|X_p(\tilde \gamma(u))| \le C q(u)^{1-Ca}$  and $|X_q(\tilde \gamma(u))| \le C $,
\[
\begin{aligned}
\left| \frac{p_1'(u)}{t_1'(u)}\right| & =
\frac{|P_u (\tilde s_{\gamma(u)})+ X_p(\tilde \gamma(u))(t_1'(u)-t_0'(u)|}{|t_1'(u)|}
\\
& \le
\frac{|P_u (\tilde s_{\gamma(u)})+ X_p(\tilde \gamma(u))(t_1'(u)-t_0'(u)|}{|t_1'(u)-t_0'(u))|}
\left(1+ \frac{|t_0'(u)|}{|t_1'(u)|}\right)\\
& \le C
\frac{|\wt P_u (\tilde s_{\gamma(u)})-\alpha (\tilde s_{\gamma(u)}) Q_u (\tilde s_{\gamma(u)})+ X_p(\tilde \gamma(u))(t_1'(u)-t_0'(u))|}{|t_1'(u)-t_0'(u)|}  \\
& =
C \frac{|\wt P_u (\tilde s_{\gamma(u)})+[(\alpha (\tilde s_{\gamma(u)}) X_q(\tilde \gamma(u))+ X_p(\tilde \gamma(u))](t_1'(u)-t_0'(u))|}{|t_1'(u)-t_0'(u)|}  \\
& =
C \left(\frac{|\wt P_u (\tilde s_{\gamma(u)})|}{|t_1'(u)-t_0'(u)|} +
\left|(\alpha (\tilde s_{\gamma(u)}) X_q(\tilde \gamma(u))+ X_p(\tilde \gamma(u))\right|\right) \\
& \le \wt C u^{1-Ca}.
\end{aligned}
\]
And, analogously, using Item 1 of Lemma~\ref{lem:solvingstraightenedvariationalequations} and the fact that $N>10$, we obtain that
\[
\begin{aligned}
\left| \frac{z_1'(u)}{t_1'(u)}\right| & \le
\frac{|Z_u (\tilde s_{w_0})+ X_z(\tilde \gamma(u))(t_1'(u)-t_0'(u)|}{|t_1'(u)-t_0'(u))|}
\left(1+ \frac{|t_0'(u)|}{|t_1'(u)|}\right)\\
& \le
C \left(\frac{|Z_u (\tilde s_{w_0}))|}{|t_1'(u)-t_0'(u)|} +| X_z (\tilde \gamma(u))|\right) \\
& \le C u^{\frac{3}{5}-Ca},
\end{aligned}
\]
which proves Item 2 of Theorem~\ref{prop:lambdalemma}.

Now we prove Item 3 of Theorem~\ref{prop:lambdalemma}. 
We first observe that, using~\eqref{bound:roughboundvariational},~\eqref{bound:tildesw0} and \eqref{eq:t1pmt0p},
\[
|t_1'(u)-t_0'(u)| \le C |Q_u (\tilde s_{\gamma(u)})| \le C q(u)^{-8-Ca}\|W_u\|.
\]
Then, using 1 of Lemma~\ref{lem:solvingstraightenedvariationalequations} and, since $N>10$, $|X_z(\tilde \gamma(u))| \le C q(u)^{N-Ca}$,
\[
\begin{aligned}
|z_1'(u)-z_0'(u)| & = |Z_{\mid t_1(u)-t_0(u)}-Z_{\mid 0} + X_z(\tilde \gamma(u))(t_1'(u)-t_0'(u)) | \\
& \le C q(u)^{N-10} \|W_u\|.
\end{aligned}
\]
Hence, 3 is proven.

We finally prove 4. Let $\tilde q_0 \in (0,\delta)$ and $\tilde w_0 = (\tilde q_0,a,\tilde z_0,\tilde t_0) \in \Lambda_{\delta}^-(K)$.
Taking into account~\eqref{bound:t1pminust0p}, which also holds in this case, by the Implicit Function Theorem, the equation
\[
t_1(q,a,z,t) -t = t_1(\tilde q_0,a,\tilde z_0,\tilde t_0)-\tilde t_0
\]
defines a function $q_0(z,t)$, with $(z,t)$ in neighborhood of $(z_0,t_0)$.
Given $(z_0(u),t_0(u))$, any curve in $K\times \TT$ with $z_0(0)=\tilde z_0$ and $t_0(0) = \tilde t_0$,
let $\gamma(u) = (q_0(z_0(u),t_0(u)),a,z_0(u),t_0(u))$ and $\tilde \gamma (a,p_1,z_1,t_1)= \Psi \circ \gamma$.
Since, by the definition of the function $q_0$, $(t_1(u) -t_0(u))' = 0$, from~\eqref{eq:gammatildeprime} we obtain that
\[
\tilde \gamma'(u)
=
\begin{pmatrix}
Q_{\mid t_1(u)-t_0(u)}  \\
P_{\mid t_1(u)-t_0(u)}  \\
Z_{\mid t_1(u)-t_0(u)}  \\
T_{\mid t_1(u)-t_0(u)}  \\
\end{pmatrix} ,
\]
%
where $(Q,P,Z,T)$ is the solution of~\eqref{eq:variationalequationsatinfinity}
along $\varphi_{t-t_0(u)}(\gamma(0))$ with initial condition $\gamma'(0)$. In particular, $Q_{\mid t_1(u)-t_0(u)} = 0$.
But, then, this implies that $Q_{\mid 0}$ has to be the value $\wt Q_0$ given by 4 of Lemma~\ref{lem:solvingstraightenedvariationalequations}. Moreover, this implies  $P_{\mid t_1(u)-t_0(u)} =\tilde P_{\mid t_1(u)-t_0(u)}$.  From the bounds of Lemma~\ref{lem:solvingstraightenedvariationalequations} follow 4 of Theorem~\ref{prop:lambdalemma}.
\end{proof}

\section{Conjugation with the Bernouilli shift}

We devote this section to proof Propositions \ref{prop:horizontalstrips} and \ref{prop:condicionsdecons}. This is done in several steps. First, in Section \ref{sec:differentialintermediatereturn}, we analyze the differential of the  return maps $\wt \Psi_{i,j}$ defined in \eqref{def:wtPsiij}. In particular, we analyze its expanding, contracting and center directions. Then, in Section \ref{sec:proofofpropositionprop:horizontalstrips}, we prove Proposition \ref{prop:horizontalstrips}.
In Section \ref{sec:conefields}, we analyze the differential of $\wt\Psi$, the high iterate of the  return map  defined in \eqref{def:Psitilde}. Finally, in Section \ref{sec:provaproposiciodelscons}, we use these cone fields to prove Proposition \ref{prop:condicionsdecons}.

\subsection{The differential of the intermediate return maps}\label{sec:differentialintermediatereturn}

In the following lemma, we will write a vector $v\in T_\omega\QQQ_\de^i$ in the basis
$\frac{\partial}{\partial p}$, $\frac{\partial}{\partial \tau}$, $\frac{\partial}{\partial z}$, where
$\frac{\partial}{\partial z}$ stands for $(\frac{\partial}{\partial z_1},\frac{\partial}{\partial z_2})$, given by the coordinates defined by $A_i$.

\begin{lemma}
\label{lem:almostinvariantvectorfields}
Let $N$ be fixed.
Assume $\de$ small enough.
\begin{enumerate}
\item
The vector field
\[
v_1 = \begin{pmatrix}
0 \\ 1 \\ 0
\end{pmatrix}
\]
satisfies that,
for any $\omega =(p,\tau,z) \in \QQQ_\de^i \cap \wt \Psi_{i,j}^{-1}(\QQQ_\de^k)$,
\begin{equation}
\label{eq:v1v2almostinvariantvectorfields}
D \wt \Psi_{i,j}(\omega)  v_1  = \lambda_1^i(\omega)(v_1+ \tilde v_1^{i,j}(\omega)),
\qquad 1 \le i,j \le 2,
\end{equation}
where $\wt \Psi_{i,j}$ is the  return map defined in \eqref{def:wtPsiij} written in coordinates $(p,\tau,z)$ and
\begin{gather*}
\lambda_1^i(\omega)\gtrsim \tau^{-\frac{3}{5}+\tilde Ca}, \\
\quad |\tilde v_{1,p}^{i,j}(\omega)|\le \OO\left(\tau^{1-\tilde Ca}\right), \quad \tilde v_{1,\tau}^{i,j}(\omega) = 0, \quad \|\tilde v_{1,z}^{i,j}(\omega)\|\le \OO\left(\tau^{\frac{3}{5}-\tilde Ca}\right).
\end{gather*}
Moreover, for any vector
\[
\hat v_1 = \begin{pmatrix}
b\\ 1 \\ c
\end{pmatrix}
\]
for $|b|,|c|\lesssim 1$ one has
\begin{equation}\label{def:v1hatstreching}
 D \wt \Psi_{i,j}(\omega)  \hat  v_1 = \lambda_1^i(\omega)(v_1+ \hat  v_1^{i,j}(\omega)),
\qquad 1 \le i,j \le 2,
\end{equation}
for some vector $\hat  v_1^{i,j}$ satisfying
\[
 \quad |\hat  v_{1,p}^{i,j}(\omega)|\le \OO\left(\tau^{1-\tilde Ca}\right), \quad |\hat v_{1,\tau}^{i,j}(\omega)| , \|\hat  v_{1,z}^{i,j}(\omega)\|\le \OO\left(\tau^{\frac{3}{5}-\tilde Ca}\right)\quad\text{and}\quad  \|\lambda_1^i(\omega)\hat  v_{1,z}^{i,j}(\omega)\|\le \OO\left(1\right).
\]
\item
There exist vector fields $v_2^{i,j}: \QQQ_\de^i \to T\QQQ_\de^i$, $i,j=1,2$, of the form
\begin{equation}\label{def:vec2Item2}
v_2^{i,j} (\omega)= \begin{pmatrix}
1 \\ \tilde v_{2,\tau}^{i,j} (\omega)\\ v_{2,z}^i(z) +  \tilde v_{2,z}^{i,j}(\omega)
\end{pmatrix},
\end{equation}
with certain functions $v_{2,z}^i$ depending only on $z$ satisfying
$\|v_{2,z}^i(z)\|=\OO(1)$ and
\[
|\tilde v_{2,\tau}^{i,j}(\omega)| = \OO(p)+\OO\left(\tau^{\frac{3}{5}-\tilde Ca}\right),  \qquad \|\tilde v_{2,z}^{i,j}(\omega)\| \le \OO\left(\tau^{\frac{3}{5}-\tilde Ca}\right), \qquad i =1,2,
\]
such that for any $\omega =(p,\tau,z) \in \QQQ_\de^i \cap \wt \Psi_{i,j}^{-1}(\QQQ_\de^j)$, the following holds.
\begin{equation}
\label{eq:v1v2almostinvariantvectorfields}
D \wt \Psi_{i,j}(\omega)  v_2^{i,j}(\omega)  = \lambda_2^{i,j}(\omega) (v_2^{i,j} (\wt \Psi_{i,j}(\omega)) + \hat v_2^{i,j}(\omega)),
\qquad 1 \le i,j \le 2,
\end{equation}
where
\begin{gather*}
\lambda_2^{i,j}(\omega)^{-1} \gtrsim \tau^{-\frac{3}{5}+\tilde Ca}, \\
\hat v_{2,p}^{i,j}(\omega) = 0, \quad |\hat v_{2,\tau}^{i,j}(\omega)|\le \OO\left(\tau^{1-\tilde Ca}\right), \quad \|\hat v_{2,z}^{i,j}(\omega)\|\le \OO\left(\tau^{\frac{3}{5}-\tilde Ca}\right).
\end{gather*}
\item
For any $v_z (z)$, $C^0$ vector field in $\RR^2$, $i,j=1,2$, there exist vector fields
\[
v^{i,j}(\omega) =
\begin{pmatrix}
0 \\ 0 \\v_z(z)
\end{pmatrix}
+\begin{pmatrix}
0 \\ \tilde v_\tau^{i,j} (\omega)\\ 0
\end{pmatrix},
\]
with $|\tilde v_\tau^{i,j}| \le \OO(\tau^{1/5-Ca})$, such that the following holds.
\begin{equation}
\label{eq:vzalmostinvariantvectorfields}
D \wt \Psi_{i,j}(\omega)  v^{i,j}(\omega) = \begin{pmatrix}
0 \\ 0 \\ D\widehat{\mathtt{S}_i}(z) v_z(z)
\end{pmatrix}+\hat v^{i,j}(\omega) ,
\qquad \|\hat v^{i,j}\| \le \OO(\tau^{1/5-Ca},a).
\end{equation}
\end{enumerate}
\end{lemma}

\begin{proof}
We start with $v_1$. For any $\omega = (p,\tau,z) \in \QQQ_\de^i$, $z= (z_1,z_2)$, we have that, in view of~\eqref{def:globalmapiinlocalcoordinates},
\[
\begin{aligned}
D\wtPsig{i}(\omega) v_1 & =
\begin{pmatrix}
\OO_1(\tau) & \nu_1^i(z) + \OO_1(p,\tau) & \OO_1(\tau) \\
\nu_2^i(z) + \OO_1(p,\tau) & \OO_1(p) & \OO_1(p) \\
\widetilde S_{i,p}(z)  + \OO_1(p,\tau) & \widetilde S_{i,\tau}(z)  + \OO_1(p,\tau)& D\wh{\mathtt{S}}_i(z) + \OO_1(p,\tau)
\end{pmatrix} \begin{pmatrix}
0 \\ 1 \\ 0
\end{pmatrix} \\
& = (\nu_1^i(z) + \OO_1(p,\tau))
\begin{pmatrix}
1 \\  \OO_1(p,\tau) \\
\wt{S}_{i,\tau}(z)\nu_1^i(z)^{-1}+ \OO_1(p,\tau)
\end{pmatrix},
\end{aligned}
\]
where $\widetilde S_{i,p}(z) = \partial_p \pi_z {\Psig{i}}(0,0,z)$ and
$\widetilde S_{i,\tau}(z) = \partial_\tau \pi_z {\Psig{i}}(0,0,z)$.
Hence, using that $(q^*,\sigma^*,z^*) = \wtPsig{i}(\omega)$ satisfies $q^* = \tau \nu_1^i(z) ( 1+ \OO_1(p,\tau))$, by Item 2 of Theorem~\ref{prop:lambdalemma}, we have that, using \eqref{def:coordenadeslocalsaSigma1} and \eqref{def:coordenadeslocalsaSigma2}
\[
\begin{aligned}
D\wt \Psi_{i,j}(\omega) v_1 & = DA_j (\Psi_{i,j}(A^{-1}_i(\omega)))  D \Psiloc{i}{j}(B_i^{-1} \circ \wtPsig{i}(\omega)) D B_i^{-1}(\wtPsig{i}(\omega)) D\wtPsig{i}(\omega) v_1\\
& = (\nu_1^i(z) + \OO_1(p,\tau)) DA_j (\Psi_{i,j}(A^{-1}_i(\omega)))  D \Psiloc{i}{j}((B_i^{-1} \circ \wtPsig{i}(\omega))\begin{pmatrix}
1 \\ \OO_0(p,\tau) \\
\OO_0(p,\tau)
\end{pmatrix} \\
& = \lambda_1(\omega) (\nu_1^i(z) + \OO_1(p,\tau))  \begin{pmatrix}
P^*_{i,j} \\ 1 \\
Z^*_{i,j}
\end{pmatrix},
\end{aligned}
\]
where, for some $C>0$,
\[
\lambda_1(\omega) \gtrsim \tau^{-\frac{3}{5}+Ca}, \qquad |P^*_{i,j}|\le \OO\left(\tau^{1- Ca}\right),\qquad \|Z^*_{i,j}\| \le \OO\left(\tau^{\frac{3}{5}-Ca}\right).
\]


This proves the claim for $v_1$, taking $\lambda_1^i (\omega)= \lambda_1(\omega) (\nu_1^i(z) + \OO_1(p,\tau))$.

The proof of the second statement of Item 1 follows exactly the same lines.

We now prove Item  2 of Lemma \ref{lem:almostinvariantvectorfields}. Let $v_{2,z}^i(z)= \widehat S_{i,\sigma}(\widehat{\mathtt{S}}_i(z))$, where $\widehat S_{i,\sigma}(z) = \partial_\sigma \pi_z {\wtPsig{i}^{-1}}(0,0,z)$.


We observe that, since $\wt \Psi_{i,j} =  \wtPsiloc{i}{j} \circ \wtPsig{i}$,
\[
\begin{split}
\left(D\wt \Psi_{i,j}(\omega)\right)^{-1}= D (\wt \Psi_{i,j}^{-1})(\wt \Psi_{i,j}(\omega)) &= D (\wtPsig{i}^{-1}) (\wtPsiloc{i}{j}^{-1} \circ \wt \Psi_{i,j}(\omega))
D (\wtPsiloc{i}{j}^{-1}) (\wt \Psi_{i,j}(\omega)) \\
&= D (\wtPsig{i}^{-1}) (\wtPsig{i}(\omega))
D (\wtPsiloc{i}{j}^{-1}) (\wt \Psi_{i,j}(\omega)).
\end{split}
\]
First of all, we notice that, denoting $\wt \Psi_{i,j}(\omega) = (\hat p,\hat \tau, \hat z)$, by~\eqref{def:globalmapiinlocalcoordinates} and Item 1 of Theorem~\ref{prop:lambdalemma},
\begin{equation}
\label{bound:hatp}
(\tau \nu_1^i(z))^{1+Ca}(1+\OO_1(p,\tau)) \le \hat p \le (\tau \nu_1^i(z))^{1-Ca}(1+\OO_1(p,\tau)).
\end{equation}
Then, since
\[
\wtPsiloc{i}{j}^{-1}  = B_i \circ \Psiloc{i}{j}^{-1} \circ A_j^{-1},
\]
applying Item 2 of Theorem~\ref{prop:lambdalemma} to $(\Psiloc{i}{j})^{-1}$ (that is, applying Item 2 of Theorem~\ref{prop:lambdalemma}, changing the sign of the vector field), evaluated at $A_j^{-1} \circ \wt \Psi_{i,j}(\omega)$, and to the vector $v_2^{i,j}$ in \eqref{def:vec2Item2},
\[
D (\wtPsiloc{i}{j}^{-1}) (\wt \Psi_{i,j}(\omega)) v_2^{i,j} (\wt \Psi_{i,j}(\omega)) =
\tilde \lambda_2^{i,j}(\omega) \begin{pmatrix}
\hat Q^*_{i,j} \\ 1 \\
\hat Z^*_{i,j}
\end{pmatrix},
\]
where, for some $\tilde C>0$,
\[
\tilde \lambda_2^{i,j}(\omega) \gtrsim \tau^{-\frac{3}{5}+\tilde Ca}, \qquad |\hat Q^*_{i,j}|\le \OO\left(\tau^{1-\tilde Ca}\right),\qquad \|\hat Z^*_{i,j}\| \le \OO\left(\tau^{\frac{3}{5}-Ca}\right).
\]
Hence, by~\eqref{def:inverseglobalmapiinlocalcoordinates},
\begin{multline*}
D\wt \Psi_{i,j}(\omega)^{-1} v_2^j(\wt \Psi_{i,j}(\omega)) =
D (\wtPsig{i}^{-1})(\wtPsig{i}(\omega))
D (\wtPsiloc{i}{j}^{-1}) (\wt \Psi_{i,j}(\omega)) v_2^j (\wt \Psi_{i,j}(\omega))\\
\begin{aligned}
& = \tilde  \lambda_2^{i,j}(\omega) D (\wtPsig{i}^{-1})(\wtPsig{i}(\omega))\begin{pmatrix}
\hat Q^*_{i,j} \\ 1 \\
\hat Z^*_{i,j}
\end{pmatrix} \\
& = \tilde \lambda_2^{i,j}(\omega) \begin{pmatrix}
\OO_1(\tau) & \mu_1^i(\widehat{\mathtt{S}}_i(z)) + \OO_1(p,\tau) & \OO_1(\tau) \\
\mu_2^i(\widehat{\mathtt{S}}_i(z)) + \OO_1(p,\tau) & \OO_1(p) & \OO_1(p) \\
\widehat S_{i,q}(\widehat{\mathtt{S}}_i(z))  +\OO_1(p,\tau)& \widehat S_{i,\sigma}(\widehat{\mathtt{S}}_i(z)) + \OO_1(p,\tau)& D(\widehat{\mathtt{S}}_i^{-1})(\widehat{\mathtt{S}}_i(z)) + \OO_1(p,\tau)
\end{pmatrix}\begin{pmatrix}
\hat Q^*_{i,j} \\ 1 \\
\hat Z^*_{i,j}
\end{pmatrix}\\
& = \tilde \lambda_2^{i,j}(\omega) (\mu_1^i(\widehat{\mathtt{S}}_i(z))+\OO_1(p,\tau))\begin{pmatrix}
1 \\ \bar T^*_{i,j} \\
\widehat S_{i,\sigma}(\widehat{\mathtt{S}}_i(z)) + \bar Z^*_{i,j}
\end{pmatrix},
\end{aligned}
\end{multline*}
where $\widehat S_{i,q}(z) = \partial_q \pi_z {\wtPsig{i}^{-1}}(0,0,z)$ and $\widehat S_{i,\sigma}(z) = \partial_\sigma \pi_z {\wtPsig{i}^{-1}}(0,0,z)$ and
\[
|\bar T^*_{i,j}|\le \OO_1(p)+\OO\left(\tau^{\frac{3}{5}-\tilde Ca}\right),\; \|\bar Z^*_{i,j}\| \le \OO\left(\tau^{\frac{3}{5}-\tilde Ca}\right).
\]
The claim follows taking $(\lambda_2^{i,j})^{-1}(\omega) = \tilde \lambda_2^{i,j}(\omega) (\mu_1^i(S_i(z))+\OO_1(p,\tau))$.

Finally, we prove Item 3. In order to find the vector fields, we look for
$v^{i,j} = v_0+ \tilde v^{i,j}(p,\tau,z)$, with $v_0 = (0,0,v_z)^\top$ and $\tilde v^{i,j} = (0, \tilde v_{\tau,1}^{i,j}+\tilde v_{\tau,2}^{i,j}, 0)^\top$. Note that both corrections appear in the $\tau$-direction. We write down them separated since they will play the different roles. Roughly speaking $\tilde v_{\tau,1}^{i,j}$ will be obtained by applying Item 4 in Theorem~\ref{prop:lambdalemma} whereas $\tilde v_{\tau,2}^{i,j}$ is obtained by applying Item 2 in the same proposition.

We have that
\[
D \wtPsig{i}(\omega) v^{i,j}(\omega) 
= \begin{pmatrix}
\nu_1^{i}(z)  \tilde v_{\tau,1}^{i,j} + \OO_1(p,\tau) v^{i,j} \\
\OO_1(p,\tau) v^{i,j} \\
D\widehat{\mathtt{S}}_i(z) v_z + \widetilde S_{i,\tau}(z)\tilde  v_{\tau,1}^{i,j}
+ \OO_1(p,\tau) v^{i,j}
\end{pmatrix}
+
 \begin{pmatrix}
\nu_1^{i}(z)  \tilde v_{\tau,2}^{i,j}  \\
0 \\
\widetilde S_{i,\tau }(z)\tilde  v_{\tau,2}^{i,j}
\end{pmatrix}.
\]
Hence, by \eqref{def:coordenadeslocalsaSigma1}
\begin{equation}
\label{eq:BinverseDPsiglob1onv}
DB_i^{-1}(\wtPsig{i}(\omega))D \wtPsig{i}(\omega) v^{i,j}(\omega)
= w_1(\omega)+w_2(\omega)
\end{equation}
with
\begin{equation}
\label{def:w1omega}
w_1(\omega) =
\begin{pmatrix}
\nu_1^{i}(z)  \tilde v_{\tau,1}^{i,j} + \OO_1(p,\tau) v^{i,j} \\
 -\beta_{i,2}(\widehat{\mathtt{S}}_i(z))DS_i(z) v_z -\left[\beta_{i,1}(\widehat{\mathtt{S}}_i(z))\nu_1^{i}(z) + \beta_{i,2}(\widehat{\mathtt{S}}_i(z)) \widetilde S_{i,\tau}(z)\right] \tilde v_{\tau,1}^{i,j}  +\OO_1(p,\tau) v^{i,j} \\
D\widehat{\mathtt{S}}_i(z) v_z + \widetilde S_{i,\tau}(z)\tilde  v_{\tau,1}^{i,j}
+ \OO_1(p,\tau) v^{i,j}
\end{pmatrix}
\end{equation}
and
\begin{equation}
\label{def:w2omega}
w_2(\omega) =
\begin{pmatrix}
\nu_1^{i}(z)  \tilde v_{\tau,2}^{i,j}  \\
-\left[\beta_{i,1}(\widehat{\mathtt{S}}_i(z))\nu_1^{i}(z) + \beta_{i,2}(\widehat{\mathtt{S}}_i(z)) \widetilde S_{i,\tau}(z)\right]\tilde v_{\tau,2}^{i,j} \\
\widetilde S_{i,\tau}(z)\tilde  v_{\tau,2}^{i,j}
\end{pmatrix}
\end{equation}
where the functions $\beta_{i,1}(z)=-\frac{\partial \wt w^u_i}{\partial q}(0,0,z)$ and $\beta_{i,2}=-(t^u_i)'(z)-
\frac{\partial \wt w^u_i}{\partial z}(0,0,z)$ (see \eqref{def:coordenadeslocalsaSigma1}).

Since $\pi_q (\wtPsig{i}(\omega)) = \nu_1^i(z) \tau (1+\OO_1(p,\tau))$, (see \eqref{def:globalmapiinlocalcoordinates}),
by Item  4 of Lemma~\ref{lem:solvingstraightenedvariationalequations} with $\wt P_0=0$, there exists a linear map $\wt Q(\wtPsig{i}(\omega))$, with $\|\wt Q(\wtPsig{i}(\omega))\| \le \OO(\tau^{1/5})$,
such that if $w_1$ satisfies
\begin{equation}
\label{eq:nongrowthcondition}
\pi_q w_1(\omega) = \wt Q(\wtPsig{i}(\omega))\left[ \pi_{t, z}  w_1(\omega)\right]
\end{equation}
then
\[
\begin{aligned}
\pi_t  D\Psiloc{i}{j}(\wtPsig{i}(\omega)) w_1(\omega) & = \pi_t w_1(\omega), \\
|\pi_p  D\Psiloc{i}{j}(\wtPsig{i}(\omega)) w_1(\omega)| & \le \OO(\tau^{3/5-Ca}) \|\pi_{t, z}  w_1(\omega)\|, \\
\|\pi_z  D\Psiloc{i}{j}(\wtPsig{i}(\omega)) w_1(\omega)- \pi_z w_1(\omega)\|
& \le \OO(\tau) \| w_1(\omega) \|.
\end{aligned}
\]
We use this fact to choose a suitable $w_1$ (by choosing a suitable  $\tilde v_{\tau,1}^{i,j}$).

Moreover, since $\nu_1^{i}$ does not vanish, from Item 2 of Theorem~\ref{prop:lambdalemma},
\[
D\Psiloc{i}{j}(\wtPsig{i}(\omega)) w_2(\omega)
= \lambda(\omega) \nu_1^{i}(z)  \tilde v_{\tau,2}^{i,j}(\omega)
\begin{pmatrix}
P^* \\ 1 \\ Z^*
\end{pmatrix}
\]
with
\[
\lambda(\omega) \gtrsim \tau^{-(3/5-Ca)}, \qquad |P^*| \le \OO(\tau^{1-Ca}), \qquad |Z^*| \le \OO(\tau^{3/5-Ca}).
\]
Then, if we assume for a moment that~\eqref{eq:nongrowthcondition} is satisfied, cumbersome but straightforward computations lead to
\begin{equation}
\label{eq:DwtPsiijvij}
D\wt \Psi_{i,j}(\omega) v^{i,j} =
\begin{pmatrix}
\lambda(\tau) \nu_1^{i}(z)  \tilde v_{\tau,2}^{i,j}(\omega) P^* + \OO(\tau^{3/5-Ca}) v^{i,j} \\
\wt T_0(\omega)
  +\OO_1(p,\tau) v^{i,j} + \lambda(\omega)\wt T_1(\omega)  \tilde v_{\tau,2}^{i,j}(\omega) \\
DS_i(z) v_z + \widetilde S_{i,p}(z)\tilde  v_{\tau,1}^{i,j}
+ \OO_1(p,\tau) v^{i,j} +   \lambda(\omega) \nu_1^{i}(z)  \tilde v_{\tau,2}^{i,j}(\omega) Z^*
\end{pmatrix}
\end{equation}
where
\[
\begin{split}
\wt T_0(\omega)=&
\left[\alpha_{i,2}(S_i(z)) -\beta_{i,2}(S_i(z))\right]DS_i(z) v_z -((\beta_{i,1}(S_i(z))-\alpha_{i,1}(S_i(z)))\nu_1^{i}(z) \\
&+ (\beta_{i,2}(S_i(z))-\alpha_{i,2}(S_i(z))) \widetilde S_{i,p}(z)) \tilde v_{\tau,1}^{i,j}\\
\wt T_1(\omega)=&  \nu_1^{i}(z)\left[1+(\alpha_{i,1}(S_i(z))+\OO_1(p,\tau))P^*+(\alpha_{i,2}(S_i(z))+\OO_1(p,\tau))Z^*\right]
\end{split}
\]
where $\al_{i,1}= - \pa_p\tilde w_i^s(p,z)$ and $\al_{i,2}=-\pa_z t_i^s(z) - \pa_z \tilde w_i^s(p,z)$.
Note that for small $(p,\tau)$,
$\wt T_1(\omega)\gtrsim 1$.

We choose
\begin{equation}
\label{def:tildevtau2ij}
\tilde v_{\tau,2}^{i,j}(\omega) = - \frac{1}{\lambda(\tau)\wt T_1(\omega)} \wt T_0(\omega).
\end{equation}
Observe that this choice of $\tilde v_{\tau,2}^{i,j}$ is linear in $v_{\tau,1}^{i,j}$ and satisfies
\[
\tilde v_{\tau,2}^{i,j}(\omega) = \OO(\tau^{3/5-Ca}) + \OO(\tau^{3/5-Ca})\tilde v_{\tau,1}^{i,j}(\omega).
\]
Inserting this choice of $\tilde v_{\tau,2}^{i,j}$ in~\eqref{eq:nongrowthcondition}, we obtain the fixed point equation for $\tilde v_{\tau,1}^{i,j}$
\[
\tilde v_{\tau,1}^{i,j} = \frac{1}{\nu_1^i(z)} \left(\wt Q(\Psig{i}(\omega))\left(
\begin{pmatrix}
\wt T(\omega)  \\ DS_i(z) v_z + \widetilde S_{i,p}(z)\tilde v_\tau^i
\end{pmatrix}
 + \OO_1(p,\tau) \tilde  v^{i,j}\right)\right).
\]
It clearly has a solution $\tilde v_{\tau,1}^{i,j} = \OO (\tau^{1/5})\|v_z\|$. Then, taking into account~\eqref{def:tildevtau2ij} and~\eqref{eq:DwtPsiijvij}, we have that
\[
\pi_p D\wt \Psi_{i,j}(\omega) v^{i,j} = \lambda(\tau) \nu_1^{i}(z)  \tilde v_{\tau,2}^{i,j}(\omega) P^* + \OO(\tau^{3/5-Ca}) v^{i,j} = \OO(\tau^{3/5-Ca}) v_z
\]
and
\[
\pi_\tau D\wt \Psi_{i,j}(\omega) v^{i,j} = \OO_1(p,\tau) v^{i,j} = \OO_1(p,\tau) v_z .
\]
This completes the proof of Item 3.
\end{proof}

\subsection{Horizontal and vertical strips: Proof of Proposition~\ref{prop:horizontalstrips}}
\label{sec:proofofpropositionprop:horizontalstrips}

Proposition~\ref{prop:horizontalstrips} is direct consequence of the following lemma.

\begin{lemma}
\label{lem:bandeshoritzontalsabandeshoritzontals}
For any \emph{horizontal surface},
\[
S_h = \{(p,\tau,\varphi,J) \in \QQQ_\de^2\mid ( p, J ) =  (h_1(\tau,\varphi),
h_2(\tau,\varphi)), \; (\tau,\varphi) \in (0,\de)\times(0,\tkk)\},
\]
 with $h$ a $\CCC^1$ function satisfying
\begin{equation}
\label{cond:horizontalband}
\begin{aligned}
\sup_{(\tau,\varphi)\in (0,\de)\times(0,\tkk)} \|\partial_\tau h_1(\tau,\varphi)\| & < \OO(1),
\sup_{(\tau,\varphi)\in (0,\de)\times(0,\tkk)} \|\partial_\varphi h_1(\tau,\varphi)\| & < \OO(\de),
\end{aligned}
\end{equation}
$\wt \Psi(S_h) \cap \QQQ_\de^2$ has an infinite number of connected components.
Moreover, $\wt \Psi(S_h) \cap \QQQ_\de^2$ contains a countable union of horizontal surfaces, $S_{h_n}$, with $\lim_{n\to \infty} h_{n,1} = 0$, in the $\CCC^0$ topology and
\[
\begin{aligned}
|\partial_\tau h_{n,1}| & \lesssim \OO(\de), & |\partial_{\varphi} h_{n,1}| & \lesssim \OO(\tkk \de), \\
|\partial_\tau h_{n,2}| & \lesssim \OO(\de)+\OO(\tkk) , & |\partial_{\varphi} h_{n,2}| & \lesssim \OO(\tkk).
\end{aligned}
\]
Analogously,
for any \emph{vertical surface},
\[
S_v = \{(p,\tau,\varphi,J) \in \QQQ_\de^2\mid (\tau,\varphi) =  (v_1( p, J ),
v_2( p, J )), \; ( p, J ) \in (0,\de)\times(0,\tkk)\},
\]
 with $v$ a $\CCC^1$ function satisfying
\begin{equation}
\label{cond:verticalband}
\begin{aligned}
\sup_{( p, J )\in (0,\de)\times(0,\tkk)} \|\partial_p v_1(\tau,\varphi)\| & < \OO(1),
\sup_{( p, J )\in (0,\de)\times(0,\tkk)} \|\partial_J v_1(\tau,\varphi)\| & < \OO(\de),
\end{aligned}
\end{equation}
$\wt \Psi^{-1}(S_v) \cap \QQQ_\de^2$ has an infinite number of connected components. Moreover, $\wt \Psi^{-1}(S_v) \cap \QQQ_\de^2$ contains a countable union of vertical surfaces, $S_{v_n}$, with $\lim_{n\to \infty} v_n = 0$, in the $\CCC^0$ topology and
\[
\begin{aligned}
|\partial_p v_{n,1}| & \lesssim \OO(\de), & |\partial_{J} h_{n,1}| & \lesssim \OO(\tkk \de), \\
|\partial_p h_{n,2}| & \lesssim \OO(1) , & |\partial_{J} h_{n,2}| & \lesssim \OO(\tkk).
\end{aligned}
\]
In particular, $\Lip h_n \lesssim \OO(\de)+\OO(\tkk)$ and $\Lip v_n \lesssim \OO(1)$, uniformly in $n$.
\end{lemma}

\begin{proof}
Let $h:(0,\de)\times(0,\tkk) \to (0,\de)\times(0,\tkk)$ be a function satisfying~\eqref{cond:horizontalband}. Let
\begin{equation}\label{def:lambda0}
\Lambda_0 (\tau,z_1) = (h_1(\tau,z_1),\tau, z_1, h_2(\tau,z_1))^\top\qquad\text{and}\qquad\wt \Lambda_1 = \wt \Psi_{2,1} \circ \Lambda_0 = (\tilde h_1, \wt T, \wt Z)^\top.
\end{equation}
By Item 1 of Theorem~\ref{prop:lambdalemma}, the definition of $\wt \Psi_{2,1}$ in~\eqref{def:wtPsiij} and the expression of $\wtPsig{2}$ in~\eqref{def:globalmapiinlocalcoordinates}, we have that
\begin{equation}
\label{def:wtTmesgranquetauunters}
\wt T(\tau,z_1) \gtrsim \tau^{-3/2+Ca},
\end{equation}
for all $z_1 \in (0,\tkk)$. Hence, for any $n\in \NN$, sufficiently large, there exist $\tau_{1,n}^- < \tau_{1,n}^+$ such that
$\wt T(\tau_{1,n}^+, z_1) \le n < n+\de \le \wt T(\tau_{1,n}^-, z_1)$, for all $z_1 \in (0,\tkk)$ and, moreover, $\tau_{1,n}^\pm\to 0$ as $n\to+\infty$.

By Item 1 of Lemma~\ref{lem:almostinvariantvectorfields},
\begin{equation}
\label{eq:ferradura_derivada_tau}
\partial_\tau \wt \Lambda_1(\tau,z_1) = D \wt \Psi_{2,1} \circ \Lambda_0 (\tau,z_1)
\begin{pmatrix}
\partial_\tau h_1(\tau,z_1) \\ 1 \\ 0 \\ \partial_\tau h_2(\tau,z_1)
\end{pmatrix}
= \begin{pmatrix}
\lambda(\tau,z_1) \varepsilon_1(\tau,z_1) \\ \lambda(\tau,z_1)  \\ Z_1(\tau,z_1) \\ Z_2(\tau,z_1)
\end{pmatrix}
\end{equation}
where,
\begin{equation}
\label{bound:fitesexpansio}
|\lambda(\tau,z_1)| \gtrsim C \tau^{-3/5+Ca}, \qquad |\varepsilon_1 (\tau,z_1) | \lesssim \tau^{1-Ca}, \qquad \|(Z_1(\tau,z_1),Z_2(\tau,z_1))\| \lesssim \OO(1).
\end{equation}
In particular, $|\partial_\tau \wt T(\tau,z_1)| = |\lambda(\tau,z_1)| > C \tau^{-3/5+Ca}$. Hence, the equation
\[
\wt T(\tau,z_1) = T
\]
defines a function $\hat \tau_1(T,z_1)$, for $T>0$ large enough and  $z_1 \in (0,\tkk)$, such that
$\tau_{1,n}^- \le \hat \tau_1(T,z_1) \le \tau_{1,n}^+$ for $T\in [n,n+\de]$, with $\partial_T \hat \tau_1(T(\tau,z_1),z_1) = \lambda^{-1} (\hat \tau_1(T,z_1),z_1)$.
Taking $\tau_{1,n}^+$ small enough by taking $n$ large enough, we can assume that $\hat \tau_1(T,z_1)^{1/5-Ca} \lesssim \OO(\de)$ for  $T\in [n,n+\de]$. Moreover $\sup_{T\in[n,n+\de]}\hat \tau_1(T,z_1)\to 0$ as $n\to +\infty$.

We define
\[
\Lambda_1(T,z_1) = \wt \Lambda_1 (\hat \tau_1(T,z_1), z_1).
\]
By construction, $\pi_\tau \Lambda_1(T,z_1) = T$ and, by Item 1 of Theorem \ref{prop:lambdalemma}, $\sup_{T\in[n,n+\de]}\pi_p \Lambda_1(T,z_1)\to 0$ as $n\to+\infty$.
In view of~\eqref{eq:ferradura_derivada_tau}, it satisfies
\[
\partial_T \Lambda_1(T,z_1) =
\partial_\tau \wt \Lambda_1(\hat \tau_1(T,z_1) ,z_1) \partial_T \hat \tau_1(T,z_1)
= \begin{pmatrix}
\OO(\de) \\ 1  \\ \OO(\de) \\ \OO(\de)
\end{pmatrix}.
\]
Moreover, by~\eqref{def:globalmapiinlocalcoordinates}  and Item~1 of Theorem~\ref{prop:lambdalemma}, the third component of $\Lambda_1$, satisfies
\begin{multline}\label{bound:componentz1deLambda1}
|\pi_{z_1} \Lambda_1(T,z_1) - \pi_{z_1} \wh \tS^2(z_1,h_2(\hat \tau_1(T,z_1),z_1))|\\
= |\pi_{z_1} \wt \Psi_{2,1} \circ \Lambda_0(\hat \tau_1(T,z_1),z_1) - \pi_{z_1} \wh \tS^2(z_1,h_2(\hat \tau_1(T,z_1),z_1))|
\lesssim \OO(\de).
\end{multline}
Now we compute the derivatives of $\Lambda_1$ with respect to $z_1$. We have that
\[
\begin{aligned}
\partial_{z_1} \Lambda_1(T,z_1) & = \partial_{z_1} \left[\wt \Lambda_1 (\hat \tau_1(T,z_1), z_1)\right] \\
& =
D \wt \Psi_{2,1} \circ \Lambda_0 (\hat \tau_1(T,z_1),z_1)
\left[ \partial_\tau \Lambda_0 (\hat \tau_1(T,z_1),z_1) \partial_{z_1} \hat \tau_1(T,z_1)+
 \partial_{z_1} \Lambda_0 (\hat \tau_1(T,z_1),z_1)\right].
\end{aligned}
\]
We write the vector $\partial_{z_1} \Lambda_0 (\hat \tau_1(T,z_1),z_1)$ as
\begin{equation}\label{def:splitpartialLambda0}
(\partial_{z_1} \Lambda_0) (\hat \tau_1(T,z_1),z_1) =
\begin{pmatrix}
\partial_{z_1} h_1 \\ 0 \\ 1 \\ \partial_{z_1} h_2
\end{pmatrix} =
\partial_{z_1} h_1
\begin{pmatrix}
1 \\  b \\  c_1 \\  c_2
\end{pmatrix} + \begin{pmatrix}
0 \\ d \\ 1 - \partial_{z_1} h_1 c_1 \\ \partial_{z_1} h_2 - \partial_{z_1} h_1 c_2
\end{pmatrix}-(d+\partial_{z_1} h_1 b) \begin{pmatrix}
0 \\ 1 \\ 0 \\ 0
\end{pmatrix},
\end{equation}
where $b$, $c=(c_1,c_2)$ are the functions given by Item 2 of Lemma~\ref{lem:almostinvariantvectorfields} such that
\[
|b(\tau,z_1)| \lesssim \OO(\de,\hat \tau_1^{3/5-Ca)}), \qquad \|(c_1,c_2)(\tau,z_1)\| \lesssim \OO(1)
\]
and
\[
D \wt \Psi_{2,1} \circ \Lambda_0 (\hat \tau_1(T,z_1),z_1) (1,b,c)^\top = \mu(\hat \tau_1,z_1) (1,b^*,c^*)^\top
\]
with
\[
|\mu(\hat \tau_1,z_1)| \lesssim \OO(\hat \tau_1^{3/5-Ca}), \qquad |b^*(\hat \tau_1,z_1)| \lesssim \OO(\de), \qquad \|(c_1^*,c_2^*)(\tau,z_1)\| \lesssim \OO(1)
\]
and $d = \OO(\hat \tau_1^{1/5-Ca})$ is given by Item 3 of Lemma~\ref{lem:almostinvariantvectorfields} satisfies
\[
D \wt \Psi_{2,1} \circ \Lambda_0 (\hat \tau_1(T,z_1),z_1) \begin{pmatrix}
0 \\ d \\ 1 - \partial_{z_1} h_1 c_1 \\ \partial_{z_1} h_2 - \partial_{z_1} h_1 c_2
\end{pmatrix} = \begin{pmatrix}
d_1^*  \\ d_2^*  \\ D \wh \tS^2 (z_1, h_2(\hat \tau_1,z_1))
\left( \begin{aligned}
1  - \partial_{z_1} h_1 c_1 \\ \partial_{z_1} h_2 - \partial_{z_1} h_1 c_2
\end{aligned} \right) + d_3^*
\end{pmatrix}
\]
with $\|(d_1^*, d_2^*,d_3^*)\| = \OO(\hat \tau_1^{1/5-Ca})$.

Then,  applying Item 1 of Lemma~\ref{lem:almostinvariantvectorfields} to $(d+\partial_{z_1} h_1 b)(0,1,0,0)^\top$, using formula \eqref{def:splitpartialLambda0} and the previous bounds for $\mu$, $b^*$ and $d_2^*$, we obtain
\begin{equation}
\label{bound:firststeptau1}
\pi_T D \wt \Psi_{2,1} \circ \Lambda_0 (\hat \tau_1(T,z_1),z_1) \partial_{z_1} \Lambda_0 (\hat \tau_1(T,z_1),z_1) =
\lambda(\hat \tau_1,z_1) (d+\partial_{z_1} h_1 b) + \OO(\hat \tau_1^{1/5-Ca}).
\end{equation}
By using the second part of  Item 1 of Lemma~\ref{lem:almostinvariantvectorfields}  (see  \eqref{def:v1hatstreching}),
\begin{equation}
\label{bound:firststeptau2}
\pi_T D \wt \Psi_{2,1} \circ \Lambda_0 (\hat \tau_1(T,z_1),z_1)
\partial_\tau \Lambda_0 (\hat \tau_1(T,z_1),z_1) \partial_{z_1} \hat \tau_1(T,z_1) = \lambda(\hat\tau_1,z_1)\left (1+\OO(\hat\tau_1^{3/5-Ca})\right)\partial_{z_1} \hat \tau_1(T,z_1).
\end{equation}
Now, since $\pi_T \Lambda_1(T,z_1) = T$, we have that $\pi_T \partial_{z_1} \Lambda_1(T,z_1) = 0$. Then, combining~\eqref{bound:firststeptau1} and~\eqref{bound:firststeptau2} we obtain that
\[
\partial_{z_1} \hat \tau_1(T,z_1) = \OO(\de).
\]
This bound is not good enough. In order to improve it, we introduce $A = \partial_\tau h_1 \partial_{z_1} \hat \tau_1 + \partial_{z_1} h_1
= \OO(\de)$ and rewrite $\partial_{z_1}\left[ \Lambda_0 (\hat \tau_1(T,z_1),z_1)\right]$ as
\[
\partial_{z_1}\left[  \Lambda_0 (\hat \tau_1(T,z_1),z_1)\right] =
A
\begin{pmatrix}
1 \\  \tilde b \\  \tilde c_1 \\  \tilde c_2
\end{pmatrix} + \begin{pmatrix}
0 \\ \tilde  d \\ 1 - A  \tilde c_1 \\ \partial_\tau h_2 \partial_{z_1} \hat \tau_1+\partial_{z_1} h_2 - A \tilde c_2
\end{pmatrix}-(\tilde d+A \tilde b-\partial_{z_1} \hat \tau_1) \begin{pmatrix}
0 \\ 1 \\ 0 \\ 0
\end{pmatrix},
\]
where, again, $\tilde b = \OO(\de)$, $\tilde c=(\tilde c_1,\tilde c_2) = \OO(1) $ are the
functions given by Item  2 of Lemma~\ref{lem:almostinvariantvectorfields} and $\tilde d = \OO (\hat \tau_1^{1/5 -Ca})$ is given by Item 3  of Lemma~\ref{lem:almostinvariantvectorfields}. Then,
\begin{multline*}
D \wt \Psi_{2,1} \circ \Lambda_0 (\hat \tau_1(T,z_1),z_1) \partial_{z_1} \left[\Lambda_0 (\hat \tau_1(T,z_1),z_1)\right] \\
= \mu A
\begin{pmatrix}
1 \\  \OO(\de) \\  \OO(1) \\  \OO(1)
\end{pmatrix} + \begin{pmatrix}
\OO(\de) \\   \OO(\de)  \\
D \wh\tS^2 (z_1, h_2(\hat \tau_1,z_1))
\left( \!\! \begin{array}{c}
1 - A  \tilde c_1 \\ \partial_\tau h_2 \partial_{z_1} \hat \tau_1+ \partial_{z_1} h_2 - A  \tilde c_2
\end{array} \!\! \right) + \OO(\de)
\end{pmatrix}
- (\tilde a+A \tilde b-\partial_{z_1} \hat \tau_1 ) \begin{pmatrix}
\lambda \OO(\de) \\ \lambda \\ \OO(1) \\ \OO(1)
\end{pmatrix},
\end{multline*}
where $|\mu(\hat \tau_1,z_1)| \lesssim \OO(\hat \tau_1^{3/5-Ca}) = \OO(\de)$. Hence, using again that $\pi_T \Lambda_1(T,z_1) = T$, we obtain that
$\tilde a+A \tilde b-\partial_{z_1} \hat \tau_1 = \OO(\de)/\lambda$. This implies that
\[
\pi_p \partial_{z_1} \Lambda_1 (T,z_1) = \mu A + \OO(\de) - (\tilde a+A \tilde b-\partial_{z_1} \hat \tau_1)\lambda \OO(\de) = \OO(\de).
\]
Summarizing,
\[
\partial_T \Lambda_1(T,z_1) =
 \begin{pmatrix}
\OO(\de) \\ 1  \\ \OO(\de) \\ \OO(\de)
\end{pmatrix}, \quad
\partial_{z_1} \Lambda_1 (T,z_1) = \begin{pmatrix}
\OO(\de) \\   \OO(\de)  \\
D \wh\tS^2 (z_1, h_2(\hat \tau_1,z_1))
\left( \!\! \begin{array}{c}
  1  \\ \partial_{z_1} h_2
\end{array} \!\! \right) + \OO(\de)
\end{pmatrix}.
\]
Now we proceed by induction, defining $\wt \Lambda_j = \wt \Psi_{1,1} \circ \Lambda_{j-1}$, for $2 \le j \le M$.
With the same argument, $\pi_\tau \wt \Lambda_j(\tau,z_1) = T$ defines a function $\hat \tau_j(T,z_1)$, with $T$ large enough, such that
$\Lambda_j(T,z_1) = \wt \Lambda_j(\hat \tau_j(T,z_1),z_1)$ satisfies
\[
\partial_T \Lambda_j(T,z_1)
= \begin{pmatrix}
\OO(\de) \\ 1  \\ \OO(\de) \\ \OO(\de)
\end{pmatrix}, \quad
\partial_{z_1} \Lambda_j (T,z_1) = \begin{pmatrix}
\OO(\de) \\   \OO(\de)  \\
D ((\wh\tS^1)^{j-1} \circ \wh\tS^2) (z_1, h_2(\hat \tau_j,z_1))
\left( \!\! \begin{array}{c}
  1  \\ \partial_{z_1} h_2
\end{array} \!\! \right) + \OO(\de)
\end{pmatrix}
\]
and
\begin{equation}
\label{bound:componentz1deLambdaj}
|\pi_{z_1} \Lambda_j(T,z_1) - \pi_{z_1} (\wh\tS^1)^{j-1} \circ \wh\tS^2(z_1,h_2(\hat \tau_j,z_1))|  \lesssim \OO(\de).
\end{equation}
Of course, the $\OO(\de)$ terms depend on $j$.

In the last step, we define $\wt \Lambda_{M+1}(\tau,z_1)  = \wt \Psi_{1,2} \circ \Lambda_{M}(\tau,z_1)$, defined for $\tau \in (0,\de)$. With the same argument to obtain~\eqref{def:wtTmesgranquetauunters}, we have that
\[
\wt T_{M+1}(\tau,z_1) = \pi_T \wt \Lambda_{M+1}(\tau,z_1) \ge C \tau^{-3/2+Ca}
\]
and, from~\eqref{eq:ferradura_derivada_tau} and~\eqref{bound:fitesexpansio},
\[
|\partial_\tau \wt T_{M+1}(\tau,z_1)| = |\lambda(\tau,z_1)| > C \tau^{-3/5+Ca}.
\]
Hence, the equation
\[
\wt T_{M+1}(\tau,z_1) = T
\]
defines a function $\hat \tau_{M+1}(T,z_1)$, for $T$ large enough and $z_1 \in  (0,\tkk)$, strictly decreasing in $\tau$ with $\lim_{T\to \infty} \hat \tau_{M+1}(T,z_1) = 0$, uniformly in $z_1$, with $\partial_T \hat \tau_{M+1}(T(\tau,z_1),z_1) = \lambda (\hat \tau_{M+1}(T,z_1),z_1)^{-1}$. With the previous arguments, $\Lambda_{M+1}(T,z_1) = \wt \Lambda_{M+1}(\hat \tau_{M+1}(T,z_1),z_1)$
\begin{equation}
\label{bound:componentz1deLambdaMm1}
|\pi_{z_1} \Lambda_{M+1}(T,z_1) - \pi_{z_1} \wh \tS (z_1,h_2(\hat \tau_{M+1},z_1))|  \lesssim \OO(\de),
\end{equation}
and
\begin{equation}
\label{bound:derivadesdelaimatgedunabandahoritzontal}
\partial_T \Lambda_{M+1}(T,z_1)
= \begin{pmatrix}
\OO(\de) \\ 1  \\ \OO(\de) \\ \OO(\de)
\end{pmatrix}, \;
\partial_{z_1} \Lambda_{M+1} (T,z_1) = \begin{pmatrix}
\OO(\de) \\   \OO(\de)  \\
D \wh \tS (z_1, h_2(\hat \tau_{M+1},z_1))
\left( \!\! \begin{array}{c}
  1  \\ \partial_{z_1} h_2
\end{array} \!\! \right)
 + \OO(\de)
\end{pmatrix},
\end{equation}
where $\wh\tS=(\wh\tS^1)^M\circ \wh\tS^2$ was introduced in of Theorem~\ref{thm:BlockScattering}. Since, by Item 3 of Theorem~\ref{thm:BlockScattering},
\begin{equation}
\label{bound:derivadesscattering}
\left|\pi_{z_1} D \wh \tS (z_1, h_2(\hat \tau_{M+1},z_1))
\begin{pmatrix}
  1  \\ \partial_{z_1}  h_2
\end{pmatrix} \right| \gtrsim \tkk^{-1},
\end{equation}
the equation $\pi_{z_1} \Lambda_{M+1}(T,z_1) = Z_1$ defines a function $\hat z_{M+1}(T,Z_1)$, with $T$ large enough and $Z_1\in(0,\tkk)$. Using~\eqref{bound:derivadesdelaimatgedunabandahoritzontal} and~\eqref{bound:derivadesscattering}, we immediately have
\[
|\partial_T \hat z_{M+1}(T,Z_1)| \lesssim \OO(\de \tkk), \qquad |\partial_{Z_1} \hat z_{M+1}(T,Z_1)| \lesssim \OO(\tkk).
\]
Then,  $\Lambda (T,Z_1) = \Lambda_{M+1} (T, \hat z_{M+1}(T,Z_1))$ is defined for $T$ large enough and $Z_1\in (0,\tkk)$, satisfies
\[
\pi_T \Lambda (T,Z_1) = T, \qquad \pi_{Z_1} \Lambda (T,Z_1) = Z_1,
\]
and, denoting $\tilde h = (\tilde h_1, \tilde h_2) = (\pi_p \Lambda, \pi_{z_2} \Lambda)$,
\[
\partial_T \tilde h =
\begin{pmatrix}
\OO(\de) \\ \OO(\de)+ \OO(\tkk)
\end{pmatrix},
\qquad
\partial_{Z_1} \tilde h =
\begin{pmatrix}
\OO(\de \tkk) \\ \OO(\tkk)
\end{pmatrix}.
\]
Moreover, if we denote by $\tilde h_n=(\tilde h_{n,1},\tilde h_{n,2})$ the restriction of $\tilde h$ to $T\in [n,n+\de]$, it satisfies that
\[
\lim_{n\to \infty} h_{n,1} = 0.
\]

This proves the claim for the horizontal bands. The claim for the vertical one is proven along the same lines.
\end{proof}

\subsection{The differential of the high iterate of the return map}\label{sec:conefields}

In this section we analyze the differential of the map $\wt\Psi$ in \eqref{def:Psitilde}. Note that this map is a composition of the return maps $\wt\Psi_{i,j}$ (see \eqref{def:wtPsiij}), whose differentials have been  studied in Section \ref{sec:differentialintermediatereturn}. In that section we have obtained a good basis at each point of the tangent space which captures the expanding, contracting and center direction for each map  $\wt\Psi_{i,j}$. Note however that the basis depends on the map (and certainly on the point!). Therefore, one has to adjust these basis so that they capture the expanding/contracting behavior for the differential of $\wt\Psi$. This is done in Proposition \ref{prop:straighteningDPsi} below.

First we state a lemma, which is an immediate consequence of  Lemma~\ref{lem:almostinvariantvectorfields} and a direct product of matrices.

\begin{lemma}
\label{lem:redressanttildepsiij}
Consider the vector fields
$\{v_1,v_2^{i,j},v_3^{i,j},v_4^{i,j}\}$, $i=1,2$, where $v_1$ and $v_2^{i,j}$ are the vectors in Items 1 and 2 of Lemma~\ref{lem:almostinvariantvectorfields} and
\[
v_3^{i,j} = e_3+ \tilde v_3^{i,j}, \qquad v_4^i = e_4+ \tilde v_4^{i,j},
\]
where $e_3 = (0,0,1,0)^\top$, $e_4 = (0,0,0,1)^\top$ and $\tilde v_3^{i,j}$ and $\tilde v_3^{i,j}$ are such $v_3^{i,j}$ and $v_4^{i,j}$ satisfy Item 3 of
    Lemma~\ref{lem:almostinvariantvectorfields}.
    They form a basis of $T_\omega \QQQ_\de^i$, at any $\omega=(p,\tau,z)  \in \QQQ_\de^i \cap \Psi_{i,j}^{-1}(\QQQ_\de^j)$. Let $C_{i,j}(\omega)$ denote the matrix of the change of coordinates from the standard basis to $\{v_1,v_2^{i,j},v_3^{i,j},v_4^{i,j}\}$. Then,
\begin{equation}
\label{eq:DPhi1adapted basis_1}
\begin{aligned}
\M_{2,1}(\omega) = C_{1,1}(\wt \Psi_{2,1}(\omega))^{-1} D\wt \Psi_{2,1}(\omega) C_{2,1}(\omega) & =
\begin{pmatrix}
\lambda^{2,1} & \mu^{2,1} \eps_{1}^{2,1} & \eps_{2}^{2,1}  \\
\lambda^{2,1} \eps^{2,1}_{3} & \mu^{2,1}  &  \eps_{4}^{2,1}  \\
\lambda^{2,1}  \eps_{5}^{2,1} & \mu^{2,1} (a_{2,1}+ \eps_{6}^{2,1}) & D\tS^2(z) +  \eps_{7}^{2,1}
\end{pmatrix}, \\
\M_{1,1}(\omega) =C_{1,1}(\wt \Psi_{1,1}(\omega))^{-1} D\wt \Psi_{1,1}(\omega) C_{1,1}(\omega) & =
\begin{pmatrix}
\lambda^{1,1} & \mu^{1,1} \eps_{1}^{1,1} & \eps_{2}^{1,1}  \\
\lambda^{1,1} \eps^{1,1}_{3} & \mu^{1,1}  &  \eps_{4}^{1,1}  \\
\lambda^{1,1}  \eps_{5}^{1,1} & \mu^{1,1}  \eps_{6}^{1,1} & D\tS^1(z) +  \eps_{7}^{1,1}
\end{pmatrix}, \\
\wt \M_{1,1}(\omega) = C_{1,2}(\wt \Psi_{1,1}(\omega))^{-1} D\wt \Psi_{1,1}(\omega) C_{1,1}(\omega) & =
\begin{pmatrix}
\tilde \lambda^{1,1} & \tilde \mu^{1,1} \delta_{1}^{1,1} & \delta_{2}^{1,1}  \\
\tilde \lambda^{1,1} \delta^{1,1}_{3} & \tilde \mu^{1,1}  &  \delta_{4}^{1,1}  \\
\tilde \lambda^{1,1} \delta_{5}^{1,1} & \tilde \mu^{1,1} (a_{1,2}+ \delta_{6}^{1,1}) & D\tS^1(z) +  \delta_{7}^{1,1}
\end{pmatrix}, \\
\M_{1,2}(\omega) = C_{2,1}(\wt \Psi_{1,2}(\omega))^{-1} D\wt \Psi_{1,2}(\omega) C_{1,2}(\omega) & =
\begin{pmatrix}
\lambda^{1,2} & \mu^{1,2} \eps_{1}^{1,2} & \eps_{2}^{1,2}  \\
\lambda^{1,2} \eps^{1,2}_{3} & \mu^{1,2}  &  \eps_{4}^{1,2}  \\
\lambda^{1,2}  \eps_{5}^{1,2} & \mu^{1,2} (a_{1,1}+ \eps_{6}^{1,2}) & D\tS^1(z) + \eps_{7}^{1,2}
\end{pmatrix},
\end{aligned}
\end{equation}
where
\[
\begin{aligned}
|\eps^{i,j}_k(\omega)|,  |\delta^{1,1}_k(\omega)|, & \le \OO(\tau^{\frac{1}{5}-Ca}, \de), \qquad k=1,\dots, 7,\\
\lambda^{i,j}(\omega), \tilde \lambda^{1,1}(\omega), \mu^{i,j}(\omega)^{-1}, \tilde \mu^{1,1}(\omega)^{-1} & \gtrsim \tau^{-\frac{3}{5}+Ca}
\end{aligned}
\]
and $a_{i,j} = a_{i,j}(z)$, $i,j = 1,2$ satisfy $|a_{i,j}|\leq\OO(1)$.
\end{lemma}
%

This lemma provides formulas for the differential of the intermediate return maps in ``good basis''. The next proposition provides a good basis for the high iterate of the return map $\wt\Psi$ in  \eqref{def:Psitilde}.

\begin{proposition}\label{prop:straighteningDPsi}
Consider $\tkk$ given by Theorem~\ref{thm:BlockScattering} and $\de>0$ small enough. Consider also   the map $\wt \Psi= \wt \Psi_{1,2}\circ \wt \Psi_{1,1}^{M-1}\circ \wt \Psi_{1,1}$ defined in $U_1 \subset \QQQ_\de$.
There exists $C:\QQQ_\de \cap \wt \Psi^{-1}(\QQQ_\de) \to \M_{4\times 4}(\RR)$ of the form
\[
C(\omega) = C_{2,1}(\omega)  \wt C(\omega) C_{\wh \tS} (\omega)
\]
where $C_{2,1}(\omega)$ is the matrix introduced in Lemma \ref{lem:redressanttildepsiij},
\[
\wt C(\omega) =
\begin{pmatrix}
1 & a(\omega) & b (\omega) \\
0 & 1 & 0 \\
0 & 0 & \Id
\end{pmatrix}, \qquad a(\omega), b (\omega) = \OO(\tau^{3/5-Ca} \de^{1/5}),
\]
and, for $\omega=(p,\tau,z)$
\[
C_{\wh \tS} (\omega) =
\begin{pmatrix}
1 & 0 & 0 & 0 \\
0 & 1 & 0 & 0 \\
0 & 0 & 1 & V_{2,1}(z) \\
0 & 0 & 0 & 1
\end{pmatrix},
\]
where $V_{2,1}$ is given by 3 of Theorem~\ref{thm:BlockScattering},
such that
\begin{equation}
\label{eq:DPhi1adapted basis_M}
C(\wt \Psi(\omega))^{-1} D \wt \Psi(\omega) C(\omega) =
\begin{pmatrix}
\lambda & \mu \eps_5 & \eps_7 & \eps_{10} \\
\lambda \eps_2 & \mu (1+\eps_6) & \eps_8 & \eps_{11} \\
\lambda \eps_3 & \mu \tilde c_1 & \la_{\wh \tS}  & \la_{\wh \tS}\ii  \eps_{12} \\
\lambda \eps_4 & \mu \tilde c_2 & \la_{\wh \tS} \eps_{9} & \la_{\wh \tS}\ii
\end{pmatrix},
\end{equation}
where
\[
\lambda(\omega) \gtrsim \tau^{-\frac{3}{5}+\tilde Ca} \de^{-3M/5},\qquad \mu(\omega)^{-1}  \gtrsim \tau^{-\frac{3}{5}+\tilde Ca},
\qquad \tilde c_1, \tilde c_2 = \OO(1), \qquad |\eps_j|  \lesssim \OO(\de^{1/5}),
\]
for $j= 2, \dots,11$, $j \neq 9$,
\[
|\eps_9|,  |\eps_{12}|  \lesssim  \tkk
\]
and $\la_{\wh \tS} \ge \tkk\ii$ was introduced in Item 3 of Theorem~\ref{thm:BlockScattering}.

Analogously, there exists $\wh C:\QQQ_\de \cap \Psi(\QQQ_\de) \to \M_{4\times 4}(\RR)$ of the form
\[
\wh C (\wt \Psi(\omega)) = C_{2,1}(\wt \Psi(\omega)) \wt  C^*(\wt \Psi(\omega)) C_{\wh \tS}^* (\omega)
\]
with
\[
\wt C^*(\wt \Psi(\omega)) =
\begin{pmatrix}
1 & 0 & 0 \\
\tilde a(\omega) & 1 & \tilde b (\omega) \\
0 & 0 & \Id
\end{pmatrix}, \qquad \tilde a(\omega), \tilde b (\omega) = \OO(\de^{4/5-Ca}),
\]
and, for $\omega(p,\tau,z)$
\[
C_{\wh \tS}^* (\omega) =
\begin{pmatrix}
1 & 0 & 0 & 0 \\
0 & 1 & 0 & 0 \\
0 & 0 & 1 & V_{2,1}((\wh \tS)^{-1}(z)) \\
0 & 0 & 0 & 1
\end{pmatrix},
\]
such that
\begin{equation}
\label{eq:DPhi1inverseadapted basis_M}
\wh C(\omega)^{-1} D (\wt \Psi)^{-1}(\wt \Psi(\omega)) \wh C(\wt \Psi(\omega)) =
\begin{pmatrix}
\tilde \mu \tilde \eps_1 & \tilde \lambda \tilde \eps_5 & \tilde \eps_8 & \tilde \eps_{11}\\
\tilde \eps_2 & \tilde \lambda & \tilde \eps_9 & \tilde \eps_{12}\\
\tilde \eps_3 & \tilde \lambda \tilde \eps_6 & \la_{\wh \tS}\ii  & \la_{\wh \tS}\ii \tilde \eps_{13}\\
\tilde \eps_4 & \tilde \lambda \tilde \eps_7 & \la_{\wh \tS}\tilde  \eps_{10} & \la_{\wh \tS}
\end{pmatrix},
\end{equation}
where
\[
\tilde \lambda(\omega) \gtrsim \de^{-3(M+1)/5},\qquad \tilde  \mu(\omega)^{-1}  \gtrsim \de^{-\frac{3}{5}},
   \qquad |\tilde \eps_j|  \lesssim \OO(\de^{1/5}),
\]
for $j= 1, \dots,12$, $j \neq 10, 13$,
\[
|\tilde \eps_{10}|,  |\tilde \eps_{13}|  \lesssim \tkk,
\]
\end{proposition}

\begin{proof}
In view of Lemma~\ref{lem:redressanttildepsiij}, we write
\begin{equation}
\label{eq:redressanttildePsiprimerpas}
C_{2,1}(\tilde \Psi(\omega))^{-1} D \tilde \Psi (\omega) C_{2,1}(\omega) = \M_{M}(\omega) \cdots \M_0(\omega),
\end{equation}
where the matrices
\[
\begin{aligned}
\M_0(\omega) &  = \M_{2,1}(\omega) \\
\M_j(\omega) & = \M_{1,1} (\wt \Psi_{1,1}^{j-1} \circ \wt \Psi_{2,1}(\omega)), \qquad j= 1,\dots, M-2, \\
\M_{M-1}(\omega) & = \wt \M_{1,1}(\wt \Psi_{1,1}^{M-2} \circ \wt \Psi_{2,1}(\omega)), \\
\M_{M}(\omega) & = \M_{1,2}(\wt \Psi_{1,1}^{M-1} \circ \wt \Psi_{2,1}(\omega)),
\end{aligned}
\]
are given by Lemma~\ref{lem:redressanttildepsiij}.
The product of the matrices in~\eqref{eq:redressanttildePsiprimerpas}, in the current form, is difficult to control. We find an adapted basis in which this product of matrices has a more convenient expression. We proceed in the following way. We claim that, for any $0 \le j \le M$, there exists a matrix
\begin{equation}
\label{def:canviCj}
C_j = \begin{pmatrix}
1 & \alpha_j & \beta_j \\
0 & 1 & 0 \\
0 & 0 & \Id
\end{pmatrix}, \qquad  |\alpha_j|, \|\beta_j\| \le K_M \tau^{3/5-Ca} \de^{(1+3j)/5},
\end{equation}
where $K_M$ is a constant depending only on $M$, such that $\wt C_j = \wt C_{j-1} C_j$, $\wt C_{-1} = \Id$, satisfies
\begin{equation}
\label{eq:inducciocons}
\M_j \dots \M_0 \wt C_j =
\begin{pmatrix}
\tilde \lambda_j & 0 & 0 \\
\tilde \lambda_j \eps_{j,1} & \mu(1+ \eps_{j,2}) & \eps_{j,3} \\
\tilde \lambda_j \eps_{j,4} & \mu c_j & \wt \SSS_j + \eps_{j,5}
\end{pmatrix}
\end{equation}
with
\begin{equation}
\label{def:SSSj}
\wt \SSS_j (z) = D((\wh\tS^1)^{j-1} \circ\wh\tS^2) (z)
\end{equation}
and
\begin{equation}
\label{bound:elementsproducteMj}
\tilde \lambda_j \gtrsim \tau^{-3/5+Ca} \de^{-3(j-1)/5}, \qquad \mu^{-1} \gtrsim \tau^{-3/5+Ca}, \qquad \eps_{j,k} = \OO(\de^{1/5}), \qquad c_j = \OO(1).
\end{equation}
The constants involved in the above equalities depend only on $M$.

We prove this claim by induction.
The case $j=0$ follows from the expression of $\M_0 =\M_{2,1}$ given Lemma~\ref{lem:redressanttildepsiij}, taking
$\wt C_0=C_0$ as in \eqref{def:canviCj} with
\[
\tilde \lambda _0 =  \lambda^{2,1}, \qquad \mu = \mu^{2,1}, \qquad  \alpha_0 =  - \frac{\mu^{2,1} \eps^{2,1}_1}{\lambda^{2,1}},
\qquad \beta_0 = - \frac{\eps^{2,1}_2}{\lambda^{2,1}}.
\]
Now assume that~\eqref{eq:inducciocons} holds for $j-1$, with $1 \le j\le M-2$, that is, there exists $\tilde C_{j-1}$ such that
\begin{multline*}
\M_j \left(\M_{j-1} \dots \M_0 \tilde C_{j-1}\right) \\
 = \begin{pmatrix}
\lambda_j & \mu_j \tilde \eps_{j,1} & \tilde \eps_{j,2}  \\
\lambda_j \tilde \eps_{j,3} & \mu_j  & \tilde \eps_{j,4}  \\
\lambda_j \tilde \eps_{j,5} & \mu_j \tilde \eps_{j,6} & \SSS_j + \tilde \eps_{j,7}
\end{pmatrix}
\begin{pmatrix}
\tilde \lambda_{j-1} & 0 & 0 \\
\tilde \lambda_{j-1} \eps_{j-1,1} & \mu (1+\eps_{j-1,2}) & \eps_{j-1,3} \\
\tilde \lambda_{j-1} \eps_{j-1,4} & \mu c_{j-1} & \wt \SSS_{j-1} + \eps_{j-1,5}
\end{pmatrix}
\end{multline*}
with
\begin{equation}\label{def:mulambdaSj}
\begin{split}
(\lambda_j, \mu_j) &= (\lambda^{1,1},\mu^{1,1}) \circ \wt \Psi_{1,1}^{j-1} \circ \wt \Psi_{2,1} \\
\SSS_j& =D\wh\tS^1 \circ \pi_z\wt \Psi_{1,1}^{j-1} \circ \wt \Psi_{2,1} \\
\tilde \eps_{j,k}& = \eps^{1,1}_k \circ \wt \Psi_{1,1}^{j-1} \circ \wt \Psi_{2,1}, \qquad k=1,\dots,7.
\end{split}
\end{equation}
where $\pi_z$ is the projection onto the $z$ component. Observe that, by hypothesis, $\wt \Psi_{1,1}^{j-1} \circ \wt \Psi_{2,1}(\omega)\in \mathcal{Q}_\de$ and therefore $\pi_\tau \wt \Psi_{1,1}^{j-1} \circ \wt \Psi_{2,1}(\omega), \pi_p\wt \Psi_{1,1}^{j-1} \circ \wt \Psi_{2,1}(\omega)\in (0,\de)$. In particular, by Lemma~\ref{lem:redressanttildepsiij},
\[
 |\la_j|,|\mu_j|\ii\gtrsim \de^{-3/5}\qquad \text{and}\qquad |\tilde\eps_{j,k}|\lesssim\de^{1/5}.
\]
Then, the elements of the top row of $\M_j \M_{j-1} \dots \M_0 \tilde C_{j-1}$ are
\begin{equation}\label{def:lambdajtilde}
\begin{split}
\tilde \lambda_j & = \lambda_j \tilde \lambda_{j-1} \left( 1+\frac{1}{\lambda_j}\left(\mu_j \tilde \eps_{j,1} \eps_{j-1,1}+
\tilde \eps_{j,2}\eps_{j-1,4}\right)\right) \\
& \gtrsim \de^{-3/5} \tau^{-3/5+Ca} \de^{-3(j-2)/5}\gtrsim\tau^{-3/5+Ca} \de^{-3(j-1)/5},
\end{split}
\end{equation}
and
\[
D_1  = \mu \left(\mu_j \tilde \eps_{j,1}(1+ \eps_{j-1,2}) +  \tilde \eps_{j,2}  c_{j-1}\right) ,  \qquad
D_2  = \mu_j \tilde \eps_{j,1} \eps_{j-1,3} + \tilde \eps_{j,2}(\wt \SSS_{j-1} + \eps_{j-1,5}).
\]
Hence, taking
\[
\alpha_j = - \frac{D_1}{\tilde \lambda_{j}} = \OO\left(\tau^{2(3/5-Ca)} \de^{(3j+1)/5} \right), \qquad \beta_j = - \frac{D_2}{\tilde \lambda_{j}} = \OO\left(\tau^{3/5-Ca} \de^{(3j+1)/5}\right),
\]
we have that
\[
\M_j \M_{j-1} \dots \M_0 \tilde C_{j-1} C_j = \begin{pmatrix}
\tilde \lambda_j & 0 & 0 \\
\tilde \lambda_j \eps_{j,1} & \mu (1+ \eps_{j,2} )& \eps_{j,3} \\
\tilde \lambda_j \eps_{j,4} & \mu c_j & \wt \SSS_j + \eps_{j,5}
\end{pmatrix},
\]
for some $\eps_{j,k}$, $k=1\ldots 5$ and $c_j$.

Note that $\eps_{j,1}$ and $\eps_{j,4}$ do not depend on the choice of $\al_j$ and $\bet_j$. Indeed, using the equality in the first row in \eqref{def:lambdajtilde}, they satisfy
\[
 \begin{aligned}
\eps_{j,1} & = \frac{\lambda_j \tilde \lambda_{j-1}}{\tilde \lambda_{j}} \left(
\tilde \eps_{j,3} + \frac{1}{\lambda_j}(\mu_j \eps_{j-1,1}+\tilde \eps_{j,4} \eps_{j-1,4})
\right)  = \OO(1) \OO(\de^{1/5}),
\\
\eps_{j,4} & = \frac{\lambda_j \tilde \lambda_{j-1}}{\tilde \lambda_{j}} \left(
\tilde \eps_{j,5} + \frac{1}{\lambda_j}(\mu_j \tilde \eps_{j,6}\eps_{j-1,1}+(\SSS_j+\tilde \eps_{j,7}) \eps_{j-1,4})
\right)  = \OO(1) \OO(\de^{1/5}). \end{aligned}
\]
Clearly,  $\alpha_j$ and $\beta_j$ satisfy the inequalities in~\eqref{def:canviCj}.
Moreover,
\begin{equation}
\label{bound:D1D2}
\tilde \lambda_j \eps_{j,1} \alpha_j = -D_1 \eps_{j,1} =\mu \OO(\de^{1/5}), \qquad \tilde \lambda_j \eps_{j,1} \beta_j = -D_2 \eps_{j,1}= \OO(\de^{1/5}).
\end{equation}
The bounds of the elements of $\M_j \M_{j-1} \dots \M_0 \tilde C_{j-1} C_j$ can be computed immediately from  Lemma~\ref{lem:redressanttildepsiij} and the induction hypotheses. Indeed,
\[
\begin{aligned}
\eps_{j,2} & = \frac{\tilde \lambda_j \eps_{j,1} \alpha_j}{\mu} + \mu_j (1+\eps_{j-1,2})+ \tilde
\eps_{j,4} c_{j-1}
= \OO(\de^{1/5}), \\
\eps_{j,3} & = \tilde \lambda_j \eps_{j,1} \beta_j + \mu_j \eps_{j-1,3}+\tilde \eps_{j,4} (\wt \SSS_{j-1}+ \eps_{j-1,5}) = \OO(\de^{1/5}), \\
c_j & = \frac{\tilde \lambda_j \eps_{j,4} \alpha_j}{\mu} + \mu_j \tilde \eps_{j,6} (1+\eps_{j-1,2}) + (\SSS_j + \tilde \eps_{j,7}) c_{j-1} = \OO(1), \\
\eps_{j,5} & = - \wt \SSS_j+ \tilde \lambda_j \eps_{j,4} \beta_j +(\SSS_j+ \tilde \eps_{j,7})(\wt \SSS_{j-1}+\eps_{j-1,5})+\mu_j\tilde\eps_{j,6}\eps_{j-1,3}
 = \SSS_j \wt \SSS_{j-1}- \wt \SSS_j + \OO(\de^{1/5})  \\
\end{aligned}
\]
where we have used~\eqref{bound:D1D2} in the bounds of $\eps_{j,2}$, $\eps_{j,3}$, $c_j$.

To get small estimates for $\eps_{j,5}$ note that, by  Statement 1 of Proposition~\ref{prop:lambdalemma}, one has that for $j=1\ldots M-2$ and $\omega=(p,\tau,z)\in\QQQ_\de$,
\[
\left|\pi_z\wt \Psi_{1,1}^{j-1} \circ \wt \Psi_{2,1}(\omega) -\wh\tS_1^{j-1}\circ \wh\tS_2(z)\right|\leq \OO(\de).
\]
Then, by the definition of $\SSS_j $ and $\wt\SSS_j$ in \eqref{def:mulambdaSj} and \eqref{def:SSSj} respectively,
\[
\left|  \SSS_j \wt \SSS_{j-1}- \wt \SSS_j\right|\leq \OO(\de).
\]
This implies that
$ |\eps_{j,5}|\leq \OO(\de^{1/5})$.

The obtained estimates prove the claim for $0\le j \le M-2$. The cases $j=M-1,M$ can be treated exactly in the same way.
They only differ in $\eps_{j,4}$, $c_j$ and $\eps_{j,5}$, where $\tilde \eps_{j-1,6}$
should be substituted by some $\tilde c_j = \OO(1)$ coming from Lemma~\ref{lem:redressanttildepsiij}. Their final bounds remain the same.

Observe that, from~\eqref{def:canviCj},
\[
\wt C_M = C_0 \cdots C_M =
\begin{pmatrix}
1 & \sum_{0\le j \le M} \alpha_j & \sum_{0\le j \le M} \beta_j \\
0 & 1 & 0 \\
0 & 0 & \Id
\end{pmatrix}
\]
where, from the bounds in~\eqref{def:canviCj},
\[
\left| \sum_{0\le j \le M} \alpha_j \right|, \left\| \sum_{0\le j \le M} \beta_j \right\|\le \OO(\tau^{3/5-Ca}\de^{1/5}).
\]
Hence, from~\eqref{eq:inducciocons} with $j=M$, we have that, for any $\omega \in \QQQ_\de \cap \wt \Psi^{-1}(\QQQ_\de) $,
\begin{multline*}
\wt C_M^{-1} ( \Psi (\omega)) C_{2,1}(\tilde \Psi(\omega))^{-1} D \tilde \Psi (\omega) C_{2,1}(\omega) \wt C_M(\omega) \\
\begin{aligned}
& = \begin{pmatrix}
1 & -\sum_{0\le j \le M} \alpha_j & -\sum_{0\le j \le M} \beta_j \\
0 & 1 & 0 \\
0 & 0 & \Id
\end{pmatrix}
\begin{pmatrix}
\tilde \lambda_M & 0 & 0 \\
\tilde \lambda_M \eps_{M,1} & \mu (1+\eps_{M,2}) & \eps_{M,3} \\
\tilde \lambda_M \eps_{M,4} & \mu c_M & \wt \SSS_M + \eps_{M,5}
\end{pmatrix}
\\
& = \begin{pmatrix}
\tilde \lambda_M (1+ \OO(\tau^{3/5-Ca}\de^{2/5})) & \mu \OO(\tau^{3/5-Ca}\de^{1/5}) & \OO(\tau^{3/5-Ca}\de^{1/5}) \\
\tilde \lambda_M \eps_{M,1} & \mu (1+\eps_{M,2}) & \eps_{M,3} \\
\tilde \lambda_M \eps_{M,4} & \mu c_M & \wt \SSS_M + \eps_{M,5}
\end{pmatrix}.
\end{aligned}
\end{multline*}
Finally, the claim follows from the properties of $V_{2,1}$ in Theorem~\ref{thm:BlockScattering}.

The proof of~\eqref{eq:DPhi1inverseadapted basis_M} is completely analogous, considering matrices $D_j$ of the form
\[
D_j = \begin{pmatrix}
1 & 0 & o \\
\alpha_j & 1 & \beta_j \\
0 & 0 & \Id
\end{pmatrix},
\]
which are also a group.
\end{proof}

\subsection{Stable and unstable cone fields: Proof of Proposition~\ref{prop:condicionsdecons}}
\label{sec:provaproposiciodelscons}
The analysis of the differential of the map $\wt\Psi$ (see \eqref{def:Psitilde}) performed in Section \ref{sec:conefields} allows us to set up stable and unstable cone fields and  prove Proposition~\ref{prop:condicionsdecons}.

The only difficulty in the study  of the behavior of the cone fields is that, of the two expanding factors $\lambda, \la_{\wh \tS} $ (see Proposition~\ref{prop:straighteningDPsi}), which satisfy $\lambda > \la_{\wh \tS}>1$, $\lambda$ is unbounded in $\QQQ_\de$ (as $\tau\to 0$).
Hence the error terms $\lambda \eps_i$ in~\eqref{eq:DPhi1adapted basis_M} are not necessarily small and, in fact, can be large.

We start by considering the unstable cone field $S^u_{\omega,\kappa_u}$, introduced in~\eqref{def:coninestable}.

Let $\M$ be the matrix in the left hand side of~\eqref{eq:DPhi1adapted basis_M}, in Proposition~\ref{prop:straighteningDPsi}, that is, such that
\[
C(\wt \Psi(\omega))^{-1} D \wt \Psi(\omega) C(\omega) = \M(\omega).
\]

For $x \in S^u_{\omega,\kappa_u}$,
let $y = C(\omega)x$, where $C$ is given by Lemma~\ref{prop:straighteningDPsi}. We will denote $y=(y_u,y_s)$, where $y_u = (y_1,y_3)$ and $y_s = (y_2,y_4)$, and use the norms $\|y_u\| = \max\{|y_1|,|y_3| \}$ and $\|y_s\| = \max\{|y_2|,|y_4| \}$. Given $\kappa >0$, we denote
\[
\wt S^u_{\omega,\kappa} = \{y; \; \|y_s\| \le \kappa \|y_u\| \}.
\]
We also denote $(\M y)_u = \M_{u,u}y_u+ \M_{u,s}y_s$ and $(\M y)_s = \M_{s,u}y_u+ \M_{s,s}y_s$, where
\begin{equation}
\label{def:Msuuuusss}
\begin{aligned}
\M_{u,u} & = \begin{pmatrix}
\lambda & \eps_7 \\
\lambda \eps_3 & \la_{\wt \tS}
\end{pmatrix}, & \quad \M_{u,s} & = \begin{pmatrix}
\mu \eps_5 & \eps_{10} \\
\mu \tilde c_1 & \la_{\wh \tS}\ii  \eps_{12}
\end{pmatrix}, \\
 \M_{s,u} & = \begin{pmatrix}
\lambda \eps_2 & \eps_8 \\
\lambda \eps_4 & \la_{\wh \tS} \eps_{9}
\end{pmatrix}, & \quad \M_{s,s} & = \begin{pmatrix}
\mu (1+\eps_6) & \eps_{11} \\
\mu \tilde c_2 & \la_{\wh \tS}\ii
\end{pmatrix}.
\end{aligned}
\end{equation}
with $|\eps_i| \lesssim \OO(\de^{1/5})$, $i \neq 9, 12$ and $|\eps_9|,|\eps_{12}| \lesssim \OO(\tkk)$.

Note that the matrix $C$ in Proposition \ref{prop:straighteningDPsi} has been defined in Proposition as the product of three matrices. The form of $\wt C$ and $C_{\wh\tS}$ has been given in Proposition \ref{prop:straighteningDPsi}.
Lemmas \ref{lem:almostinvariantvectorfields} and  \ref{lem:redressanttildepsiij}  imply that
$C_{2,1}$ is of the form
\[
C_{2,1} =
\begin{pmatrix}
0 & 1 & 0 & 0 \\
1 & \OO(\de) & \OO(\de)& \OO(\de) \\
0 & \tilde a & 1 & 0 \\
0 & \tilde b & 0 & 1
\end{pmatrix}.
\]
Let $\hat a = \sup_{\omega \in \QQQ_\de} |\tilde a( \omega)|\lesssim \OO(1)$ and $\hat b = \sup_{\omega \in \QQQ_\de} |\tilde b( \omega)|\lesssim \OO(1)$.

Then, it is immediate to check that
\begin{equation}
\label{bound:deconsacons}
C(\omega) S^u_{\omega,\kappa_u} \subset \wt S^u_{\omega,\wt \kappa_u}, \qquad \wt \kappa_u \geq \frac{(1+\OO(\de^{1/5}))\kappa_u + \hat b }{1+\hat a} = \OO(1),
\end{equation}
and, for $\beta >0$, small enough,
\begin{equation}
\label{bound:deconsaconsinvers}
C(\omega)^{-1} \wt S^u_{\omega,\beta}  \subset S^u_{\omega,\wt \beta}, \qquad \wt \beta \le \frac{(1+\hat b)\beta}{1-\OO(\de^{1/5})-(1+\hat a+\OO(\de))\beta } = \OO(\beta).
\end{equation}

We claim that
\begin{enumerate}
\item
if $y \in \wt S^u_{\omega,\wt \kappa_u}$, $\|(\M y)_u \| \ge \la_{\wh \tS} (1+\OO(\de^{1/5}))\|y_u\|$ and
\item
$\M(\omega) \wt S^u_{\omega,\wt \kappa_u} \subset \wt S^u_{\omega,\hat \kappa_u}$, with $\hat \kappa_u = \OO(\de^{1/5}) +  \OO(\tkk)$.
\end{enumerate}

Indeed, let $y \in \wt S^u_{\omega,\wt \kappa_u}$. We first observe that $\|\M_{u,u}^{-1}\| \leq \la_{\wh \tS}\ii$. This implies that
\[
\|\M_{u,u} y_u\| \ge \|\M_{u,u}^{-1}\|^{-1}\|y_u\| \ge \la_{\wh \tS} \|y_u\|.
\]
Hence, using that $\|y_s\| \le \wt \kappa_u \|y_u\|$,
\begin{equation}
\label{bound:expansiominima}
\|(\M y )_u\| \ge \|\M_{u,u} y_u\| - \|\M_{u,s} y_s\| \ge ( \la_{\wh \tS}  - \OO(\de^{1/5}) \wt \kappa_u ) \|y_u\|.
\end{equation}
Taking into account the bound on $\wt \kappa_u$ given by~\eqref{bound:deconsacons}, this last inequality implies Item 1. However, this is the minimum expansion in the unstable directions. If $\|y_u\| = |y_1|$, the expansion is much larger, as follows from
\begin{multline}
\label{bound:expansiomaxima}
\|(\M y )_u\| \ge |(\M y)_1| = | \lambda y_1 + \eps_7 y_3 + \mu \eps_5 y_2 + \eps_{10} y_4 | \\ \ge
(\lambda - \OO(\de^{1/5})- \OO(\de^{1/5}) \wt \kappa_u) |y_1| = \lambda(1 - \OO(\de^{1/5})/\lambda- \OO(\de^{1/5}) \wt \kappa_u/\lambda)\|y_u\|.
\end{multline}
Also, if $\|y_u\| = |y_1|$ and $|y_1| \ge (\la_{\wh \tS}/\lambda) |y_3|$, we have that
\begin{equation}
\label{bound:expansioy1sinormayuy3peronomassa}
|(\M y)_1| = | \lambda y_1 + \eps_7 y_3 + \mu \eps_5 y_2 + \eps_{10} y_4 |  \ge
\lambda(1 - \OO(\de^{1/5})/\la_{\wh \tS}- \OO(\de^{1/5}) \wt \kappa_u) |y_1|.
\end{equation}
Now we prove Item 2. We first claim that, if $y \in \wt S^u_{\omega,\wt \kappa_u}$,
\begin{equation}
\label{bound:Msuyu}
\|\M_{s,u} y_u \| \le
\begin{cases}
\lambda (\OO(\de^{1/5})+ \OO(\tkk)) \|y_u\|, \quad & \text{if $\|y_u\| = |y_1|$}, \\
\lambda (\OO(\de^{1/5})+ \OO(\tkk)) |y_1|, \quad & \text{if $\|y_u\| = |y_3|$ and $|y_1| \ge (\la_{\wt \tS}/\lambda)|y_3|$}, \\
\la_{\wt \tS} (\OO(\de^{1/5})+ \OO(\tkk)) \|y_u\|, \quad & \text{if $\|y_u\| = |y_3|$ and $|y_1| \le (\la_{\wt \tS}/\lambda)|y_3|$}.
\end{cases}
\end{equation}
Indeed, if $\|y_u \| = |y_1|$, by the definition of $\M_{s,u}$ in~\eqref{def:Msuuuusss} and the fact that $\la_{\wt \tS}/ \lambda < 1$,
\[
\|\M_{s,u} y_u \| \le \lambda \OO(\de^{1/5})|y_1|+ \la_{\wt \tS} \OO(\tkk)|y_3|
 \le \lambda ( \OO(\de^{1/5}) +  \OO(\tkk)) \|y_u\|.
\]
In the case $\|y_u\| = |y_3|$ and $|y_1| \ge (\la_{\wt \tS}/\lambda)|y_3|$,  we have that
\[
\|\M_{s,u} y_u \| \le \lambda \OO(\de^{1/5})|y_1|+ \la_{\wt \tS} \OO(\tkk)|y_3|
 \le \lambda ( \OO(\de^{1/5}) +  \OO(\tkk))|y_1|.
\]
Finally, if $\|y_u\| = |y_3|$ and $|y_1| \le (\la_{\wt \tS}/\lambda)|y_3|$,
\[
\|\M_{s,u} y_u \| \le \lambda \OO(\de^{1/5})|y_1|+ \la_{\wt \tS} \OO(\tkk)|y_3|
 \le \la_{\wt \tS} ( \OO(\de^{1/5}) +  \OO(\tkk)) \|y_u\|,
    \]
which proves~\eqref{bound:Msuyu}.

Hence, if $y \in \wt S^u_{\omega,\wt \kappa_u}$ and  $\|y_u \| = |y_1|$, by the first inequality in~\eqref{bound:Msuyu}, using that
$\|\M_{s,s}\| \le \OO(\de^{1/5})+\OO(\tkk)$ and~\eqref{bound:expansiomaxima}, we have that
\[
\|(\M y)_s\| \le  \|\M_{s,u} y_u \| + \|\M_{s,s} y_s\| \le \lambda ( \OO(\de^{1/5}) +  \OO(\tkk)) \|y_u\| \le ( \OO(\de^{1/5}) +  \OO(\tkk)) \|(\M y)_u\|.
\]
In the case $\|y_u\| = |y_3|$ and $|y_1| \ge (\la_{\wt \tS}/\lambda)|y_3|$, by \eqref{bound:expansiominima}, the second inequality in~\eqref{bound:Msuyu} and~\eqref{bound:expansioy1sinormayuy3peronomassa},
\begin{multline*}
\|(\M y)_s\| \le   \|\M_{s,u} y_u \| + \|\M_{s,s} y_s\|  \le \lambda ( \OO(\de^{1/5}) +  \OO(\tkk)) |y_1| + ( \OO(\de^{1/5}) +  \OO(\tkk)) \|y_s\| \\
 \le ( \OO(\de^{1/5}) +  \OO(\tkk)) |(\M y)_1|+ ( \OO(\de^{1/5}) +  \OO(\tkk)) \|(\M y)_u\| \le ( \OO(\de^{1/5}) +  \OO(\tkk)) \|(\M y)_u\|.
\end{multline*}
Finally, in the case $\|y_u\| = |y_3|$ and $|y_1| \le (\la_{\wt \tS}/\lambda)|y_3|$, by the third inequality in~\eqref{bound:Msuyu} and~\eqref{bound:expansiominima},
\begin{multline*}
\|(\M y)_s\| \le   \|\M_{s,u} y_u \| + \|\M_{s,s} y_s\|  \le \la_{\wt \tS} (\OO(\de^{1/5})+ \OO(\tkk)) \|y_u\| + ( \OO(\de^{1/5}) +  \OO(\tkk)) \|y_s\| \\
  \le ( \OO(\de^{1/5}) +  \OO(\tkk)) \|(\M y)_u\|.
\end{multline*}
This proves 2. Then, taking $\beta = \OO(\de^{1/5}) +  \OO(\tkk)$ in~\eqref{bound:deconsaconsinvers},  the claim for the unstable cones follows.

The proof of the claim for the stable cones is completely analogous. It is only necessary to use~\eqref{eq:DPhi1inverseadapted basis_M} instead of~\eqref{eq:DPhi1adapted basis_M}. We simply emphasise that~\eqref{bound:deconsaconsinvers} is replaced by~\eqref{bound:deconsacons}.

\appendix

\section{Proof of Proposition \ref{prop:invariance}}
We devote this section to proof Proposition \ref{prop:invariance}. Separating the linear and non-linear terms, 
the invariance equation \eqref{eq:Invariance1} 
can be rewritten as
\begin{equation*}
\mathcal{L}Z= D_ZX_Z^0(x,0)Z+ F(x,Z)
\end{equation*}
where 
\[
F(x,Z) = X_Z^0(x,Z)-X_Z^0(x,0)-DX_Z^0(x,0)Z+ X_Z^1(x,Z)-DZ\left( X_x^0(x,Z)- \Omega+ X_x^1(x,Z)  \right)
\]
with $\Omega=(1,\nu G_0^3/L_0^3)^\top$.

Observe that, using the definition of $Q_0$ in \eqref{def:Q0} and defining $q=\Lambda-\alpha\beto-\alo\beta-\alpha\beta$, one has
\[
\begin{split}
\lefteqn{DZ\left( X_x^0(x,Z)- \Omega+ X_x^1(x,Z)\right)}\\
 &=
 \partial_u Z \left(\partial_Y Q_0(u,  Y,q)-1\right)
+ \partial_\lo Z \left(\frac{G_0^3\nu}{(L_0+\Lambda)^3}-\frac{G_0^3\nu}{L_0^3}
+ \partial_q Q_0(u,  Y,q)
+\partial_\La \PP_1(u,\lo,\La,\al,\bet)\right)\\
&= \partial_u Z
 \left(\frac{ Y}{G_0y_\h^2(u)}+f_1(u)q_1\right)+ \partial_\lo Z \left(\frac{G_0^3 \nu}{(L_0+\Lambda)^3}-\frac{G_0^3\nu}{L_0^3}+f_1(u)Y+f_2(u)q
 +\partial_\La \PP_1(u,\lo,\La,\al,\bet)\right).
\end{split}
\]
Therefore,
\begin{eqnarray*}
DZ\left( X_x^0(x,Z)- \Omega+ X_x^1(x,Z)\right) =
 \partial_u Z \GG _1 (u,\lo,Z)+ \partial_\lo Z\GG_2 (u,\lo,Z)
 \end{eqnarray*}
where $\GG_1$ and $\GG_2$ are the functions introduced in \eqref{def:G1G2}.
Moreover, $X^0_Z(x,0)=0$ and
\begin{equation*}
\begin{split}
A(u)&=DX^0_Z(u,\ga,0)=\left(\begin{array}{cccc}
          -\frac{\partial ^2 \PP_0}{\partial  Y \partial u} & -\frac{\partial ^2 \PP_0}{\partial \Lambda \partial u} & -
          \frac{\partial ^2 \PP_0}{\partial \alpha \partial u} & -\frac{\partial ^2 \PP_0}{\partial \beta \partial u} \\
          0&0&0&0\\
          -i \frac{\partial ^2 \PP_0}{\partial  Y \partial \beta} & -i\frac{\partial ^2 \PP_0}{\partial \Lambda \partial \beta} &
          -i\frac{\partial ^2 \PP_0}{\partial \alpha \partial \beta} & -i\frac{\partial ^2 \PP_0}{\partial \beta \partial \beta} \\
          i \frac{\partial ^2 \PP_0}{\partial  Y \partial \alpha} & i\frac{\partial ^2 \PP_0}{\partial \Lambda \partial \alpha} &
          i\frac{\partial ^2 \PP_0}{\partial \alpha \partial \alpha} & i\frac{\partial ^2 \PP_0}{\partial \beta \partial \alpha}
        \end{array}\right)(u,\gamma,0)
        =
\begin{pmatrix}0&0\\ \mathcal{A}(u)& \mathcal{B}(u) \end{pmatrix}.
\end{split}
\end{equation*}
We obtain the expression of $\AAA(u)$ and $\BB(u)$ using the formula for $\PP_0$ in \eqref{eq:P0} and \eqref{def:Q0}, \eqref{def:f1f2}.
%
Defining
\[
\QQ(Z)=X_Z^0(x,Z)-X_Z^0(x0)-AZ+X_Z^1(x,Z)
\]
we obtain the formulas \eqref{def:Q}.

\section{Estimates for the perturbing potential}\label{sec:Potential}
The goal of this appendix is to give estimates for the Fourier coefficents of the potential $\PP_1$ introduced in \eqref{def:P1}. This estimates are thoroughly used in the proof of Theorem \ref{thm:ExistenceManiFixedPt} given in Section \ref{sec:proofexistencemani} and in the analysis of the Melnikov potential given in Appendix \ref{sec:Melnikov}.

Using \eqref{def:potentialwtilde}, \eqref{def:rvInPoinc} and \eqref{def:ChangeThroughHomo}, the potential $\PP_1$  in \eqref{def:P1} satisfies
\begin{equation*}
\begin{array}{rcl}
\PP_1(u,\gamma,\Lambda, \alpha, \beta)&=&G_0^3\wt W\left(\gamma+\phi_\h(u), L_0+
\Lambda,e^{i\phi_\h(u)}(\alo+\alpha),
e^{-i\phi_\h(u)}(\beto+\beta), G_0^2\wh  r_\h(u)\right)
\\
&=&
\frac{\wt\nu G_0}{\wh r_\h(u)}\left(\frac{m_0}{\left|
1+\frac{\wt\sigma_0}{G_0^2}\frac{\wt \rr e^{iv}}{\wh r_\h(u)} \sqrt{\frac{\alo+\alpha}{\beto+\beta}} e^{i\phi_\h(u)}\right|}
+\frac{m_1}{\left|1-\frac{\wt\sigma_1}{G_0^2}\frac{\wt
\rr e^{iv}}{\wh r_\h(u)} \sqrt{\frac{\alo+\alpha}{\beto+\beta}} e^{i\phi_\h(u)}\right|}-(m_0+m_1)\right),\\
\end{array}
\end{equation*}
where the function $\wt \rr(\ell,L, \Gamma) e^{iv(\ell,L, \Gamma)}$ is evaluated at 
\[
\ell=\gamma-\frac{1}{2i}\log{\frac{\alo+\alpha}{\beto+\beta}}, \quad L=L_0+\Lambda,\quad \Gamma= L_0+\La-(\alo+\alpha)(\beto+\beta).
\]
By  \eqref{def:Delvg} and \eqref{def:Dellu},  it can be written also as function of $(\ell,L,\e )$, where $\e$ is the function introduced in
\eqref{def:DeleLGamma}. In  coordinates \eqref{def:ChangeThroughHomo}, it can be written as
\begin{equation}\label{eq:ecalbet}
\e=\mathcal{E}(\Lambda,\al,\bet)\sqrt{(\alo+\alpha)(\beto+ \beta)}\quad \text{where}\quad \mathcal{E}(\Lambda,\al,\bet)=\frac{\sqrt{2(L_0+\Lambda)-(\alo+\alpha)(\beto+ \beta)}}{L_0+\Lambda}.
\end{equation}

Next lemma gives some crucial information about the Fourier expansion of the perturbed Hamiltonian $\PP_1$ in the domains
$D^{u}_{\kk,\de}$, $D^{s}_{\kk,\de}$ in \eqref{def:DomainOuter}.

\begin{lemma}\label{lem:P1k}
Assume $|\alpha_0|<\zeta_0$,  $|\bet_0|<\zeta_0$, where $\zeta_0 G_0^{3/2}\ll 1$, $1/2 \le L_0 \le 2$.
Then, there exists $\sigma>0$ such that for $u\in D^{u}_{\kk,\de} \cup D^{s}_{\kk,\de}$,
$\gamma \in \TT_\sigma$, $|\La|\le 1/4$, $|\al|\le \zeta_0/4$,
$|\bet|\le \zeta_0/4$,the function $\PP_1$ can be written as
\[
\PP_1(u,\gamma, \Lambda,\alpha, \beta) = \sum_{q\in\ZZ}  \PP_1 ^{[q]}(u, \Lambda,\alpha, \beta) e^{i q\gamma}\qquad \text{where}\qquad  \PP_1 ^{[q]}(u, \Lambda,\alpha, \beta)= \wh \PP_1 ^{[q]}(u, \Lambda,\alpha, \beta) e^{i q\phi_\h(u)}
\]
for some  coefficients $\wh \PP_1 ^{[q]}$ satisfying
\begin{equation}\label{eq:boundsP1q}
\begin{aligned}
 \left|\wh \PP_1^{[q]}\right|  & \le \frac{K}{G_0^3 |\wh r_\h(u)|^3}e^{-|q|\sigma}&
 \left|\partial_u \wh \PP_1^{[q]}\right|  & \le \frac{K}{G_0^3 |\wh r_\h(u)|^4}|\wh y_\h(u)|e^{-|q|\sigma}\\
 \left|\partial_\gamma \wh \PP_1^{[q]}\right|  & \le \frac{K}{G_0^3 |\wh r_\h(u)|^3}e^{-|q|\sigma}&
  \left|\partial_\La \wh \PP_1^{[q]}\right|  & \le \frac{K}{G_0^3 |\wh r_\h(u)|^3}e^{-|q|\sigma}\\
 \left|e^{-i\phi_\h(u)}\partial_\al \wh \PP_1^{[q]}\right|  & \le \frac{K}{ G_0^3 |\wh r_\h(u)|^3}e^{-|q|\sigma}&
 \left|e^{i\phi_\h(u)}\partial_\bet \wh \PP_1^{[q]}\right|  & \le \frac{K}{ G_0^3 |\wh r_\h(u)|^3}e^{-|q|\sigma}  \end{aligned}.
\end{equation}
\end{lemma}
We devote the rest of this appendix to prove this lemma.
\begin{proof}[Proof of Lemma \ref{lem:P1k}]
To estimate the Fourier coefficients of $\PP_1$ it is convenient to analyze first the Newtonian potential in Delaunay coordinates. Indeed, we consider the potential the potential $\wt W$ in Delaunay coordinates, that is 
\[
 V( \ell, L,\phi, \Ga,\wt r)=\wt W\left( \ell+\phi, L,\sqrt{L-\Ga}e^{i\phi}, \sqrt{L-\Ga}e^{-i\phi},\wt r\right)
\]
which reads
\begin{equation}\label{def:potentialDelaunay}
V( \ell, L,\phi, \Ga,\wt r)=\frac{\wt\nu}{\wt r}\left(\frac{m_0}{|
1+\wt\sigma_0 n e^{i\phi}|}+\frac{m_1}{|1-\wt\sigma_1 n e^{i\phi}|}-(m_0+m_1)\right),
\end{equation}
where $n=n(\ell,L,\phi,\Gamma,\wt r)=\frac{\wt \rr}{\wt r} e^{iv}$.

This potential can be rewritten as $V(\ell,L,\phi,\Gamma, \wt r) = \sum_{q\in\ZZ} V ^{[q]} e^{i q\ell}$ with
\[
\begin{split}
V^{[q]}( L,\phi, \Ga,\wt r)&=\frac{1}{2\pi}\int_0^{2\pi} V(\ell, L,\phi, \Ga, \wt r) e^{-i q\ell} d\ell\\
&= \frac{\wt\nu}{2\pi\wt r} \int_0^{2\pi} \left(\frac{m_0}{|
1+\wt\sigma_0 n e^{i\phi}|}+\frac{m_1}{|1-\wt\sigma_1 n e^{i\phi}|}-(m_0+m_1)\right) e^{-i q\ell} d\ell.
\end{split}
\]
Following \cite{DelshamsKRS19}, to estimate these integrals, we perform the change to the excentric anomaly
\begin{equation}\label{def:eccentric}
\ell =E-\e\sin E, \quad\ 
d\ell=(1-\e\cos E)dE
\end{equation}
and use that
\begin{equation}\label{def:a}
\wt \rr e^{iv}= L^2\left( a^2 e^{iE}-\e+\frac{\e^2}{4a^2}e^{-iE}\right),\quad \ a=\frac{\sqrt{1+\e}+\sqrt{1-\e}}{2}
\end{equation}
where $\e=\frac{1}{L}\sqrt{L^2-\Gamma^2}$. Then we do a second change of variables $E+\phi=\s$ to obtain
\[
 \begin{split}
V^{[q]} & = \frac{\wt\nu}{\wt r} \int_0^{2\pi} \left(\frac{m_0}{|
1+\wt\sigma_0 \wt n e^{i\phi}|}+\frac{m_1}{|1-\wt\sigma_1 \wt n e^{i\phi}|}-(m_0+m_1)\right)
e^{-i q(\s-\phi-\e\sin (\s-\phi))}(1-\e\cos (\s-\phi)) d\s \\
&=e^{i q\phi}  \frac{\wt\nu}{\wt r} \int_0^{2\pi} \left(\frac{m_0}{|
1+\wt\sigma_0 \wt n e^{i\phi}|}+\frac{m_1}{|1-\wt\sigma_1 \wt n e^{i\phi}|}-(m_0+m_1)\right)
e^{-i q(\s-\e\sin (\s-\phi))}(1-\e\cos (\s-\phi)) d\s \\
& = e^{i q\phi} \wh V^{[q]}( L,\phi, \Ga,\wt r)
\end{split}
\]
where
\[
\wt n(\s,L,\phi,\Gamma,\wt r)e^{i\phi}= n(\s-\phi-e\sin (\s-\phi),L,\phi,\Gamma,\wt r)e^{i\phi}=\frac{1}{\wt r}
L^2\left(a^2e^{i\s}-\e e^{i\phi}+\frac{\e^2}{4a^2}e^{-i(\s-2\phi)}\right).
\]
Now we relate these Fourier coefficients to those of $\PP_1$. To this end we relate Delaunay coordinates to the coordinates  introduced in \eqref{def:ChangeThroughHomo}. One can see that the
%
\[
(\ell,L, \phi, \Gamma,\wt r, \wt y)\to (u, Y,\gamma,\La, \alpha,\beta)
\]
is given by
\begin{equation}\label{def:DelaunayToAdapted}
\begin{aligned}
\ell &= \gamma-\frac{1}{2i}\log{\frac{\alo+\alpha}{\beto+\beta}} &\quad L &=\, L_0+\Lambda\\
\phi &=\, \phi_\h(u)+\frac{1}{2i}\log{\frac{\alo+\alpha}{\beto+\beta}} &\quad \Gamma &=\, L-(\alo+\alpha)(\beto+\beta)\\
\wt r &=\, G_0^2\wh r_\h(u) & \quad\wh y&=\,\frac{\wh y_\h(u)}{G_0}+\frac{ Y}{G_0^2\wh y_\h(u)}+\frac{\Lambda-(\alo+\alpha)(\beto+\beta)+\alo\beto}{G_0^2\wh y_\h(u)(\wh r_\h(u))^2}.
\end{aligned}
\end{equation}
Then,
\[
 \begin{split}
\lefteqn{\PP_1(u,\gamma,\La,\alpha, \beta)}\\
&=G_0^3 V\left(\gamma-\frac{1}{2i}\log{\frac{\alo+\alpha}{\beto+\beta}} ,L_0+\La,\phi_\h(u)+\frac{1}{2i}\log{\frac{\alo+\alpha}{\beto+\beta}} ,  L_0+\La-(\alo+\al)(\beto+ \bet),G_0^2\wh r_\h(u)\right)\\
&=\sum_{q\in\ZZ}
e^{iq( \phi_\h(u)+\frac{1}{2i}\log{\frac{\alo+\alpha}{\beto+\beta}} )}
e^{iq(\gamma-\frac{1}{2i}\log{\frac{\alo+\alpha}{\beto+\beta}} )} \wh \PP_1^{[q]}(u,\La,\alpha, \beta)\\
&=\sum_{q\in\ZZ} e^{iq( \phi_\h(u)+\gamma)} \wh \PP_1^{[q]}(u,\La,\alpha, \beta)
\end{split}
\]
where
\[
\begin{split}
& \wh \PP_1^{[q]}(u,\La,\alpha, \beta) =
G_0^3\wh V^{[q]}\left(L_0+\La,\phi_\h(u)+\frac{1}{2i}\log{\frac{\alo+\alpha}{\beto+\beta}},L_0+\La-(\alo+\al)( \beto+\bet),G_0^2\wh r_\h(u)\right) \\
& =\frac{\wt\nu G_0}{2\pi\wh r_\h(u)} \int_0^{2\pi} \left(\frac{m_0}{\left|
1+\wt\sigma_0 \wh n e^{i(\phi_\h(u)+\frac{1}{2i}\log{\frac{\alo+\alpha}{\beto+\beta}} )}\right|}
+\frac{m_1}{\left|1-\wt\sigma_1 \wh n e^{i(\phi_\h(u)+\frac{1}{2i}\log{\frac{\alo+\alpha}{\beto+\beta}} )}\right|}-(m_0+m_1)\right) \\
& e^{-i q(\s-\e\sin (\s-(\phi_\h(u)+\frac{1}{2i}\log{\frac{\alo+\alpha}{\beto+\beta}} )))}\left(1-\e
\cos \left(\s-\left(\phi_\h(u)+\frac{1}{2i}\log{\frac{\alo+\alpha}{\beto+\beta}} \right)\right)\right) d\s
\end{split}
\]
where now
\[
\wh n e^{i(\phi_\h(u)+\frac{1}{2i}\log{\frac{\alo+\alpha}{\beto+\beta}} )}=
\frac{1}{G_0^2 \wh r_\h(u)}
(L_0+\Lambda)^2\left(a^2 e^{i\s}-\e e^{i(\phi_\h(u)+\frac{1}{2i}\log{\frac{\alo+\alpha}{\beto+\beta}})}
+\frac{\e^2}{4a^2}e^{-i(\s-2(\phi_\h(u)+\frac{1}{2i}\log{\frac{\alo+\alpha}{\beto+\beta}}) )}\right).
\]
The first important observation is that, using the expresion for the eccentricity in \eqref{eq:ecalbet} one has
\[
 \e e^{\frac{1}{2}\log{\frac{\alo+\alpha}{\beto+\beta}}}= (\alo+\al) \mathcal{E},\qquad 
 \e e^{-\frac{1}{2}\log{\frac{\alo+\alpha}{\beto+\beta}}}= (\beto+\bet) \mathcal{E}
 \]
which implies
\begin{align*}
  \e\sin (\s-(\phi_\h(u)+\frac{1}{2i}\log{\frac{\alo+\alpha}{\beto+\beta}})))&=\frac{\mathcal{E}}{2i}\left(
(\beto+\bet) e^{i(s-\phi_\h(u))}-(\alo+\al) e^{-i(s-\phi_\h(u))} \right)\\
 \e\cos (\s-(\phi_\h(u)+\frac{1}{2i}\log{\frac{\alo+\alpha}{\beto+\beta}})))&=\frac{\mathcal{E}}{2}\left(
(\beto+\bet) e^{i(s-\phi_\h(u))}+(\alo+\al) e^{-i(s-\phi_\h(u))} \right)\\
\end{align*}
and 
\[
\wh n e^{i(\phi_\h(u)+\frac{1}{2i}\log{\frac{\alo+\alpha}{\beto+\beta}})}=
\frac{(L_0+\Lambda)^2}{G_0^2 \wh r_\h(u)}
\left(a^2 e^{i\s}-
(\alo+\al) \mathcal{E}e^{i\phi_\h(u)}
+\mathcal{E}^2\frac{(\alo+\al)^2}{4a^2}e^{-i(\s-2\phi_\h(u))}\right).
\]
Now we observe that, taking into account that $1/2\le a \le 2$, the asymptotics provided by Lemma \ref{lemma:homounperturbed} for $\wh r_\h$ and $\phi_\h$, and the fact that $|\alo+\al|+|\beto+\bet| \le \zeta_0$ imply
\[
\left|(|\alo+\al|+|\beto+\bet|) \mathcal{E}e^{\pm i\phi_\h(u)}\right| \le K \zeta_0 G_0^{3/2} \ll 1
\]
we have that
\[
\left|\wh n e^{i(\phi_\h(u)+\frac{1}{2i}\log{\frac{\alo+\alpha}{\beto+\beta}})}\right| \le \frac{K}{G_0^2 |\wh r_\h(u)|}\le  G_0^{-1/2}\le  \frac{1}{2}
\]
under the hypotheses of the lemma.
Therefore we have that, using the cancellations (observe that $\wt \sigma_0 m_0-\wt \sigma_1 m_1=0$),
\[
\left|\wh \PP_1^{[q]} (u,\La,\alpha,\beta)
\right| \le K
\frac{1}{G_0^3 |\wh r_\h(u)|^3}.
\]
The bounds for the derivatives can  be made analogously differentiating the expressions for $\PP_1$.
\end{proof}

\section{The Melnikov potential: Proof of Proposition \ref{prop:MelnikovPotential}}\label{sec:Melnikov}
We devote this section to prove Proposition \ref{prop:MelnikovPotential} which gives estimates for the Melnikov potential $\LL$ introduced in \eqref{def:MelnikovPotential}. Note that $\LL$ can be rewritten in terms of the perturbing potential $\PP_1$ introduced in \eqref{def:P1} as
\[
\LL(\vm,\alo,\beto)=\int_{-\infty}^{+\infty} \PP_1(s, \vm+\omega s,0,0,0)\ ds\qquad \text{where}\qquad \omega=\frac{\nu G_0^3}{L_0^{3}}
\]
(see \eqref{eq:omega}).

%
First, we obtain estimates for the harmonics different from $0, \pm 1$ of the potential $\LL$.
\begin{lemma}\label{lemma:Melnikov:qlarge}
The $q$-Fourier coefficient of the Melnikov potential \eqref{def:MelnikovPotential} with $|q|\geq 2$ can be bounded as
\begin{equation}\label{boundLq}
\left|\LL^{[q]}\right|\le K^q G_0^{\frac{3|q|+1}{2}}\ex^{-\frac{|q|\tilde \nu G_0^3}{3 L_0^3}}
\end{equation}
for some $K>0$ independent of $G_0$ and $q$.
\end{lemma}
\begin{proof}
The proof of this lemma is straighforward if one writes the Fourier coefficients of $\LL$ in terms of the Fourier coefficients of the perturbing Hamiltonian $\PP_1$ introduced in \eqref{def:P1} (which can be also expressed in terms of the function $\wh \PP_1^{[q]}$ introduced in Lemma  \ref{lem:P1k}) as
\[
\LL^{[q]}=\int_{-\infty}^{+\infty}\PP_1^{[q]}(u,0,0,0)\ex^{i\frac{q\tilde \nu G_0^3}{ L_0^3}u}du=
\int_{-\infty}^{+\infty}\wh \PP_1^{[q]}(u,0,0,0) \ex^{iq\phi_\h(u)} \ex^{i\frac{q\tilde \nu G_0^3}{ L_0^3}u}du.
\]
Then, it is enough to change the path of the integral to $\Im u=\frac{1}{3}-  G_0^3$ and use the bounds of Lemma
\ref{lem:P1k} and that, by Lemma \ref{lemma:homounperturbed},
\[
|\ex^{q i \phi_\h(u)}|\le K^q G_0^{\frac{3|q|}{2}}.
\]
\end{proof}
Now we give asymptotic formulas for the harmonics $q=0,\pm 1$. It is easy to check that
\[
\LL^{[q]}(\alo,\beto)=\ol{\LL^{[-q]}}(\beto,\alo)
\]
and therefore it is enough to compute $q=0,1$.

\begin{lemma}\label{lemma:Melnikov:q01}
The Fourier coefficients  $\LL^{[0]}$ and $\LL^{[1]}$ satisfy
\[
\begin{split}
\LL^{[0]}(\alo,\beto) =&\,
\frac{\tilde \nu \pi}{8} L_0^4 G_0^{-3} \Bigg[ N_2\left(1 +\frac{3}{L_0}\alo\beto-\frac{3}{2L_0^2}(\alo\beto)^2\right)\\
&-\frac{15 N_3  L_0^2}{8\sqrt{2L_0}}  G_0^{-2}(\alo+\beto)+\alo\beto\OO_1\left(\alo,\beto\right)\Bigg]\\
 \LL^{[1]}(\alo,\beto) =&\,
  \frac{\tilde\nu\sqrt{\pi}}{4}\ex^{-\frac{\tilde \nu G_0^3}{3 L_0^3}}L_0^4G_0^{3/2}\left[
 \left( \frac{N_3L_0^2}{ 8\sqrt{2}}G_0^{-2}-\frac{3N_2}{\sqrt{L_0}}\alo
 \right)+
   \OO\left( G_0^{-5/2}, G_0^{-3/2}\alo, \alo\beto, \alo^2 G_0, \beto^2 G_0\right)\right]
\end{split}
\]
where $N_2$ and $N_3$ are given in~\eqref{def:massterm}
\end{lemma}
\begin{proof}
Proceeding as in the proof of Lemma \ref{lem:P1k}, we use Delaunay coordinates. To this end, we first compute expansions of the potential $V$ introduced in  \eqref{def:potentialDelaunay} in powers of $\wt r$ as $ V= V_{1}+V_{2}+ V_{\geq}$ with
\[
 \begin{split}
V_{1}( \ell, L,\phi, \Ga,\wt r)&=N_2\wt \nu\frac{\wt\rr^2}{4\wt r^3}\left(3\cos 2(v+\phi)+1\right), \\
V_{2}( \ell, L,\phi, \Ga,\wt r)&=-N_3\wt \nu\frac{\wt\rr^3}{8\wt r^4}\left(3\cos(v+\phi)+5\cos3 (v+\phi)\right),
 \end{split}
\]
(see~\eqref{def:massterm}) and $V_\geq$ is of the form $V_\geq=\frac{1}{\wh r}E$ and $E$ is a function of $z=\frac{1}{\wt r}\wt \rr e^{i(v+\phi)}$ of order 4.

Accordingly,  we write the potential $\PP_1(u,\ga,\La,\al,\bet)$ in \eqref{def:P1} as
\[
 \PP_1= \PP_{11}+ \PP_{12}+ \PP_{1\geq}.
\]
Now we compute formulas for the Fourier coefficents of $\PP_1^{[q]}$ with $q=0,1$. For the coefficients $\PP_{1\geq}^{[q]}$, proceeding as in Lemma \ref{lem:P1k}, one can prove that
\begin{equation}\label{def:P1qhat}
\PP^{[q]}_{1\geq} (u,0,0,0)=\wh\PP^{[q]}_{1\geq}(u,0,0,0) e^{iq\phi_\h(u)}\qquad\text{with}\qquad\left|\wh\PP^{[q]}_{1\geq}(u,0,0,0)\right|\lesssim \frac{1}{G_0^7|\wh r^5_\h(u)|}.
\end{equation}
For $\PP_{11}$ and $\PP_{12}$ we have explicit formulas. To compute their Fourier coefficients, we introduce the coefficients $C_q^{n,m}$, defined, following   \cite{DelshamsKRS19},  by
\begin{equation}\label{def:c}
\wt \rr^n(\ell,L,\Gamma) \ex^{imv(\ell,L,\Gamma)}= \sum_{q\in\ZZ} C_q^{n,m}(L,\e)\ex^{iq\ell}=\sum_{q\in\ZZ}\left(\frac{\beto}{\alo}\right)^{q/2} C_q^{n,m}(L_0,\e)\ex^{iq\ga}
\end{equation}
where
\[
 \ell=\ga-\frac{1}{2i}\log\frac{\eta_0}{\xi_0},\quad L=L_0,\quad \Ga=L_0-\eta_0\xi_0 \qquad \text{(see \eqref{def:DelaunayToAdapted})}.
\]
The coefficients $C_q^{n,m}$ depend on $L$ and  $\e$ is the eccentricity (see \eqref{eq:ecalbet}).

Then, recalling also $\phi=\phi_\h(u)+\frac{1}{2i}\ln\left(\frac{\alo}{\beto}\right)$, we obtain
\[
 \begin{split}
  \PP_{11}^{[q]}(u,0,0,0)=&\,\frac{N_2\wt\nu}{4G_0^3\wh r^3_\h(u)}\left(C_q^{2,0}+\frac{3}{2} C_q^{2,2} e^{2i\phi}+\frac{3}{2} C_q^{2,-2} e^{-2i\phi}\right)\left(\frac{\beto}{\alo}\right)^{q/2}\\
=&\,\frac{N_2\wt\nu}{4G_0^3\wh r^3_\h(u)}\left(C_q^{2,0}+\frac{3}{2} C_q^{2,2} e^{2i\phi_\h(u)}\left(\frac{\beto}{\alo}\right)\ii+\frac{3}{2} C_q^{2,-2} e^{-2i\phi_\h(u)}\left(\frac{\beto}{\alo}\right)\right)\left(\frac{\beto}{\alo}\right)^{q/2}\\
  \PP_{12}^{[q]}(u,0,0,0)=&\,-\frac{N_3\wt\nu}{8G_0^5\wh r^4_\h(u)}\left(\frac{3}{2}C_q^{3,1} e^{i\phi}+\frac{3}{2}C_q^{3,-1} e^{-i\phi}+\frac{5}{2}C_q^{3,3} e^{3i\phi}+\frac{5}{2}C_q^{3,-3} e^{-3i\phi}\right)\left(\frac{\beto}{\alo}\right)^{q/2}\\
=&\,-\frac{N_3\wt\nu}{8G_0^5\wh r^4_\h(u)}\Bigg(\frac{3}{2}C_q^{3,1} e^{i\phi_\h(u)}\left(\frac{\beto}{\alo}\right)^{-1/2}+\frac{3}{2}C_q^{3,-1} e^{-i\phi_\h(u)}\left(\frac{\beto}{\alo}\right)^{1/2}\\
  &\,+\frac{5}{2}C_q^{3,3} e^{3i\phi_\h(u)}\left(\frac{\beto}{\alo}\right)^{-3/2}+\frac{5}{2}C_q^{3,-3} e^{-3i\phi_\h(u)}\left(\frac{\beto}{\alo}\right)^{3/2}\Bigg)\left(\frac{\beto}{\alo}\right)^{q/2}
  \end{split}
\]
where $N_2$ and $N_3$ are given by in \eqref{def:massterm}.

The functions $\LL^{[q]}$, $q=0,1$, can be written as $\LL^{[q]}=\LL_1^{[q]}+\LL^{[q]}_2+\LL^{[q]}_\geq$ with
\[
 \LL_i^{[q]}=\int_{-\infty}^{+\infty}\PP_{1i}^{[q]}(s,0,0,0)e^{iq\omega s}ds,\qquad i=1,2,\geq.
\]
We first give estimates for the last terms. For $\LL^{[1]}_\geq$ it is enough to change the path of integration to $\Im s=1/3-G_0^{-3}$ and use \eqref{def:P1qhat} and the properties in Lemma \ref{lemma:homounperturbed}. For $\LL^{[0]}_\geq$ one can estimate the integral directly. Then, one can obtain
\begin{equation}\label{def:FitaL1>}
\left|\LL^{[0]}_\geq\right|\lesssim \,G_0^{-7}\qquad\text{and}\qquad\left|\LL^{[1]}_\geq\right|\lesssim G_0^{-1}\ex^{-\frac{\tilde \nu G_0^3}{3 L_0^3}}.
\end{equation}

Thus, it only remains to obtain formulas for $\LL_i^{[q]}$, $q=0,1$, $i=1,2$. Using the formulas in \eqref{def:homoclinic}, they are given in terms of the integrals
\[
\begin{aligned}
\int_{-\infty}^{+\infty}\frac{1}{\wh r_\h^{2l-k+1}(t)}\ex ^{-i k\phi_\h(t)}  \ex^{iq\omega t} dt&=&
\int_{-\infty}^{+\infty}\frac{\ex^{iq\omega (\tau+\frac{\tau ^3}{3})}}{(\tau -i)^{2l-2k}(\tau+i)^{2l}}d\tau    &\qquad\mbox{for}\ l \ge k \ge 0\\
\int_{-\infty}^{+\infty}\frac{1}{\wh r_\h^{k-2l+1}(t)}\ex ^{-i k\phi_\h(t)}  \ex^{iq\omega t} dt&=&
\int_{-\infty}^{+\infty}\frac{\ex^{iq\omega (\tau+\frac{\tau ^3}{3})}}{(\tau -i)^{-2l}(\tau+i)^{2k-2l}}d\tau &\qquad\mbox{for}\ l\le k\le -1
\end{aligned}
\]
Therefore, following the notation of \cite{DelshamsKRS19}, if we introduce
\begin{equation}\label{eq:Nmn}
 N(q,m,n)=\frac{2^{m+n}}{G_0^{2m+2n-1}}{ -1/2\choose m}{-1/2\choose n}
\int_{-\infty}^{+\infty}\frac{e^{iq\omega (\tau+\frac{\tau ^3}{3})}}{(\tau -i)^{2m}(\tau+i)^{2n}}d\tau,
\end{equation}
one can write these functions as
\begin{equation}\label{def:MelnikovImportant}
\begin{split}
\LL_1^{[0]}=&\frac{N_2\wt\nu}{4}\left(C_0^{2,0}N(0,1,1)+ \left(\frac{\beto}{\alo}\right)\ii C_0^{2,2}N(0,2,0) +\left(\frac{\beto}{\alo}\right) C_0^{2,-2}N(0,0,2)\right)\\
\LL_1^{[1]}=&\frac{N_2\wt\nu}{4}\left(\left(\frac{\beto}{\alo}\right)^{1/2}C_1^{2,0}N(1,1,1)+\left(\frac{\beto}{\alo}\right)^{-1/2} C_1^{2,2}N(1,2,0) +\left(\frac{\beto}{\alo}\right)^{3/2} C_1^{2,-2}N(1,0,2)\right)\\
\LL_2^{[0]}=&\frac{N_3\wt\nu}{8}\Bigg(\left(\frac{\beto}{\alo}\right)^{-1/2}C_0^{3,1}N(0,2,1)+\left(\frac{\beto}{\alo}\right)^{1/2}C_0^{3,-1}N(0,1,2)\\
&+\left(\frac{\beto}{\alo}\right)^{-3/2}C_0^{3,3}N(0,3,0)+\left(\frac{\beto}{\alo}\right)^{3/2}C_0^{3,-3} N(0,3,0)\Bigg)\\
\LL_2^{[1]}=&\frac{N_3\wt\nu}{8}\Bigg(C_1^{3,1}N(1,2,1)+\left(\frac{\beto}{\alo}\right)C_1^{3,-1}N(1,1,2)\\
&\,+\left(\frac{\beto}{\alo}\right)^{-1}C_1^{3,3}N(1,3,0)+\left(\frac{\beto}{\alo}\right)^{2}C_1^{3,-3} N(1,0,3)\Bigg)
\end{split}
\end{equation}
It can be easly check that the functions $N$ satisfy
\[
N(0,0,k)=\,N(0,k,0)=0, \quad\mbox{for}\quad k\ge 2 \qquad \text{ and }\qquad N(-q,m,n)= N(q,n,m).
\]
In \cite{DelshamsKRS19}, the leading terms of these integrals are computed (see Lemma 30 and the proof of Lemma 36),
\[
\begin{split}
N(0,1,1)=&\,\frac{\pi}{2} G_0^{-3}\\
N(0,2,1)=&\,N(0,1,2)=\frac{3 \pi}{8} G_0^{-5}\\
N(1,2,1) =&\, \frac{1}{4}\sqrt{\frac{\pi}{2}}G_0^{-\frac12}\ex^{-\frac{\tilde \nu G_0^3}{3 L_0^3}}\left(1+\OO(G_0^{-3/2})\right)\\
N(1,2,0) =&\, \sqrt{\frac{\pi}{2}}G_0^{\frac32}\ex^{-\frac{\tilde \nu G_0^3}{3 L_0^3}}\left(1+\OO(G_0^{-3/2})\right)\\
N(1,3,0) =&\, \frac{1}{3}\sqrt{\frac{\pi}{2}}G_0^{\frac52}\ex^{-\frac{\tilde \nu G_0^3}{3 L_0^3}}\left(1+\OO(G_0^{-3/2})\right)
\end{split}
\]
and
\[
\begin{aligned}
N(1,1,1) =&\,\ex^{-\frac{\tilde \nu G_0^3}{3 L_0^3}}\OO\left(G_0^{-3/2}\right),&\quad N(1,0,2) =&\,\ex^{-\frac{\tilde \nu G_0^3}{3 L_0^3}}\OO\left(G_0^{-9/2}\right)\\
N(1,1,2) =&\,\ex^{-\frac{\tilde \nu G_0^3}{3 L_0^3}}\OO\left(G_0^{-7/2}\right),&\quad N(1,0,3) =&\,\ex^{-\frac{\tilde \nu G_0^3}{3 L_0^3}}\OO\left(G_0^{-13/2}\right).
\end{aligned}
\]
Now it remains to compute/estimate some of the coefficients $C_q^{n,m}$ in \eqref{def:c}. Observe that these functions can be defined as 
\[
 C_q^{n,m}(L_0,e_c)=\frac{L_0^{2n}}{2\pi}\int_0^{2\pi}\wt\rr^n(\ell,L_0,e_c)e^{imv(\ell,L_0,e_c)}e^{-iq\ell}d\ell.
\]
Proceeding as in \eqref{def:eccentric}, we change variables to the eccentric anomaly
as
\[\ell =E-\e\sin E, \quad\ 
d\ell=(1-\e\cos E)dE.\]
Then, using \eqref{def:Delvg} and \eqref{def:Dellu} to express $\wt\rr$ and $v$ in terms of the eccentric anomaly, one obtains
\begin{equation}\label{def:CqnmMeanAnomaly}
 C_q^{n,m}(L_0,e_c)=\frac{L_0^{2n}}{2\pi}\int_0^{2\pi}\left(a^2 e^{iE}+\frac{e_c^2}{4a^2}e^{-iE}-e_c\right)^m\left(1-e_c\cos E\right)^{n+1-m} e^{-iq(E-\e\sin E)}dE.
\end{equation}
where $a$ has been defined in \eqref{def:a}.
Using this formulas easily imply  the symmetries
\[
  C_q^{n,m}(\e)=C_{-q}^{n,-m}(\e)=\ol{ C}_q^{n,m}(\e)
\qquad \text{and}\qquad C_q^{n,m}(\e)=(-1)^{q+m} C_q^{n,m}(-\e).
 \]
Moreover, they allow to compute them, as was already done in  \cite{DelshamsKRS19}, to obtain
\[
\begin{split}
C_0^{2,0}&=\, L_0 ^4\left(1+ \frac{3}{L_0}\alo\beto-\frac{3}{2L_0^2}(\alo\beto)^2\right) \\
C_0^{3,-1}\left(\frac{\beto}{\alo}\right)^{1/2}&=\,
-\frac{5}{\sqrt{2L_0}}L_0 ^6\left(\beto + \mathcal{O}(\beto^{2}\alo)\right)\\
C_0^{3,1}\left(\frac{\beto}{\alo}\right)^{-1/2}&=\,
-\frac{5}{\sqrt{2L_0}}L_0 ^6\left(\alo + \mathcal{O}(\alo^{2}\beto)\right)\\
C_1^{3,1}&=\,L_0^6\left(1+ \OO(\alo\beto)\right)\\
\left(\frac{\beto}{\alo}\right)^{-1/2} C_1^{2,2}&=\,-3L_0^4\sqrt{\frac{2}{L_0}}\alo+\OO\left(\alo^2\beto\right).
\end{split}
\]
They also can be easily bounded, switching the integration path in \eqref{def:CqnmMeanAnomaly} to either  $\Im E=\log \e$ (if $q-m>0$) or $\Im E=-\log \e$ (if $q-m<0$), as 
\[
 |C_q^{n,m}(\e)|\lesssim \e^{|m-q|}.
 \]
Using this estimate, one obtains
\[
\begin{aligned}
 \left(\frac{\beto}{\alo}\right)^{1/2}C_1^{2,0}=&\,\OO(\beto),&  \left(\frac{\beto}{\alo}\right)^{3/2} C_1^{2,-2}=&\,\OO\left(\beto^3\right), &  \left(-\frac{\beto}{\alo}\right)C_1^{3,-1}=&\,\OO\left(\beto^2\right) \\
\left(\frac{\beto}{\alo}\right)^{-1}C_1^{3,3}=&\, \OO\left(\alo^2\right),& \left(\frac{\beto}{\alo}\right)^{2}C_1^{3,-3}=&\, \OO\left(\beto^4\right).
\end{aligned}
\]
Then,
\[
\begin{split}
\LL_1^{[0]}=&\frac{N_2\wt\nu\pi}{8}L_0 ^4  G_0^{-3}\left(1+ \frac{3}{L_0}\alo\beto-\frac{3}{2L_0^2}(\alo\beto)^2\right)\\
\LL_1^{[1]}=&-\frac{N_2\wt\nu}{4}\sqrt{\frac{\pi}{L_0}}G_0^{\frac32}\ex^{-\frac{\tilde \nu G_0^3}{3 L_0^3}}3L_0^4\alo \left(1+\OO(\alo\beto,G_0^{-3/2})\right)\\
\LL_2^{[0]}=&-\frac{15N_3\wt\nu\pi L_0^6}{64\sqrt{2L_0}}G_0^{-5}\left(\alo+\beto+\alo\beto\OO_1(\alo,\beto)\right)\\
\LL_2^{[1]}=&\frac{N_3\wt\nu}{32}\sqrt{\frac{\pi}{2}}L_0^6G_0^{-\frac12}\ex^{-\frac{\tilde \nu G_0^3}{3 L_0^3}}\left(1+\OO\left(\alo\beto,\beto^2 G_0^{-3},G_0^{-3/2}, \alo^2 G_0^3\right)\right).
\end{split}
\]
These formulas and \eqref{def:FitaL1>} give the asymptotic expansions stated in Lemma \ref{lemma:Melnikov:q01}.
\end{proof}

To finish the proof of Proposition \ref{sec:Melnikov} it is enough to use the relation between $\Theta$, $\tte$ and $G_0$ given in \eqref{def:ThetaTilde} and \eqref{eq:omega}.

\bibliography{references}
\bibliographystyle{alpha}

\end{document}